\documentclass[11pt]{amsart}

\usepackage{amssymb}

\usepackage{enumitem}

\usepackage[utf8]{inputenc}
\usepackage{color}
\usepackage[colorlinks]{hyperref}

\usepackage{amsmath}
\usepackage{amsthm}
\usepackage{amsfonts}
\usepackage{subfigure}
\usepackage{graphicx,psfrag}
\usepackage{epstopdf}
\usepackage{esvect}
\usepackage{epigraph}
\usepackage[spanish, english]{babel}
\usepackage{tikz-cd}
\usepackage{mathrsfs}
\usepackage{mathtools}

\usepackage{import}
\usepackage{xifthen}
\usepackage{pdfpages}

\usepackage{comment}

\makeatletter
\@namedef{subjclassname@2020}{%
	\textup{2020} Mathematics Subject Classification}
\makeatother

\newtheorem{theorem}{Theorem}[section]

\newtheorem*{theorem*}{Theorem}
\newtheorem{teorema}{Theorem}
\newtheorem{corolario}[teorema]{Corollary}

\newtheorem{corollary}[theorem]{Corollary}
\newtheorem{lemma}[theorem]{Lemma}
\newtheorem{proposition}[theorem]{Proposition}

\newtheorem*{claim*}{Claim}
\newtheorem{definition*}{Definition}
\newtheorem{claim}[theorem]{Claim}

\DeclareFontFamily{U}{mathx}{}
\DeclareFontShape{U}{mathx}{m}{n}{<-> mathx10}{}
\DeclareSymbolFont{mathx}{U}{mathx}{m}{n}
\DeclareMathAccent{\widecheck}{0}{mathx}{"71}

\newcommand{\TT}{\textnormal{T}}
\newcommand{\RR}{\textnormal{R}}
\newcommand{\BB}{\textnormal{B}}
\newcommand{\LL}{\textnormal{L}}
\newcommand{\R}{\mathbb{R}}
\newcommand{\T}{\mathbb{T}}
\newcommand{\D}{\mathbb{D}}
\newcommand{\Z}{\mathbb{Z}}
\newcommand{\F}{\mathcal{F}}
\newcommand{\dH}{\mathrm{d}_{\mathrm{H}}}
\newcommand{\rar}{\rightarrow}
\newcommand{\uar}{\uparrow}
\newcommand{\lar}{\leftarrow}
\newcommand{\dar}{\downarrow}
\newcommand{\wh}{\widehat}
\newcommand{\tl}{\tilde}
\newcommand{\wt}{\widetilde}
\newcommand{\wc}{\widecheck}	
\newcommand{\dd}{\,\mathrm{d}}

\newcommand{\Ine}{\textnormal{Ine}}
\newcommand{\Ess}{\textnormal{Ess}}

\newcommand{\bck}{\color{black}}

\theoremstyle{definition}
\newtheorem{definition}[theorem]{Definition}
\newtheorem{remark}[theorem]{Remark}

\numberwithin{equation}{section}

\begin{document}

	\baselineskip=17pt
	
	
	\title[Chaotic conservative models for homeomorphisms of $\T^2$]{Fully chaotic conservative models for some torus homeomorphisms}
	
	\author[A. Garc\'ia-Sassi, F. A. Tal]{Alejo Garc\'ia-Sassi, Fábio Armando Tal}
	\address{Instituto de Matem\'atica y Estad\'istica\\ Facultad de Ingenier\'ia\\ Universidad de la Rep\'ublica\\
		Julio Herrera y Reissig 565\\
		11300 Montevideo, Uruguay.\\
		Instituto de Matemática e Estatísica \\ Universidade de São Paulo\\
		Rua do Matão 1010 \\ 05508-090 Butantã, São Paulo - SP, Brasil.}
	\email{alejog@fing.edu.uy}
	\address{Instituto de Matemática e Estatísica \\ Universidade de São Paulo\\
		Rua do Matão 1010 \\ 05508-090 Butantã, São Paulo - SP, Brasil.}
	\email{fabiotal@usp.br}

	\date{}

	\begin{abstract}
		We study homotopic-to-the-identity torus homeomorphisms whose rotation set has nonempty interior. We prove that any such map is monotonically semiconjugate to a homeomorphism that preserves the Lebesgue measure and that has the same rotation set. Furthermore, the dynamics of the quotient map has several interesting chaotic traits: for instance, it is topologically mixing, has a dense set of periodic points and is continuum-wise expansive. In particular, this shows that a convex compact set of $\R^2$ with nonempty interior is the rotation set of the lift of a homeomorphism of $\T^2$ if and only if it is the rotation set of the lift of a conservative homeomorphism.
	\end{abstract}
	
	
	
	\maketitle

	\setcounter{tocdepth}{2}
	\tableofcontents
	
	\section{Introduction}\label{sec:Intro}
	
	\subsection{Motivation} The study of arbitrary dynamical systems on a manifold is a very hard task, as there can be too much information and features for one to grasp. One of the most well-established ways to parse all that information is to try to relate a given dynamical system to some simpler model dynamics which still encompass central attributes from the original system, but where our understanding of the phenomena is more robust. This is typically done by finding semiconjugacies between the original system or some restriction of it, to certain factor dynamics. Information may be lost in this process, but several aspects of the original system become clearer in the study of these factors. 
	
	This strategy has so far been very successfully used in the study of one-dimensional dynamics. For instance, the Poincaré Rotation Theory provides simple models for homeomorphisms of the circle, and the Kneading Theory of Milnor and Thurston showed the existence of semiconjugacies from positive entropy multimodal maps of the interval, to simple piecewise-linear constant-slope models.
	
	The use of models by factors is also a prevalent strategy in the study of surface homeomorphisms. It started with the study of Smale's Horseshoes and their semiconjugacy with the shift in the symbolic space, one of the most well understood models in dynamics. Katok in \cite{katok} showed that, if a $\mathcal{C}^{1+\alpha}$ diffeomorphism $f$ of a surface has positive topological entropy, then there exists some invariant subset for a power $f^k$ of $f$ such that the restriction of $f^k$ to this invariant set is \textcolor{black}{topologically} semiconjugate to the Bernoulli shift. 
	
	Other deep interesting models are obtained when one considers the isotopy classes of maps in surfaces. Nielsen-Thurston Theory has classified all the possible isotopy classes, and it has been shown that maps which are isotopic to Linear Anosov diffeomorphisms of $\T^2$, are extensions of these (\cite{franks69}). For closed surfaces it is also known that if a map is isotopic to a pseudo-Anosov map, then there exists a semiconjugacy between some invariant closed set of the original dynamics to the pseudo-Anosov homeomorphism (\cite{handel85}). {\color{black}Therefore, when the induced action in the homotopy group is sufficiently complex, one has very good models to understand the dynamics. The original dynamics is at least as complicated as a classical model with the same induced action in the homotopy.} 
	
	But none of those models help in understanding the rotational behaviour of a dynamical system on a surface if it lies in the isotopy class of the identity. In this paper we attempt to provide models that preserve this specific feature, which has been increasingly relevant in Surface Dynamics. We will concentrate our work on the study of homeomorphisms of $\T^2$ which are homotopic to the identity, a set we denote by $\textnormal{Homeo}_0(\mathbb{T}^2)$, but we imagine most of what is done here can be extended in the correct circumstances to other closed surfaces. {\color{black}Our work will mostly deal with maps with complex rotational behaviour. For homeomorphisms of $\T^2$ with more restricted rotational behaviour, there have been nice results identifying when they could be modeled by either conjugating to an ergodic $\T^2$-translation (\cite{jager}) or by semiconjugating to an irrational rotation in the circle (\cite{TobiasFabio, Kocsardsemiconjugado}).}

 For this study, we will make use of Rotation Theory and the modern Equivariant Brouwer Theory techniques. We will postpone the description of the latter to the body of the text. Let us say a few words on Rotation Theory. 
	
	\bigskip 
	
	\subsection{Torus Rotation Theory} Given a surface $S$ with a nontrivial fundamental group and an isotopy $(f_t)_{t\in [0,1]}$ connecting the identity $f_0$ to a homeomorphism $f=f_1$ of $S$, one can ask how the trajectories of points by the isotopy \textit{loop} around the surface. For the torus case, the classical notion of rotation set was introduced in \cite{mz89}.

	\begin{definition}[{\color{black}\textbf{Rotation set}, \cite{mz89}}]
		Given $f$ in  $\textnormal{Homeo}_0(\mathbb{T}^2)$ and a lift $\wh f$ of $f$ to $\R^2$, one defines the \emph{rotation set} of $\wh f$ as
		$$ \rho(\wh f) := \Bigl\{ v \in \R^2 :  \exists \  \wh z_k \in \R^2, n_k \to +\infty \text{ such that } \lim \limits_{k \to \infty} \frac{\wh f^{n_k}(\wh z_k) - \wh z_k}{n_k} = v \Bigr\}.$$
	\end{definition} 
	
	{\color{black}In \cite{mz89}, this set is proved to be equal} to the set of mean rotations of $f$-invariant measures, that is, 
	\[\rho_{\textnormal{inv}}(\wh f) := \Bigl \{v \in \R^2 : \exists \mu \text{ which is } f- \text{invariant s.t.} \int \wh f(\wh z) - \wh z \dd \mu = v \Bigr\} \]
	
	In \cite{mz89} it is proven that $\rho(\wh f)$ is a compact convex subset of $\R^2$ (in particular it is connected), which is then either a point, a segment or a set with nonempty interior. Given {\color{black}that} every lift $\wh f$ commutes with the deck transformations, the rotation set of two different lifts will be the same up to integer translations. We will then say that the rotation set of $f$ has nonempty interior if this is true for the rotation set of a lift $\wh f$ of $f$. One notes that the notion of rotation sets was already introduced by Schwartzmann using invariant measures in \cite{schwartzmann}, and a topological version was later given by Fried in \cite{fried}. 
	
	Rotation sets have been a very effective tool in the study of torus homeomorphisms, as they can in some cases encode several dynamical properties. For instance, when the rotation set has nonempty interior, it is known that the homeomorphism has positive topological entropy, as proved in \cite{llibremackay}, using Nielsen-Thurston classification of surface homeomorphisms from \cite{thurston88}, after having punctured the surface on an $f$-invariant set and restricting the dynamics to the resulting surface. Franks shows the abundance of periodic orbits in \cite{franks89}, and ergodic measures realizing each rotation vector in the interior of the rotation set are found in \cite{mz}. New proofs of these results and several improvements are recently obtained using Forcing Theory by Le Calvez and the second author, found in \cite{lecalveztalforcing, lct2}, using Brouwer-Le Calvez's dynamically transverse decomposition built in \cite{lecalvezequivariante}. 
	
	Also, assuming some extra regularity, inside the class of $\mathcal{C}^{1+\alpha}$ diffeomorphisms, if $f$ has a lift $\wh f$ such that $\rho(\wh f)$ has the origin in its interior, then Addas-Zanata in \cite{addas14grilla} and \cite{addas15consecuencias} uses Pesin Theory (see for example \cite[Part I]{pollicott93pesin}) to find saddle periodic points with different rotation vectors such that, for any lift of {\color{black}such a point}, its unstable manifold intersects the stable manifolds of each of its integer translates in a topologically transverse way. If, furthermore, one assumes that $f$ is transitive, then one obtains that $\wh f$ (and therefore $f$ itself) must be topologically mixing.
	
	Given all {\color{black}these applications}, the study of the \emph{realization} of rotation sets, i.e. understanding which subsets of $\R^2$ are rotation sets of lifts of elements of $\textnormal{Homeo}_0(\mathbb{T}^2)$, became also a popular topic. One says that a subset $D$ of $\R^2$ is realized as a rotation set of a homeomorphism $f$ if $f$ has a lift whose rotation set is $D$.
	For the case in which it has nonempty interior, Kwapisz proved in \cite{kwapisz92} that every rational polygon (i.e. {\color{black}whose} vertices have rational coordinates) can be realized as the rotation set of a torus homeomorphism: Passeggi proves that rational polygons are yielded as rotation sets for a generic set of homeomorphisms (\cite{passeggi13}), and Guihéneuf and Koropecki prove (\cite{pakoro}) that this is a necessary condition for the rotation set to be stable under small perturbations. 
	
	Kwapisz found in 1995 the first example of a rotation set with infinitely many extremal points (\cite{kwapisz95}), and in 2017 Boyland, De Carvalho and Hall find a one-parameter family of examples, many of them also having this property (\cite{boyland_2015}). They also describe explicitly how the rotation set changes as the parameter changes. Also, a result by Béguin, Crovisier and Le Roux proves that every irrational vector can appear as the extremal point of a rotation set with nonempty interior (\cite{beguin06}). 
	
	There are examples of compact convex sets that have been proven to never be a rotation set of a torus homeomorphism, but only in a very specific case (\cite{lecalveztalforcing}). But it is not known if there exists a rotation set with uncountably many extremal points. In particular, one does not know if for instance the closed unit ball is realized as a rotation set.

	\subsection{Statement of results}
	
	Our main result shows that, whenever $f$ has a rotation set with nonempty interior, then it is always an extension of another homeomorphism with much better understood dynamics, and with the same rotation set. Let us first introduce some notions for the sake of precision. 
	
	\bigskip
	
	\paragraph{\textbf{Factors and extensions.}} Given two metric spaces $X$ and $Y$, and two continuous maps $f: X \to X$, $g:Y \to Y$, we will say that $g$ is a \emph{factor of $f$} (and equivalently, we will also say that $f$ is an \emph{extension of $g$}) if there exists a semiconjugacy $h$ from $f$ to $g$, that is, a continuous surjection $h: X \to Y$ such that $h \circ f = g \circ h$. A map $h: X \to Y$ is \textit{monotone} if the preimage of every point is connected. 
	
	\medskip

	\paragraph{\textbf{Diameter of continua.}} Let $\wh \pi:\R^2\to\T^2$ be the classical covering projection. 
	
	\begin{definition}\label{def:Inessential}
		Given an open set $U \subset \T^2$, we will say that $U$ is \textit{inessential} if every loop in $U$ is homotopically trivial in $\T^2$. Otherwise, we will say it is {\color{black}\emph{essential}}. 
		
		We will say an arbitrary set $X \subset \T^2$ is \textit{inessential} if it has an inessential neighbourhood. Otherwise we will say it is {\color{black}\emph{essential}}.

	\end{definition}
	
	Note that {\color{black}for an inessential} connected set $D\subset \T^2$, every lift of $D$ (which is a connected component of $\wh \pi^{-1}(D)$), is homeomorphic to $D$. {\color{black}Moreover,} a lift $\wh D$ of {\color{black}an inessential connected set $D \subset \T^2$} satisfies that $\wh D \cap \wh D+v = \varnothing$ for every non-null $v\in\mathbb{Z}^2$.
	
	\begin{definition}
		Given a continuum $K \subset \T^2$, we will define its lifted diameter $\textnormal{diam}(K)$ as the euclidean diameter of any connected component of its lifts $\wh K$ to the universal covering $\R^2$ of $\T^2$. 	
	\end{definition}
	Note that the diameter of $K$ in this case is finite if, and only if, it is inessential. Thus, for the context of inessential continua we will abuse notation and write \emph{diameter} for its lifted diameter. 
	
	\begin{definition}[\textbf{Dynamical diameter}]\label{def:DynamicalDiameter}
		Let $K \subset \T^2$ be a continuum, and $f: \T^2 \to \T^2$ a homeomorphism. We will define the \emph{dynamical diameter} $\mathscr{D}_f(K)$ as
		$$ \mathscr{D}_f(K) := \sup \limits_{j \in \mathbb{Z}} \{\textnormal{diam}(f^j(K))\} $$
	\end{definition}
	
	A continuum $K$ will be \textit{dynamically bounded} if $\mathscr{D}_f(K) < \infty$, and it will be \textit{dynamically unbounded} otherwise.
	
	A homeomorphism $h$ of a metric space is \emph{continuum-wise expansive} if there exists $\varepsilon>0$ such that, for any nontrivial continuum $K$, the supremum of the diameters of $h^{j}(K)$ with $j\in\mathbb{Z}$ is greater than $\varepsilon$. Continuum-wise expansive maps have several strong properties, see for instance \cite{jana04}.
	For $f$ in  $\textnormal{Homeo}_0(\mathbb{T}^2)$ we say that it is \emph{infinitely continuum-wise expansive} if, for any 	{\color{black}nontrivial} continuum $K$, we have that $\mathscr{D}_f(K) = \infty$. 
	
	\bigskip

	\paragraph{\textbf{Topological entropy}}
	
	Let us recall a definition of topological entropy by Bowen in \cite{bowen71}. Let $f$ be a uniformly continuous map of a metric space $(X,\mathrm{d})$. For every $n \in \Z^+$, we define the distances 
	\[\dd_n(x,x') = \sup \{\dd(f^j(x),f^j(x')) \ : 0 \leq j \leq n\},\]
	\[\dd_{\pm n}(x,x') = \sup \{\dd(f^j(x),f^j(x')) \ : -n \leq j \leq n\}\]
	We will say a subset $Y \subset X$ is $(n,\varepsilon)$-separated (similarly $(\pm n,\varepsilon)$-separated) with respect to $f$, if for every pair $y,y' \in Y$ we have that $\dd_n(y,y') > \varepsilon$ (similarly $\dd_{\pm n}(y,y') > \varepsilon$).

	{\color{black}For a continuum $K \subset X$, we} define $s_n(\varepsilon, K)$ as the maximal number of elements of an $(n,\varepsilon)$-separated subset of $K$, with respect to $f$, and define $s_{\pm n}(\varepsilon, K)$ in the same fashion. Let us define
	\begin{equation}\label{eq:ContinuumEntropy}
		h(f,K) = \lim_{\varepsilon \to 0} \  \Bigl (\limsup_{n \to +\infty} \frac{1}{n} \mathrm{log} \bigl ( s_n(\varepsilon, K) \bigr ) \Bigr )	
	\end{equation}
	We define the \textit{{\color{black}two-sided} (topological) entropy of $f$ carried by $K$} as 
	\begin{equation}\label{eq:TwoSidedEntropyK}
		h_{\pm}(f,K) = \lim_{\varepsilon \to 0} \  \Bigl (\limsup_{n \to + \infty} \frac{1}{n} \mathrm{log} \bigl ( s_{\pm n}(\varepsilon, K) \bigr ) \Bigr )
	\end{equation} 
	We will say that $K$ \textit{carries positive 	{\color{black}two-sided} entropy} whenever 	{\color{black}$h_{\pm}(f,K)$} is positive. We then define the \textit{topological entropy} $h(f)$ of $f$ as  
	\[h(f) = \sup_{K \text{compact}} h(f,K),\]
	{\color{black}and the \textit{two-sided topological entropy} $h_{\pm}(f)$ of $f$ as  
	\[h_{\pm}(f) = \sup_{K \text{compact}} h_{\pm}(f,K).\]
	}
	
	We will say $f$ is \textit{tight} if for every nontrivial continuum $K$, we have that $h_{\pm}(f,K)$ is positive. Note that this notion has been of interest during these last years: De Carvalho and Paternain define the 0-entropy quotient to recover tight dynamics in a more complex topological space in \cite{decarvalhopaternain02}, this result has been retaken for example in \cite{boylandcarvalhohall23} by Boyland, De Carvalho and Hall.

	\begin{teorema}\label{thmA:semiconjugation}
		Let $f \in \textnormal{Homeo}_0(\mathbb{T}^2)$ have a lift $\wh f$ such that $\rho(\wh f)$ has nonempty interior. Then, there exists a factor map $g \in \textnormal{Homeo}_0(\mathbb{T}^2)$ by a monotone semiconjugacy, having a lift $\wh g$ with the following properties:
		\begin{itemize}
			\item $\rho(\wh g) = \rho(\wh f)$.
			\item $g$ is area preserving.
			\item $g$ is topologically mixing, and if $\vec{0} \in \textnormal{int}(\rho(\wh g))$, $\wh g$ is topologically mixing.
			\item $g$ is infinitely continuum-wise expansive.
			\item $g$ is tight: for each nontrivial continuum $K \subset \T^2$, we have $h_{\pm}(g,K) > 0$.  
			\item For each open set $U \subset \T^2$, $g$ has a Markovian horseshoe $X \subset U$ (see Definition \ref{def:RotationalHorseshoe}). In particular, $\overline{\textnormal{Per}(g)} = \mathbb{T}^2$.
		\end{itemize}
	\end{teorema}

	Note that Theorem A, in addition to presenting dynamical consequences for homeomorphisms of $\T^2$, has a direct consequence for the study of the realization of rotation sets in the form of the following corollary:
	
	\begin{corolario}\label{cor:realizacao}
		A convex compact subset of $\R^2$ with nonempty interior is realized as the rotation set of a homeomorphism of $\T^2$ if {\color{black}and} only if it is also realized as the rotation set of an area-preserving homeomorphism of $\T^2$. 
	\end{corolario}
	In particular, this holds for the closed unit disk.

 {\color{black}It is not known if the same result as the corollary holds when the compact convex set has empty interior. {\color{black}Indeed, even the question of which line segments can be realized as the rotation set of a torus homeomorphism (tackled by the Franks-Misiurewicz conjecture in \cite{franksmisiurewicz}) is not completely settled. For instance, it is not known whether there exists an example in which the rotation set is a line segment with rational slope and no bi-rational points. On the other hand, it is known that a line segment with irrational slope and a bi-rational point in its relative interior, is not realized as the rotation set of any torus homeomorphism (\cite{lecalveztalforcing}).} 
 	
But every compact line segment $L\subset\R^2$ that is known to be realized as the rotation set of a homeomorphism of $\T^2$, is also realized as the rotation set of an area-preserving homeomorphism (the case of a point $p \in \R ^2$ being realized by a translation). Indeed, it is not hard to construct examples whenever $L$ has infinitely many bi-rational points. When $L$ has a single bi-rational point that is an extremal point, one can produce examples having an invariant measure with full support and no atoms by the composition of a translation $T$ and the time $t$ map of a flow $\Phi$ which commutes with $T$ and is a reparametrization of the flow obtained by a constant vector field with irrational slope. Oxtoby-Ulam's Theorem in \cite{oxtobyulam41}, then yields the desired example. Likewise, the only other known cases of a line segment being realized as a rotation set are those produced by A. Avila's technique, which are minimal and as such also could be assumed to be area-preserving, again by Oxtoby-Ulam's Theorem.}

	We remark that Theorem A for the case where $f$ is a $\mathcal{C}^{1+\alpha}$ diffeomorphism was announced by De Carvalho, Koropecki and the second author, in an unfinished pre-print. That work relied on Pesin's theory and the work of Addas-Zanata in \cite{addas14grilla}, and the main structure of the proof is similar to ours. There are some technical difficulties that appear in {\color{black}our} proof since we are working with {\color{black}$\mathcal{C}^0$} regularity, particularly in showing the existence of dense horseshoes, but the most crucial point is the proof of Theorem \ref{PropC:maximumdiameter} below. In any case, we outline here every step so that the work is self-contained.
	
	\begin{teorema}\label{PropC:maximumdiameter}
		Let $f \in \textnormal{Homeo}_0(\mathbb{T}^2)$ {\color{black}have a lift $\wh f$ such that $\rho(\wh f)$} has nonempty interior. 
		
		Then, there exists $M>0$ such that, if  {\color{black}$\wh K \subset \R^2$} is a continuum and $\mathrm{diam}({\color{black}\wh K})>M$, then {\color{black}either $\lim \limits_{j \to +\infty} \textnormal{diam}(\wh f^{j}(\wh K)) = +\infty$, or $\lim \limits_{j \to -\infty} \textnormal{diam}(\wh f^{j}(\wh K)) = +\infty$.}
	\end{teorema}
	The main improvement here is that we showed this result using equivariant Brouwer techniques and Forcing Theory, bypassing Pesin's Theory and the need for regularity.
	
	Note that if $f$ belongs to the isotopy class of a linear Anosov automorphism $A$, then any lift $\wh f$ to $\R^2$ has that $\dd(\wh f(\wh z) {\color{black},} A(\wh z))$ is uniformly bounded, and therefore we obtain that:
	
	\begin{remark}
		If we take $f \in \textnormal{Homeo}(\mathbb{T}^2)$ to be in the isotopy class of an Anosov linear automorphism, the thesis of Theorem \ref{PropC:maximumdiameter} still holds.  
	\end{remark}  
	
	This allows us to get a monotone semiconjugacy for this isotopy class, and obtain a factor $g$ which is again infinitely continuum-wise expansive, is area-preserving and has a dense set of topological horseshoes.   

	\textcolor{black}{If $f$ is homotopic to a Dehn twist, the proof of De Carvalho-Koropecki-Tal would show that the thesis of Theorem~C holds whenever $f$ is a $\mathcal{C}^{1+\epsilon}$ diffeomorphism with a lift having a vertical rotation set (see \cite{addas14grilla} for a definition) that is not a single point. We do not know if this can be extended to homeomorphisms. If the vertical rotation set is a single point, then it is easy to construct examples where the thesis of Theorem \ref{PropC:maximumdiameter} does not hold.}

	Apart from its use in the proof of Theorem A, Theorem C has also a few other applications. For instance, one can now obtain for homeomorphisms a result previously known  only for diffeomorphisms (see \cite{addas14grilla}):
	\begin{teorema}\label{thmD:TransitiveThenMixing}
		Let $f \in \textnormal{Homeo}_0(\mathbb{T}^2)$ be a transitive homeomorphism. 
		
		\begin{enumerate}
			\item If $\rho(f)$ has nonempty interior, then $f$ is topologically mixing. 
			\item Furthermore, if $f$ has a lift $\wh f$ with $\vec{0} \in \mathrm{int}(\rho(\wh f))$, then $\wh f$ is also topologically mixing.
		\end{enumerate}
	\end{teorema}
	
	Finally, Theorem \ref{PropC:maximumdiameter} also solves a relevant question  (Question G from \cite{korotal}). It was known that, if $f$ is a non-wandering isotopic-to-the-identity homeomorphism whose rotation set has nonempty interior, then the lifted diameter of any periodic open topological disk (usually thought of as elliptic islands) is bounded, and the bound only depends on the period of the disk. But it was not known if the bound could be independent of the period, except in the $\mathcal{C}^{1+\alpha}$ setting where it was proven in \cite{addas14grilla}.
	
	\begin{teorema}\label{thmE:UniformlyBoundedIslands}
		Let $f \in \textnormal{Homeo}_0(\mathbb{T}^2)$ be a non-wandering homeomorphism whose rotation set has nonempty interior. Then, there exists $M > 0$ such that the lifted diameter of any periodic topological disk is less than or equal to $M$. 
	\end{teorema}
	{\color{black}
	\begin{definition}[\textbf{General Hypothesis}]\label{def:GeneralHyp}
		We will say a homeomorphism $f \in \textnormal{Homeo}_0(\mathbb{T}^2)$ is under the \emph{General Hypothesis} when 
		it has a lift $\wh f \in \textnormal{Homeo}_0(\mathbb{R}^2)$ such that $\textnormal{int}(\rho(\wh f))\neq \varnothing$.  
	\end{definition}
	}

	\subsection{Organization of the article}

	Looking at the big picture, the proof of Theorem \ref{thmA:semiconjugation} can be split into two main parts. In the first one we will build a factor $g$, by collapsing dynamically bounded continua for $f$ (assuming that even makes sense). This is done in Sections \ref{sec:CFS} through \ref{sec:EssFactor}. The second part (Sections \ref{sec:StableSets} through \ref{sec:DenseTopHorseshoes}) is dedicated to fully understanding the dynamics of $g$ and recovering every piece for the proof of Theorem \ref{thmA:semiconjugation}. Let us now comment on the role of each section.   
	
	As usual, Section \ref{section:Preliminaries} is devoted to settle notation and explain some preliminary results and techniques we will use later.

	In Section \ref{sec:CFS} we introduce the notion of \textit{canonically foliated strip}, which is the central notion of the article and will aid us throughout the whole work. In Section \ref{section:stretching} we already use it to develop \emph{anchoring} techniques for inessential continua, to prove that there is a uniform bound for the diameter of dynamically bounded continua for $f$, which is Theorem \ref{PropC:maximumdiameter} (we also obtain Theorem \ref{thmE:UniformlyBoundedIslands} as a result). The whole existence of the factor $g$ relies on this theorem: this is shown in Section \ref{sec:EssFactor}, where we build the mentioned factor and immediately obtain some basic dynamical properties for it:
	
	\begin{itemize}
		\item There exists a lift $\wh g$ of $g$ such that $\rho(\wh g) = \rho(\wh f)$,
		\item $g$ is infinitely continuum-wise expansive, 
		\item $g$ is tight.
	\end{itemize}

	In Section \ref{sec:StableSets} we prove some topological and dynamical properties for stable and unstable sets of points, which we use in the following three sections. In Section \ref{section:RotMixing} we make a refinement of the anchoring arguments from Section \ref{section:stretching} which we call \textit{total anchoring}, and use it to prove that $g$ is rotationally mixing. 
	
	Sections \ref{sec:HPR} and \ref{sec:DenseTopHorseshoes} are somewhat technical, and dedicated to proving the density of Markovian horseshoes for $g$: in Section \ref{sec:HPR} we introduce the notion of \textit{heteroclinic pseudo-rectangle} using stable and unstable sets, and prove their density in Proposition \ref{prop:ChaoticRectanglesAreDense}; in Section \ref{sec:DenseTopHorseshoes} we show that inside any heteroclinic pseudo-rectangle there exists a Markovian horseshoe, using the total anchoring we previously developed.

	The article finishes in Section \ref{sec:Conservative}, where we show {\color{black}that} up to choosing a topological model for the quotient torus, $g$ is area-preserving. This allows us to prove Theorem \ref{thmA:semiconjugation}, which we also do in this section.  
	
	\subsection{Some open questions}
	
	Here are some questions which appeared with the development of the techniques used in this article. 
	\begin{enumerate}
		\item Can the uniform bound $M$ for the diameter of dynamically bounded continua, be in turn bounded by only using some dynamical invariants of $f$? Our construction heavily depends on the choice of an isotopy $I$ from the identity to $f$, and a transverse foliation $\mathcal{F}$ (see Definition \ref{def:MDTD} for details). 
		\item Can we give a lower bound for the period for $f$, of the realizing points built in Proposition \ref{prop:RotatingPointsNoTransverseIntersection}, using the size of the rotation set? This could yield lower bounds for the topological entropy of $f$, and (more subtly) for the topological entropy carried by continua. 
		\item {\color{black}Consider the actions of $f$ and $g$ in the fine curve graph of the torus 
		(see \cite[Definition 3.1]{bowden}). How similar are they? Do they have the same asymptotic translation length? Does their quasi-axis remain at bounded distance from each other?}
		\item Suppose on top of the General Hypothesis that $f$ is nonwandering, in which case it is also strictly toral. Is the set of periodic points dense in $\Ess(f)$? (see Definitions \ref{def:StrictlyToral} and \ref{def:EssIne}). 
		\item Even simpler: when $f$ is transitive (and therefore $\Ess(f) = \T^2$), is it true that $f$ has a dense set of periodic points? (Note that the answer is indeed positive for the factor $g$). This was the very first motivating question we faced in the production of this article, and there are no proofs known to us, even for the $\mathcal{C}^{\infty}$, area-preserving case.    
		\item Let us go to the context of closed orientable hyperbolic surfaces $S_g$, of genus $g \geq 2$. It seems reasonable to search for an analogous semiconjugacy to the one in Theorem \ref{thmA:semiconjugation}. Which conditions on the homological or homotopical rotation sets 
		would allow us to replicate the anchoring techniques from Section \ref{section:stretching} and get a uniform bound for the diameter of dynamically bounded continua, as in Theorem \ref{PropC:maximumdiameter}? We imagine that a condition as the one in the main result in \cite{salvadorbruno} is sufficient, but this can probably be improved.
		\item A very recent pre-print by Militon \cite{preprintmiliton} provides, for a homeomorphism $g$ of a hyperbolic surface $S$ of genus $g\geq 2$, homotopic to a pseudo-Anosov (pA) map, a partition of the surface $S$ into stable sets as well as a partition of the surface into unstable sets, which imitates the well known foliations of a pA map. It appears that, in that context, a continuum $K$ is {\color{black}dynamically} bounded if and only if it is contained in a connected component of the intersection of a stable set and an unstable set. Is this sufficient to obtain a result {\color{black}like} Theorem~\ref{PropC:maximumdiameter} in this case? If that is the case, a factor with similar properties would appear. 
	\end{enumerate}
	
	\section*{Acknowledgements}

	The first author would like to thank Pierre-Antoine Guihéneuf, Alejandro Passeggi, Rafael Potrie and Nelson Schuback for helpful conversations. The second author would like to thank Andres Koropecki for several discussions and ideas.
	\footnote{The first author was partially funded by ANII, CSIC, FAPESP, IFUMI, IMERL, MathAmSud and PEDECIBA. The second author was partially funded by CNPq and FAPESP.}

	\medskip

	\section{Preliminaries}\label{section:Preliminaries}

	\subsection{Basic notation}

	\paragraph{\textbf{Surfaces}}
	Throughout this paper, a \textit{surface} $S$ will be a two dimensional orientable topological manifold. We will say that $S$ is respectively a \textit{plane, annulus, sphere} or \textit{torus} if it is homeomorphic to $\R^2, \R^2 \backslash \{0\}, \mathbb{S}^2$ or $\mathbb{T}^2$. A set $U \subset S$ will be a \textit{disk} if it is homeomorphic to $\mathbb{D} = \{(x,y) \in \mathbb{R}^2: x^2+y^2 < 1\}$. \textcolor{black}{While a disk is homeomorphic to a plane, we will usually reserve \emph{plane} for when referring to the whole surface.}
	\medskip
	
	\paragraph{\textbf{Homeomorphisms}}
	All maps considered in this work are continuous. Moreover, we will say that a map $f: S \to S$ is a \textit{homeomorphism} if $f$ is bijective, and both $f$ and its inverse are continuous. 
	
	\begin{definition}
		Given two surfaces $S, S'$ and two endomorphisms $f: S \to S$ and $f': S' \to S'$ we will say that $f$ and $f'$ are topologically conjugate if there exists a homeomorphism $h: S \to S'$ such that $h \circ f = f'\circ h$. 
	\end{definition}
	
	\paragraph{\textbf{Covering spaces}}
	Given two surfaces $S$ and $\tilde{S}$, we will say that $\tilde{\pi} : \tilde{S}\to S$ is a \textit{covering projection} if every point $z \in S$ has a neighbourhood $U$, such that $\tilde{\pi}^{-1}(U)$ is a disjoint union of disks $\tilde{D}_i$ in $\tilde{S}$, such that for every value of $i$, we have that $\tilde{\pi}|_{\tilde{D_i}}$ is a homeomorphism. 
	
	Under these conditions, we will say that $\tilde{S}$ is a \textit{covering space} of $S$, and it will be the \textit{universal covering} if it is simply connected. It is very well known that the universal covering of a surface is unique up to homeomorphisms, and it is homeomorphic to $\R^2$ whenever $S \neq \mathbb{S}^2$. When there is no place for ambiguity, we will lighten notation and simply call $\tilde{\pi}$ and $\tilde{S}$ a \textit{covering}, indistinctively. 
	
	For the particular case when $S = \wc{\mathbb{A}} \simeq \mathbb{R}^2 \backslash \{0\}$ is an annulus, we will take the universal covering ${\color{black}\wc{\mathbb{A}}}$ which is a plane, and the covering projection will be ${\color{black}\wc{\pi}}(x,y) = (e^{-y} \textnormal{cos}(2\pi x), e^{-y}\textnormal{sin}(2\pi x))$. We will then have the deck transformation $T \in \textnormal{Homeo}^+({\color{black}\wc{\mathbb{A}}})$ defined by $T(x,y) = (x+1,y)$. Similarly, for the case of the torus $\T^2 = \R^2 / \Z^2$, simply take the universal covering $\R^2$ with the natural projection, and we will have that the group of deck transformations will be generated by $T_x(x,y) = (x+1,y)$ and $T_y(x,y) = (x,y+1)$. 
	
	Given $f: S \to S$, and considering $\tilde{S}$ the universal covering of $S$, we may take $\tl{f}: \tilde{S}\to \tilde{S}$ a \textit{lift} of $f$, as a map such that $f\tilde{\pi} = \tilde{\pi}\tl f$ (we know that $\tl f$ is determined up to postcomposition with a deck transformation). Notice that, for the case in which $S$ is hyperbolic, there exists a unique lift $F$ of $f$, which commutes with the deck transformations.
	
	\medskip

	\paragraph{\textbf{Curves and foliations}}
	
	\begin{definition}
		Given two points $p, p' \in S$, a \textit{path} or \textit{curve} from $p$ to $p'$ will be a continuous function $\gamma: [0,1] \rightarrow S$ such that $\gamma(0)= p$, $\gamma (1) = p'$. We will say the path is \textit{simple} when $\gamma$ is injective, and \textit{closed} when $p = p'$. 
	\end{definition}
	
	We will say $\gamma$ is respectively a \textit{segment, line, circle} if it is homeomorphic to $[0,1], \mathbb{R}, \mathbb{S}^1$ by a proper application {\color{black}$h : \gamma \to \mathbb{R}^2$.}
	
	\begin{definition}
		Given a surface $S$, an \textit{oriented topological foliation} $\mathcal{F}$ is a partition of $S$ in one-dimensional manifolds such that for each $p \in S$, there exists a neighbourhood $U_p$ and a homeomorphism $h : U_p \rightarrow (- 1, 1)\times(- 1, 1)$ which preserves orientation and sends $\mathcal{F}$ into the foliation by vertical lines, oriented from top to bottom. 
	\end{definition}

	Given {\color{black}an oriented topological foliation, and} $z \in S$, we define $\phi_z$ as the leaf going through $z$, similarly we define $\phi ^{+}_z, \phi^{-}_z$ as the semi-leaves starting at $z$, which respect the orientation.
	
	Given a leaf $\phi$ of a singular oriented topological foliation in a surface $S$, we define $\alpha(\phi), \omega(\phi)$ as the usual $\alpha$ and $\omega$ limits of curves in surfaces. 
	
	\medskip

	\paragraph{\textbf{Topological notions}}
	
	\begin{definition}
		Let $X \subset \T^2$ be a set, $U \subset \T^2$ an open set, and $f: \T^2 \to \T^2$ a homeomorphism. 
		
		\begin{itemize}
			\item An open set $U \subset \T^2$ is said to be \emph{inessential} if every closed curve in $U$ is homotopically trivial. $X$ is said to be \emph{inessential} when it has an inessential neighbourhood. For the case of connected sets, this is equivalent to Definition \ref{def:Inessential}. 
			\item $X$ is \textit{essential} when it is not inessential. Moreover, it is \emph{fully essential} when $\T^{2} \backslash X$ {\color{black}is} inessential. 
			\item $U$ is said to be \emph{annular} if it is homeomorphic to an open annulus.
			\item We define the \emph{filling} $\textnormal{Fill}(X)$ of $X$, as the union of $X$ with every inessential component of its complement. We will say that $X$ is \emph{filled} whenever $X = \textnormal{Fill}(X)$.  
			\item We will define the \emph{orbit} of $X$ as the union of the past and future iterates of $X$ by the homeomorphism $f$. 
		\end{itemize}
	\end{definition}
	
	\begin{definition}
		Let $f$ be a uniformly continuous homeomorphism of the surface $S$. We will say that $f$ is \textit{mixing} if given two open sets $U, V \subset S$, there exists $n_0$ such that 
		\[f^n(U) \cap V \neq \varnothing, \ \text{ for every }\ n \geq n_0.\]
	\end{definition}
	
	\begin{lemma}\label{lem:MixingPower}
		If $f$ has a power $f^j$ which is topologically mixing, then $f$ is also topologically mixing. 
	\end{lemma}
	
	\begin{proof}
		Let $U, V \subset S$ be two open sets. And let us write $V_k = f^{-k}(V)$, where $k> 0$. By hypothesis, there exists $m_0$ such that 
		\[(f^j)^m (U) \cap V_k \neq \varnothing, \ \text{ for every } \ 0 \leq k \leq j-1, \ m \geq m_0.\] 
		This means that we may take $n_0 = jm_0$ and obtain that
		\[f^n(U) \cap V \neq \varnothing, \ n \geq n_0,\]
		which concludes the proof.  
	\end{proof}
	
	\begin{definition}
		Let $f \in \mathrm{Homeo}(\mathbb{T}^2)$. We will say that $f$ is \emph{rotationally mixing} if there exists {\color{black}$j > 0$ and a} lift $\widehat{f^j}$ of a power $f^j$ to $\mathbb{R}^2$, such that $\widehat{f^j}$ is topologically mixing. 
	\end{definition}

	\begin{definition}[\textbf{Hausdorff distance}]
		Let $(X,d)$ be a metric space, and let $K_1, K_2 \subset X$ be two nonempty subsets. We define the \emph{Hausdorff distance} $\dH(K_1,K_2)$ between $K_1$ and $K_2$ as
		\[\dH(K_1,K_2) = \inf \{r \geq 0 : K_2 \subset \mathrm{B}(K_1,r), K_1 \subset \mathrm{B}(K_2,r)\}, \] 
		where 
		\[\mathrm{B}(K,r) = \{x \in X : d(x,K) < r\}.\]
	\end{definition}
	
	We recall that whenever $X$ is compact, the metric space $(\mathcal{K}(X), \dH)$ of continua of $X$ with the Hausdorff distance, is also compact.  
	
	The following definition makes precise the notion of when a curve crosses another one from left to right, or vice versa.
	
	\begin{definition}[\textbf{Intersection number}]
		Given a surface $S$, and two curves $\gamma, \gamma': [0,1] \to S$ with a simple isolated intersection point $\gamma(t_0) = \gamma'(t'_0)$ in their interior, we will say that $\gamma \wedge \gamma' = 1$ whenever there exists a homeomorphism $h$ from a neighbourhood of $z$ to $(0,1)^2$, such that $h(\gamma(t)) = (t, \frac{1}{2})$ and such that the vertical coordinate of $h \circ \gamma'$ is strictly increasing in a neighbourhood of $t'_0$. 
	\end{definition}
	
	Note that this notion extends naturally to pairs of curves with finitely many transversal intersections.
	
	\medskip
	
	\begin{definition}[\textbf{Lines' relative position}]
		Let $\lambda_0, \lambda_1, \lambda_2 \subset \R^2$ be oriented lines. We will say that $\lambda_2$ is \emph{above} $\lambda_1$ relative to $\lambda_0$ (or equivalently, $\lambda_1$ is \emph{below} $\lambda_2$ relative to $\lambda_0$), if
		\begin{itemize}
			\item The three lines are pairwise disjoint
			\item None of the lines separate the other two, 
			\item If $\eta_1, \eta_2$ are two disjoint paths that respectively join $z_1 = \lambda_0(t_1)$ to $z'_1 \in \lambda_1$ and $z_2 = \lambda_0(t_2)$ to $z'_2 \in \lambda_2$, and the paths do not meet the three lines except at their endpoints, then $t_2 > t_1$. 
		\end{itemize}
	\end{definition}
	
	Note that this definition only depends on the orientation of $\lambda_0$. 
	
	\medskip

	\subsection{Strictly Toral Dynamics}
	
	The results for this subsection come from \cite{korotal}. The following notion tries to encode torus homeomorphisms whose dynamics can not be reduced to simpler subsurfaces. Throughout this subsection we will assume that $f$ is a \emph{non-wandering, isotopic-to-the identity} torus homeomorphism.  
	
	\begin{definition}
		We will say $f$ is \emph{irrotational} if there exists a lift $\wh f$ in $\R^2$ such that $\rho(\wh f) = (0,0)$. 
	\end{definition} 
	
	\begin{definition}
		We will say that $f$ is \emph {annular} if there exists a lift $\wh f$ of $f$, a vector $v \in \mathbb{Z}^2_*$ and some $M \in \R$, such that for every $\wh z \in \R^2$ and every $j \in \mathbb{Z}$, we have that   
		\[-M < \langle  \wh f^j (\wh z) - \wh z\; , \;  v \rangle  < M,\]
		that is, the displacement of any orbit in one particular rational direction is uniformly bounded.
	\end{definition}
	
	It can be proved in this case that there exists a finite covering of $\T^2$ by another torus, such that the corresponding lift of $f$ acting in this torus has an invariant annular set (see for example \cite[Remark 3.10]{jager}). 
	
	\begin{definition}[\textbf{Strictly toral}, \cite{korotal}]\label{def:StrictlyToral}
		Let $f \in \textnormal{Homeo}_0(\T^2)$, and assume that $f$ is non-wandering. We will say that $f$ is \emph{strictly toral} if {\color{black}none} of the following properties hold:
		\begin{enumerate}
			\item There exists a power $f^k$ of $f$ such that $\textnormal{Fix}(f^k)$ is fully essential, and such that $f^k$ is irrotational, 
			\item There exists a power $f^k$ of $f$ such that $f^k$ is annular 
		\end{enumerate}
	\end{definition}
	
	Note that the first item is essentially saying that $f$ is \emph{planar}, in the sense that, up to taking a power of $f$, understanding its dynamics is as difficult as understanding plane dynamics (i.e. the return maps for each disk in the complement of the set of fixed points). For the annular case, up to a finite covering we have an invariant annular set, whose complement is also annular, and therefore it suffices to understand annulus dynamics. 
	
	Let us now explain how the dynamics in the complement of these two cases is rightfully called {strictly toral}. 
	
	\begin{definition}\label{def:EssIne}
		We define the set $\textnormal{Ine}(f)$ of \emph{inessential points} (or simply \emph{inessential set}) of $f$, as
		\[\textnormal{Ine}(f) = \{z \in \T^2 : \exists \text{ a neighbourhood }U \text{ of } z \text{ such that the orbit of } U \text{ is inessential}\}.\] 
		
		We define $\textnormal{Ess}(f)$ the \emph{set of essential points} of $f$ as
		\[\textnormal{Ess}(f) = \T^2 \backslash \textnormal{Ine}(f)\] 
	\end{definition}
	
	Note that both of these sets are invariant. Moreover, we have by \cite[Theorems A and B]{korotal} that, if $f$ is strictly toral, then $\Ine (f)$ is an open inessential set, which is in turn the union of bounded periodic disks, whose diameter is uniformly bounded if we look at the set of disks with a fixed period $q$. Therefore, $\Ess(f)$ is a fully essential closed set. We also {\color{black}know} from \cite{korotal} that
	
	\begin{remark}\label{rem:TrasitiveIffEss}
		If $f$ is strictly toral, then $f$ is transitive $\iff$ $\Ess(f) = \T^2$. 
	\end{remark}
	
	Note that if the rotation set $\rho(f)$ has nonempty interior, then $f$ is strictly toral. 
	
	\medskip

	\subsection{Equivariant Brouwer Theory}
	
	This theory was developed by Le Calvez, and generalizes the foliations by Brouwer lines for fixed-point-free orientation-preserving homeomorphisms of the plane. 
	
	\begin{definition}
		Given a closed surface and an isotopy $I: [0,1] \times S \to S$, we will define the set $\textnormal{Fix}(I)$ of fixed points of $I$ by
		\[\textnormal{Fix}(I) = \{z \in S : I(t,z) = z \text{ for every } t \in [0,1]\}.\]
		We will also define $\textnormal{Dom}(I) = S \backslash \textnormal{Fix}(I)$.  
	\end{definition}

	\begin{definition}
		Given $f: S \to S$ and an isotopy $I : [0,1] \times S \to S$ from the identity to $f$, we will say that $I$ is \emph{maximal} if for every $z \in  \textnormal{Fix}(f) \cap \textnormal{Dom}(I)$, we have that the path $I_z(t)$ is not homotopically trivial in $\textnormal{Dom}(I)$.  
	\end{definition}
	
	The following result is due to Béguin, Crovisier and Le Roux, and can be found in \cite{beguin20isotopies}.
	
	\begin{theorem}[\cite{beguin20isotopies}]
		Given an isotopy $I'$ from the identity to a homeomorphism $f$ of a closed surface $S$, there exists a new isotopy $I$ which is homotopic to $I'$ with fixed endpoints (seen as paths in the space of homeomorphisms of $S$) such that $I$ is a maximal isotopy.  
	\end{theorem}
	
	This allows us to use the equivariant version of the Brouwer-Le Calvez transverse foliation built in \cite{lecalvezequivariante} and work with the following object:
	
	\begin{definition}[\textbf{Maximal dynamically transverse decomposition}]\label{def:MDTD}
		Let $f$ be a homeomorphism of the surface $S$ which is isotopic to the identity, by the maximal isotopy $I$. We will say that a singular oriented foliation $\mathcal{F}$ of $S$ is {\color{black}\emph{dynamically transverse}} to $I$ if
		\begin{itemize}
			\item $\textnormal{Sing}(\mathcal{F}) = \textnormal{Fix}(I)$
			\item Every path by the isotopy $\gamma_z(t) : [0,1] \to \textnormal{Dom}(I), \ \gamma_z(t) = I(t,z)$ has a representative $\gamma'_z:[0,1] \to \mathrm{Dom}(I)$ which is homotopic to {\color{black}$\gamma_z$} with fixed endpoints in $\textnormal{Dom}(I)$, such that $\mathcal{F}$ is positively transverse to $\gamma'_z$ (i.e. the new path only crosses leaves from right to left).  
		\end{itemize}
		
	\end{definition}

	In this context, we will say that $(f,I,\mathcal{F})$ is a {\color{black}\emph{maximal dynamically transverse decomposition}} (MDTD) for $f$. The desired foliation appears as the projection of an equivariant version of the plane foliation by Le Calvez for the universal covering $\tilde{M}$ of $M=\mathrm{Dom}(I)$, that is, we first obtain a maximal dynamically transverse decomposition $(\tl f, \tl I, \tl{\mathcal{F}})$ in the universal covering of $\mathrm{Dom(I)}$ (where $\tl{\mathcal{F}}$ is a non-singular equivariant foliation), and then we project it to obtain $(f,I,\mathcal{F})$.  
	
	We will define $I^t_{\mathcal{F}}(z) = \gamma'_z(t)$ the transverse path for $z$, which is defined for times in $[0,1]$, and we will simply write $I_{\mathcal{F}}(z)$ for said path. We will extend this map in the usual way for other values of $t$, by concatenating respective transverse paths of iterates by $f$:
	\[I^n_{\mathcal{F}}(z) = \prod \limits_{0 \leq j \leq n-1}I_{\mathcal{F}}(f^j(z)), \ I^\Z_{\mathcal{F}}(z) = \prod \limits_{j \in \Z}I_{\mathcal{F}}(f^j(z))\] 
	We can also lift these transverse paths to the universal covering of the surfaces we are working on.  
	
	We will write $\phi$ to denote leaves in $\F$, and similarly $\tl \phi$ for their lifts in $\tl \F$. Given a transverse path $\gamma: [a,b] \to \mathrm{Dom}(I)$ and $t \in [a,b]$, we will write $\phi_{\gamma(t)}$ to denote the leaf such that $\gamma(t) \in \phi_{\gamma(t)}$ (and similarly for their respective lifts).

	\subsection{Forcing Theory}
	
	The Theory of Forcing for transverse trajectories has less than ten years of existence, and has been developed in \cite{lecalveztalforcing}, and \cite{lct2}. It has already yielded several strong results in the description of surface topological dynamics, and uses the elements from Definition \ref{def:MDTD}.  	
	
	We say that two transverse paths $\tl \gamma, \tl \gamma '$ whose images are in $\mathrm{Dom}(\tl I)$ are {\bck \emph{equivalent}} if they have lifts to $\tilde{M}$ that start and finish at the same leaf. 
	
	We say a path $\tl \gamma: [a,b] \to \tilde{M}$ from $\tl \phi_{\tl \gamma(a)}$ to $\tl \phi_{\tl \gamma(b)}$ is \textit{admissible of order $n$} if it is equivalent to a transverse path $\tilde{I}^{n}_{\tl \F}(\tl z)$, that is, if
	\[\tl f^{n}(\tl \phi_{\tl \gamma(a)}) \cap \tl \phi_{\tl \gamma(b)} \neq \varnothing.\] 
	
	We say a path $\gamma: [a,b] \to \mathrm{Dom(I)}$ is  \textit{admissible of order $n$} if it has a lift that is admissible of order $n$.
	
	The following definition comes from \cite[Section 3]{lecalveztalforcing}. 
	
	\begin{definition}[\textbf{$\tl{\mathcal{F}}$-transverse intersection}, \cite{lecalveztalforcing}]
		Let $\tl{\mathcal{F}}$ be an oriented foliation of the plane and let $(\tl f, \tl I, \tl{\mathcal{F}})$ be a {\bck \emph{maximal}} dynamically transverse decomposition. Let $\tl \gamma_1: J_1 \to \R^2$, $\tl \gamma_2: J_2 \to \R^2$ be two transverse paths such that $\tl \phi_{\gamma_1(t_1)} = \tl \phi_{\gamma_2(t_2)} = \tl \phi$. 
		
		We will say that $\tl \gamma_1$ intersects $\tl \gamma_2$ {\bck \emph{$\tl{\mathcal{F}}$-transversely and positively}} (or equivalently, that $\tl \gamma_2$ intersects $\tl \gamma_1$ {\bck \emph{$\tl{\mathcal{F}}$-transversely and negatively}}) if there exist $a_1, b_1 \in J_1$ with $a_1 < t_1 < b_1$, and $a_2, b_2 \in J_2$ with $a_2 < t_2 < b_2$, such that 
		\begin{itemize}
			\item $\tl \phi_{\tl{\gamma}_2(a_2)}$ is below $\tl \phi_{\tl{\gamma}_1(a_1)}$ relative to $\tl \phi$,
			\item $\tl \phi_{\tl{\gamma}_2(b_2)}$ is above $\tl \phi_{\tl{\gamma}_1(b_1)}$ relative to $\tl \phi$.
		\end{itemize}
		
		In this context we can also say that $\tl \gamma_1$ and $\tl \gamma_2$ have an {\bck \emph{$\tl{\mathcal{F}}$-transverse intersection}} 		 
	\end{definition}
	
	We will say that two transverse paths $\gamma_1, \gamma_2$ intersect $\mathcal{F}$-transversally whenever they have lifts $\tl \gamma_1, \tl \gamma_2$ that intersect $\tl{\mathcal{F}}$-transversally.
	
	If two lifts of $\gamma_1$ have a $\tl{\mathcal{F}}$-transverse intersection, we will say that $\gamma_1$ has a transverse self-intersection. 
	
	The fundamental result of \cite{lecalveztalforcing} is without a doubt Proposition 20 which roughly says we may concatenate pieces of transverse trajectories, whenever these paths intersect $\tilde{\mathcal{F}}$-transversally:
	
	\begin{proposition}[\textbf{Forcing}, \cite{lecalveztalforcing}]\label{prop:Forcing}
		Let $\gamma_1:[a_1,b_1] \to S, \gamma_2:[a_2,b_2] \to S$ be two admissible transverse paths, which are respectively admissible of order $n_1$ and $n_2$, and intersect $\F$-transversally at $\gamma_1(t_1) = \gamma_2(t_2)$. Then, 
		\begin{itemize}
			\item $\gamma_1|_{[a_1,t_1]}\gamma_2|_{[t_2,b_2]}$ and $\gamma_2|_{[a_2,t_2]}\gamma_1|_{[t_1,b_1]}$ are both admissible of order $n_1 + n_2$. 
			\item Furthermore, either one of these paths is admissible of order $\min(n_1,n_2)$, or both are admissible of order $\max(n_1,n_2)$.
		\end{itemize} 
	\end{proposition}
	
	We will use \cite[Propositions 7 and 26]{lecalveztalforcing} and \cite[Theorem M]{lct2} to build the canonically foliated strips in Section \ref{sec:CFS}. 
	For the reader who wants some insight on how these techniques work, going through \cite[Sections 3 and 4]{lecalveztalforcing}, could be a useful practice.

	\subsection{Topological and Markovian horseshoes}
	
	Here we introduce two objects with horseshoe-like dynamics, the former is easier to detect, the latter yields stronger consequences.  
	
	\medskip
	
	\paragraph{\textbf{Topological horseshoes}}We start by defining a purely topological notion of horseshoe, introduced by Kennedy and Yorke in \cite{kennedy}. 
	
	\begin{definition}
		Let $C$ be a compact set of a surface $S$, and let $C_0, C_1$ be two fixed disjoint compact subsets of $C$, each of them intersecting each connected component of $C$. We will say that a continuum $K$ of $C$ is a \emph{connection} if it intersects both $C_0$ and $C_1$. 
		
		We will say that a continuum $K'$ of $C$ is a \emph{preconnection} (for the homeomorphism $f$) when $f(K')$ is a connection. 
	\end{definition}
	
	\begin{definition}
		We will say that the homeomorphism $f$ has \emph{crossing number} equal to $m$ for the compact set $C$, if $m$ is the greatest integer for which there exist $C_0, C_1 \subset C$ such that each connection contains {\bck at least} $m$ disjoint preconnections.
	\end{definition}
	
	\begin{theorem}[\cite{kennedy}]\label{thm:TopHorseshoe}
		Suppose that a power $f^j$ of a homeomorphism $f$ of the surface $S$ has crossing number $m > 1$ for the locally connected, compact set $C \subset S$. Then, there exists a compact subset $Q \subset C$, such that $f^j |_Q$ is semiconjugate to the Bernoulli shift in $m$ symbols. 
		
		In particular, $h_{\textnormal{top}}(f)$ is positive. 
	\end{theorem}
	
	In this case, we will say that $Q$ is a topological horseshoe in $m$ symbols for $f$, with period $j$. 
	
	\medskip
	
	This notion generalizes the classical Smale's Horseshoe, which satisfies the hypothesis of Theorem \ref{thm:TopHorseshoe}, see Figure \ref{figure:Smale} for a brief explanation. 
	
	Note that although we obtain positive topological entropy, we do not recover the existence of periodic points for $f$ as we can not make sure that the preimages by the semiconjugacy of periodic points for the Bernoulli shift, contain periodic points for $f$.
	
	\begin{figure}[ht]
		\centering
		
		\def\svgwidth{.92\textwidth}
\begingroup%
  \makeatletter%
  \providecommand\color[2][]{%
    \errmessage{(Inkscape) Color is used for the text in Inkscape, but the package 'color.sty' is not loaded}%
    \renewcommand\color[2][]{}%
  }%
  \providecommand\transparent[1]{%
    \errmessage{(Inkscape) Transparency is used (non-zero) for the text in Inkscape, but the package 'transparent.sty' is not loaded}%
    \renewcommand\transparent[1]{}%
  }%
  \providecommand\rotatebox[2]{#2}%
  \newcommand*\fsize{\dimexpr\f@size pt\relax}%
  \newcommand*\lineheight[1]{\fontsize{\fsize}{#1\fsize}\selectfont}%
  \ifx\svgwidth\undefined%
    \setlength{\unitlength}{487.54495456bp}%
    \ifx\svgscale\undefined%
      \relax%
    \else%
      \setlength{\unitlength}{\unitlength * \real{\svgscale}}%
    \fi%
  \else%
    \setlength{\unitlength}{\svgwidth}%
  \fi%
  \global\let\svgwidth\undefined%
  \global\let\svgscale\undefined%
  \makeatother%
  \begin{picture}(1,0.35191562)%
    \lineheight{1}%
    \setlength\tabcolsep{0pt}%
    \put(0,0){\includegraphics[width=\unitlength,page=1]{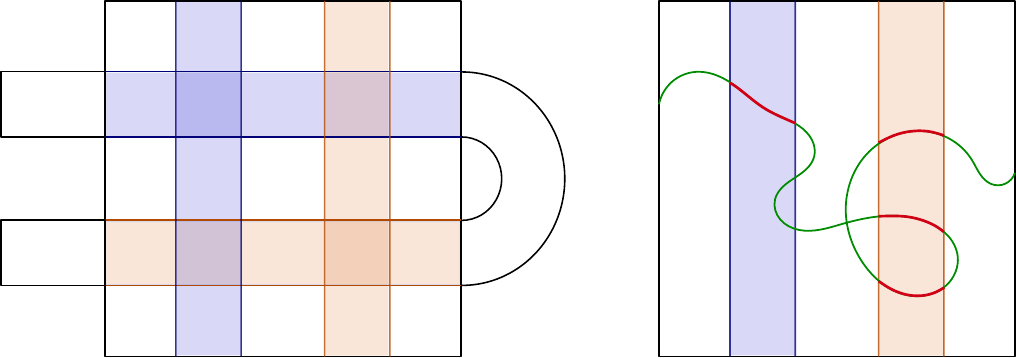}}%
    \put(0.00660414,0.15010128){\makebox(0,0)[lt]{\lineheight{1.25}\smash{\begin{tabular}[t]{l}$f(C_0)$\end{tabular}}}}%
    \put(0.46219073,0.01168524){\makebox(0,0)[lt]{\lineheight{1.25}\smash{\begin{tabular}[t]{l}$C_1$\end{tabular}}}}%
    \put(0.06340612,0.01168524){\makebox(0,0)[lt]{\lineheight{1.25}\smash{\begin{tabular}[t]{l}$C_0$\end{tabular}}}}%
    \put(0.00660414,0.29124144){\makebox(0,0)[lt]{\lineheight{1.25}\smash{\begin{tabular}[t]{l}$f(C_1)$\end{tabular}}}}%
  \end{picture}%
\endgroup%

		\smallskip
		
		\caption {For the classical Smale Horseshoe, every connection (green) has at least two respective subcontinua (red) connecting the vertical sides of the blue and orange regions, which implies it has at least two disjoint preconnections.}
		\label{figure:Smale}
	\end{figure}
	
	\medskip
	
	\paragraph{\textbf{Markovian horseshoes}} Stronger notions of horseshoe than the one in \cite{kennedy} are described for example in \cite{zbMATH06864334}, \cite{lct2} and \cite[Section 9.2]{pamiliton}. The notion used in \cite{pamiliton} is slightly stronger and adapted to the particular context of hyperbolic surface homeomorphisms. 
	
	This will be our case for the surface $S = \R^2 \backslash \textnormal{Dom}(\wh I)$, whenever Dom$(\wh I)$ is nonempty. Furthermore, be aware that although this notion is called \textit{rotational horseshoe} in its original definition, in our context the horseshoes will rotate around singularities {\bck \emph{of}} $\wh \F$, which is not itself rotation in our original torus.

	\begin{definition}
		We will say a compact connected set $R \subset S$ is a \emph{rectangle} if it is homeomorphic to $[0,1]^2$ by a homeomorphism $h : [0,1]^2 \to h([0,1]^2) \subset S$. We will call \textit{sides} of $R$, the image of the sides of $[0,1]^2$ by $h$, the \textit{horizontal sides} being $R^- = h([0,1]\times \{0\})$ and $R^+ = h([0,1]\times \{1\})$. 
	\end{definition}
	
	\begin{definition}\label{def:MarkovianIntersection}
		Let $R_1, R_2 \subset S$ be two rectangles. We will say that $R_1 \cap R_2$ is a \textit{Markovian intersection} if there exists a homeomorphism $h$ from a neighbourhood of $R_1 \cup R_2$ to an open subset of $\R^2$, such that
		\begin{itemize}
			\item $h(R_2) = [0,1]^2$;
			\item Either $h(R_1^+) \subset \{(x,y) \in \R^2 : y > 1\}$ and $h(R_1^-) \subset \{(x,y) \in \R^2 : y < 0\}$; or $h(R_1^+) \subset \{(x,y) \in \R^2 : y < 0\}$ and $h(R_1^-) \subset \{(x,y) \in \R^2 : y > 1\}$;	
			\item $h(R_1) \subset \{(x,y) \in \R^2 : y > 1\} \cup [0,1]^2 \cup \{(x,y) \in \R^2 : y < 0\}$.
		\end{itemize}
	\end{definition}
	
	{\color{black}
	The next statement generalizes the Jordan-Schöenflies theorem using a result due to Homma (see \cite{homma53}). It will be used in the construction of rectangles and Markovian intersections. This particular formulation was obtained from \cite[Theorem 9.4]{pamiliton}.

	\begin{lemma}[Homma, \cite{homma53}]\label{lem:homma}
		Any homeomorphism from 
		\[\Big( \big( (\R \times \{0\}) \cup (\R \times \{1\}) \cup (\{0\} \times [0,1]) \cup (\{1\} \times [0,1]) \big) \cap B(\vec{0}, 10) \Big) \cup \partial B(\vec{0}, 10),\]
		to its image in $\R^2$, can be extended to a homeomorphism $h : \R^2 \to \R^2$. 
	\end{lemma}
	}

	\begin{definition}[\textbf{Markovian horseshoe}. \cite{pamiliton}]\label{def:RotationalHorseshoe}
		Let $f$ be a homeomorphism of a hyperbolic surface $S$, let $\tl f$ be its lift to the universal covering $\tl S$, and let $m > 1$ be an integer. 
		
		We will say that $f$ has a \emph{Markovian horseshoe} in $m$ symbols with deck transformations $W_1, \hdots, W_m$, if there exists a rectangle $\tl R \subset \tl S$, and some $j \geq 1$ such that
		$$ \text{for every } 1 \leq i \leq m, \text{ the intersections } W_i(\tl R) \cap \tl f^j(\tl R) \text { are Markovian}.$$ 
		
		In this context, we will say that $j$ is the period of the horseshoe. Whenever $\tl R$ projects injectively by the covering projection, we may also say that $R$ contains a Markovian horseshoe in $m$ symbols for $f$. In this case, the set
		\[\bigcap_{n \in \Z} f^{nj}(R)\]
		is the Markovian horseshoe, and $j$ is again its period. 
		
		\smallskip
		
		For the purpose of this work, we will adapt this notion and say that a homeomorphism $f \in \mathrm{Homeo}_0(\T^2)$ has a Markovian horseshoe if that happens for the restriction $\wh {f^k} |_{\mathrm{Dom}(\wh I)}$ of a lift of some power $f^k$, where $\wh I$ is taken as in Definition \ref{def:MDTD}. 
	\end{definition}
	
	\medskip

	The proof of the following result can be found in \cite{pamiliton}. 
	
	\begin{proposition}\label{prop:topologicalhorseshoe}
		Let $f \in \textnormal{Homeo}(S)$. Suppose that $f$ has a Markovian horseshoe $X$ of period $j$ contained in the rectangle $R$, and with associated deck transformations $W_1, \dots , W_m \ (m \geq 2)$, which generate a free group. Then, there exists a compact subset $\tl Y \subset \tl S$ such that
		\begin{itemize}
			\item The map $f^j|_X$ has an extension $\tl g: \tl Y \to \tl Y$, which in turn is an extension of the Bernoulli shift $\sigma: \{1,\dots,m\}^{\mathbb{Z}} \to \{1,\dots,m\}^{\mathbb{Z}}$ in $m$ symbols by the semiconjugacy $h$.
			
			In particular, the following diagram commutes:

			\begin{center}
				\begin{tikzpicture}[scale=1]
					\node (A) at (0,0) {$X$};
					\node (B) at (3,0) {$X$};
					\node (C) at (0,1.5) {$\tl Y$};
					\node (D) at (3,1.5) {$\tl Y$};
					\node (E) at (0,3) {$\{1,\dots,m\}^\mathbb{Z}$};
					\node (F) at (3,3) {$\{1,\dots,m\}^\mathbb{Z}$};
					
					\draw[->,>=latex,shorten >=3pt, shorten <=3pt] (A) to node[midway, above]{$f^j$} (B);
					\draw[->,>=latex,shorten >=3pt, shorten <=3pt] (C) to node[midway, above]{$\tl g$} (D);
					\draw[->,>=latex,shorten >=3pt, shorten <=3pt] (E) to node[midway, above]{$\sigma$} (F);
					\draw[->,>=latex,shorten >=3pt, shorten <=3pt] (C) to node[midway, left]{$h$} (E);
					\draw[->,>=latex,shorten >=3pt, shorten <=3pt] (D) to node[midway, right]{$h$} (F);
					\draw[->,>=latex,shorten >=3pt, shorten <=3pt] (C) to node[midway, left]{$\pi$} (A);
					\draw[->,>=latex,shorten >=3pt, shorten <=3pt] (D) to node[midway, right]{$\pi$} (B);
				\end{tikzpicture}
			\end{center}
			
			\item the preimage by $h$ of every periodic sequence by the shift $\sigma$, contains a periodic point of $g$ (which then projects to a periodic point of $f^j$). 
		\end{itemize}

	\end{proposition}

	\begin{remark}
		If $f \in \textnormal{Homeo}(S)$ has a Markovian horseshoe $X$ of period $j$ in $m$ symbols, then $h_{\textnormal{top}}(f) \geq \frac{\textnormal{log}(m)}{j} $. In particular, the existence of a Markovian horseshoe for $f$ implies positivity of topological entropy of $f$.  
	\end{remark}

	Note that the existence of a Markovian horseshoe implies the existence of a topological one. The Markovian ones behave better, in particular
	
	\begin{remark}
		Markovian horseshoes coming from deck transformations forming a free group yield exponential growth of the number of periodic orbits, as the period goes to infinity. Topological horseshoes do not ensure the existence of periodic points.    
	\end{remark}

	\subsection{Prime End Theory} We will use classical Prime End Theory by Carathéodory in order to prove the density of horseshoes, in Section \ref{sec:DenseTopHorseshoes}. An exhaustive construction of the definitions and results we exhibit can be found for example in \cite[Chapter 17]{milnor} or \cite[Chapter 9]{collingwood}. We include the results we will use later for the sake of completeness.

	Let us take $\mathcal{O} \subset \mathbb{S}^2$ a simply connected open set, whose complement has more than one point. By the Riemann mapping theorem, it is conformally equivalent to the unit disk $\D = \mathrm{B}(0,1) \subset \R^2$, {\color{black}by the Riemann map $h: \mathcal{O} \to \D$}.

	\begin{definition}
		We will say $\gamma: (0,1) \to \mathcal{O}$ is a \textit{crosscut} if it is homeomorphic to (0,1), and its closure is homeomorphic to [0,1], with its two endpoints in $\partial \mathcal{O}$. 
	\end{definition}
	
	Using the Jordan Curve Theorem, we know that every crosscut $\gamma$ separates $\mathcal{O}$ into two connected components. Each component $N$ of $\mathcal{O} \backslash \gamma$ will be called a \textit{cross-section}, which {\color{black}satisfies} $\gamma = {\bck \mathcal{O}} \cap \partial N$. 
	
	\begin{definition}
		A \textit{fundamental chain} $\mathcal{N} = \{N_j\}_{j \in \mathbb{N}}$ will be a decreasing sequence ($N_{j+1} \subset N_j$ for every $j$) of cross-sections $N_i \subset \mathcal{O}$, such that the diameter of the corresponding crosscuts tends to 0 as $j$ tends to infinity.
	\end{definition}
	
	Two fundamental chains $\mathcal{N} = \{N_j\}_{j \in \mathbb{N}}$ and $\mathcal{N'} = \{N'_j\}_{j \in \mathbb{N}}$ will be equivalent if every $N_j$ contains some $N'_k$ and conversely every $N'_{j'}$ contains some $N_{k'}$. 
	
	\begin{definition}
		A \textit{prime end} $\breve{z}$ for the pair $(\mathrm{cl}(\mathcal{O}), \partial \mathcal{O})$ is an equivalence class of fundamental chains. 
	\end{definition}
	
	\begin{definition}
		We define the \textit{impression} $\mathrm{Imp}(\breve{z}) \subset \mathrm{cl}(\mathcal{O})$ of a prime end $\breve{z}$ as 
		\[\bigcap_{j \in \mathbb{N}} \mathrm{cl}(N_j),\]
		where $\{N_j\}$ is any fundamental chain which defines $\breve{z}$. 
	\end{definition}
	
	\begin{remark}
		For every prime end $\breve z$, we have that $\mathrm{Imp}(\breve{z})$ is a compact connected subset of $\partial \mathcal{O}$. 	
	\end{remark}

	\begin{definition}\label{def:Caratheodory}
		We define the \textit{Caratheodory compactification} $\breve{\mathcal{O}}$ of $\mathcal{O}$ as the disjoint union of $\mathcal{O}$ and the set of prime ends of $\mathcal{O}$, with the topology having $\mathscr{\mathcal{O}} \cup \mathscr{N}$ as a base, where 
		\begin{itemize}
			\item $\mathscr{\mathcal{O}}$ is the set of open subsets of $\mathcal{O}$,
			\item $\mathscr{N} = \bigcup \underline{N}$, where $\underline{N}$ is the union of any cross-section $N$ with the set of prime ends which are given by a fundamental chain $\{N_j\}$ with $N_1 \subset N$. 
		\end{itemize}  	
	\end{definition}

	\begin{theorem}[\textbf{Caratheodory compactification}]
		Any conformal isomorphism $h: \D \to \mathcal{O}$ extends uniquely to an isomorphism from $\breve{\D} = \mathrm{cl}(\D)$ to $\breve{\mathcal{O}}$. 
	\end{theorem}
	
	We will write $\breve{z}$ for points in $\breve{\D}$, note that the set of prime ends is then homeomorphic to $\partial \breve{\D} \simeq \mathbb{S}^1$.  
	
	{\color{black} 
		
		\begin{definition}{\label{def:ray}}
			A \textit{ray} $r$ is a continuous injective curve $r : [0,1) \to \D$, such that $r(0) = \vec{0}$. Moreover, when there is no room for ambiguity, we will call $r$ to the image of $r$.	
		\end{definition}

		\begin{definition}\label{def:LandingRay}
			We will say a curve $\eta: [0,1) \to \mathcal{O}$ \textit{lands} at $\breve{z} \in \partial \mathbb{D}$, when $\lim \limits_{t \to 1} \eta(t) = \breve{z}$. Similarly, we will say that $h(\eta) : [0,1] \to \mathcal{O}$ \textit{lands} at $z \in \partial \mathcal{O}$, when $\lim \limits_{t \to 1} h(\eta(t)) = z$
		\end{definition}

	
	\begin{definition}[\textbf{Accessible point}]\label{def:AccessiblePoint}
	We will say $z \in \partial \mathcal{O}$ is \textit{accessible} if there exists a simple curve in $\mathcal{O}$ which lands at $z$.	
	\end{definition}
	}
	
	Note that any neighbourhood $U$ of a point $z$ in ${\bck \partial}\mathcal{O}$ intersects both $\mathcal{O}$ and its complement. By drawing a curve inside $U$ between points in each of these components, and which does not go through $z$, we obtain the following result. 
	
	\begin{remark}\label{rem:AcessiblePoints}
		The set of accessible points is dense in $\partial \mathcal{O}$. 
	\end{remark}
	
	{\color{black}
	\begin{definition}[\textbf{Accessible prime end}]\label{def:AccessiblePrimeEnd}
		We will say a prime end $\breve{z} = e^{i\theta}$ is \textit{accessible}, if there exists a simple curve $\breve{\gamma}:[0,1] \to \D$ with $\breve{\gamma}|_{[0,1)} \subset \mathbb{D}$, $\breve{\gamma}(1) =e^{i\theta}$, {\color{black}satisfying} 
		\[h(\breve{\gamma}) \text{ lands at some }z\in \partial \mathcal{O}\]
	\end{definition}

\begin{remark}\label{rem:AccessibleImpression}
	The impression of an accessible prime end is an accessible point.  
\end{remark}
}
	
	\begin{theorem}[17.4, \cite{milnor}]\label{thm:milnor}
		The set of accessible prime ends has full Lebesgue measure in $\mathbb{S}^1$. In particular, it is dense. 
	\end{theorem}

	\section{Canonically foliated strips}\label{sec:CFS}
	
	The purpose of this section is to build the main tool for the proof of Theorem \ref{PropC:maximumdiameter}, which is called a \textit{canonically foliated strip}. Its construction uses some notions of Le Calvez-Tal Forcing Theory for surface homeomorphisms. Given our homeomorphism $f$ under the hypothesis of Theorem \ref{PropC:maximumdiameter}, we may assume up to taking a power and an adequate lift, that the origin $\vec{0} \in {\bck {\textnormal{int}}}(\rho(\wh f))$ for some lift $\wh f$.

	Let us settle notation for the remainder of the current and the following sections. We will take a maximal dynamically transverse decomposition {\bck ({MDTD})} $(f,I,\mathcal{F})$ for $f$, where $\mathcal{F}$ is a singular foliation with $\textnormal{Sing}(\mathcal{F}) = \textnormal{Fix}(I)$, and such that $I$ lifts to an isotopy $\wh I$ from the identity to $\wh f$. We may then take a lift of this decomposition to the universal covering $\mathbb{R}^2$ of $\mathbb{T}^2$, obtaining $(\widehat{f}, \widehat{I}, \widehat{\mathcal{F}})$. Restricting to our domain $\textnormal{Dom}(\widehat{I}) = \mathbb{R}^2 \backslash \textnormal{Fix}(\widehat{I})$, we may then lift again to the universal covering $\tl \D$ of $\textnormal{Dom}(\widehat{I})$, and get the corresponding $(\tilde{f}, \tilde{I}, \tilde{\mathcal{F}})$, where $\tl f$ acts as a Brouwer homeomorphism on the topological plane $\tl \D$.

	Unless explicitly stated, we shall use for the remainder of the work the diacritic $\ \wh{\mbox{}} \ $ (i.e. $\wh f$) to denote objects in the universal covering $\R^2$ of $\T^2$, and we shall use the diacritic $\ \wt{\mbox{}} \ $ (i.e. $\tl f$) to denote objects in the universal covering $\tl \D$ of $\textnormal{Dom}(\widehat{I})$.  
	
	Be aware that we will develop the proofs without going to the -more frequent- universal covering of the punctured torus $\mathrm{Dom}(I)$. In our context, for instance, the group of deck transformations acting in $\tl \D$ is necessarily not finitely generated, as $\mathrm{Dom}(\wh I)$ has infinitely many punctures for each puncture of $\mathrm{Dom}(I)$.   
	
	For the particular case of Proposition \ref{prop:RotatingPointsNoTransverseIntersection} we will use maximal dynamically transverse decompositions in two other auxiliary spaces, which will be explained in its proof. 
	
	\begin{definition}[\textbf{Canonically foliated strip}]
		Let $\wh f$ be a lift to $\R^2$ of a homeomorphism $f \in \mathrm{Homeo}_0(\T^2)$, and let $(\wh f, \wh I, \wh \F)$ be an MDTD for $\wh f$. We will say that $\wh A \subset \R^2$ is a \textit{canonically foliated strip} (CFS) for $\wh f$ if 
		
		\begin{enumerate}
			\item There exists $v \in \Z^2$, $\wh z \in \R^2$ and some $j > 0$, such that $\wh f^j(\wh z) = \wh z + v$, and such that $\wh A$ is the saturated set of $\wh I^{\Z}_{\wh \F}(\wh z)$ by leaves of $\wh \F$. 
			\item There exists a homeomorphism $h: \wh A \to \R^2$, such that $h(\wh \F|_{\wh A})$ is the foliation by vertical lines, oriented from top to bottom. 
			
			In this context, we will say that $\wh z$ is a \textit{realizing point} for $\wh A$. 
		\end{enumerate}
	\end{definition}
	
	\begin{remark}
		If $\wh A$ is a CFS, there exists a non-null vector $v \in \Z^2$ such that $\wh A + v = \wh A$.  
	\end{remark}
	
	The goal is then to prove the following result.
	
	\begin{proposition}\label{prop:ExistenceofCFS}
		Let $\wh f$ be a lift to $\R^2$ of $f \in \mathrm{Homeo}_0(\T^2)$, with $\vec{0} \in \mathrm{int}(\rho(\wh f))$. 
		
		Then, there exist $j, p \in \Z^+$, and four families 
		$$\wh {\mathcal{A}}^{\rar} = \{\wh A^{\rar}_k\}_{k \in \Z}, \ \wh {\mathcal{A}}^{\uar} = \{\wh A^{\uar}_k\}_{k \in \Z}, \ \wh {\mathcal{A}}^{\lar} = \{\wh A^{\lar}_k\}_{k \in \Z}, \ \wh {\mathcal{A}}^{\dar} = \{\wh A^{\dar}_k\}_{k \in \Z}$$ of canonically foliated strips $\wh A^{\rar}_k, \wh A^{\uar}_k, \wh A^{\lar}_k, \wh A^{\dar}_k$, with respective realizing points $\wh z^{\rar}_k, \wh z^{\uar}_k, \wh z^{\lar}_k, \wh z^{\dar}_k$, such that
		\begin{enumerate}
			\item For every $k \in \Z$, we have that 
			\[\wh f^j(\wh z^{\rar}_k) = \wh z^{\rar}_k + (p,0), \ \wh f^j(\wh z^{\uar}_k) = \wh z^{\uar}_k + (0,p),\]
			\[\wh f^j(\wh z^{\lar}_k) = \wh z^{\lar}_k + (-p,0), \ \wh f^j(\wh z^{\dar}_k) = \wh z^{\dar}_k + (0,-p),\]
			\item For every pair of values $k, k' \in \Z$ with $k \neq k'$, we have that
			\[\mathrm{cl}(\wh A^{\rar}_k) \cap \mathrm{cl}(\wh A^{\rar}_{k'}) = \varnothing, \ \mathrm{cl}(\wh A^{\uar}_k) \cap \mathrm{cl}(\wh A^{\uar}_{k'}) = \varnothing,\]
			\[ \mathrm{cl}(\wh A^{\lar}_k) \cap \mathrm{cl}(\wh A^{\lar}_{k'}) = \varnothing, \ \mathrm{cl}(\wh A^{\dar}_k) \cap \mathrm{cl}(\wh A^{\dar}_{k'}) = \varnothing. \]
		\end{enumerate}
		
	\end{proposition}
	

	\medskip
	
	Let us start with an elementary result regarding transverse foliations {\bck {which can be found in \cite{lecalveztalforcing}. We prove it here for the sake of completeness.}}
	
	\medskip
	
	\begin{lemma}\label{lem:UniformlyBoundedLeaves}
		If $\vec{0} \in \rho(\wh f)$, the diameter of leaves $\wh \phi \in \wh{\mathcal{F}}$ is uniformly bounded. 
	\end{lemma}
	
	\begin{proof}
		By Franks realization result in \cite{franks89}, let us take $\wh z_1, \wh z_2, \wh z_3$ three periodic points realizing different rotation directions and such that $\vec{0}$ belongs to the convex hull of the set containing the respective three rotation vectors. Let $\wh \gamma_1, \wh \gamma_2, \wh \gamma_3$ be their respective transverse paths for their whole orbits, 
		\[\wh \gamma_j = \wh I ^\Z_{\wh \F} (\wh z_j), \ \ 1 \leq j \leq 3.\]
		which we will build as periodic by respective integer translations, given that they come from periodic orbits. Furthermore, note that these paths have a well defined slope, again because they come from periodic orbits (and the three slopes are different from each other). We may then take three sufficiently long pieces $\wh \gamma'_1, \wh \gamma'_2, \wh \gamma'_3$ {\bck {respectively of $\wh \gamma_1, \wh \gamma_2, \wh \gamma_3$}} such that their concatenation defines a region $\wh R$ on its right (i.e. we obtain a \textit{triangle} with its sides being positively and cyclically oriented), such that $\wh R$ contains a fundamental domain. Given that the used paths are positively transverse with respect to the foliation, we obtain that the future half-leaf of any point inside the region $\wh R$, must also stay inside $\wh R$, which proves that the diameter of these half-leaves is uniformly bounded since $\wh R$ contains a fundamental domain. The same argument reiterated implies the diameter of past half-leaves is also uniformly bounded, which completes the proof. 
	\end{proof}

	\subsection{Construction} We will use this subsection to prove Proposition \ref{prop:ExistenceofCFS}, by constructing the respective families. The following result is key for this construction. 
	
	\begin{proposition}\label{prop:RotatingPointsNoTransverseIntersection}
		If $\vec{0} \in \mathrm{int}(\rho(\wh f))$, then 
		\begin{enumerate}
			\item There exists $\wh z^{\rar} \in \R^2$ and two integers $p^{\rar},q^{\rar} > 0$, {\color{black}with} $\wh f^{q^{\rar}}(\wh z^{\rar}) = \wh z^{\rar} + (p^{\rar},0)$, and such that the transverse path $\wh I^\Z_{\wh \F}(\wh z^{\rar})$ has no $\wh \F$-transverse self-intersections.
			\item Moreover, for every $q'^{\rar} \geq q^{\rar}$, there exists $\wh z'^{\rar} \in \R^2$ such that 
			\[\wh f^{q'^{\rar}}(\wh z'^{\rar}) = \wh z'^{\rar} + (p^{\rar},0)\] and such that the transverse path $\wh I^\Z_{\wh \F}(\wh z'^{\rar})$ has no $\wh \F$-transverse self-intersections.
		\end{enumerate}

	\end{proposition}
	
	{\color{black}The proof of this proposition yields a similar statement for the other three directions $\uar, \lar, \dar$, respectively replacing $(p^{\rar},0)$ with $(0,p^{\uar})$, $(-p^{\lar},0)$, or $(0,-p^{\dar})$.}  
	
	\begin{proof}
		
		\textbf{(1)}. By the realization result in \cite{franks89}, there exists $\wh z_0$ such that $\wh f^{q_0}(\wh z_0) = \wh z_0 + (1,0)$ {\bck{for some $q_0 > 0$}}. Let us take a dynamically transverse decomposition $(\wh f, \wh I, \wh \F)$.
		
		Given the foliation $\wh {\mathcal{F}}$ is equivariant by integer translations, we can project our decomposition to the annulus $\widecheck{\mathbb{A}} = \mathbb{S}^1 \times \R$, by taking {\color{black}$\wh z \sim \wh z'$ whenever $\wh z' - \wh z \in \Z\times\{0\}$}, and obtain the corresponding $(\widecheck f, \widecheck I, \widecheck \F)$. Let us then take $\widecheck z$ as the projection to the annulus of a point $\wh z \in \R^2$. 
		
		Let us define $\widecheck{\gamma}_0 : \R \to \widecheck{\mathbb{A}}$ by 
		\[ \widecheck{\gamma}_0(t) = \widecheck{I}^{q_0t}_{\widecheck{\F}}(\widecheck z_0)\]
		
		Note that $\wc \gamma_0(0) = \wc \gamma_0(1)$, we may then think of it as a loop in the annulus (with the same name), with domain equal to $\mathbb{S}^1$. Note that it suffices to prove that $\widecheck{I}^{\Z}_{\widecheck{\F}}(\wc z^{\rar})$ has no $\widecheck \F$-transverse self-intersection: this fact will then imply that $\widehat{I}^{\Z}_{\widehat{\F}}(\wh z^{\rar})$ has no $\wh \F$-transverse intersection with any of its horizontal integer translates, which is in turn stronger than our conclusion.
		
		Compactifying the annulus in the classical way to get a sphere lets us apply Proposition 7 in \cite{lecalveztalforcing}, from which we know that $\widecheck{\gamma}_0$ has an $\widecheck{\mathcal{F}}$-transverse self-intersection if and only if $\widecheck{\gamma}_0 |_{[0,2]}$ has an $\widecheck{\mathcal{F}}$-transverse self-intersection. 
		
		Given the path $\widecheck \gamma_0$ is at a positive distance from the set $\mathrm{Sing}(\widecheck \F)$, we can assume up to taking an equivalent path, that {\bck{$\widecheck \gamma_0 |_{[0,2]}$}} has a finite number $\ell_0$ of self-intersections counted with multiplicity (without looking at $\widecheck \F$-transversality). If $\widecheck \gamma_0 |_{[0,2]}$ has no $\widecheck \F$-transverse self-intersection, the proof is finished.
		
		Let us then suppose that there exists an $\widecheck \F$-transverse self-intersection at $\widecheck \gamma_0(s) = \widecheck \gamma_0(t)$, with $s,t \in [0,2], \ s < t$. In this case, we will build a process which finds $\wh z_1 \in \R^2$ (its projection to the annulus being $\widecheck z_1$), {\color{black}satisfying} the following properties:
		\begin{itemize}
			\item There exists $q_1 > 0$ such that $\wh f ^{q_1}(\wh z_1) = \wh z_1 + (p_1,0)$, where $p_1 \in \Z^+$,
			\item The transverse trajectory $\widecheck I^{q_1}_{\widecheck \F}(\widecheck z_1)$ is a subloop of {\bck{$\widecheck \gamma_0$}}, which has at most $\ell_1 \leq \ell_0 - 1$ self-intersections counted with multiplicity (without looking at $\widecheck \F$-transversality).
		\end{itemize} 
		
		We shall use $\rho(\widecheck \gamma_0|_{[s,t]})$ to denote  the integer such that, when taking lifts to $\R^2$, we obtain
		\[\wh \gamma_0(t) = \wh \gamma_0(s) + (\rho(\widecheck \gamma_0|_{[s,t]}),0).\]
		We define $\rho(\widecheck \gamma_0|_{[t,s+1]})$ in the same fashion. Note that one of these two numbers is positive, given their sum is equal to 1.  
		
		Let us write $\pitchfork_{\wc \F}$ to denote $\wc \F$-transverse intersection between curves. Note that from 
		\[\wc \gamma_0 |_{[0,2]} (s) \pitchfork_{\wc \F} \wc \gamma_0 |_{[0,2]} (t)\]
		we derive that 
		\[\wc \gamma_0 |_{[1,3]} (s+1) \pitchfork_{\wc \F} \wc \gamma_0 |_{[0,2]} (t), \ \wc \gamma_0 |_{[0,2]} (s) \pitchfork_{\wc \F} \wc \gamma_0 |_{[-1,1]} (t-1),\]  \[\wc \gamma_0 |_{[1,3]} (s+1) \pitchfork_{\wc \F} \wc \gamma_0 |_{[-1,1]} (t-1)\]
		
		This means that the closed curve $\wc \gamma_0 |_{[-1,3]}$ which contains all the previous ones, has transverse intersections at the pairs of times $(s,t), (s+1,t), (s,t-1)$ and $(s+1, t-1)$. 
		
		Let us go to the universal covering $\breve{\D}$ of $\mathrm{Dom}(\wc I)$, and name $(\breve f, \breve I, \breve \F)$ to the corresponding dynamically transverse decomposition. Let $\breve \gamma_0$ be a lift of $\widecheck \gamma_0$ to this covering. By hypothesis on $\widecheck \gamma_0$, we know there exists a deck transformation $\breve T$ such that there exists an $\breve \F$-transverse intersection at $\breve T^{j} (\breve \gamma_0(t)) = \breve T^{j+1} (\breve \gamma_0(s))$ for every $j \in \Z$.

		Let us first assume that $t-s < 1$, the other case is analogous. Given that $\breve \gamma_0 |_{[-1,3]}$ is admissible of order $4q_0$, we know by \cite[Theorem M]{lct2} the path $\breve{\gamma}_0|_{[s,t]}$ is realized by a periodic point $\breve z_1$ of period $q_1 = 4q_0$, and therefore by definition we have that its projection $\wc{\gamma}_0|_{[s,t]}$ is realized by a periodic point $\wc z_1 = \breve{\pi}(\breve{z}_1)$ of period $q_1 = 4q_0$. 
		
		If $\rho(\widecheck \gamma_0|_{[s,t]})$ is positive, then we define $\wh z_1$ as any lift of $\wc{z}_1$ to $\R^2$. Note that 
		\begin{itemize}
			\item $\wh f^{q_1}(\wh z_1) = \wh z_1 + (\rho(\wc{\gamma}_0|_{[s,t]}),0)$
			\item $\widecheck I^{q_1}_{\widecheck \F}(\widecheck z_1) = \wc{\gamma}_0|_{[s,t]}$ is a subloop of $\wc{\gamma}_0$, and therefore the number $\ell_1$ of self-intersections counted with multiplicity, is strictly lower than $\ell_0$. 
		\end{itemize}
		
		If $\rho(\widecheck \gamma_0|_{[s,t]})$ is negative, we have that  $\rho(\widecheck \gamma_0|_{[t,s+1]})$ is positive, and the solution is analogous, because again by \cite[Theorem M]{lct2} we get that the path $\breve{\gamma}_0|_{[t,s+1]}$ is realized by a periodic point of period $q_1 = 4q_0$. 
		
		We are left with the case where $1 < t-s < 2$. In this case, note that $(t-1) - s < 1$, and again we will look at the two subpaths $\widecheck \gamma_0|_{[s,t-1]}$ and $\widecheck \gamma_0|_{[t-1,s+1]}$. Both are realized by periodic points of period $4q_0$, and one of them rotates positively in the annulus $\wc A$, which will be our desired $\wc z_1$.   
		
		If $\wc \gamma_1$ has an $\widecheck \F$-transverse self intersection, this process can be iterated to obtain $\wh z_2, q_2, \ell_2$ in the same fashion as before. Note that this process ends in at most $k \leq \ell_0$ iterations. This means that the point $\wh z^{\rar} = \wh z_k$ obtained after the $k$-th and last iteration of the process, must have a transverse path with no $\widecheck \F$-transverse self-intersections (otherwise we would be able to reiterate the process), and also satisfies  
		\[\wh f^{q_k}(\wh z_k) = \wh z_k + (p_k,0), \text{ for some } q_k > 0,\]
		which concludes the proof of the first item. 
		
		\medskip
		
		\textbf{(2)}. Let $\wh z^{\rar} \in \R^2$ be as obtained in the previous item, and let us take its projection $\wc z^{\rar}$ to the annulus. Define its transverse path $\wc \gamma : \R \to \wc {\mathbb{A}}$ by 
		\[ \wc \gamma(t) = \wc I^{q^{\rar}t}_{\wc \F}(\wc z^{\rar}),\]
		and note that for $n \in \Z^+$, $\wc \gamma |_{[0,n]}$ is an admissible path of order $nq^{\rar}$ because $\wc z^{\rar}$ has period $q^{\rar}$. 
		
		By \cite[Proposition 26]{lecalveztalforcing}, it suffices to prove that there exists an admissible path which intersects $\wc \gamma$ $\wc \F$-transversely and positively, and an admissible path which intersects $\wc \gamma$ $\wc \F$-transversely and negatively: in that case we will obtain that for every $q'^{\rar} \geq q^{\rar}$, the path $\wc \gamma |_{[0,1]}$ is realized up to equivalence, as the transverse path of a periodic point of period $q'$, whose transverse trajectory will have no $\wc \F$-transverse self-intersection, and the proof will then be finished. 
		
		Let us find these two paths. Let $\wc A^{\rar}$ be the saturated set by leaves of $\wc \F$ of $\wc \gamma|_{[0,1]}$. Using the result from the first item, take $\wh z^{\uar}$ and $\wh z^{\dar}$ lifts of periodic points of $f$ which respectively rotate $(0,p^{\uar}/q^{\uar})$ and $(0,-p^{\dar}/q^{\dar})$ {\bck{($p^{\uar}, p^{\dar} >0$)}}, and take their respective projections $\wc z^{\uar}$, $\wc z^{\dar}$ to the annulus $\wc {\mathbb{A}}$. Using that 
		\begin{itemize}
			\item The diameter of leaves of $\wc \F$ is uniformly bounded because $\wc \F$ is the projection of $\wh \F$ (see Lemma \ref{lem:UniformlyBoundedLeaves}),
			\item The transverse paths
			\[\wc \gamma^{\uar}(t) = \wc I^{q^{\uar}t}_{\wc \F}(\wc z^{\uar}), \ \wc \gamma^{\dar}(t) = \wc I^{q^{\dar}t}_{\wc \F}(\wc z^{\dar}),\]
			both go through $\wc A^{\rar}$ in opposite directions, 
		\end{itemize}
		\smallskip
		
		\noindent we obtain that the following sets 
		\[J^{\uar} = \{t \in \R : \wc{\phi}_{\wc \gamma (t)} \cap \wc \gamma^{\uar} \neq \varnothing\}, J^{\dar} = \{t \in \R : \wc{\phi}_{\wc \gamma (t)} \cap \wc \gamma^{\dar} \neq \varnothing\}\]

		are both bounded. Finally, by \cite[Proposition 15]{salomaotal}, we obtain that for sufficiently long pieces $\wc \gamma^{\uar}|_{[s^{\uar},t^{\uar}]}, \wc \gamma^{\dar}|_{[s^{\dar},t^{\dar}]}$, we have that
		\begin{equation}
			{\bck{\wc \gamma^{\uar}|_{[s^{\uar},t^{\uar}]}}} \text{ intersects } \wc \gamma^{\rar} \ \wc \F\text{-transversally and positively,}
		\end{equation}
		\begin{equation}
			{\bck{\wc \gamma^{\dar}|_{[s^{\dar},t^{\dar}]}}} \text{ intersects } \wc \gamma^{\rar} \ \wc \F\text{-transversally and negatively,}
		\end{equation}	
		which concludes the proof. 
		
	\end{proof}

	\medskip
	
	We emphasize that the saturated set $\wh X$ of $\wh I^\Z_{\wh \F}(\wh z^{\rar})$ by leaves of $\wh \F$ is a topological plane with a trivial foliation: \textcolor{black}{First, as $\wh I^\Z_{\wh \F}(\wh z^{\rar})$ lifts the transverse trajectory of a periodic point with non-zero rotation, it can be taken as a proper path. Now, as the leaves of $\wh \F$ have uniformly bounded diameter and $\wh I^\Z_{\wh \F}(\wh z^{\rar})$ has no transverse $\wh \F$-self intersection, Proposition 4 of \cite{lecalveztalforcing} implies that every leaf $\wh \phi \subset X$ intersects $\wh I^\Z_{\wh \F}(\wh z^{\rar})$ in exactly one point.} {\color{black}{From this and from $\wh I^\Z_{\wh \F}(\wh z^{\rar})$ being proper follows that no leaf contained in $\wh X$ can be a closed curve. Furthermore, if $\tl X$ is a connected component of the lift of $\wh X$, then $\tl X$ is an open, connected and foliated set, and as $\tl \F$ has no singularities, its complement, which is also a foliated set, cannot have bounded components. Thus $\tl X$ is simply connected and homeomorphic to a plane. Since $\wh I^\Z_{\wh \F}(\wh z^{\rar})$ intersects every leaf once, and since these leaves are not closed, the projection of $\tl X$ to $\wh X$ is injective and therefore $\wh X$ and $\tl X$ are homeomorphic. Finally, by \cite{kaplan40}, we have a global homeomorphism $h: \R^2 \to \wh X$ which gives us coordinates $(s,t)$ for every point in $\wh X$, where $\wh y_1$ and $\wh y_2$ are in the same leaf if and only if $h^{-1}(\wh y_1)-h^{-1}(\wh y_2)$ is in $\{0\}\times\R$.}}
	
	\medskip
	
	We now build the four families of closure-disjoint CFS. 
	
	\begin{proof}[Proof of Proposition \ref{prop:ExistenceofCFS}]
		
		Using the first item in Proposition \ref{prop:RotatingPointsNoTransverseIntersection}, we find the existence of four periodic points $\wh z^{\rar}, \wh z^{\uparrow}, \wh z^{\leftarrow}$ and $\wh z^{\downarrow}$, with respective periods $q^{\rar}, q^{\uar}, q^{\lar}, q^{\dar}$, which respectively rotate $(p^{\rar}/q^{\rar},0)$, $(0,p^{\uar}/q^{\uar})$, $(-p^{\lar}/q^{\lar},0)$ and $(0,-p^{\dar}/q^{\dar})$, (where $p^{\rar}, p^{\uar}, p^{\lar}, p^{\dar} >0$), and whose transverse trajectories by the isotopy $\wh I$ have no $\wh \F$-transverse self-intersection.

		Using the second item from Proposition \ref{prop:RotatingPointsNoTransverseIntersection}, let us take $q'^{\rar} = (q^{\rar}q^{\uar}q^{\lar}q^{\dar}) p^{\rar} \geq q^{\rar}$ to obtain a new periodic point which we will still call $\wh z^{\rar}$ for the sake of simplicity, such that its transverse trajectory has no $\wh \F$-transverse self-intersection, and moreover
		\[ \wh f^{q'^{\rar}}(\wh z^{\rar})  = \wh z^{\rar} + (p^{\rar},0).\]
		Therefore, $\wh z^{\rar}$ has rotation vector $(1/(q^{\rar}q^{\uar}q^{\lar}q^{\dar}), 0)$. In the same fashion we obtain $\wh z^{\uar}$, $\wh z^{\lar}$ and $\wh z^{\dar}$ with the same rotation speed $1/(q^{\rar}q^{\uar}q^{\lar}q^{\dar})$ (the speed being the norm of the rotation vector).

		\bigskip
		
		Note that if we take a power $f^j$ of $f$, then we may take a new isotopy $\wh I^j$ from the identity to $\wh f^j$ defined as $\wh I^j(t, \wh z) = \wh I(jt, \wh z)$, and we then have that $(\wh f^j, \wh I^j, \wh \F)$ is a maximal dynamically transverse decomposition for $\wh f^j$ (notice that we have used the same singular foliation). We then also have the following equality for the transverse paths 
		\[ (\wh I^j)^n_{\wh \F}(\wh z) = \wh I^{jn}_{\wh \F}(\wh z), \text{ where } n \in \Z\]
		
		This implies that if $\wh I^{\Z}_{\wh \F}(\wh z^{\rar})$ has no $\wh \F$-transverse self-intersection, then $(\wh I^j)^{\Z}_{\wh \F}(\wh z^{\rar})$ also has no $\wh \F$-transverse self-intersection. 
		
		This in turn shows that if we take 
		\[q = q^{\rar} q^{\uar} q^{\lar} q^{\dar}, \ \underline{p} = |\mathrm{lcm}(p^{\rar}, p^{\uar}, p^{\lar}, p^{\dar})|,\]
		we have that 
		\[\wh f^{q\underline{p}}(\wh z^{\rar}) = \wh z^{\rar} + (\underline{p},0), \ \wh f^{q\underline{p}}(\wh z^{\uar}) = \wh z^{\uar} + (0,\underline{p}),\] 
		\[\wh f^{q\underline{p}}(\wh z^{\lar}) = \wh z^{\lar} + (-\underline{p},0), \ \wh f^{q\underline{p}}(\wh z^{\dar}) = \wh z^{\dar} + (0,-\underline{p}).\]

		For this $\widehat{z}^{\rightarrow}$, we can define $\wh A^{\rar}$ as the saturated set by leaves of $\wh \F$, of the transverse path $\wh I_{\wh \F}(\wh z^{\rar})$. By construction, we obtain that $\wh A^{\rar}$ is a trivially foliated \textit{horizontal strip} (\textit{i.e.} homeomorphic to a plane, and invariant under integer horizontal translations). In this context, we will say that $\wh A^{\rar}$ is a \textit{canonically foliated strip} (CFS), and that $\wh z^{\rightarrow}$ is a \textit{realizing point} for $\wh A^{\rightarrow}$.   
		
		Given that the diameter of leaves of $\wh \F$ is uniformly bounded (see Lemma \ref{lem:UniformlyBoundedLeaves}) and that $\wh A^{\rar}$ is invariant under integer horizontal translations, there exists $d^{\rar} \in \Z^+$ such that 
		\begin{equation}\label{eq:dRar}
			\mathrm{cl}(\wh A^{\rar}) \cap \mathrm{cl}(\wh A^{\rar} + (0,d')) = \varnothing, \text{ for every } d' \geq d^{\rar}
		\end{equation}

		In the same way as before, we get three analogous canonically foliated strips, which we will call $A^{\uparrow}, A^{\leftarrow}$ and $A^{\downarrow}$, with their respective realizing points $\wh z^{\uparrow}, \wh  z^{\leftarrow}$ and $\wh z^{\downarrow}$, and integers $d^{\uar}, d^{\lar}, d^{\dar}$ as in Equation \ref{eq:dRar}.

		Let us take $p$ the smallest multiple of $\underline{p}$ such that 
		\[p \geq \max\{d^{\rar}, d^{\uar}, d^{\lar}, d^{\dar}\} := d.\] 
		Again, we obtain that 
		\begin{equation}\label{eq:NiceRotatingPoint}
			\wh f^{qp}(\wh z^{\rar}) = \wh z^{\rar} + (p,0),
		\end{equation}
		and similarly for the other three realizing points. 
		
		Let us then take $\wh z^{\rar} = (\wh x^{\rar},\wh y^{\rar})$, such that $0 \leq \wh y^{\rar} < 1$, and build its corresponding CFS $\wh A^{\rar}$. Now, consider the family $\wh{\mathcal{A}}^{\rar} = \{\wh A^{\rar}_k\}_{k \in \Z}$ of closure-disjoint horizontal CFS, defined by 
		\begin{equation}\label{eq:HorizontalCFS}
			\wh A^{\rar}_k = \wh A^{\rar} + (0,pk),	
		\end{equation}
		and note that $(\wh x^{\rar}_k, \wh y^{\rar}_k) = \wh z^{\rar}_k = \wh z^{\rar} + (0,pk)$ is a realizing point for $\wh A^{\rar}_k$. We then have that 
		$$k = \lfloor \frac{\wh y^{\rar}_k}{p} \rfloor.$$ 
		We define in the same fashion the families 
		\begin{equation}\label{eq:CFSThree}
			\wh{\mathcal{A}}^{\uar} = \{\wh A^{\uar}_k\}_{k \in \Z}, \  \wh{\mathcal{A}}^{\lar} = \{\wh A^{\lar}_k\}_{k \in \Z} \  \text{and} \ \wh{\mathcal{A}}^{\dar} = \{\wh A^{\dar}_k\}_{k \in \Z},
		\end{equation}
		each made of closure-disjoint CFS, with the desired realizing points, which concludes the proof. 
	\end{proof}
	
	We will work with these four families in the following sections. We shall write 
	\begin{equation}\label{eq:CFSFamily}
		\wh{\mathcal{A}} = \wh{\mathcal{A}}^{\rar} \cup \wh{\mathcal{A}}^{\uar} \cup \wh{\mathcal{A}}^{\lar} \cup \wh{\mathcal{A}}^{\dar}.	
	\end{equation}
	
	We shall write $\wh A$ for an arbitrary canonically foliated strip in $\wh {\mathcal{A}}$.
	
	\begin{remark}\label{rem:ClosureOfStrip}
		We will sometimes work with the closure of our canonically foliated strips and use the same notation $\wh {A}^{\rightarrow, \uparrow,\leftarrow,\downarrow}$ for the sake of simplicity. In this context, we will refer to this closure as a closed CFS. However, we will refer to the open ones unless explicitly stated.
	\end{remark}

	\medskip
	
	\subsection{Crossing continua}
	
	We will introduce the notion of crossing continuum for a closed CFS, which is central for the remainder of our work, and describe the topology of the strips.  
	
	\medskip

	\paragraph{\textbf{Boundary components: Top and Bottom}}
	
	Note that by the Uniformization Theorem, we know that the complement of a CFS $\wh A$ has no bounded connected components. Moreover, there are exactly two connected components on this complement, because $\wh A$ is invariant under the integer translation given by its realizing point, but bounded in the orthogonal direction (recall that the diameter of leaves of $\wh \F$ is uniformly bounded by Lemma \ref{lem:UniformlyBoundedLeaves}). With this uniform bound we also obtain that the boundary $\partial \wh A$ has exactly two connected components.
	
	\begin{definition}
		Given a CFS $\widehat{A}$ with a realizing point $\widehat{z}$, a point $\widehat{w}$ on its complement will be \textit{on its top} (and we will write $\widehat{w} \in \textnormal{T}(\widehat{A})$) if it belongs to the connected component of its complement included in $\mathrm{R}(\widehat{I}_{\widehat{F}}^{\mathbb{Z}}(\widehat{z}))$, and it will be \textit{on its bottom} ($\widehat{w} \in \textnormal{B}(\widehat{A})$) otherwise.
	\end{definition}

	In addition, we will write $\partial \widehat{A} = \partial_{\textnormal{T}} \widehat{A} \cup \partial_{\textnormal{B}} \widehat{A}$, where 
	$$ \partial_{\textnormal{T}} \widehat{A}= \partial  (\textnormal{T} (\widehat{A})), \ \textnormal{and} \  \partial_{\textnormal{B}} \widehat{A} = \partial  ({\textnormal{B}} (\widehat{A})).$$
	
	Note that both $\partial_{\TT} \wh A^{\rar}$ and $\partial_{\BB} \wh A^{\rar}$ are connected, since $\wh A^{\rar}$ is bounded in the vertical direction. 
	
	{\bck{Given that $\wh A$ is equivariant by integer translations given by its realizing point, and that it is crossed by infinitely many CFS in both orthogonal directions}}, we get that each boundary component $\partial_{\text{T}} \widehat{A}$ and $\partial_{\text{B}} {\bck {\wh A}}$ of a {\bck{CFS}} contains infinitely many singularities of the foliation. This means that if we take a lift $\tilde{A} \subset \tl \D$ of $\widehat{A} \subset \R^2$ (which we will also call CFS), then its boundary and its complement will both have infinitely many connected components, each one determined by one leaf of $\partial\widehat{A}$. We will now settle notation for these components. 
	
	\begin{definition}
		Let $\widehat{A}$ be a canonically foliated strip, $\tilde{A}$ a lift of $\widehat{A}$ to ${\tl {\mathbb{D}}}$.  
		We will write $$\partial_{\textnormal{T}}\tilde{A} := \partial \tilde{A} \cap \{\tilde{\pi}^{-1}(\widehat{\phi}_{\textnormal{T}}) : \widehat{\phi}_{\textnormal{T}} \in \partial_{\textnormal{T}}\widehat{A} \}, \ \partial_{\textnormal{B}}\tilde{A} := \partial \tilde{A} \cap \{\tilde{\pi}^{-1}(\widehat{\phi}_{\textnormal{B}}) :  \widehat{\phi}_{\textnormal{B}} \in \partial_{\textnormal{B}}\widehat{A} \}.$$
	\end{definition}
	
	Moreover, we will write $\text{T}(\tilde{A})$ for the {\bck{union of the}} (closed) connected components of the complement of $\tilde{A}$ which are bounded by leaves $\tilde{\phi}_{\text{T}} \subset \partial_{\text{T}}\tilde{A}$, and similarly we will write $\text{B}(\tilde{A})$ for the {\bck{union of the}} (closed) connected components of the complement of $\tl A$ which are bounded by leaves $\tilde{\phi}_{\text{B}} \subset \partial_{\text{B}}\tilde{A}$. 
	
	\begin{definition}
		Given a lift $\tilde{A}$ of a canonically foliated strip, and $\tilde{\phi}_{\textnormal{T}} \subset \partial_{\textnormal{T}} \tilde{A}$, we will write $\textnormal{T}_{\tilde{\phi}_{\textnormal{T}}}(\tilde{A})$ for the closed connected component of $\tilde{A}^{\textnormal{C}}$ bounded by $\tilde{\phi}_{\textnormal{T}}$ (similar fashion for $\tilde{\phi}_{\textnormal{B}} \subset \partial_{\textnormal{B}} \tilde{A}$).
	\end{definition}

	\begin{figure}[ht]
		\centering
		
		\def\svgwidth{.92\textwidth}
		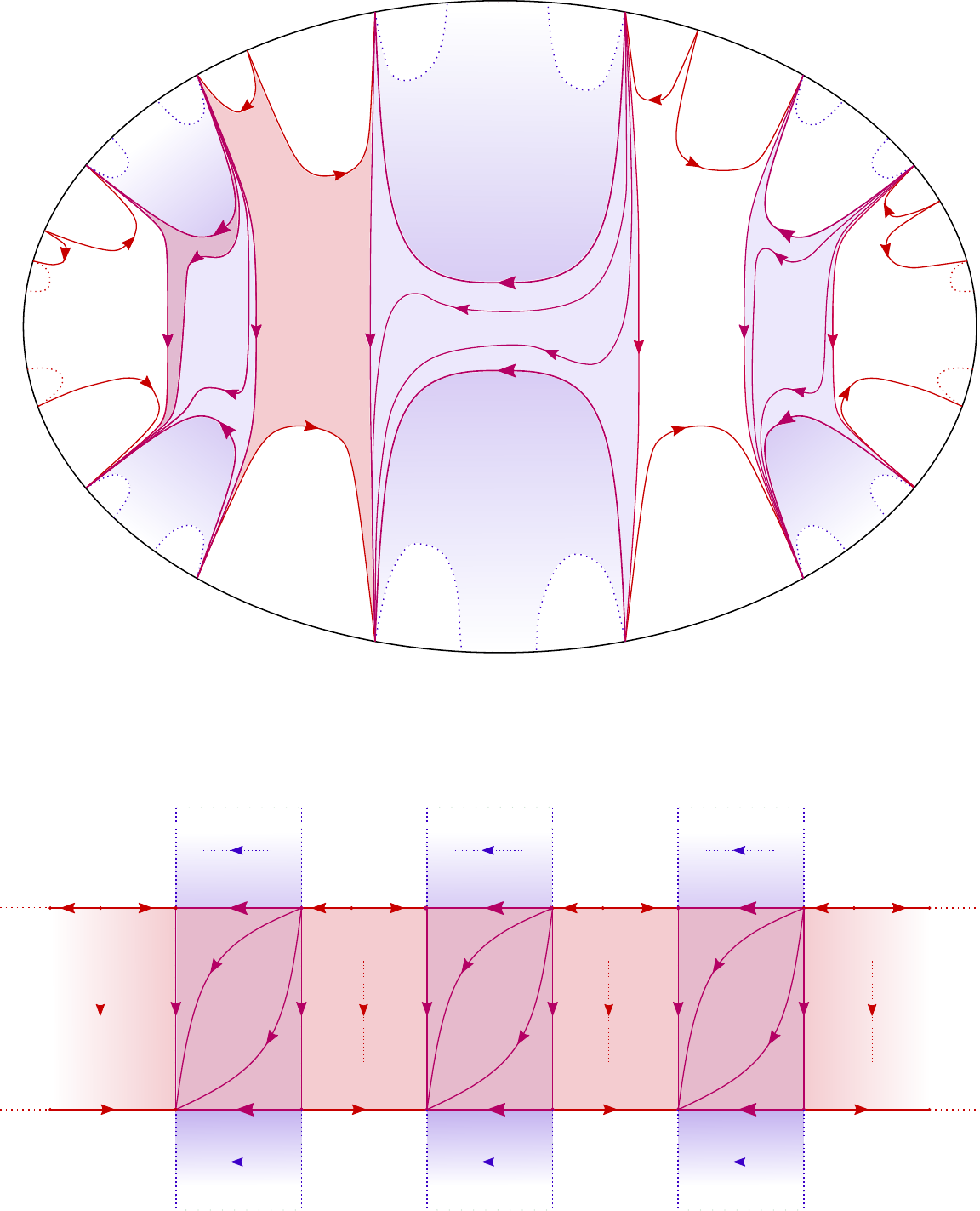
		
		\bigskip

		\caption{A canonically foliated strip $\widehat{A}^{\rightarrow}$, with two boundary connected components, and one of its lifts, with infinitely many of them.  }
		\label{figure:BaseCFS}
	\end{figure}
	
	\medskip
	
	\paragraph{\textbf{Crossing continua and diameter measurement}}
	
	\begin{definition}[\textbf{Crossing continuum}]
		Given a closed CFS ${\widehat{A}}$ and a continuum $\widehat{K} \in \mathbb{R}^2$, we will say that $\widehat{K}$ is a  \textit{crossing continuum} (relative to ${\widehat{A}}$) if it intersects both connected components of $\partial \wh A$. In addition, a crossing continuum will be \textit{minimal} if it does not strictly contain any crossing subcontinua. 
	\end{definition}
	
	\begin{remark}
		Given $\widehat{A}$ a closed CFS, any crossing continuum $\widehat{K}$ has a minimal crossing subcontinuum $\widehat{K}' \subset \widehat{A}$. 
		In particular, $\widehat{K}'$ intersects {\bck{at least}} one leaf (or singularity of the foliation) on each connected component of $\partial \widehat{A}$.  
	\end{remark}
	
	\begin{definition}[\textbf{Vertical diameter}]
		Given a continuum $\widehat{K} \in \mathbb{R}^2$, we will define its \textit{vertical diameter} $\textnormal{Vdiam}(\widehat{K})$ as the supremum of the difference of the $y$-coordinates in pairs of points belonging to $\widehat{K}$. 
	\end{definition}
	
	The \textit{horizontal diameter} of a continuum, $\textnormal{Hdiam}(\widehat{K})$ is defined analogously, looking at the $x$-coordinates. For inessential continua $K \subset \T^2$ we recall that we take their lifted diameter as $\text{diam}(K) = \text{diam}(\wh K)$, where $\wh K$ is any of its lifts to $\R^2$. Accordingly, if we go to the universal covering $\tl \D$ of $\textnormal{Dom}(\wh I)$ and take $\tl K \subset \tl {\mathbb{D}}$, we will write $\text{diam}(\tl K) = \text{diam}(\wh K) = \text{diam}(\tl \pi (\tl K))$. The same holds for horizontal and vertical diameter.
	
	We emphasize that the (lifted) diameter of any continuum (either included in $\T^2$ or in $\tl \D$) is always measured in $\R^2$, respectively taking a lift or projecting.
	
	{\bck
	{
	\begin{remark}\label{rem:LargeCrosses}
		By Lemma \ref{lem:UniformlyBoundedLeaves} and Proposition \ref{prop:ExistenceofCFS}, there exists $C > 0$ such that if $\textnormal{Vdiam}(\wh K) > C$, there exists $k_0$ such that $\wh K$ is a crossing continuum of both $\wh A^{\rar}_{k_0}$ and $\wh A^{\lar}_{k_0}$. 
	\end{remark}
	}
	}
	
	\bigskip
	
	\paragraph{\textbf{Inside the strip: Left and Right}}
	
	\begin{definition}
		Given a canonically foliated strip $\widehat{A}$ with a realizing point $\wh z$ and a crossing continuum $\wh K$, we say a point $\widehat{w} \in \widehat{A} \backslash \widehat{K}$ is \textit{on its left} ($\widehat{w} \in \text{L}_{\widehat{A}}(\widehat{K})$) if it belongs to the connected component of $\widehat{A} \backslash \widehat{K}$ which contains $\wh f^j (\widehat{z})$ for arbitrarily large values of $j$, and it is \textit{on its right} ($\wh w \in \text{R}_{\widehat{A}}(\widehat{K})$) if it belongs to the connected component of $\widehat{A} \backslash \widehat{K}$ which contains $\wh f^{j} (\widehat{z})$ for arbitrarily large negative values of $j$. 
	\end{definition}
	
	These notions lift naturally to the lifts $\tl A \subset \tl \D$ of our canonically foliated strips. Moreover, note that each of the two ends of the strip belong to one of these connected components (different from each other). For the case of CFS $\tilde{A}$ in the universal covering $\tilde{\mathbb{D}}$, we may identify each of these ends $l_{\tilde{A}}$ (the future) and $r_{\tilde{A}}$ (the past) with a point in the boundary of $\tilde{\mathbb{D}}$. 
	
	Note that the closure of a leaf ${\widehat{\phi}} \subset \wh A$ is a crossing continuum for $\wh A$. We may then apply the \textit{left-right} notion to leaves $\widehat{\phi} \subset \widehat{A}$ by taking their closure: we will say that $\widehat{z} \in \textnormal{L}_{\widehat{A}}(\widehat{\phi})$ whenever $\widehat{z} \in \textnormal{L}_{\widehat{A}}(\mathrm{cl}(\widehat{\phi}))$.  Analogously, $\widehat{z} \in \textnormal{R}_{\widehat{A}}(\widehat{\phi})$ if $\widehat{z} \in \textnormal{R}_{\widehat{A}}(\mathrm{cl}(\widehat{\phi}))$   
	
	This leads us to the following:
	
	\begin{remark}
		Let $\widehat{A}$ be a CFS and $\widehat{\phi} \subset \widehat{A}$ a leaf. Then, if $\tilde{A} \subset \tl \D$ is a lift of $\widehat{A}$, we obtain that  $$\textnormal{L}_{\widehat{A}}(\widehat{\phi}) = \tilde{\pi}(\text{L}(\tilde{\phi}) \cap \tilde{A}), \text{ and } \textnormal{R}_{\widehat{A}}(\widehat{\phi}) = \tilde{\pi}(\text{R}(\tilde{\phi}) \cap \tilde{A})$$
	\end{remark}

	\begin{remark}
		Given the relative orientation between the transverse foliation and the paths by the isotopy, the notions of Top, Bottom, Left and Right we just defined coincide with the visual concepts, for canonically foliated strips $\wh A^{\lar}$, where the realizing point goes \textit{from right to left}.  
	\end{remark}

	\paragraph{\textbf{Fundamental domains}}
	
	Recall the definition of $\wh {\mathcal
		A}$ from Equations \ref{eq:HorizontalCFS} through \ref{eq:CFSFamily}, and let us explain the notion of fundamental domains for a CFS $\wh A^{\rar} \subset \wh {\mathcal{A}}^{\rar}$. Recall that every $\wh A^{\dar}_k \in \wh {\mathcal{A}}^{\dar}$ must cross $\wh A^{\rar}$ from bottom to top at least once. 
	
	Then, we may take $\wh \phi^{\dar}_{\TT,0} \subset \wh A^{\rar} \cap \partial_{\TT}\wh A^{\dar}_0$, and then define 
	$$\wh \phi^{\dar}_{\TT,k} = \wh \phi^{\dar}_{\TT,0} + (pk,0),$$
	where $p$ is as in Equation \ref{eq:NiceRotatingPoint}. 
	
	These leaves define the fundamental domains 
	$$ \wh D^{\rar}_k = \text{L}_{\wh A^{\rar}}(\wh \phi^{\dar}_{\TT,k}) \cap \text{R}_{\wh A^{\rar}}(\wh \phi^{\dar}_{\TT,k+1}).$$
	
	Note that all of these fundamental domains behave nicely when taking natural lifts, therefore obtaining that any $\tl A^{\rar}$ is divided in the following domains:
	$$ \tl D^{\rar}_k = \text{L}_{\tl A^{\rar}}(\tl \phi^{\dar}_{\TT,k}) \cap \text{R}_{\tl A^{\rar}}(\tl \phi^{\dar}_{\TT,k+1}).$$
	
	This construction can be accordingly made with any other canonically foliated strip  $\wh A^{\uar}, \wh A^{\lar}, \wh A^{\dar}$ and their respective lifts to the universal covering. We emphasize that each of these fundamental domains has \textit{length} $p$, more precisely, we have that $\wh \phi^{\dar}_{\TT,k} \in \partial_{\TT}\wh A^{\dar}_k$. 
	
	Let us finish the section with a result on how points travel through fundamental domains. 
	
	\begin{lemma}\label{lem:PointsMoveSlowly}
		Let $\wh A \in \wh {\mathcal{A}}$ be a CFS, $\wh \phi \subset \wh A$ a leaf in the fundamental domain $\wh D_k$, and $n > 0$ an integer.
		
		Then, there exists $k'$ such that for any leaf $\wh \phi' \subset \wh D_{k'} \subset \wh A$, when we take a set of natural lifts $\tl A, \tl \phi, \tl \phi'$ to $\tl \D$, we have that
		
		\[\tl f^{j}(\tl \phi') \subset \mathrm{R}(\tl \phi), \text{ for every } 0 \leq j \leq n.\]

	\end{lemma}
	
	Let us write the proof for a horizontal CFS $\wh A^{\rar}$, as the other cases are analogous. 
	
	\begin{proof}
		Fix $n > 0$, a leaf $\phi \subset \wh A^{\rar} \in \wh{\mathcal{A}}^{\rar}$, and let $\wh D_k$ be the fundamental domain of $\wh A$ such that $\phi \subset \wh D^{\rar}_k$ (see the construction above). Given that $\wh f$ is the lift to $\R^2$ of an isotopic-to-the-identity torus homeomorphism (which is uniformly continuous), we have that there exists $M > 0$ such that
		
		\begin{equation}\label{eq:UniformlycontinuousWhF}
			\sup \limits_{\wh z \in \R^2} \mathrm{d}(\wh f(\wh z), \wh z) \leq M, \text{  which implies that  } \sup \limits_{\wh z \in \R^2} \mathrm{d}(\wh f^{n}(\wh z), \wh z) \leq nM.	
		\end{equation}

		Now, recall that by Lemma \ref{lem:UniformlyBoundedLeaves}, the diameter of leaves is uniformly bounded, which implies that the closure $\mathrm{cl}(\wh D^{\rar})$ of every fundamental domain is compact. Given fundamental domains are obtained by taking horizontal translations of one another, we recover that there exists $k' < k$, such that 
		\[\mathrm{d}(\wh z', \wh D^{\rar}_k) > 2nM, \ \text{for every } \wh z' \in \wh D^{\rar}_{k'}.\]
		
		This, together with Equation \ref{eq:UniformlycontinuousWhF}, implies that taking any leaf $\wh \phi' \subset \wh D^{\rar}_{k'}$, and a natural set of lifts $\tl \phi', \ \tl \phi$ of $\wh \phi'$ and $\wh \phi$ to $\tl \D$, we have that  
		\[\tl f^{j}(\tl \phi') \subset \mathrm{R}(\tl \phi), \text{ for every } 0 \leq j \leq n,\]
		which concludes the proof.   
	\end{proof}

	\section{Anchoring techniques}\label{section:stretching}
	
	The goal of this section is to prove Theorem \ref{PropC:maximumdiameter}, whose statement we now recall (see Definition \ref{def:DynamicalDiameter} for details),
	
	\medskip  
	
	\paragraph{\textbf{Theorem C}} \textit{Let $f \in \textnormal{Homeo}_0(\mathbb{T}^2)$, such that  $\rho(f)$ has nonempty interior. Then, there exists $M >0$ such that for any continuum $K \subset \mathbb{T}^2$, $\mathscr{D}_f(K) > M \implies\mathscr{D}_f(K) = \infty$.} 
	
	\medskip
	
	The idea of the proof resides in the canonically foliated strips we built in the previous section, and heavily uses Brouwer-Le Calvez's dynamically transverse decomposition. We will use the construction and notation of Section \ref{sec:CFS}, and recall that up to taking a power, we may assume that {\bck{$\vec{0} \in \mathrm{int}(\rho(\wh f))$.}} 
	
	\begin{remark}
		Let $f \in \mathrm{Homeo_0}(\mathbb{T}^2)$, a continuum $K \subset \mathbb{T}^2$ and $n > 0$. Then ${\mathscr{D}}_f(K) = \infty \iff {\mathscr{D}}_{f^n}(K) = \infty$
	\end{remark}
	
	This remark shows it is enough to prove Theorem \ref{PropC:maximumdiameter} for {\bck{some}} power of $f$. 
	
	\begin{remark}\label{rem:NicePowerOfF}
		Given each of the properties from Theorem \ref{thmA:semiconjugation} can be checked up to taking a power of $f$, we will assume that $f$ has a lift $\wh f$ with $\vec{0}$ in $\mathrm{int}(\rho(\wh f)) $, and we will work with such a lift from now on. 
	\end{remark}
	
	We will now prove Theorem \ref{PropC:maximumdiameter} for $f^{qp}$, where $q$ and $p$ are as in Equation \ref{eq:NiceRotatingPoint}. Let us then rename $f := f^{qp}$, $\wh f = (\wh f)^{qp}$ (which still has the origin in the interior of its rotation set). This implies that 
	
	\begin{equation}
		\wh f(\wh z^{\rar}) = \wh z^{\rar} + (p,0),
	\end{equation}
	
	and analogously for any other realizing point in any direction. The key fact here is 
	
	\begin{remark}\label{rem:FastAdvancement}
		Each realizing point from a CFS advances one fundamental domain per iteration of $\wh f$. 
	\end{remark}

	\subsection{Semianchors}\label{subsection:Preparingtheground}

	Note that if $\text{diam}(\widehat{K}) > M$, then we either have $\textnormal{Vdiam}(\widehat{K}) > \frac{\sqrt{2}M}{2}$ or $\textnormal{Hdiam}(\widehat{K}) > \frac{\sqrt{2}M}{2}$. This, together with the symmetry of this section's hypothesis, shows that in order to prove Theorem \ref{PropC:maximumdiameter}, it is enough to prove that there exists $M > 0$ such that for every continuum $K \subset \mathbb{T}^2$,
	$$ \textnormal{Vdiam}(K) > M \implies \mathscr{D}_f(K) = \infty. $$
	
	In particular, note that continua with large vertical diameter will be crossing, relative to some horizontal canonically foliated strip. Let us assume that $\textnormal{Vdiam}(\widehat{K})$ is sufficiently large so that there exists {\bck{$k_0 \in \mathbb{Z}$}} such that $\widehat{K}$ goes through $\widehat{A}^{\rightarrow}_{k_0}$ {\bck{(see Remark \ref{rem:LargeCrosses})}}. Whenever there is no ambiguity for this value of $k_0$, we will fix this horizontal strip and denote it as $\widehat{A}^{\rightarrow}$.  
	
	For the remainder of this subsection and the following one, let us assume that $K \subset \text{Dom}(I)$, that is, $K$ does not intersect any singularity.  
	
	\begin{lemma}\label{lemma:ThereExistsM_n_o}
		Let $n \in \mathbb{Z}^{+}$. Then, there exists $M_{n} > 0$ such that:
		
		For any continuum $\widehat{K}$ with $\textnormal{Vdiam}(\widehat{K}) \geq M_{n}$, there exists a horizontal $\textnormal{CFS}$ $\widehat{A}^{\rightarrow} \in \wh{\mathcal{A}}^{\rar}$, and two points $\widehat{z}_{\textnormal{TT}}, \widehat{z}_{\textnormal{BB}} \in \wh K$ such that
		
		\begin{equation}\label{equation:SemiAnchoringWidehat}
			\widehat{I}^t(\widehat{\phi}_{\widehat{z}_{\textnormal{TT}}}) \in \textnormal{T}(\widehat{A}^{\rightarrow}) \ \textnormal{and} \ \widehat{I}^t(\widehat{\phi}_{\widehat{z}_{\textnormal{BB}}}) \in \textnormal{B}(\widehat{A}^{\rightarrow}),
			\ \text{for every} \ -n \leq t \leq n.  
		\end{equation}	
		
	\end{lemma}

	\begin{proof}

		Fix $n > 0$. Let us then start by stating the three following facts:
		
		\begin{itemize}
			\item The vertical diameter of any CFS $\widehat{A}^{\rightarrow} \in \wh {\mathcal{A}}^{\rar}$ is uniformly upper-bounded. Let us then take $d_1$ an upper bound for that quantity. 
			\item As proven in Lemma \ref{lem:UniformlyBoundedLeaves}, the diameter of leaves is also uniformly bounded. Let us then take $d_2$ a bound for this diameter. 
			\item The length of the vectors $\{\widehat{I}^t(\widehat{z}) - \widehat{z} \ : \ \widehat{z} \in \mathbb{R}^2,  |t| < n \}$ is also uniformly bounded. This is because $\wh I$ is $\Z^2$-periodic, given it is defined as the lift of an isotopy in the torus, which is in turn compact. Let us then take $d_3$ an upper bound for this family of lengths.   
		\end{itemize}
		
		Finally, let us take
		\[M_n = 2(d_1 + d_2 + d_3) + p,\]
		where $p$ is as in Equation \ref{eq:NiceRotatingPoint}. Given a continuum $\wh K$ with $\mathrm{Vdiam}(\wh K) \geq M_n$, we may take {\bck{$\wh z_{\TT \TT}, \wh z_{\BB \BB} \in \wh K$}} with vertical distance equal to $M_n$. This means that  
		\[\mathrm{Vdist} \Bigl( \bigcup \limits_{|t| \leq n} \widehat{I}^t(\widehat{\phi}_{\widehat{z}_{\textnormal{TT}}}), \bigcup \limits_{|t| \leq n} \widehat{I}^t(\widehat{\phi}_{\widehat{z}_{\textnormal{BB}}}) \Bigr) \geq 2d_1 + p.\]
		This means that there exists a strip $\R \times [y,y+2d_1+p] \subset \R^2$ between these two sets, which allows us to take a realizing point $\wh z^{\rar}$ with vertical coordinate between $y+d_1$ and $y+d_1 + p$, whose associated canonically foliated strip $\wh A^{\rar}$ will satisfy Equation \ref{equation:SemiAnchoringWidehat}, which concludes the proof.  
	\end{proof}

	\begin{remark}
		An identical proof works with $\wh A^{\lar}$ for continua with large vertical diameter, and with $\wh A^{\uar}, \wh A^{\dar}$ for continua with large horizontal diameter. 
	\end{remark}
	
	\begin{definition}[\textbf{Semianchor}]
		Given $n \in \mathbb{Z}^+$, a CFS $\tl A$, a crossing continuum $\tl K$ and two leaves $\tl \phi_{\BB} \subset \partial_{\BB}\tl A, \phi_{\TT} \subset \partial_{\TT}\tl A$, we will say that the quartet $(\tilde{A}, \tilde{K}, \tilde{\phi}_{\textnormal{B}}, \tilde{\phi}_{\textnormal{T}})$ is an {\bck \emph{n-semianchor}} when there exist $\tl z_{\TT \TT}, \tl z_{\BB \BB} \in \tl K$ 
		such that 
		$$\tilde{f}^j(\tilde{\phi}_{\tilde{z}_{\textnormal{TT}}}) \subset \textnormal{T}_{\tilde{\phi}_{\textnormal{T}}}(\tilde{A}), \ \tilde{f}^j(\tilde{\phi}_{\tilde{z}_{\textnormal{BB}}}) \subset \textnormal{B}_{\tilde{\phi}_{\textnormal{B}}}(\tilde{A}), \ \text{for every} \ j \in \mathbb{Z}; \ |j| \leq n.$$
		
		{{In this context, we will say that $\tl K$ is an \emph{n-semianchored} continuum (or simply that it is \textit{n-semianchored}).}}

	\end{definition}
	
	The following consequence of Lemma \ref{lemma:ThereExistsM_n_o} is key to prove Theorem \ref{PropC:maximumdiameter}.
	
	\begin{lemma}\label{lemma:SemiAnchoring}
		Let $n \in \mathbb{Z}^{+}$. Then, there exists $M_n > 0$ such that for any $\tl K$ in $\tilde{\mathbb{D}}$ with $\textnormal{Vdiam}(\tl K) > M_n$, there exist $\tl A, \tl \phi_{\TT} \subset \partial_{\TT}\tl A, \tl \phi_{\BB} \subset \partial_{\BB}\tl A$, such that
		$$ (\tl A, \tl K, \tl \phi_{\BB}, \tl \phi_{\TT}) \text{ is an \textit{n}-semianchor.} $$

	\end{lemma}
	
	\begin{proof}
		Take $\wh A$, $\wh z_{\text{TT}}, \wh z_{\text{BB}}$ as in Lemma \ref{lemma:ThereExistsM_n_o}, and take natural lifts $\tl A, \tl z_{\text{TT}}, \tl z_{\text{BB}}$. Note that $\tilde{z}_{\text{TT}} \in \text{T}_{\tilde{\phi}_{\text{T}}}(\tilde{A}), \tilde{z}_{\text{BB}} \in \text{B}_{\tilde{\phi}_{\text{B}}}(\tilde{A})$, for some unique leaves $\tilde{\phi}_{\text{T}} \subset \partial_{\text{T}}\tilde{A}, \  \tilde{\phi}_{\text{B}} \subset \partial_{\text{B}}\tilde{A}$. By the statement \ref{equation:SemiAnchoringWidehat} in Lemma \ref{lemma:ThereExistsM_n_o}, we immediately obtain that 
		$$\tilde{I}^j(\tilde{\phi}_{\tilde{z}_{\textnormal{TT}}}) \in \textnormal{T}_{\tilde{\phi}_{\textnormal{T}}}(\tilde{A}^{\rightarrow}) \ \textnormal{and} \ \tilde{I}^j(\tilde{\phi}_{\tilde{z}_{\textnormal{BB}}}) \in \textnormal{B}_{\tilde{\phi}_{\textnormal{B}}}(\tilde{A}^{\rightarrow}), \text{ for every } j \text{ with } |j| \leq n,$$
		which concludes the proof.
	\end{proof}

	\begin{lemma}\label{lemma:semianchorstillpushes}
		Let $n \in \mathbb{Z}^+$, $(\tilde{A}, \tilde{K}, \tilde{\phi}_{\textnormal{B}}, \tilde{\phi}_{\textnormal{T}})$ an $n-$semianchor, and $\tilde{z}$ a realizing point for $\tilde{A}$, with $\tilde{z} \in \textnormal{R}_{\tilde{A}}(\tilde{K})$. 
		
		Then, for every $j \in \mathbb{Z}$ with $|j| \leq n$, we have that $\tl f^{j}(\tl K)$ is a crossing continuum for $\tl A$, and  $\tilde{f}^j(\tilde{z}) \notin \textnormal{L}_{\tilde{A}}(\tilde{f}^{j}(\tilde{K}))$. 
	\end{lemma}

	\begin{proof}
		Note that $\Gamma_0 := \tilde{\phi}^{-}_{\tl z_{\text{TT}}} \cup \tilde{K} \cup \tilde{\phi}^+_{\tl z_{\text{BB}}}$ separates the ends of $\tilde{A}$, and leaves $\tilde{z}$ on the same component as $r_{\tilde{A}}$. Name $\text{R}(\Gamma_0)$ to this component, and let $\text{L}(\Gamma_0)$ be the one containing the other end. 
		
		{\bck{Given $(\tilde{A}, \tilde{K}, \tilde{\phi}_{\textnormal{B}}, \tilde{\phi}_{\textnormal{T}})$ is an $n-$semianchor,}} we know that if $|j| < n$, then $\tl f^{j}(\tl \phi_{z_{\TT \TT}}) \cap \tl A
		= \varnothing$ and also $\tl f^{j}(\tl \phi_{z_{\BB \BB}}) \cap \tl A = \varnothing$, which implies that $\tl f^{j}(\tl K)$ is crossing for $\tl A$ through the leaves $\tl \phi_{\TT}$ and $\tl \phi_{\BB}$. Furthermore, $\tl f^{j}(\tl z) \in \tl f^{j}(\text{R}(\Gamma_0))$, and $\textnormal{L}_{\tilde{A}}(\tilde{f}^{j}(\tilde{K})) \subset \tl f^{j}(\text{L}(\Gamma_0))$, which concludes the proof.   	 
	\end{proof}

	We do the whole dynamical arguments in the universal covering $\tl {\mathbb{D}}$, 
	but we measure distances in $\R^2$. Let us then define:
	
	\begin{definition}
		We will say that $(\wh A, \wh K, \wh \phi_{\BB}, \wh \phi_{\TT})$ is an \textit{n}-semianchor if there exist lifts $(\tl A, \tl K, \tl \phi_{\TT}, \tl \phi_{\BB})$ forming an \textit{n}-semianchor. We will then say that $\wh K$ is \textit{n}-semianchored.  
	\end{definition}

	\subsection{The Anchoring Lemma}
	
	\begin{lemma}[\textbf{Anchoring Lemma}]\label{lemma:AnchoringLemma}
		Let $\tilde{A}$ be a canonically foliated strip, $\tl K$ a crossing continuum, and $\tl \phi_{\TT} \subset \TT (\tl A), \tl \phi_{\BB} \subset \BB (\tl A)$ such that 
		\begin{itemize}
			\item $\tilde{K} \cap \tilde{\phi}_{\textnormal{T}} \neq \varnothing$, $\tilde{K} \cap \tilde{\phi}_{\textnormal{B}} \neq \varnothing$
			\item $\tilde{A} \subset \textnormal{R}(\tilde{\phi}_{\textnormal{T}}) \cap \textnormal{R}(\tilde{\phi}_{\textnormal{B}})$.  
		\end{itemize}
		
		Then, we have that 
		\begin{enumerate}
			\item For every $n > 0$, $\tl f ^n (\tl K)$ intersects both $\tl \phi_{\TT}$ and $\tl \phi_{\BB}$
			\item For every leaf $\tl \phi \subset \tl A$ with $\tl \phi_{\TT}, \tl \phi_{\BB} \subset \RR(\tl \phi)$, there exists $n_0 > 0$ such that 
			$$ \tl f^{n}(\tl K) \cap \tl \phi \neq \varnothing, \text{ for every } n \geq n_0.$$
		\end{enumerate}
	\end{lemma}
	
	\smallskip
	
	\begin{proof} Let $\tl z_{\TT} \in \tl \phi_{\TT} \cap \tl K, \tl z_{\BB} \in \tl \phi_{\BB} \cap \tl K$ 
		\begin{enumerate}
			\item Fix an integer $n>0$. Note that we have $$\tilde{f}^n(\tilde{z}_{\text{T}}) \in \text{L}(\tilde{\phi}_{\text{T}}) \textnormal{  and  }  \tilde{f}^n(\tilde{z}_{\text{B}}) \in \text{L}(\tilde{\phi}_{\text{B}}),$$ from where we obtain that $$\text{L}(\tilde{\phi}_{\text{T}}) \cap \tilde{f}^{n}(\tilde{K}) \neq \varnothing \textnormal{ and } \text{L}(\tilde{\phi}_{\text{B}}) \cap \tilde{f}^{n}(\tilde{K}) \neq \varnothing.$$
			Given that $\tilde{A}^{\rightarrow} \subset \text{R}(\tilde{\phi}_{\text{T}}) \cap \text{R}(\tilde{\phi}_{\text{B}})$, we conclude that $\tilde{f}^{n}(\tilde{K})$ is a crossing continuum for $\tilde{A}^{\rightarrow}$ which intersects both $\tilde{\phi}_{\text{T}}$ and $\tilde{\phi}_{\text{B}}$ {\bck{(see Figure \ref{figure:AnchoringLemma} for details).}}
			
			\item Let $\tilde{z}_0$ be a realizing point for $\tl A$ such that $\tl z_0 \in \text{R}_{\tl A}(\tl K)$, and let us name $\tl z_n = \tl f^{n}(\tl z_0)$.  Take two semi-leaves $\tilde{\gamma}_{\text{T}} = \tilde{\phi}^-_{\tilde{z}_\text{T}}, \tilde{\gamma}_{\text{B}} = \tilde{\phi}^+_{\tilde{z}_\text{B}}$, and note that   $$\Gamma_0 = \tilde{\gamma}_{\text{T}} \cup \tilde{K} \cup \tilde{\gamma}_{\text{B}}$$ 
			separates the disk $\tl \D$ into at least two connected components, one containing $\textnormal{L}_{\tilde{A}}(\tilde{K})$, which we will call $\text{L}(\Gamma_0)$, and another one containing $\textnormal{R}_{\tilde{A}}(\tilde{K})$, which we will call $\text{R}(\Gamma_0)$ {\bck{(once again, see Figure \ref{figure:AnchoringLemma}).}}
			
			Fix a leaf $\tl \phi \subset \tl A$ {\bck{with $\tl \phi_{\TT}, \tl \phi_{\BB} \subset \RR(\tl \phi)$}}. Take a sufficiently large value of $n_0$, such that $\tl z_{n_0} \in \text{L}_{\tl A}(\tl \phi)$. Note that for every $n>0$, $\tilde{z}_n$ still belongs to $\tilde{A}$ and also $$\tilde{z}_n \in \text{R}(\tilde{f}^{n}(\Gamma_0)) = \tilde{f}^{n} (\text{R}(\Gamma_0)),$$ which  {\bck{given that $\tl \phi_{\TT}, \tl \phi_{\BB} \subset \RR(\tl \phi)$}}, implies that 
			$$ \tl f^{n}(\Gamma_0)\cap \tl \phi \neq \varnothing, \text{ for every } n \geq n_0.$$ 
			Recalling that $\tilde{f}^{n}(\tilde{\phi}_{\text{T}}) \in \text{L}(\tilde{\phi}_{\text{T}})$,  $\tilde{f}^{n}(\tilde{\phi}_{\text{B}}) \in \text{L}(\tilde{\phi}_{\text{B}})$, we obtain that 
			$$ \tl f^{n}(\tl \phi_{\TT}) \cap \tl \phi = \varnothing; \ \tl f^{n}(\tl \phi_{\BB}) \cap \tl \phi = \varnothing \text{ for every }n > 0$$
			which concludes the proof.  
		\end{enumerate}
	\end{proof}	
	
	\begin{figure}[ht]
		\centering
		
		\def\svgwidth{.92\textwidth}
		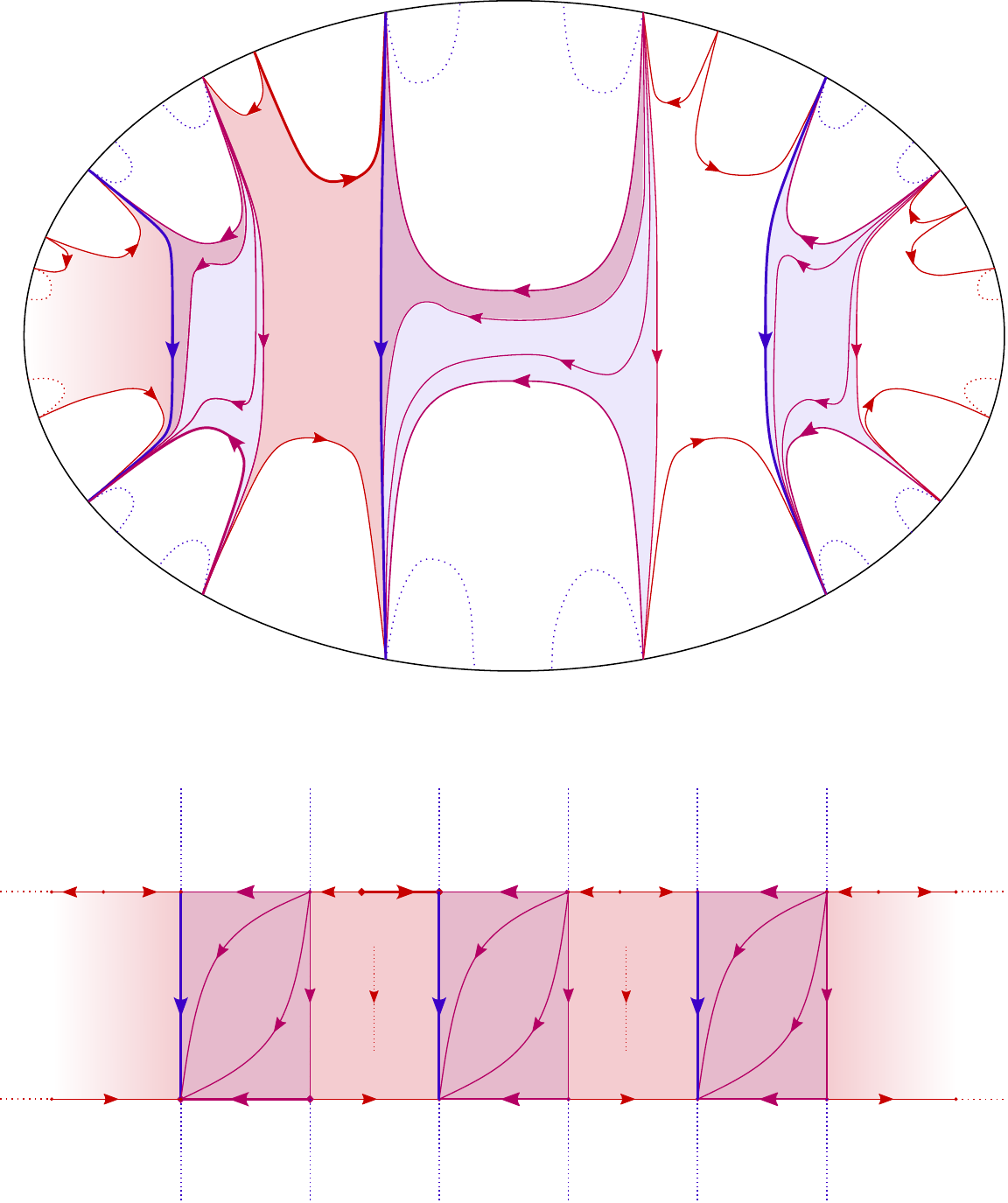
		
		\smallskip
		
		\caption {The sequence of iterates $\{\tilde{z}_n\}$ \emph{pushes} $\tilde{K}$, the leaves $\tilde{\phi}_{\text{T}}, \tilde{\phi}_{\text{B}}$ serve as anchors fixing it. The blue leaves in $\R^2$ define fundamental domains in the canonically foliated strip. The intersections between $\tl A^{\dar}_k$ and $\tl A^{\rar}$ are in purple, the intersections between $\tl A^{\uar}_k$ and $\tl A^{\rar}$ are included in the red regions (and similarly for their projections to $\R^2$.)}
		\label{figure:AnchoringLemma}
	\end{figure}
	
	\begin{remark}
		If we exchange $$\tilde{A} \subset \textnormal{R}(\tilde{\phi}_{\textnormal{T}}) \cap \textnormal{R}(\tilde{\phi}_{\textnormal{B}}) \text{ for } \tilde{A} \subset \textnormal{L}(\tilde{\phi}_{\textnormal{T}}) \cap \textnormal{L}(\tilde{\phi}_{\textnormal{B}}),$$
		in the hypothesis of the Anchoring Lemma, we may proceed in the exact same fashion iterating $\tl f$ to the past, to obtain that 
		\begin{enumerate}
			\item For every $n < 0$, $\tl f ^n (\tl K)$ intersects both $\tl \phi_{\TT}$ and $\tl \phi_{\BB}$
			\item For every leaf $\tl \phi \subset \tl A$ with $\tl \phi_{\TT}, \tl \phi_{\BB} \subset \text{L}(\tl \phi)$, there exists $n_0 < 0$ such that 
			$$ \tl f^{n}(\tl K) \cap \tl \phi \neq \varnothing, \text{ for every } n \leq n_0.$$
		\end{enumerate}
	\end{remark}	
	
	\begin{definition}[\textbf{Stable and Unstable Anchors}]
		\hspace{5cm}
		\begin{itemize}
			\item We will say the quartet $(\tl A, \tl K, \tl \phi_{\BB}, \tl \phi_{\TT})$ is an unstable anchor (or simply \textit{u}-anchor) when it satisfies the hypothesis of the Anchoring Lemma.
			
			\item Accordingly, we will say the quartet $(\tl A, \tl K, \tl \phi_{\BB}, \tl \phi_{\TT})$ is a stable anchor (or \textit{s}-anchor) when it satisfies the hypothesis of last remark.
			
			\item For these two cases, we will say respectively that $\tl K$ is \textit{u}-anchored (\textit{s}-anchored) to $\tl A$, by the leaves $\tl \phi_{\BB}$ and $\tl \phi_{\TT}$, or simply that $\tl K$ is \textit{u}-anchored (\textit{s}-anchored). 
			
			\item We will say that $\tl K$ is an anchored continuum (or that it is anchored) if it is \textit{u}-anchored or \textit{s}-anchored. 	
		\end{itemize}
		
	\end{definition}

	\begin{definition}
		We will say that $(\wh A, \wh K, \wh \phi_{\BB}, \wh \phi_{\TT})$ is a \textit{u}-anchor (respectively \textit{s}-anchor) whenever there exist respective lifts $(\tl A, \tl K, \tl \phi_{\BB}, \tl \phi_{\TT})$ forming a \textit{u}-anchor (respectively \textit{s}-anchor). 
	\end{definition}
	
	The following is an almost direct consequence of the Anchoring Lemma. 
	
	\begin{proposition}\label{prop:u-anchoredThenDynamicallyUnbounded}
		Let $(\wh A, \wh K, \wh \phi_{\BB}, \wh \phi_{\TT})$ be a \textit{u}-anchor. Then $$\lim \limits_{n \to +\infty} \textnormal{diam}(\wh f^{n}(\wh K)) = \infty \text{ and in particular, } \mathscr{D}_{\wh f}(\wh K) = \infty$$ 
	\end{proposition}
	
	\begin{proof}
		Recall that by Lemma \ref{lem:UniformlyBoundedLeaves}, we know that {\bck{the}} diameter of leaves $\wh \phi \in \widehat{\mathcal{F}}$ is uniformly bounded. For every $j \in \mathbb{Z}^{+}$, let us take $\wh \phi _j \subset \wh A$ such that 
		$$ \wh \phi_{\TT}, \wh \phi_{\BB} \subset \partial \text{R}_{\wh A}(\wh \phi_j); \text{ and } \text{d} (\wh \phi_j , \wh \phi_{\TT}) > j$$
		
		By the Anchoring Lemma, there exists $n_j > 0$ such that for every $n > n_j$, 
		$$ \wh f^{n}(\wh K) \cap \wh \phi_{\TT} \neq \varnothing, \  \wh f^{n}(\wh K) \cap \wh \phi_j \neq \varnothing \ \implies \ \textnormal{Hdiam}(\wh f^{n}(\wh K)) > j$$
		
		which concludes the proof. 
	\end{proof}

	Proceeding in the same fashion we obtain:
	
	\begin{corollary}\label{coro:SAnchoredGrowsPast}
		Let $\wh K$ be an \textit{s}-anchored continuum. Then $$\lim \limits_{n \to -\infty} \textnormal{diam}(\wh f^{n}(\wh K)) = \infty, \text{ and in particular, } \mathscr{D}_{\wh f}(\wh K) = \infty$$
		
	\end{corollary}

	\begin{remark}\label{remark:ThroughSingularitiesThenAnchored}
		If a continuum $\wh K$ is crossing for $\wh A$ \textit{through singularities}, that is, when we take $\tl K, \tl A$ we have that  
		$$ \text{either } \tl K \cap \partial_{\TT} \tl A = \varnothing \text{ or } \tl K \cap \partial_{\BB} \tl A = \varnothing,$$
		
		then the singularity serves as an anchor (remember that the boundary is fixed by $\tl f$), and therefore we may proceed in the exact same fashion as in Lemma \ref{lemma:AnchoringLemma} to obtain  $\mathscr{D}_{\wh f}(\wh K) = \infty$.  
	\end{remark}
	
	We can then assume without loss of generality, that our continua $\wh K$ cross our canonically foliated strips $\wh A$ through actual leaves of $\wh \F$. 
	
	\subsection{Proofs of Theorems C and E}.
	
	We have already proved that any anchored continuum has infinite dynamical diameter, and that every sufficiently large continuum is \textit{n}-semianchored for a certain $n$. Let us finish the proof of Theorem \ref{PropC:maximumdiameter} by gluing these two results together.

	\begin{lemma}\label{lemma:3semianchoredThenAnchored}
		For any 3-semianchored continuum $\tl K$, there exists $0 \leq j \leq 3$ such that $\tl f^{j}(\tl K)$ is anchored. 
	\end{lemma}

	\begin{proof}
		Let us assume without loss of generality that our 3-semianchor is given by $(\tilde{A}^{\rightarrow}, \tilde{K}, \tilde{\phi}_{\text{B}}^{\rightarrow}, \tilde{\phi}_{\text{T}}^{\rightarrow})$. 
		We will face one of the following four cases:
		\[ \tilde{A}^{\rightarrow} \subset \text{R}(\tilde{\phi}_{\text{T}}^{\rightarrow}) \cap  \text{R}(\tilde{\phi}_{\text{B}}^{\rightarrow}), \ \  \tilde{A}^{\rightarrow} \subset \text{L}(\tilde{\phi}_{\text{T}}^{\rightarrow}) \cap  \text{L}(\tilde{\phi}_{\text{B}}^{\rightarrow}),\]
		\[\tilde{A}^{\rightarrow} \subset \text{R}(\tilde{\phi}_{\text{T}}^{\rightarrow}) \cap  \text{L}(\tilde{\phi}_{\text{B}}^{\rightarrow}), \ \  \tilde{A}^{\rightarrow} \subset \text{L}(\tilde{\phi}_{\text{T}}^{\rightarrow}) \cap  \text{R}(\tilde{\phi}_{\text{B}}^{\rightarrow}).\]
		
		For the first two cases the lemma is proved, as the original quartet would already be respectively a \textit{u}-anchor or an \textit{s}-anchor. We will then assume that the fourth case holds, as the other one is analogous.
		
		Recall that $\tilde{A}^{\rightarrow}$ is crossed by infinitely many vertical CFS $\tilde{A}_k^{\downarrow}$ with their corresponding realizing points $\tilde{z}^{\downarrow}_k$ (see the construction of $\wh {\mathcal{A}}^{\dar}$ in Section \ref{sec:CFS} for details). Moreover, we may take leaves $\tilde{\phi}^{\downarrow}_{\text{T},k} {\color{black}\ \subset \ } \partial_{\text{T}}\tilde{A}^{\downarrow}_k \cap \tilde{A}^{\rightarrow}$ and recall the construction of fundamental domains $\tilde{D}^{\rar}_{k}$ of our strip $\tilde{A}^{\rightarrow}$, built in Section \ref{sec:CFS}: 
		$$\tilde{D}^{\rar}_k = \tilde{A}^{\rightarrow} \cap \text{L}(\tilde{\phi}^{\downarrow}_{\text{T},k}) \cap \text{R}(\tilde{\phi}^{\downarrow}_{\text{T},k+1})$$

		Write $\tilde{\phi}^{\downarrow}_{\textnormal{B},k} {\color{black}\ \subset \ }  \partial_{\text{B}}\tilde{A}^{\downarrow}_k \cap \tilde{A}^{\rightarrow}$ for the other boundary component of the connected component of $\tilde{A}^{\rightarrow} \cap \tilde{A}^{\downarrow}_{\color{black}k}$  which is bounded by $\tilde{\phi}^{\downarrow}_{\textnormal{T},k}$ (see Figure \ref{figure:AnchoringSubtle1}).  Note that $\tilde{\phi}_{\text{B}}^{\rightarrow}$ belongs to the closure of a fundamental domain $\tilde{D}^{\rar}_{k}$ for a certain value of $k$, similarly $\tilde{\phi}_{\text{T}}^{\rightarrow}$ belongs to $\tilde{D}^{\rar}_{k'}$ for a certain $k'$. 
		
		\medskip
		
		\paragraph{\textbf{Case 1.}} $\tilde{K} \not\subset \text{R}(\tilde{\phi}^{\downarrow}_{\text{B},k+1})$. For this case we obtain that $\tilde{K} \cap \tilde{\phi}^{\downarrow}_{\text{B},k+1} \neq \varnothing$ since $\tilde{K}$ is connected. Then, we have that 
		$$(\tilde{A}^{\downarrow}_{k+1}, \tilde{K},   \tilde{\phi}^{\downarrow}_{\text{B},k+1}, \tilde{\phi}^{\rightarrow}_{\text{B}}) \text{ is a \textit{u}-anchor,}$$ which finishes the proof for this case. 
		
		In identical fashion, we easily solve
		
		\medskip
		
		\paragraph{\textbf{Case 2.}} $\tilde{K} \not\subset \text{L}(\tilde{\phi}^{\downarrow}_{\text{T},k'})$. For this case we get that 
		$$ (\tilde{A}^{\downarrow}_{k'}, \tilde{K},   \tilde{\phi}^{\rightarrow}_{\text{T}}, \tilde{\phi}^{\downarrow}_{\text{T},k'}) \text{ is an \textit{s}-anchor,} $$
		which also ends this case's analysis. 
		
		\medskip
		
		\paragraph{\textbf{Case 3.}} $\tilde{K} \subset \text{R}(\tilde{\phi}^{\downarrow}_{\text{B},k+1})$ and $\tilde{K} \subset \text{L}(\tilde{\phi}^{\downarrow}_{\text{T},k'})$. For this case, we must have that $k' \leq k$: otherwise, with $k' > k$ we would have that $\tl K$ simultaneously satisfies the hypotheses of Cases 1 and 2, and it would then be $\textit{s}$-anchored and $\textit{u}$-anchored. See Figure \ref{figure:AnchoringSubtle1} for details. 
		\begin{figure}[h]
			\centering
			
			\def\svgwidth{.92\textwidth}
			\import{./Figures/}{Anchoringsubtle1.pdf_tex}
			
			\bigskip

			\caption{An example with $k' = k+1$. Note that $(\tilde{A}^{\downarrow}_{k'}, \tilde{K},   \tilde{\phi}^{\rightarrow}_{\text{T}}, \tilde{\phi}^{\downarrow}_{\text{T},k'})$ is an \textit{s}-anchor, and $(\tilde{A}^{\downarrow}_{k'}, \tilde{K},   \tilde{\phi}^{\downarrow}_{\text{B},k'}, \tilde{\phi}^{\rightarrow}_{\text{B}})$ is a \textit{u}-anchor. This implies that $\lim \limits_{n \to \pm \infty} \textnormal{diam} (\widehat{f}^{n}(\widehat{K})) = \infty$.}
			\label{figure:AnchoringSubtle1}
		\end{figure}
		
		\smallskip
		
		\textbf{3.1.} $k' = k$. Let us take $\tl z^{\rar}$ a realizing point for $\tl A^{\rar}$ with $\tilde{z}^{\rightarrow} \in \tilde{D}^{\rar}_{k-1}$, and note that we must have  $\tilde{z}^{\rightarrow} \in \text{R}_{\tilde{A}}(\tilde{K})$, because $\tl K \subset \text{L}(\tilde{\phi}^{\downarrow}_{\text{T},k'})$. 
		Now, given that 
		
		\begin{itemize}
			\item $\tilde{f}^3 (\tilde{z}^{\rightarrow}) \in \tilde{D}^{\rar}_{k+2} \subset \text{L}(\tilde{\phi}^{\downarrow}_{\text{B},k+1})$
			\item $(\tilde{A}^{\rightarrow}, \tilde{K}, \tilde{\phi}_{\text{B}}^{\rightarrow}, \tilde{\phi}_{\text{T}}^{\rightarrow})$ is a 3-semianchor,
		\end{itemize}
		we obtain that 
		\begin{itemize}
			\item $\tl f^{3}(\tilde{K}) \not\subset \text{R}(\tilde{\phi}^{\downarrow}_{\text{B},k+1})$ 
			\item {\bck{$\tl f^{3}(\tilde{K}) \cap \tilde{\phi}^{\downarrow}_{\text{B},k+1} \neq \varnothing, \ \tl f^{3}(\tilde{K}) \cap \tilde{\phi}^{\rightarrow}_{\text{B}} \neq \varnothing$}}
		\end{itemize}
		
		from where we proceed as in Case 1 and conclude that $$(\tilde{A}^{\downarrow}_{k+1}, \tilde{f}^{3}(\tilde{K}),  \tilde{\phi}^{\downarrow}_{\text{B},k+1}, \tilde{\phi}^{\rightarrow}_{\text{B}}) \text{ is a \textit{u}-anchor,}$$ which concludes the analysis for this case. See Figure \ref{figure:AnchoringSubtle2} for details.  
		
		\begin{figure}[h]
			\centering
			
			\def\svgwidth{.92\textwidth}
			\import{./Figures/}{Anchoringsubtle2.pdf_tex}
			
			\bigskip

			\caption{An example with $k' = k$. $\tilde{f}^3(\tilde{K})$ is \textit{u}-anchored in many different ways, in particular $(\tilde{A}^{\downarrow}_{k+1}, \tilde{f}^{3}(\tilde{K}),   \tilde{\phi}^{\downarrow}_{\text{B},k+1}, \tilde{\phi}^{\rightarrow}_{\text{B}})$ is a \textit{u}-anchor.}
			\label{figure:AnchoringSubtle2}
		\end{figure}
		
		\textbf{3.2.} $k' < k$. This implies that $\tilde{A}^{\downarrow}_{k}$ separates $\tilde{\phi}^{\rightarrow}_{\text{B}}$ from $\tilde{\phi}^{\rightarrow}_{\text{T}}$. Let us define $$\tilde{\phi}^{\rightarrow}_{\text{B},k} = \partial_{\text{B}}\tilde{A}^{\rightarrow} \cap \tilde{A}^{\downarrow}_{k}, \ \ \tilde{\phi}^{\rightarrow}_{\text{T},k} = \partial_{\text{T}}\tilde{A}^{\rightarrow} \cap \tilde{A}^{\downarrow}_{k}.$$ 
		
		If we have $\tilde{K} \cap \tilde{\phi}^{\rightarrow}_{\text{B},k} \neq \varnothing$ (See Figure \ref{figure:AnchoringSubtle3}), we then get that $$(\tilde{A}^{\rightarrow}, \tilde{K}, \tilde{\phi}^{\rightarrow}_{\text{B},k}, \tilde{\phi}^{\rightarrow}_{\text{T}}) \text{ is an \textit{s}-anchor}.$$
		
		Finally, let us assume that $\tilde{K} \cap \tilde{\phi}^{\rightarrow}_{\text{B},k} = \varnothing$. Let us take $\tilde{z}^{\downarrow} \in \text{R}(\tilde{\phi}^{\rightarrow}_{\text{B},k}) \cap \tilde{A}^{\downarrow}_{k}$, being the lift of $z^{\downarrow}$ which {\color{black}satisfies} $\tilde{f}^2(\tilde{z}^{\downarrow}) \in \text{L}(\tilde{\phi}^{\rightarrow}_{\text{T},k})$. 
		Given that $$(\tilde{A}^{\rightarrow}, \tilde{K}, \tilde{\phi}_{\text{B}}^{\rightarrow}, \tilde{\phi}_{\text{T}}^{\rightarrow}) \text{ is a 2-semianchor}$$ we obtain that $$(\tilde{A}^{\downarrow}_{k}, \tilde{K}, \tilde{\phi}^{\downarrow}_{\text{B},k}, \tilde{\phi}^{\downarrow}_{\text{T},k}) \text{ is also a 2-semianchor}$$
		
		This fact implies that $\tilde{f}^2(\tilde{z}^{\downarrow}) \notin \text{L}_{\tilde{A}^{\downarrow}_{k}}(\tilde{K})$, which in particular shows that 
		$$\tilde{f}^2(\tilde{K}) \cap \tilde{\phi}^{\rightarrow}_{\text{T},k} \neq \varnothing$$ 
		from where we obtain that  $$(\tilde{A}^{\rightarrow}, \tilde{f}^2(\tilde{K}), \tilde{\phi}^{\rightarrow}_{\text{B}}, \tilde{\phi}^{\rightarrow}_{\text{T},k}) \text{ is a \textit{u}-anchor},$$ which concludes the proof.
	\end{proof}
	
	\begin{figure}[h]
		\centering
		
		\def\svgwidth{.92\textwidth}
		\import{./Figures/}{Anchoringsubtle3.pdf_tex}
		
		\bigskip

		\caption{An example with $k' = k-1$. Either $\tilde{K}$ intersects $\tilde{\phi}^{\rightarrow}_{\text{B},k}$, or $\tilde{f}^2(\tilde{K})$ intersects $\tilde{\phi}^{\rightarrow}_{\text{T},k}$.} 
		\label{figure:AnchoringSubtle3}
	\end{figure}

	In Case 3.2, $\tl f^{2}(\tl K)$ is anchored to $\tl A^{\rar}$ by construction. For Cases 1, 2 and 3.1, $f^{j}(\tl K)$ is anchored to some $\tl A^{\dar}_k$ (with $j = 0$ in Cases 1 and 2, $j = 3$ in Case 3.1). Repeating the argument from the last step of Case 3.2, it is easy to check that up to using two more iterations, we can anchor either $\tl f^{j+2}(\tl K)$ or $\tl f^{j-2}(\tl K)$ to $\tl A^{\rar}$
	
	This implies that
	
	\begin{remark}\label{rem:AnchoredToThatParticularLeaf}
		If $\wh K$ is 3-semianchored to $\wh A$, then at least one of the following occurs:
		\begin{itemize}
			\item For every $j \geq 5$, $\wh f^j(\wh K)$ is \textit{u}-anchored to $\wh A$, and so $\wh f^j(\wh K) \cap \mathrm{cl}(\wh A) \neq \varnothing$,
			\item For every $j \leq -2$, $\wh f^j(\wh K)$ is \textit{s}-anchored to $\wh A$, and so $\wh f^j(\wh K) \cap \mathrm{cl}(\wh A) \neq \varnothing$.
		\end{itemize}
	\end{remark}
	
	Also, note that combining Lemmas \ref{lemma:SemiAnchoring} and \ref{lemma:3semianchoredThenAnchored}, we immediately obtain that 
	
	\begin{remark}\label{rem:LargeThenWillBeAnchored}
		Let $\wh f$ be a lift to $\R^2$ of $f \in \mathrm{Homeo}_0(\T^2)$, with $\vec{0} \in \mathrm{int}(\rho(\wh f))$. 
		
		Then, there exists $M' > 0$ such that, for any continuum with $\mathrm{diam}(\wh K) > M'$, there exists $0 \leq j \leq 3$ such that $\wh f ^j(\wh K) \text{ is anchored.}$ 
	\end{remark}

	Let us now reiterate this argument and put everything together. 
	
	\begin{proof}[Proof of Theorem \ref{PropC:maximumdiameter}]
			By Remark \ref{remark:ThroughSingularitiesThenAnchored}, we may assume that $\wh K$ goes through no singularities. Take $M = 2M_3$ from Lemma \ref{lemma:ThereExistsM_n_o} to see that any continuum $\wh K$ with $\text{diam}(\wh K) > M$ is 3-semianchored, then use Lemma \ref{lemma:3semianchoredThenAnchored} to see that $\wh f^{\color{black}j_0}(\wh K)$ is anchored for some $0 \leq {\color{black}j_0} \leq 3$, which by Proposition \ref{prop:u-anchoredThenDynamicallyUnbounded} implies that {\color{black}either $\lim \limits_{j \to +\infty} \textnormal{diam}(\wh f^{j}(\wh K)) = +\infty$ or $\lim \limits_{j \to -\infty} \textnormal{diam}(\wh f^{j}(\wh K)) = +\infty$,} which concludes the proof.
	\end{proof}
	
	Theorem \ref{thmE:UniformlyBoundedIslands} is now obtained as a consequence of Theorem \ref{PropC:maximumdiameter}. 
	
	\begin{proof}[Proof of Theorem \ref{thmE:UniformlyBoundedIslands}]
		By \cite[Theorem B]{korotal}, we know that the lifted diameter of any periodic topological disk is finite (there is a shorter improved version of the proof of this result in \cite[Theorem 6]{koropecki19triple}). This implies that this disk is dynamically bounded, which then implies that its closure {\bck{in $\T^2$}} must be a dynamically bounded (inessential) continuum, whose diameter is uniformly bounded from Theorem \ref{PropC:maximumdiameter}, which concludes the proof.
		
	\end{proof}

	\section{The essential factor}\label{sec:EssFactor}
	
	The goal of this section is to build an explicit monotone semiconjugacy, for a torus homeomorphism $f$ under the hypotheses of Theorem \ref{thmA:semiconjugation}, that is, a torus homeomorphism $f$ which is isotopic to the identity, and such that $\rho(f)$ has nonempty interior. For the remainder of this section, let us assume these hypotheses hold and use the notation of said theorem. We will dedicate the following section to the basic understanding of the dynamics in the resulting quotient. In particular, let us recall that the dynamical diameter $\mathscr{D}$ of a continuum $\wh K \subset \R^2$ is defined as the supremum of the diameters of the iterates of $\wh K$ by a lift $\wh f$.

	\subsection{Dynamically bounded splitting}

	Let us define a relation $\sim_{\wh f}$ in $\mathbb{R}^2$ as follows:
	
	\begin{definition}
		Let $\wh z,\wh z' \in \mathbb{R}^2$. We will say that $\wh z \sim_{\wh f} \wh z'$ if there exists a continuum {\bck{$\wh K \subset \mathbb{R}^2$}} such that $\wh z, \wh z' \in \wh K, \text{ and } \mathscr{D}_{\wh f}(\wh K) < \infty$. 
	\end{definition} 
	
	It follows immediately that $\sim_{\wh f}$ is an equivalence relation, where we write $[\wh z]_{\wh f}$ for the equivalence class of $\wh z$ under $\sim_{\wh f}$ (or simply $[z]$ and $\sim$ when there is no ambiguity). Let us now understand the structure of these classes.
	
	\begin{remark}
		For every $\wh z \in \R^2$, $[\wh z] = \bigcup \limits_{\wh w \in [\wh z]} \wh K_{\wh w}$, where $\wh K_{\wh w}$ is a dynamically bounded continuum containing both $\wh w$ and $\wh z$. This implies $[\wh z]$ is connected. 
	\end{remark}
	
	\begin{lemma}[Uniformly bounded classes]\label{lem:UniformlyBoundedClasses}
		There exists $M > 0$, such that for every $\wh z \in \mathbb{R}^2$, we have $\textnormal{diam}([\wh z]) < M$.  
	\end{lemma}

	\begin{proof}
		Take $M$ as in Theorem \ref{PropC:maximumdiameter}. Suppose by contradiction that we have $\textnormal{diam}([\wh z]) > M$. Then, there exist $\wh z_1, \wh z_2 \in [\wh z]$ such that $d(\widehat{z}_1, \widehat{z}_2) > M$. 
		
		This implies that any continuum $\widehat{K}$ which contains both $\widehat{z}_1$ and $\widehat{z}_2$ has $\mathrm{diam}(K) > M$. Using Theorem \ref{PropC:maximumdiameter}, we instantly get that {\bck{$\mathscr{D}_{\wh f}(\widehat{K}) = \infty$}}, which contradicts the fact that $\widehat{z}_1$ and $\widehat{z}_2$ belong to the same equivalence class. 
	\end{proof}
	
	\begin{corollary}
		For every $\wh z \in \R^2$, $[\wh z]$ is dynamically bounded. Furthermore, its closure is also dynamically bounded. This implies that $\textnormal{cl}([\wh z]) \subset [\wh z]$, and in particular every class is closed. 
	\end{corollary}
	
	\begin{proof}
		The fact that every class is dynamically bounded is due to Lemma \ref{lem:UniformlyBoundedClasses} and the fact that $\wh f$ sends classes to classes. The closure of each class is also dynamically bounded because $\textnormal{diam}(\textnormal{cl}([\wh z])) = \textnormal{diam}([\wh z])$.  
	\end{proof}
	
	\begin{remark}
		For every $\wh z$, $[\wh z]$ is filled, because $\textnormal{diam}(\wh f^j([\wh z])) = \textnormal{diam}(\textnormal{Fill}(\wh f^j([\wh z])))$. 
	\end{remark}

	\begin{lemma}\label{lem:DynBoundedSetsAreInessential}
		Let $\wh K \subset \R^2$ be a dynamically bounded continuum for $\wh f$. Then, for every nontrivial deck transformation $T$, we have that $T(\wh K) \cap \wh K = \varnothing$. 
	\end{lemma}
	
	\begin{proof}
		Suppose by contradiction that there exists a deck transformation $T(\wh z) = \wh z + v$ where $v \in \mathbb{Z}^2$ {\bck{is}} non-zero, such that $T(\wh K) \cap \wh K \neq \varnothing$. Start by building a \textit{necklace} 
		$$\wh K_{T} = \bigcup \limits_{n \in \mathbb{Z}} T^n(\wh K),$$
		
		and note that $\mathrm{diam}(\wh K_{T}) = \infty$. Let us denote by $\pi_{v^{\perp}}$ to the projection onto the subspace generated by the vector $v^{\perp}$ which is orthogonal to $v$. Given that $\mathscr{D}(\wh K) < M$, we obtain that
		$$ \sup \limits_{j \in \mathbb{Z}} \textnormal{diam}(\pi_{v^\perp}(\wh f^j(\wh K_T))) = \sup \limits_{j \in \mathbb{Z}} \textnormal{diam}(\pi_{v^\perp}(\wh f^j(\wh K))) < M.$$
		
		Note that the complement of $\wh K_T$ has two unbounded connected components $U_+, U_-$ (think respectively of points with large positive and negative second coordinate in the positively-oriented base ${v, v^{\perp}}$). We shall use $\textnormal{coord}_{v^{\perp}}$ to denote this second coordinate. Again, up to taking a power and an adequate lift we will assume that $\vec{0} \in \mathrm{int}(\rho(\wh f))$. Then, we may take two periodic points $\wh z_+ \in U_-, \ \wh z_- \in U_+$ such that $\textnormal{coord}_{v^{\perp}}(\rho(\wh z_+)) > 0$ and $\textnormal{coord}_{v^{\perp}}(\rho(\wh z_-)) < 0$. This implies that for a sufficiently large value of $j$, we have that $\textnormal{coord}_{v^{\perp}}(\wh f^j(\wh z_+)) - \textnormal{coord}_{v^{\perp}}(\wh f^j(\wh z_-)) > 2M$, which in turn implies that $\textnormal{diam}(\pi_{v^\perp}(\wh f^j(\wh K_T))) > M$, which is a contradiction.    
	\end{proof}
	
	\begin{corollary}\label{cor:ClassesAreInessential}
		For every $\wh z$, we have that $\wh{\pi}([\wh z])$ is an inessential filled continuum, which is dynamically bounded.  
	\end{corollary}

	\begin{proof}
		This is due to Lemmas \ref{lem:DynBoundedSetsAreInessential} and \ref{lem:UniformlyBoundedClasses} put together. 
	\end{proof}

	\begin{proposition}
		The equivalence relation $\sim_{\wh f}$ is upper semicontinuous. 
	\end{proposition}

	\begin{proof}
		First, note that it suffices to prove that the Hausdorff limit of a sequence of equivalence classes is entirely contained in an equivalence class.
		
		Let {\bck{$\{[\wh z_n]\}_{n \in \mathbb{N}}$}} be a sequence of equivalence classes such that {\bck{$[\wh z_n] \underset{n \rightarrow \infty}{\rightarrow} \wh K$}} for the Hausdorff topology. This implies that for every $\delta > 0$, there exists $n_{\delta}$ such that {\bck{$\wh K \subset \mathrm{B}([\wh z_{n_{\delta}}], \delta)$}}. 
		
		Given that $\wh f$ is uniformly continuous, we get that for every $j \in \mathbb{N}$  
		there exists $\delta_{j}$ such that $$ \sup \limits_{|i| \leq j} \mathrm{d}_{\mathrm{H}}(f^{i}([\wh z]), f^{i}([\wh z'])) < {1}, \  \text{whenever} \ \mathrm{d}_{\mathrm{H}} ([\wh z], [\wh z']) < \delta_{j}.$$
		
		For every $j \in \mathbb{Z}^+$, given that {\bck{$\mathscr{D}_{\wh f}([\wh z_{n_{\delta_j}}]) \leq M$}} and {\bck{$\mathrm{d}_{\mathrm{H}} ([\wh z_{n_{\delta_j}}], \wh K) < \delta_j$}}, we obtain that 
		$$\sup \limits_{|i| \leq j} \mathrm{diam}(\wh f^j(\wh K)) < M +1,$$ and therefore $\mathscr{D}(\wh K) < M + 1$. We conclude that $\wh K$ is dynamically bounded, which concludes the proof.
	\end{proof}

	\paragraph{\textbf{Projecting the classes.}} We may now induce an equivalence class $\sim_f$ for the torus, defined as 
	$$ [z]_f = \wh{\pi}([\wh z]_{\wh f}), \text{ where } \wh z \text{ is any lift of } z,$$
	which is equivalent to saying $z \sim_f z'$ if there exists a continuum $K$ containing both $z$ and $z'$, which is dynamically bounded. 
	
	From the results of last subsection, we recover the following
	
	\begin{proposition}[\textbf{Dynamically bounded Splitting}]\label{prop:EqRelationEveryProperty}
		Let $\sim_f$ be an equivalence relation on the torus, whose classes are defined by $[z]_f = \wh \pi ([\wh z]_{\wh f})$. Then, 
		\begin{itemize}
			\item Each equivalence class of $\sim_f$ is an inessential filled continuum. 
			\item Equivalence classes of $\sim_f$ are uniformly dynamically bounded. 
			\item $\sim_f$ is upper semicontinuous. 
		\end{itemize}
	\end{proposition}

	\subsection{Monotone quotient}
	
	We have built a partition given by an equivalence relation which is upper semicontinuous, and whose classes are inessential filled continua. By a main result in \cite{roberts}, we know that the quotient topological space  $\mathrm{T}^2 = \mathbb{T}^2 _{/ \sim}$, equipped with the metric $\mathrm{d}_{\sim}([z],[w]) = \min \{d(z',w')  : z' \sim z, w' \sim w \}$ is homeomorphic to $\mathbb{T}^2$ (this is actually a generalization of Moore's result for the sphere, which can be found in \cite{moore}).  Note that this induces a distance $\widehat{\mathrm{d}}_{\sim}$ on the universal covering $\mathrm{R}^2$ of $\mathrm{T}^2$ (which is in turn homeomorphic to $\mathbb{R}^2$), and allows us to measure the diameter of continua in $\mathrm{T}^2$. When there is no room for confusion, we will abuse notation and use $\mathrm{d}$ for all of these distances.

	Given that $f$ sends equivalence classes to equivalence classes, we immediately get natural dynamics $g$ on the quotient space $\mathrm{T}^2$, defined as $g([z]) = [f(z)]$. Note that the quotient map $\pi_{\sim}: \mathbb{T}^2 \rightarrow \mathrm{T}^2$ is monotone. We obtain the following diagram:

	\tikzset{every picture/.style={line width=0.75pt}} 
	
	\begin{tikzpicture}[x=0.75pt,y=0.75pt,yscale=-1,xscale=1]
		
		\draw    (273.07,38.23) -- (388.5,38.23) ;
		\draw [shift={(390.5,38.23)}, rotate = 180] [color={rgb, 255:red, 0; green, 0; blue, 0 }  ][line width=0.75]    (10.93,-3.29) .. controls (6.95,-1.4) and (3.31,-0.3) .. (0,0) .. controls (3.31,0.3) and (6.95,1.4) .. (10.93,3.29)   ;
		\draw    (273.67,121.19) -- (389.1,121.19) ;
		\draw [shift={(391.1,121.19)}, rotate = 180] [color={rgb, 255:red, 0; green, 0; blue, 0 }  ][line width=0.75]    (10.93,-3.29) .. controls (6.95,-1.4) and (3.31,-0.3) .. (0,0) .. controls (3.31,0.3) and (6.95,1.4) .. (10.93,3.29)   ;
		\draw    (253.83,50.25) -- (253.83,106.57) ;
		\draw [shift={(253.83,108.57)}, rotate = 270] [color={rgb, 255:red, 0; green, 0; blue, 0 }  ][line width=0.75]    (10.93,-3.29) .. controls (6.95,-1.4) and (3.31,-0.3) .. (0,0) .. controls (3.31,0.3) and (6.95,1.4) .. (10.93,3.29)   ;
		\draw    (412.54,50.86) -- (412.54,107.17) ;
		\draw [shift={(412.54,109.17)}, rotate = 270] [color={rgb, 255:red, 0; green, 0; blue, 0 }  ][line width=0.75]    (10.93,-3.29) .. controls (6.95,-1.4) and (3.31,-0.3) .. (0,0) .. controls (3.31,0.3) and (6.95,1.4) .. (10.93,3.29)   ;
		
		\draw (324.18,20.41) node [anchor=north west][inner sep=0.75pt]    {$f$};
		\draw (325.38,102.77) node [anchor=north west][inner sep=0.75pt]    {$g$};
		\draw (244.64,29.23) node [anchor=north west][inner sep=0.75pt]    {$\mathbb{T}^{2}$};
		\draw (400.34,28.63) node [anchor=north west][inner sep=0.75pt]    {$\mathbb{T}^{2}$};
		\draw (245.03,111.59) node [anchor=north west][inner sep=0.75pt]    {$\mathrm{T}^{2}$};
		\draw (402.54,111.59) node [anchor=north west][inner sep=0.75pt]    {$\mathrm{T}^{2}$};
		\draw (82,63.67) node [anchor=north west][inner sep=0.75pt]   [align=left] { \ \ \ };
		\draw (232,69.07) node [anchor=north west][inner sep=0.75pt]    {$\pi _{\sim }$};
		\draw (389,68.73) node [anchor=north west][inner sep=0.75pt]    {$\pi _{\sim }$};

	\end{tikzpicture}

	\begin{definition}\label{def:EssentialFactor}
		Given $f \in \mathrm{Homeo}_{0}(\mathbb{T}^2)$ with $\textnormal{int}(\rho(f)) \neq \varnothing$, and $\sim$ the dynamically bounded equivalence relation, we will say that $g = f _{/ \sim}$ is its essential factor.
	\end{definition}
	
	\begin{remark}
		Given $f$ under the General hypothesis and its essential factor $g$, we obtain lifts $\widehat{g}: \mathbb{R}^2 \rightarrow \mathbb{R}^2$, $\tilde{g}: \tilde{\mathbb{D}} \rightarrow \tilde{\mathbb{D}}$, proceeding in the same fashion as we did for $f$. 	
	\end{remark}
	
	\begin{remark}
		A continuum $K$ is essential in $\mathbb{T}^2$ if and only if $K_{\sim} = \pi_{\sim}(K)$ is essential in $\textnormal{T}^2$. 
	\end{remark} 
	
	Keep in mind that the distance functions in $\T^2$ and $\text{T}^2$ are very much different. However, we may recover the following facts from the construction:
	
	\begin{remark}
		Let $\gamma^{\downarrow}, \gamma^{\rightarrow}$ be the canonical $\pi_1$ generators of $\mathbb{T}^2$, based at the origin. Then, their projections $\gamma^{\downarrow}_{\sim}, \gamma^{\rightarrow}_{\sim}$ to the quotient, are a $\pi_1$ generator of $\textnormal{T}^2$. 
	\end{remark}
	
	\begin{remark}[\textbf{Canonical torus model}]\label{rem:TorusModel}
		Our new torus $\textnormal{T}^2$ is topologically \textcolor{black}{homeomorphic} to a canonical torus $\mathbb{S}^1 \times \mathbb{S}^1$. \textcolor{black}{We can further choose the homeomorphism $h:\textnormal{T}^2\to\T^2$ such that $h(\gamma^{\dar}_{\sim}) = \gamma^{\dar}$, and $h(\gamma^{\rar}_{\sim}) = \gamma^{\rar}$ {(that is, the action in the fundamental group induced by $h$ with the natural identifications, is the identity).}}   
	\end{remark} 
	
	\subsection{Quick dynamical consequences}
	
	We will start by stating some properties inherited by the essential factor $g$, due to the structure of $f$. The following is a purely topological fact. 
	
	\begin{claim}\label{claim:quotientdynamicallyunboundedcontinua}
		A continuum $K$ has $\mathscr{D}_f(K) = \infty$ if and only if its projection $K_{\sim}$ has $\mathscr{D}_g(K_\sim)=\infty$. 
	\end{claim}
	
	\begin{proof}
		Having $\mathscr{D}_f(K) = \infty$ is equivalent to saying that there exists a sequence $\{n_j\}_{j \to \infty}$ such that $f^{n_j}(K)$ intersects a growing-to-infinity number of lifts $\widehat{\gamma}$ of either $\gamma^{\downarrow}$ or $\gamma^{\rightarrow}$. in turn, this is equivalent to having that $g^{n_j}(K_{\sim})$ intersects a growing-to-infinity number of lifts $\widehat{\gamma}_{\sim}$ of either $\gamma^{\downarrow}_{\sim}$ or $\gamma^{\rightarrow}_{\sim}$, which is equivalent to having $\mathscr{D}_g(K_{\sim}) = \infty$ 
	\end{proof}
	
	Note that our quotient has not only preserved the topological space, it has also preserved free homotopy classes of curves. We may also send the isotopy downstairs and note that 
	
	\begin{remark}\label{rem:IsotopicToId}
		The essential factor $g$ is isotopic to the identity. 
	\end{remark}
	
	\paragraph{\textbf{Rotation in the quotient.}} 
	Recall that from the classical definition, given a standard plane torus $\mathbb{T}^2$, we say that $v = (v_1,v_2) \in \rho(\widehat{f})$ if and only if there exist $\wh z_k \in \R^2$, $n_k \to + \infty$ such that 
	$$ \lim \limits_{k \to \infty} \frac{\widehat{f}^{n_k}(\widehat{z}_k) - \widehat{z}_k}{n_k} = v \ $$
	
	Note that this is equivalent to \textit{counting the intersections} with the canonical generators $\gamma^{\downarrow}, \gamma^{\rightarrow}$, that is, if we define $\widehat{\gamma}_k$ to be a curve from $\widehat{z}_k$ to $\widehat{f}^{n_k}(\widehat{z}_k)$ in general position respect to the infinite lifts $\widehat{\gamma}^{\downarrow}, \widehat{\gamma}^{\rightarrow}$, and take $\gamma_k = \wh \pi(\wh \gamma_k)$ its projection to the torus, then  
	\begin{equation}\label{equation:rotationvectorequivalent}
		\lim \limits_{k \to \infty} \frac{\widehat{f}^{n_k}(\widehat{z}_k) - \widehat{z}_k}{n_k} = v \  \iff \ \lim \limits_{k \to \infty}  \frac{\gamma^{\downarrow} \wedge \gamma_k}{n_k} = v_1; \  \lim \limits_{k \to \infty} \frac{\gamma^{\rightarrow} \wedge \gamma_k}{n_k} = v_2	
	\end{equation}

	where $\wedge$ denotes the classical intersection number with sign between two curves. This allows us to understand rotation, which is more of a topological quantity than a metric one, for a homeomorphism of any torus, in particular for $\text{T}^2$. Note that this notion is well defined, as the quotient yields the same result for the whole dynamically bounded class of $\widehat{z}_k$.

	\begin{proposition}\label{prop:elementprop.essfactor}
		Let $\wh f$ be a lift of $f \in \mathrm{Homeo}_{0}(\mathbb{T}^2)$ with $\textnormal{int}(\rho(f)) \neq \varnothing$, and let $g$ be its essential factor. Then,
		
		\begin{enumerate} 
			\item There exists a lift $\wh g$ of $g$ such that $\rho(\wh g) = \rho(\wh f)$.
			\item $\mathscr{D}_g(K_{\sim}) = \infty$ for every nontrivial continuum $K_{\sim} \subset \mathrm{T}^2$.
		\end{enumerate}
	\end{proposition}
	
	\begin{proof}
		\begin{enumerate}
			\item As we have just explained above, $v = (v_1,v_2) \in \rho(\wh g)$ if and only if $$\lim \limits_{k \to \infty}  \frac{\gamma^{\downarrow}_{\sim} \wedge \gamma_{\sim,k}}{n_k} = v_1; \  \lim \limits_{k \to \infty} \frac{\gamma^{\rightarrow}_{\sim} \wedge \gamma_{\sim,k}}{n_k} = v_2$$
			as in Equation \ref{equation:rotationvectorequivalent} which lets us obtain the desired result given that the quotient $\pi_{\sim}$ preserves the intersection number between curves, that is, 
			$$\gamma^{\downarrow} \wedge \gamma_k = \gamma^{\downarrow}_{\sim} \wedge \gamma_{\sim,k}; \ \gamma^{\rightarrow} \wedge \gamma_k = \gamma^{\rightarrow}_{\sim} \wedge \gamma_{\sim,k}.$$
			
			\item This is immediate from Claim \ref{claim:quotientdynamicallyunboundedcontinua} and the fact that  the preimage $\pi_{\sim}^{-1}(K_{\sim})$ of any nontrivial continuum $K_{\sim} \subset \textnormal{T}^2$ is dynamically unbounded.  
			
		\end{enumerate}
	\end{proof}
	
	From now on, and having taken the canonical model we already explained, we will abuse notation and write $\mathbb{T}^2$ for $\textnormal{T}^2$, and similarly for the coverings $\mathbb{R}^2$ and $\tilde{\mathbb{D}}$. We will also eliminate the subindex $\sim$ and write $K$ instead of $K_{\sim}$ for continua, $z$ instead of $z_{\sim}$ for points, and similarly for their lifts.

	Let us now prove that $g$ is tight, which follows from the fact that $g$ is continuum-wise expansive. {\color{black} For that proof, we will use a result from Continuum Theory, \cite[Theorem 5.6]{nadler92continuum}. We will adapt the notation to our context
		
		\begin{lemma}[Boundary Bumping Theorem II, \cite{nadler92continuum}]\label{lem:nadler}
			Let $S$ be a surface, let $K \subset S$ be a continuum, let $U \subset S$ be an open set such that $K \cap U \neq \varnothing$, $K \cap U^{\mathrm{C}} \neq \varnothing$, and let $z \in K \cap U$ be a point. Let $X$ be the connected component of $K \cap U$ which contains $z$.
			
			Then, $\overline{X} \cap \partial U \neq \varnothing$. 
		\end{lemma}
		
		\begin{corollary}\label{cor:SplittingContinua}
			Let $\varepsilon>0$ and let $K \subset \T^2$ be a continuum with $\mathrm{diam}(K) > 3 \varepsilon$. Then, there exist two subcontinua $K_0, K_1 \subset K$ such that 
			
			\begin{itemize}
				\item $\mathrm{diam}(K_0) \geq \varepsilon$, $\mathrm{diam}(K_1) \geq \varepsilon$
				\item $\forall \ z_0 \in K_0$ we have $\mathrm{d}(z_0, K_1) > \varepsilon$, and $\forall \ z_1 \in K_1$ we have $\mathrm{d}(z_1, K_0) > \varepsilon$. 
			\end{itemize}
		\end{corollary}
		
		\begin{proof}
			Take $z, z' \in K$ such that $\mathrm{d}(z,z') > 3 \varepsilon$. Let $U_0$ be the open disk centered in $z$, with radius equal to $\varepsilon$, and define $K_0$ as the closure of the connected component of $K \cap U_0$ which contains $z$. Define $U_1$ and $K_1$ in a similar fashion, using $z'$. 
			
			Now, the sets $K_0$ and $K_1$ {\color{black}satisfy} the second property because $K_0 \subset \overline{U_0}$, $K_1 \subset \overline{U_1}$, and they {\color{black}satisfy} the first property because of Lemma \ref{lem:nadler}, which concludes the proof. 
		\end{proof}
		
	}
	
	{\bck{Recall the definition of the two-sided entropy $h_{\pm}(g,K)$ for $g$, carried by a continuum $K$, given in Equation \ref{eq:TwoSidedEntropyK}.}}
	
	\begin{proposition}\label{prop:Tight}
		For each nontrivial continuum $K \subset \T^2$ we have that $h_{\pm}(g,K) > 0$.
	\end{proposition}
	
	\begin{proof}
		Let us first prove that for every $\varepsilon > 0$, there exists $n_{\varepsilon} > 0$ such that: for every nontrivial continuum $K \subset \T^2$ with $\mathrm{diam}(K) \geq \varepsilon$, there exists $-n_{\varepsilon} \leq n_K \leq n_{\varepsilon}$ such that 
		\begin{equation}\label{eq:ContinuumWise}
			\mathrm{diam}(g^{n_K}(K)) > 3 \varepsilon
		\end{equation}

		Suppose by contradiction that is not true. Then, there would exist a sequence of continua $\{K_n\}_{n \in \mathbb{N}}$, such that for every $n \in \mathbb{N}$ we have that
		\[ \mathrm{diam}(K_n) \geq \varepsilon, \ \mathrm{diam}(g^j(K_n)) \leq 3 \varepsilon, \text{ for every } -n \leq j \leq n \] 
		In this case, up to taking a subsequence we may assume that $K_n \to K$ in the Hausdorff topology, {\bck{with $\textnormal{diam}(K) \geq \varepsilon$}}. Then by uniform continuity of $\wh g$ we would have that 
		\[ \mathscr{D}_g(K) \leq 3\varepsilon,\]
		which contradicts the fact that $g$ is infinitely continuum-wise expansive (see Proposition \ref{prop:elementprop.essfactor}). 
		
		Let us now use the notation from Equation \ref{eq:ContinuumEntropy}, and prove that for small values of $\varepsilon$, the number {\bck{$s_{\pm n}(\varepsilon, K)$}} grows exponentially, that is,
		{\bck{\[\limsup_{n \to +\infty} \frac{1}{n} \mathrm{log}(s_{\pm n}(\varepsilon,K)) > 0.\]}}
		
		We may assume that $\mathrm{diam}(K) \geq \varepsilon$ up to taking an iterate by some power of $g$. It is now enough to prove that for every $j \in \Z^+$ we have that
		\begin{equation}\label{eq:Separated}
			s_{\pm jn_{\varepsilon}}(\varepsilon,K) \geq 2^j
		\end{equation}

		Take $n_K$ from Equation \ref{eq:ContinuumWise}. Given that $\mathrm{diam}(g^{n_K}(K)) > 3\varepsilon$, we may take two subcontinua $K_0, K_1 \subset g^{n_K}(K)$ 
		{\color{black} as in Corollary \ref{cor:SplittingContinua}.}
		This means that we can take $z_0, z_1 \in K$ which are $(n_{\varepsilon}, \varepsilon)$-separated, that is {\bck{$s_{\pm n_{\varepsilon}}(\varepsilon, K) \geq 2$}}. 
		
		Reiterating this process for $K_0$ and $K_1$ {\color{black}and reusing Corollary \ref{cor:SplittingContinua},} we obtain two respective subcontinua $K_{00}, K_{01} \subset g^{n_{K_0}}(K_0)$, $K_{10}, K_{11} \subset g^{n_{K_1}}(K_1)$ of an iterate of each of them, such that points in $K_{00}$ and points in $K_{01}$ are at distance greater than $\varepsilon$ from each other, and similarly for points in $K_{10}$ with respect to points in $K_{11}$. This allows us to take four points $z_{00}, z_{01}, z_{10}, z_{11}$ in $K$ which are $(2n_{\varepsilon}, \varepsilon)$-separated. By repeating this process we obtain the equality from Equation \ref{eq:Separated}, which concludes the proof.  
		
	\end{proof}

	Guided by the two properties held by the essential factor which were proved in Proposition \ref{prop:elementprop.essfactor}, we shall introduce the following notion. 
	
	\begin{definition}\label{def:FullyChaotic}
		We will say a (torus) homeomorphism $g \in \mathrm{Homeo}_0(\T^2)$ is \emph{frice} (i.e. fully rotational, infinitely continuum-wise expansive) if 
		\begin{itemize} 
			\item $\mathrm{int}(\rho(g)) \neq \varnothing$,
			\item $\mathscr{D}_g(K) = \infty$ for every nontrivial continuum $K \subset \T^2$.	
		\end{itemize}
	\end{definition}
	
	By Proposition \ref{prop:elementprop.essfactor} and Remark \ref{rem:IsotopicToId}, we have that the essential factor of a torus homeomorphism in the General Hypothesis, is frice. Moreover, if we take a frice torus homeomorphism, then its essential factor will be the map itself, because the equivalence classes for $\sim$ will be trivial. We then have that 
	
	\begin{remark}
		A torus homeomorphism $g$ is frice, if and only if it is the essential factor of a torus homeomorphism $f$ in the General Hypothesis.  
	\end{remark}
	
	For the sake of clarity, we will generally state future results for the essential factor by using the definition of frice (torus) homeomorphism instead. 
	
	Throughout the following sections, we will use some of the techniques and results we developed in Sections \ref{sec:CFS} and \ref{section:stretching} for the context of frice homeomorphisms, as they satisfy the General Hypothesis. Keep in mind that in this context, any essential continuum is dynamically unbounded both to the past and to the future, therefore, all the anchoring techniques will be done for inessential continua.

	\section{Stable and unstable sets}\label{sec:StableSets}
	
	This brief section is devoted to introducing and studying the following three stability-related notions, which will be strongly used in Sections \ref{section:RotMixing} and \ref{sec:HPR}:
	
	\begin{itemize}
		\item Weakly stable and weakly unstable continua,  
		\item Stable and unstable sets of a point, 
		\item Stable and unstable continua.
	\end{itemize}
	
	\subsection{Weak stability}
	
	\begin{definition}
		Let $\wh f$ be a lift to $\R^2$ of a homeomorphism $f$ in the General Hypothesis. We will say that a continuum $\wh K$ is \textit{weakly stable} for $\wh f$ if there exists $n < 0$ such that $\wh f^{n} (\wh K)$ is \textit{s}-anchored, and analogously it is \textit{weakly unstable} for $\wh f$ if there exists $n > 0$ such that $\wh f^{n} (\wh K)$ is \textit{u}-anchored.
		
		In this context, the projection $K = \wh \pi(\wh K)$ to the torus will also be called respectively a weakly stable (unstable) continuum for $f$. 
	\end{definition}
	
	We will generally use $\wh K^{ws}$ and $\wh K^{wu}$ for weakly stable and weakly unstable continua of the plane, respectively.

	We saw in Proposition \ref{prop:elementprop.essfactor}, that the dynamical diameter for a frice torus homeomorphism $g$, of every nontrivial continuum, is infinite. This fact, together with the structure of canonically foliated strips we built in Section \ref{section:stretching} applied for $g$, gives us the following:

	\begin{lemma}\label{rem:EveryContinuumIsWeaklySomething}
		Let $\wh g$ be a lift to $\R^2$ of a frice torus homeomorphism $g$. 
		
		Then, any nontrivial continuum $\wh K \subset \R^2$ is either weakly stable or weakly unstable for $\wh g$ (it may very well be both).
	\end{lemma}
	
	\begin{proof}
		Up to taking a power and changing the lift, we assume that $\vec{0} \in \mathrm{int}(\rho(\wh g))$. Let $\wh K \subset \R^2$ be a nontrivial continuum. By Proposition \ref{prop:elementprop.essfactor}, we know that $\mathscr{D}_{\wh g}(\wh K) = \infty$. In particular, there exists {\bck{$j_0 \in \Z$}} such that $\mathrm{diam}(\wh g^{j_0}(\wh K)) > M'$, where $M'$ is taken as in Remark \ref{rem:LargeThenWillBeAnchored}. By this same remark, we obtain that there exists $j_0 \leq j \leq j_0+3$ such that $\wh g ^j(\wh K) \text{ is anchored,}$
		which concludes the proof. 
	\end{proof}

	\subsection{Shape and size}
	
	\begin{definition}
		Let $\wh f$ be a lift to $\R^2$ of a homeomorphism $f$ in the General Hypothesis, and let $\wh z \in \R^2$. Then,
		
		\begin{enumerate}
			\item The \textit{$\varepsilon$-stable set of $\wh z$ for $\wh f$}, is defined as 
			$$ {\color{black}X^{s}_{\varepsilon,f}(\wh z)} = \{\wh z' \in \R^2 \ : \ \text{d}(\wh f^j(\wh z'),\wh f^{j}(\wh z)) \leq \varepsilon, \text{ for every } j \geq 0\},$$ 
			\item The \textit{$\varepsilon$-unstable set of $\wh z$ for $\wh f$} is defined as
			$$ {\color{black}X^{u}_{\varepsilon,f}(\wh z)} = \{\wh z' \in \R^2 \ : \ \text{d}(\wh f^j(\wh z'),\wh f^{j}(\wh z)) \leq \varepsilon, \text{ for every } j \leq 0\},$$
			\item For $z \in \T^2, \ z = \wh \pi (\wh z)$, we define
			$$ X^{s}_{\varepsilon,f}(z)  = \wh \pi(X^{s}_{\varepsilon,f}(\wh z)), \ \ X^{u}_{\varepsilon,f}(z)  = \wh \pi(X^{u}_{\varepsilon,f}(\wh z)).$$	
			\item $K^s_{\varepsilon,f} (\wh z)$ is defined as the connected component of {\color{black}$X^{s}_{\varepsilon,f}(\wh z)$} containing $\wh z$, 
			\medskip
			
			\item $K^u_{\varepsilon,f} (\wh z)$ is defined as the connected component of {\color{black}$X^{u}_{\varepsilon,f}(\wh z)$} containing $\wh z$,
			\medskip
			
			\item For $z \in \T^2, \ z = \wh \pi (\wh z)$, we define
			$$K^s_{\varepsilon,f} (z) = \wh \pi (K^s_{\varepsilon,f} (\wh z)), \ \  K^u_{\varepsilon,f} (z) = \wh \pi (K^u_{\varepsilon,f} (\wh z)).$$			
		\end{enumerate}

	\end{definition}

	We will omit the homeomorphism as a subscript when there is no room for ambiguity. 
	\medskip
	
	{\color{black}
		\begin{definition}[\textbf{Stable set of a point}]\label{defi:StableSet}
			Let $\wh f$ be a lift to $\R^2$ of a torus homeomorphism $f$ in the General Hypothesis, and let $\wh z \in \R^2$. Then, 
			
			\begin{enumerate}
				\item the \textit{stable set of $\wh z$} for $\wh f$ is defined as $$ W^s(\wh z) = \bigcup \limits_{\varepsilon > 0} K^{s}_{\varepsilon}(\wh z)$$
				\item the \textit{unstable set of $\wh z$} for $\wh f$ is defined as $$W^u(\wh z) = \bigcup \limits_{\varepsilon > 0} K^{u}_{\varepsilon}(\wh z).$$
				
			\end{enumerate}
			
		\end{definition}
	}
	
	{\color{black}
		\begin{remark}
			Neither $W^s(\wh z)$ nor $W^{u}(\wh z)$ is necessarily a continuum. {\color{black}Indeed, it could be neither bounded nor closed.}
		\end{remark}
	}
	
	Note that the orbit of every $\wh z' \in {\color{black}W^s(\wh z)}$ remains at a bounded distance from the orbit of $\wh z$. Again, for $z \in \T^2$ with $z = \wh \pi (\wh z)$, we define 
	{\color{black}\[  W^s(z) = \wh \pi (W^s(\wh z)).\] Once again, this set is not necessarily a continuum.}

	\begin{definition}[\textbf{Stable Continuum}]\label{def:StableContinuum}
		Let $\wh f$ be a lift to $\R^2$ of a torus homeomorphism $f$ in the General Hypothesis. We will say that {\color{black}a continuum} $\wh K \subset \R^2$ is a \textit{stable continuum} if {\color{black}$\wh K \subset K^s_{\varepsilon}(\wh z)$ for some $ \wh z \in \mathbb{R}^2$, $\varepsilon > 0$}, and it will be an \textit{unstable continuum} if {\color{black}$\wh K \subset K^u_{\varepsilon}(\wh z)$ for some $\wh z \in \mathbb{R}^2$, $\varepsilon > 0$}.
		
		{\color{black}We will say that a continuum $K \subset \T^2$ is a \textit{stable continuum} if $K = \wh \pi (\wh K)$, where $\wh K$ is a stable continuum (same fashion for unstable continua).}
	\end{definition}
	
	{\color{black}
	Let us state some quick consequences of this definition.
		
	\begin{lemma}\label{lem:StableDoesNotGrow}
	Let $\wh f$ be a lift of a torus homeomorphism $f$ in the General Hypothesis. If $\wh K$ is a stable continuum for $\wh f$, then $\sup \limits_{j \in \Z^+} \textnormal{diam}(\wh f^{j}(\wh K)) < + \infty$. 
	\end{lemma}
	
	\begin{proof}
		By definition $\wh K \subset \wh K^s_{\varepsilon}(\wh z)$ for some $\wh z \in \R^2, \ \varepsilon > 0$. This implies that  
		\begin{equation}\label{eq:StableContinuumFutureDynBounded}
			\sup \limits_{j \in \Z^+} \ \mathrm{diam} (\wh g^j(\wh K)) \leq 2\varepsilon,
		\end{equation}
		which concludes the proof. 
	\end{proof}
	
	\begin{corollary}\label{cor:StableNotWeaklyUnstable}
		Let $\wh K \subset \R^2$ be a stable continuum. Then, $\wh K$ is not weakly unstable.   
	\end{corollary}
	
	\begin{proof}
		Suppose by contradiction that $\wh K$ is weakly unstable. By Proposition \ref{prop:u-anchoredThenDynamicallyUnbounded}, we have that $\lim \limits_{j \to +\infty} \textnormal{diam}(\wh f^{j}(\wh K)) = + \infty$, which contradicts Lemma \ref{lem:StableDoesNotGrow} and concludes the proof. 
	\end{proof}
	
	An analogous argument shows that an unstable continuum is not weakly stable. 
	
		{\color{black}
		\begin{corollary}\label{cor:StableContinuumInessential}
			Any (un)stable continuum $K \subset \T^2$ is inessential.
		\end{corollary}
		
		\begin{proof}
			Suppose by contradiction that $K$ is essential. By Lemma \ref{lem:DynBoundedSetsAreInessential}, we would obtain that $\lim \limits_{j \to +\infty} \textnormal{diam}(f^{j}(K)) = + \infty$, {\color{black}which contradicts Lemma \ref{lem:StableDoesNotGrow}} and concludes the proof.
		\end{proof}	
	}
	}
	\medskip
	
	While the definitions in this section can be applied in a broader context, the work in what follows will be done for frice torus homeomorphisms, which we shall denote by $g$. 
	
	Recall that by Proposition \ref{prop:elementprop.essfactor}, we have that any such $g$ is continuum-wise expansive (i.e. there exists a positive uniform lower bound for the dynamical diameter of nontrivial continua), and so is any lift $\wh g$ of $g$ to $\R^2$. We may then apply a result from {\color{black}Kato, which can be found in \cite[Theorem 1.6]{kato93continuumwiseexpansive}}, and obtain the following:
	
	\begin{lemma}\label{lem:NonemptyStableSets}
		Let $g$ be a frice torus homeomorphism. Then, for any $\varepsilon > 0$, there exists $\delta > 0$ such that
		$$ \inf \limits_{z \in \T^2} \mathrm{diam}(K^s_{\varepsilon}(z)) \geq \delta, \ \inf \limits_{z \in \T^2} \mathrm{diam}(K^u_{\varepsilon}(z)) \geq \delta.$$
		
		In particular, for every $z \in \T^2$ we have that ${\color{black}W^s(z)}$ is nonempty. 
	\end{lemma} 
	
	Note that this result for the diameter of stable sets has its natural analogue in $\R^2$ when applied to a lift $\wh g$ of $g$. 
	
	{\color{black}
	\begin{lemma}\label{lem:StableThenNotUnstable}
		Let $\wh g$ be a lift to $\R^2$ of a frice torus homeomorphism. Let $\wh K \subset \R^2$ be a stable continuum which is also an unstable continuum. Then $\wh K$ is trivial (i.e. it is a point).
	\end{lemma}
	
	\begin{proof}
		Since $\wh K$ is both stable and unstable, by Lemma \ref{lem:StableDoesNotGrow} we obtain that
		$$\sup \limits_{j \in \Z^+} \textnormal{diam}(\wh f^{j}(\wh K)) < + \infty, \ \sup \limits_{j \in \Z^-} \textnormal{diam}(\wh f^{j}(\wh K)) < + \infty,$$
		which implies that $\wh K$ is dynamically bounded, which implies that $\wh K$ is trivial because $\wh g$ is infinitely continuum-wise expansive.  
	\end{proof}

	\begin{corollary}\label{cor:IntersectionStableUnstableTotDisconnected}
		For frice homeomorphisms, the intersection of stable and unstable continua is totally disconnected.
	\end{corollary}
	}
	
	\begin{lemma}\label{lem:StableSetsAreWeaklyStable}
		Let $\wh g$ be a lift to $\R^2$ of a frice torus homeomorphism. Then, any nontrivial stable continuum is also weakly stable, and any nontrivial unstable continuum is also weakly unstable. 
	\end{lemma}
	
	\begin{proof}
		It suffices to prove the result for stable continua, as the other case is analogous. Let $\wh K \subset {\color{black}W^s(\wh z)}$ be a nontrivial stable continuum for $\wh g$. 
		{\color{black}By Corollary \ref{cor:StableNotWeaklyUnstable}, 
		we obtain that $\wh K$ is not weakly unstable.} Then, Remark \ref{rem:EveryContinuumIsWeaklySomething} yields that $\wh K$ is weakly stable, which concludes the proof.  
	\end{proof}

	\begin{lemma}\label{lem:stableunstablecontinuadense}
		Let $\wh g$ be a lift to $\R^2$ of a frice torus homeomorphism. Then, for every open set $\widehat{V} \subset \mathbb{R}^2$, there exist two nontrivial continua $\wh K^{ws}, \wh K^{wu} \subset \wh V$, which are respectively weakly stable and weakly unstable continua for $\wh g$.  
	\end{lemma}
	
	\begin{proof}
		Let $\wh V \subset \R^2$ be an open set. Take $\wh z \in \wh V$, let $\varepsilon < \mathrm{d}(\wh z, \partial \wh V)$, and note that $K^s_{\varepsilon}(\wh z)$ is a stable continuum, which is nontrivial by Lemma \ref{lem:NonemptyStableSets}, and is also included in $\wh V$. By Lemma \ref{lem:StableSetsAreWeaklyStable}, we obtain that $\wh K^{ws} = K^s_{\varepsilon}(\wh z)$ is a weakly stable continuum. We proceed for weakly unstable continua in identical fashion and thus conclude the proof. 
	\end{proof}
	
	As a consequence we obtain the following two results:
	
	\begin{remark}
		Let $\wh g$ be a lift to $\R^2$ of a frice torus homeomorphism $g$. Then, any closed ball in $\R^2$ is simultaneously weakly stable and weakly unstable for $\wh g$.  
	\end{remark}
	
	\begin{remark}\label{rem:FilledStableContinuum}
		Let $\wh g$ be a lift to $\R^2$ of a frice torus homeomorphism $g$. Then, for every $\varepsilon > 0$, both $K^s_{\varepsilon}(\wh z)$ and $K^u_{\varepsilon}(\wh z)$ are nonempty, filled and have empty interior. Moreover, any stable or unstable continuum is filled and has empty interior. 
	\end{remark}

	\subsection{Dynamical properties}
	
	We will study some properties of the stable sets, which have their natural analogous versions for unstable sets. We obtain results using the anchoring techniques developed in Section \ref{section:stretching}, the most important one of the subsection being Proposition \ref{prop:HffLimitSmallstableSets}.
	
	\begin{lemma}\label{lemma:StableSetGoesToZero}
		For every $\wh z' \in {\color{black}W^s(\wh z)}$, we have that 
		$$\lim \limits_{j \to +\infty} d(\wh g^j(\wh z'),\wh  g^j(\wh z)) = 0$$
	\end{lemma}
	
	\begin{proof}
		{\bck{Suppose by contradiction that there exist $\delta, \varepsilon > 0$, a point $\wh  z' \in K^s_{\varepsilon}(\wh z)$ and a sequence $\{j_n\}_{n \in \mathbb{N}}$ of future iterates with $j_n \to +\infty$, such that $d(\wh g^{j_n}(\wh z'),\wh  g^{j_n}(\wh z)) \geq \delta$. This on its turn implies for every $n \in \mathbb{N}$, that 
		$$\text{diam}(\wh g^{j_n}(K^s_{\varepsilon}(\wh z))) \geq \delta.$$}
		}
		We shall use $\wh K_n$ to denote the connected component of $\wh  g^{j_n}(K^s_{\varepsilon}(\wh z)) \cap \text{B}(\wh g^{j_n}(\wh z),\delta)$ containing $\wh g^{j_n}(\wh z)$, which by Lemma \ref{lem:nadler} is a continuum of diameter between $\delta$ and $2\delta$, and we will write $K_n= \wh \pi (\wh K_n) \subset \T^2$. Now, given that the space of continua of a compact space, equipped with the Hausdorff topology, is compact, we have that up to taking a subsequence, 
		$$ K_n \xrightarrow[n \to +\infty]{\text{Hff}} K.$$
		where $\delta \leq \text{diam}(K) \leq 2\delta$. 
		
		Fix {\bck{$m \in \Z$}}, and let us check that $\text{diam}(g^m)(K) \leq 4\varepsilon$. For this, take $\delta_{m,\varepsilon} > 0$ from the uniform continuity of $g$, such that any two points $\delta_{m,\varepsilon}$-near, must be $\varepsilon$-near for $m$ iterates of $g$. Now, simply take $n$ large enough such that 
		{\bck{$$\text{d}_{\text{H}}(K_n,K) < \delta_{m,\varepsilon}, \text{ and } -j_n \leq m \leq j_n.$$}
		}
		Take any pair of points $w', w'' \in K$ and note that there exist $z', z'' \in K_n$ with 
		$$\text{d}(z',w') < \delta_{m,\varepsilon}, \text{ and } \text{d}(z'',w'') < \delta_{m,\varepsilon}.$$
		Moreover, {\bck{given that $z', z'' \in g^{j_n}(K^s_{\varepsilon}(z))$, we obtain that }}
		$$\text{d}(g^m(z'),g^m(z'')) \leq \text{d}(g^{m+j_n}(z),g^m(z'')) + \text{d}(g^m(z'),g^{m+j_n}(z)) \leq 2\varepsilon,$$
		from where we obtain that  
		$$ \text{d}(g^m(w'),g^m(w''))$$ 
		$$\leq \text{d}(g^m(w'),g^m(z')) + \text{d}(g^m(z'),g^m(z'')) + \text{d}(g^m(z''),g^m(w'')) \leq \varepsilon + 2\varepsilon + \varepsilon \leq 4\varepsilon.$$

		We have just constructed a nontrivial dynamically bounded continuum $K$, which is a direct contradiction with the second item in Proposition \ref{prop:elementprop.essfactor}. 
	\end{proof}

	
	{\color{black}Recall that by Corollary \ref{cor:StableContinuumInessential}, every stable or unstable continuum is inessential.}

	The following result strongly uses the techniques developed in Section \ref{section:stretching}. 
	
	\begin{lemma}
		Fix $d > 0$. Then, there exists $L_d > 0$, such that for every stable continuum $\wh K$ with $\textnormal{diam}(\wh K) > L_d$, we have that 
		$$ \textnormal{diam}(\wh g^{j}(\wh K)) > d, \text{ for every } j < 0.$$ 	
	\end{lemma}

	\begin{proof}
		As in Section \ref{section:stretching}, let us work with vertical diameter $\mathrm{Vdiam}$ without loss of generality. Take $\wh A^{\rar}$ a canonically foliated strip. Take $M_3$ from Lemma \ref{lemma:SemiAnchoring}, and taking
		$$ L_d = 2M_3 + 2d + 1,$$ 
		we obtain that any continuum $\wh K$ with $\textnormal{Vdiam}(\wh K) > L_d$ is 3-semianchored to two lifts $\wh A^{\rar}_{k}, \wh A^{\rar}_{k'}$ of the same canonically foliated strip $A^{\rar}$, and such that the vertical distance $\mathrm{d}_{\mathrm{V}}$ between the closure of these strips {\color{black}satisfies} 
		$$ \text{d}_{\text{v}}(\wh A^{\rar}_{k},\wh A^{\rar}_{k'}) \geq 2d. $$ 
		
		Given that no iterate of $\wh K$ can be \textit{u}-anchored since $\wh K$ is {\color{black} a stable continuum} {\color{black}(see Proposition \ref{prop:u-anchoredThenDynamicallyUnbounded} and Lemma \ref{lem:StableDoesNotGrow})}, we obtain from Corollary \ref{coro:SAnchoredGrowsPast} and Remark \ref{rem:AnchoredToThatParticularLeaf} that the past iterates of $\wh K$ intersect the closure of both strips, from where we obtain 
		$$ \textnormal{diam}(\wh g^{j}(\wh K)) \geq 2d, \text{ for every } j < 0,$$
		which concludes the proof.   
	\end{proof}
	
	The following result is the contrapositive of what we have just proved. 
	
	\begin{corollary}\label{corollary:SmallStableContinuaCantGrowMuch}
		Fix $d > 0$. Then, there exists $L_d > 0$, such that for any stable continuum $K$ with $\textnormal{diam}(\wh K) \leq d$, we have that 
		$$ \textnormal{diam}(\wh g^j(\wh K)) \leq L_d, \text{ for every } j > 0. $$
	\end{corollary}
	
	We now prove a key result for the construction in Section \ref{sec:HPR}. 
	
	\begin{proposition}[\textbf{Hausdorff limit of small stable continua}]\label{prop:HffLimitSmallstableSets}
		Let $\{\wh K_n\}_{n \in \mathbb{N}}$ be a sequence of stable continua $\wh K_n \subset {\color{black}W^s(\wh z_n)}$, such that $\textnormal{diam}(\wh K_n) \leq d$ for every $n > 0$ and some $d > 0$. Suppose that 
		$$\wh  K_n \xrightarrow[n \to +\infty]{\textnormal{Hff}} \wh K.$$
		Then, $\wh K$ is also a stable continuum. 
	\end{proposition}
	
	The proof is essentially the same as in Lemma \ref{lemma:StableSetGoesToZero}, {\bck{using Corollary \ref{corollary:SmallStableContinuaCantGrowMuch}}.}
	
	\begin{proof}
		Fix $\wh w \in \wh K$. We will prove that for any other $\wh w' \in \wh K$ and every $j > 0$, we have that $\text{d}(\wh g^j(\wh w), \wh g^j(\wh w')) \leq L_d$, which is enough to finish the lemma as it implies that $\wh K \subset K^s_{L_d}(w)$. 
		
		Fix $j_0 > 0$ and $\varepsilon > 0$, and take $\delta_{j_0, \varepsilon}$ from the uniform continuity of $\wh g$, as in Lemma \ref{lemma:StableSetGoesToZero}. Take a large enough value of $n$ such that $\text{d}_{\text{H}}(\wh K_n,\wh K) < \delta_{j_0, \varepsilon}$. We may then finish with a similar argument to the one in Lemma \ref{lemma:StableSetGoesToZero}: first, take $z,z' \in \wh K_n$ such that 
		$$ \mathrm{d}(\wh z, \wh w) < \delta_{j_0, \varepsilon}, \ \ \mathrm{d}(\wh z',\wh w') < \delta_{j_0, \varepsilon},$$
		which means that 
		$$ \mathrm{d}(\wh g^j(\wh z),\wh g^j(\wh w)) < \varepsilon, \ \ \mathrm{d}(\wh g^j(\wh z'),\wh g^j(\wh w')) < \varepsilon, \text{ for every } 0 \leq j \leq j_0.$$
		By Corollary \ref{corollary:SmallStableContinuaCantGrowMuch}, we know that 
		$$ \textnormal{d}(\wh g^{j}(\wh z), \wh g^j(\wh z')) \leq L_d, \text{ for every } j \geq 0.$$

		Then, by the triangle inequality, we obtain that 
		$$ \textnormal{d}(\wh g^{j}(\wh w), \wh g^j(\wh w')) < L_d + 2\varepsilon, \text{ for every } 0 \leq j \leq j_0.$$
		Taking $j_0$ growing to $\infty$ and $\varepsilon$ going to 0, we obtain the desired result.	   
	\end{proof}

	We finish the section with a sufficient condition for continua to be weakly (un)stable. 
	
	\begin{lemma}\label{lemma:AnchorTypeA}
		Let $\wh f$ be a lift of a torus homeomorphism $f \in \mathrm{Homeo}_0(\T^2)$ with $\vec{0} \in \mathrm{int}(\rho(\wh f))$. Let $\wh z_1, \wh z_2$ be two periodic points with different associated rotation directions, and let {\bck{$d, m > 0$}}.

		Then, every continuum $\wh K$ such that 
		\[\wh K \cap K^s_m(\wh z_1) \neq \varnothing, \ \wh K \cap K^s_m(\wh z_2) \neq \varnothing,\]
		is weakly unstable. 
	\end{lemma}
	
	\begin{proof}
		
		The proof uses the anchoring techniques from Section \ref{section:stretching}. Let the respective rotation vectors for the periodic points be given by
		\[\rho(\wh z_1) = (a_1,b_1), \ \rho(\wh z_2) = (a_2,b_2), \text{ with } \frac{b_1}{a_1} \neq \frac{b_2}{a_2}.\]
		We will assume without loss of generality that the respective rotation vectors {\color{black}satisfy} $a_1, a_2, b_1, b_2 < 0$. The proof for the other cases is analogous, up to taking a different CFS {\bck{(possibly in the vertical direction, for example if either $b_1 = 0$ or $b_2=0$)}}. 
		
		Let us take a continuum $\wh K \subset \wh V$ such that 
		$$\wh K \cap K^s_m(\wh z_1) \neq \varnothing, \wh K \cap K^s_m(\wh z_2) \neq \varnothing$$ and note that 	
		\[\wh C := K^s_m(\wh z_1) \cup \wh K \cup K^s_m(\wh z_2) \text{ is connected.}\]
		
		Note that
		\[\sup_{j \in \Z^+} \mathrm{diam}(\wh f^j(K^s_m(\wh z_1))) \leq 2m, \sup_{j \in \Z^+} \mathrm{diam}(\wh f^j(K^s_m(\wh z_2))) \leq 2m.\]
		Up to taking sufficiently large negative values of $k$, we have that $\wh C \subset \mathrm{B}(\wh A^{\rar}_k)$. 
		
		Note that, given the diameter of leaves $\wh \phi \in \wh \F$ is uniformly bounded, and that the slopes of the vectors $\rho(\wh z_1)$ and $\rho(\wh z_2)$ are different, we know that we may take $k < 0$ with $\wh C \subset \mathrm{B}(\wh A^{\rar}_k)$, and such that the sets of leaves $\wh {\Phi}_1$ and $\wh {\Phi}_2$ are nonempty and disjoint (see Figure \ref{figure:Drift}), where 
		$$\wh {\Phi}_1 = \{\wh \phi \subset \partial_{\BB} \wh A^{\rar}_k : \ \wh \phi \cap \wh I^{+}_{\wh {\mathcal{F}}}(\wh z_1) \neq \varnothing\},$$ 
		$$\wh {\Phi}_2 = \{\wh \phi \subset \partial_{\BB} \wh A^{\rar}_k : \ \wh \phi \cap \wh I^{+}_{\wh {\mathcal{F}}}(\wh z_2) \neq \varnothing\}.$$
		
		\begin{figure}[h]
			\centering
			
			\def\svgwidth{.8\textwidth}
			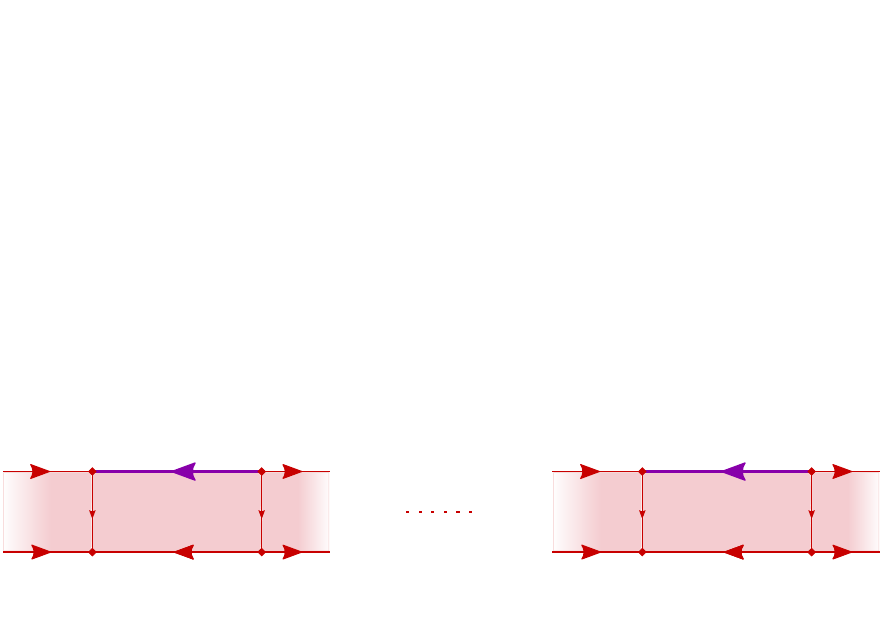
			
			\smallskip
			
			\caption {The paths $\wh I^{+}_{\wh {\mathcal{F}}}(\wh z_1)$ and $\wh I^{+}_{\wh {\mathcal{F}}}(\wh z_2)$ drift linearly at the large scale, thus for a sufficiently distant CFS we obtain that $\wh {\Phi}_1$ and $\wh {\Phi}_2$ are disjoint.}
			\label{figure:Drift}
		\end{figure}	
		
		Given that $b_1, b_2 < 0$, we can take $j_0 > 0$ such that 
		\[\wh f^{j_0}(\wh z_1) \in \mathrm{T}(\wh A^{\rar}_k), \ \wh f^{j_0}(\wh z_2) \in \mathrm{T}(\wh A^{\rar}_k),\]
		\[ \dd (\wh f^{j_0}(\wh z_1), \wh A^{\rar}_k) > m, \ \dd (\wh f^{j_0}(\wh z_2), \wh A^{\rar}_k) > m,\]
		which means that 
		\begin{equation}\label{eq:ContinuumOnTheTop}
			\wh f^{j_0}(K^s_m(\wh z_2)) \subset \TT(\wh A^{\rar}_k), \ \wh f^{j_0}(K^s_m(\wh z_1)) \subset \TT(\wh A^{\rar}_k) 
		\end{equation} 
		
		Let us check that $\wh f^{j_0}(\wh K)$ is $u$-anchored. Take two leaves $\wh \phi_1 \in \wh {\Phi}_1, \ \wh \phi_2 \in \wh {\Phi}_2$. The key observation here is that given $\wh {\Phi}_1$ and $\wh {\Phi}_2$ are disjoint, we have that for every pair of lifts $\tl \phi_1, \tl \phi_2$ of $\wh \phi_1, \wh \phi_2 \subset \wh A^\rar_k$ which {\color{black}satisfy} $\tl I^{+}_{\tl {\mathcal{F}}}(\tl z_1) \cap \tl \phi_1 \neq \varnothing$, $\tl I^{+}_{\tl {\mathcal{F}}}(\tl z_2) \cap \tl \phi_2 \neq \varnothing$, we have that $\tl \phi_2 \subset \mathrm{R}(\tl \phi_1), \tl \phi_1 \subset \mathrm{R}(\tl \phi_2)$, which means that if we take two lifts $\tl A^{\rar}_{k,1}, \ \tl A^{\rar}_{k,2}$ of $\wh A^{\rar}_k$ such that 
		\[\tl I^{+}_{\tl {\mathcal{F}}}(\tl z_1) \cap \tl A^{\rar}_{k,1} \neq \varnothing, \ \ \tl I^{+}_{\tl {\mathcal{F}}}(\tl z_2) \cap \tl A^{\rar}_{k,2} \neq \varnothing,\]
		we will have that 
		\[ \tl A^{\rar}_{k,2} \subset \BB(\tl A^{\rar}_{k,1}), \ \ \tl A^{\rar}_{k,1} \subset \BB (\tl A^{\rar}_{k,2}).\] 
		See Figure \ref{figure:AnchorTypeA} for details.  
		
		\begin{figure}[h]
			\centering
			
			\def\svgwidth{.8\textwidth}
			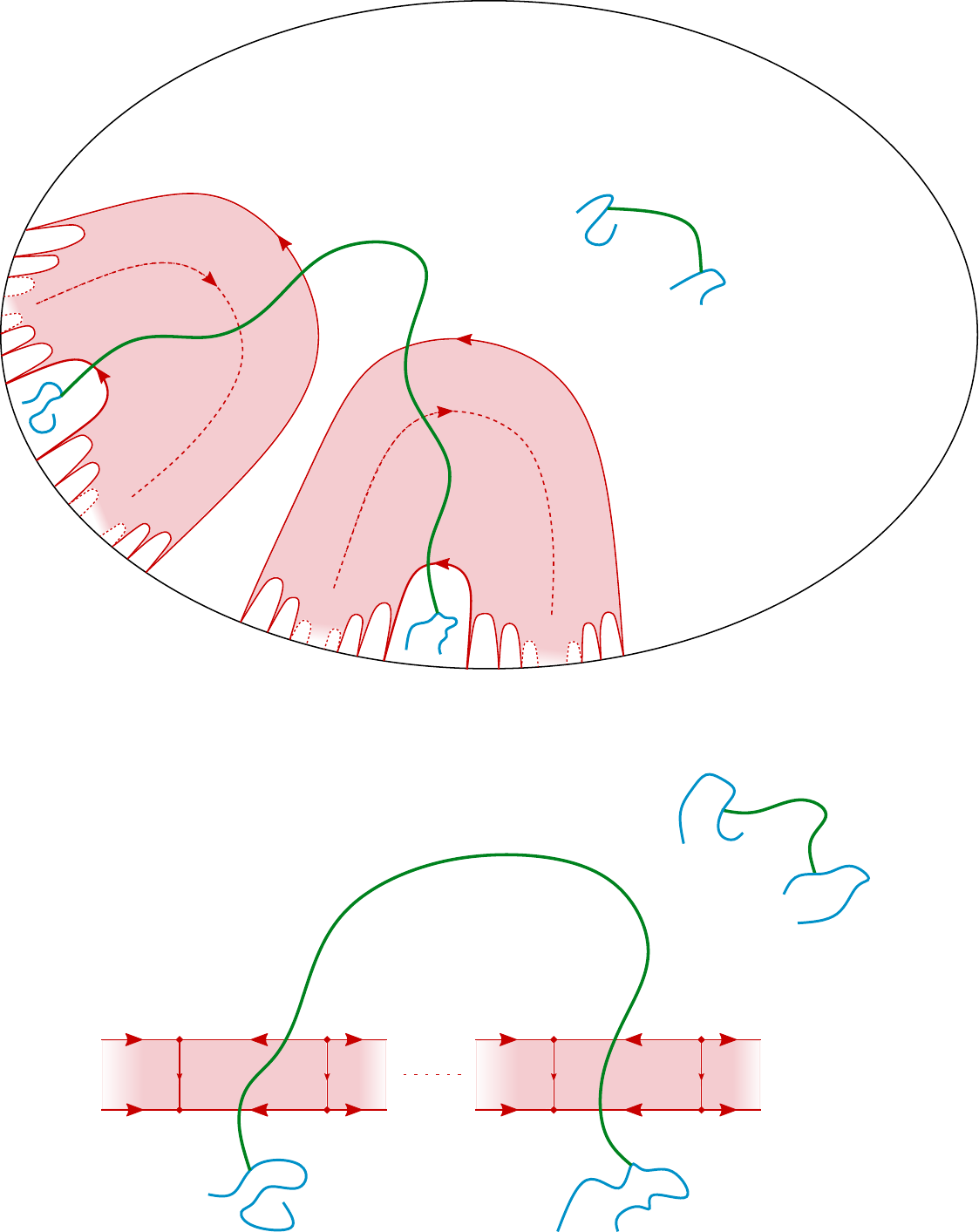
			
			\smallskip
			
			\caption {Stable sets appear in light blue, the continua connecting each respective pair of stable sets appear in green. Note that both $(\tl A^{\rar}_{k,1}, \tl f^{j_0}(\tl K), \tl \phi^2_{\TT}, \tl \phi^1_{\TT})$ and $(\tl A^{\rar}_{k,2}, \tl f^{j_0}(\tl K), \tl \phi^1_{\TT}, \tl \phi^2_{\TT})$ are \textit{u}-anchors.}
			\label{figure:AnchorTypeA}
		\end{figure}

		This fact, together with equation \ref{eq:ContinuumOnTheTop}, shows that there exist two lifts $\tl A^{\rar}_{k,1}, \ \tl A^{\rar}_{k,2}$ of $\wh A^{\rar}_k$, two leaves $\wh \phi^1_{\TT}, \wh \phi^2_{\TT} \subset \TT(\wh A^{\rar}_k)$ with their respective lifts $\tl \phi^1_{\TT} \subset \TT(\tl A^{\rar}_{k,1})$ and $\tl \phi^2_{\TT} \subset \TT(\tl A^{\rar}_{k,2}) \subset \BB (\tl A^{\rar}_{k,1})$, such that by taking natural lifts we obtain
		\begin{equation}\label{eq:PAnchor}
			\tl f^{j_0}(K^s_m(\tl z_1)) \subset \TT_{\tl \phi^1_{\TT}}(\tl A^{\rar}_{k,1}), \ \ \   \tl f^{j_0}(K^s_m(\tl z_2)) \subset \TT_{\tl \phi^2_{\TT}}(\tl A^{\rar}_{k,2}) \subset \BB(\tl A^{\rar}_k),
		\end{equation}
		which implies that 
		\[ \tl f^{j_0}(\tl K) \cap \tl \phi^1_{\TT} \neq \varnothing, \ \tl f^{j_0}(\tl K) \cap \tl \phi^2_{\TT} \neq \varnothing.\]
		We also obtain that
		\[ \tl A^{\rar}_{k,1} \subset \RR (\tl \phi^1_{\TT}), \ \tl A^{\rar}_{k,1} \subset \RR (\tl \phi^2_{\TT}),\]
		which means that
		\[ (\tl A^{\rar}_{k,1}, \tl f^{j_0}(\tl K), \tl \phi^2_{\TT}, \tl \phi^1_{\TT}) \text{ is a \textit{u}-anchor},\]
		which shows that $\wh K$ is a weakly unstable continuum, and thus finishes the proof.  
		
	\end{proof}

	\section{The essential factor is rotationally mixing} \label{section:RotMixing}
	
	So far, we have obtained the following three of the six properties described in Theorem \ref{thmA:semiconjugation} for the essential factor, proved in Proposition \ref{prop:elementprop.essfactor}:
	\begin{itemize}
		\item There exists a lift $\wh g$ of $g$ such that $\rho(\wh g) = \rho(\wh f)$,
		\item $g$ is infinitely continuum-wise expansive.
		\item $g$ is tight: for each nontrivial continuum $K \subset \T^2$, we have $h_{\pm}(g,K) > 0$,
	\end{itemize} 
	
	One of the goals of the section is to prove the following property:
	\begin{itemize}
		\item $g$ is topologically mixing, and if $\vec{0} \in \mathrm{int}(\rho(\wh g))$, $\wh g$ is topologically mixing,
	\end{itemize}
	
	the other one being to prove Theorem \ref{thmD:TransitiveThenMixing}, which we will do at the end of the Section. 
	\medskip
	
	Let us show that for the first purpose, it is enough to prove Proposition \ref{prop:essfactor.isrot.mixing}, as we immediately discuss. 
	
	\begin{proposition}\label{prop:essfactor.isrot.mixing}
		Let $g \in \mathrm{Homeo}_0(\T^2)$ be a frice torus homeomorphism.
		Then, if $\wh g$ is a lift of $g$ with $\vec{0} \in \text{int}(\rho(\wh g))$, then $\wh g$ is topologically mixing.   
	\end{proposition}
	
	\begin{corollary}\label{coro:EssFactorIsMixing}
		Any frice torus homeomorphism $g$ is rotationally mixing. In particular, {\color{black}the essential factor of a torus homeomorphism in the General Hypothesis is rotationally mixing.}  
	\end{corollary}
	
	\begin{proof}
		Recalling that $\rho(\wh g^j) = j\rho(\wh g)$, we know that there exists a power $g^j$ and an adequate lift $\wh {g^j}$ which has $\vec{0} \in \rho(\wh{g^j})$, and is therefore topologically mixing, which means that $g$ is rotationally mixing. 
	\end{proof}
	
	\begin{remark}
		Assuming Proposition \ref{prop:essfactor.isrot.mixing} holds, we have obtained that $g^j$ is topologically mixing, which implies that the original $g$ is also topologically mixing (see Lemma \ref{lem:MixingPower}).
	\end{remark}
	
	\begin{corollary}
		The essential factor $g$ is strictly toral, and $\mathrm{Ess}(g) = \T^2$. 
	\end{corollary}
	
	\begin{proof}
		Note that $g$ is non-wandering because it is topologically mixing. Given that $g$ already satisfies the General Hypothesis, we conclude that it is strictly toral. Moreover, $g$ is also transitive, again because it is topologically mixing. By Remark \ref{rem:TrasitiveIffEss}, we conclude that $\mathrm{Ess}(g) = \T^2$, as desired. 
	\end{proof}
	
	\medskip
	
	The idea for the proof of Proposition \ref{prop:essfactor.isrot.mixing} is to use the existence of stable and unstable sets for every point, together with a refinement of the anchoring techniques, mainly developed in Proposition \ref{prop:AnchoredEverywhere}. For this matter, we will use the results and notation developed in Sections \ref{sec:CFS} and \ref{section:stretching}.

	\subsection{Total anchoring}
	
	{\color{black} In this subsection, we will explain for homeomorphisms $\wh f$ with $\vec{0} \in \mathrm{int}(\rho(\wh f))$, how the orbit of any weakly unstable continuum is \textit{u}-anchored to \textit{every canonically foliated strip} in $\wh{\mathcal{A}}$, and even more, the pair of leaves which serve as anchors can be chosen to be almost anywhere in the boundary of the chosen strip.
		
	The proof of this fact is somewhat technical but it is at its heart, a careful iterative use of the anchoring techniques developed in Section \ref{section:stretching}. We will split the proof in four pieces so it becomes easier to understand. 
	
	{\color{black}Roughly speaking: consider $\wh K \subset \R^2$ a continuum which is \textit{u}-anchored to a horizontal CFS. The first piece, Lemma \ref{lem:InductiveAnchoring}, shows that some future iterate of $\wh K$ is \textit{u}-anchored to a vertical CFS. The second one, Lemma \ref{lem:VeryAnchored}, shows that future iterates of $\wh K$ are \textit{u}-anchored to every CFS. The third one, Lemma \ref{lem:TheRoute}, is a technical statement which allows us to control the leaves through which the future iterates of $\wh K$ will enter each CFS. The final one, Proposition \ref{prop:AnchoredEverywhere}, shows that future images of $\wh K$ can be made to be \textit{u}-anchored to any CFS before any prescribed fundamental domain.}

	For the remainder of the subsection, recall the numbering of the families of CFS given by Equations \ref{eq:HorizontalCFS} and \ref{eq:CFSThree}.

		\begin{lemma}[\textbf{Inductive Anchoring}]\label{lem:InductiveAnchoring}
			Let $\wh f$ be a lift to $\R^2$ of $f \in \mathrm{Homeo}_0(\T^2)$, with $\vec{0} \in \mathrm{int}(\rho(\wh f))$. Fix a CFS $\widehat{A}^{\rar} \in {\mathcal{A}}^{\rar}$, and a continuum $\widehat{K} \subset \R^2$ that is \textit{u}-anchored to $\wh{A}^{\rar}$.
			
			Then, there exists $k_0$ such that, for every $k \geq k_0$, there exists $j_k$ such that $\wh{f}^{j_k}(\wh K)$ is \textit{u}-anchored to $\wh{A}^{\dar}_k$ and it is also \textit{u}-anchored to $\wh{A}^{\uar}_k$.
		\end{lemma}
		
		\begin{proof}
			Let us take two leaves $\widehat{\phi}^{\rightarrow}_{\textnormal{B}}, \widehat{\phi}^{\rar}_{\textnormal{T}}$ such that  
			$$ (\widehat{A}^{\rightarrow}, \widehat{K}, \widehat{\phi}^{\rightarrow}_{\textnormal{B}}, \widehat{\phi}^{\rar}_{\textnormal{T}}) \text{ is a \textit{u}-anchor}.$$

			Fix a lift $\tilde{A}^{\rightarrow}$, and the natural lifts $\tilde{K}$, $\tilde{\phi}^{\rightarrow}_{\text{B}}, \tilde{\phi}^{\rightarrow}_{\text{T}}$ such that 
			$$(\tilde{A}^{\rightarrow}, \tilde{K}, \tilde{\phi}^{\rightarrow}_{\text{B}},  \tilde{\phi}^{\rightarrow}_{\text{T}}) \text{ is a \textit{u}-anchor.}$$ 
			
			As in Lemma \ref{lemma:3semianchoredThenAnchored}, recall that $\tilde{A}^{\rightarrow}$ is crossed by infinitely many vertical CFS $\tilde{A}_k^{\downarrow}$ with their corresponding realizing points $\tilde{z}^{\downarrow}_k$ (see the construction of $\wh {\mathcal{A}}^{\dar}$ in Section \ref{sec:CFS} for details). 
			Again as in Lemma \ref{lemma:3semianchoredThenAnchored}, we write $\tilde{\phi}^{\downarrow}_{\textnormal{B},k} \subset \partial_{\text{B}}\tilde{A}^{\downarrow}_k \cap \tilde{A}^{\rightarrow}$ 
			(see Figure \ref{figure:InductiveAnchoring}). 
			Note that here exists $k_0$ such that $\tilde{\phi}^{\rightarrow}_{\text{B}} \subset \TT(\tl A^{\dar}_{k_0}) \subset  \text{R}(\tilde{\phi}^{\downarrow}_{\text{B},k_0}), \  \tilde{\phi}^{\rightarrow}_{\text{T}} \subset \TT(\tl A^{\dar}_{k_0}) \subset \text{R}(\tilde{\phi}^{\downarrow}_{\text{B},k_0})$. By the Anchoring Lemma, we have that for every $k \geq k_0$, there exists $j_k > 0$ such that 
			$$ \tilde{f}^{j_k}(\tilde{K}) \cap \tilde{\phi}^{\dar}_{\text{B},k} \neq \varnothing. $$
			
			Thus, we can conclude that 
			\begin{equation}\label{equation:driftinganchor}
				(\tilde{A}^{\dar}_k, \tilde{f}^{j_k}(\tilde{K}), \tilde{\phi}^{\dar}_{\text{B},k}, \tilde{\phi}^{\rar}_{\text{T}}) \text{ is a \textit{u}-anchor,}
			\end{equation}
			
			which implies that $\wh{f}^{j_k}(\wh K)$ is \textit{u}-anchored to $\wh{A}^{\dar}_k$ (again, see Figure \ref{figure:InductiveAnchoring} for details). We proceed in the same fashion for the infinitely many vertical CFS $\tilde{A}_k^{\uar}$ that cross $\tilde{A}^{\rar}$, which concludes the proof.  
			
			\begin{figure}[h]
				\centering
				
				\def\svgwidth{.92\textwidth}
				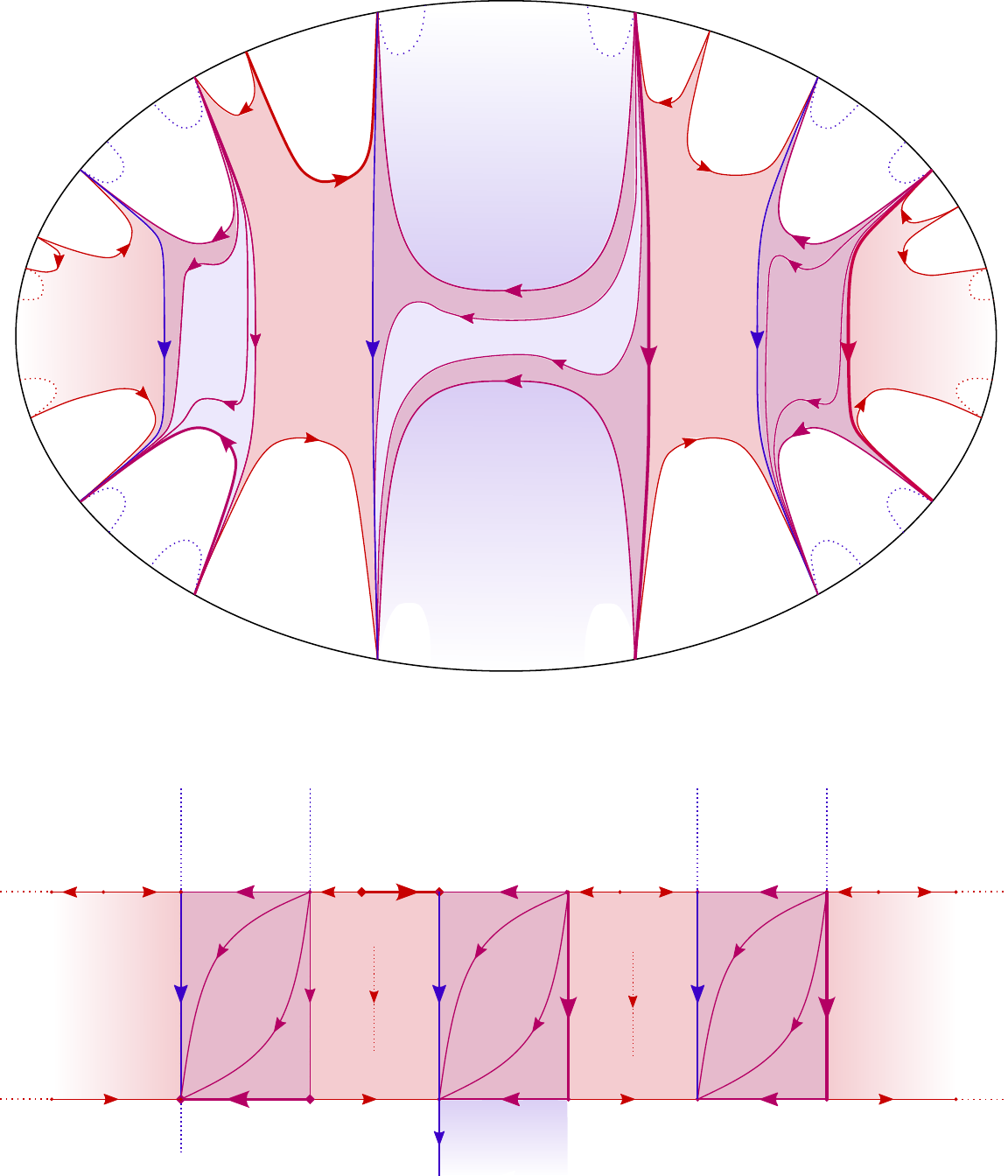
				
				\smallskip
				
				\caption {Here, $k = k_0$. Notice that $(\tilde{A}^{\dar}_k, \tilde{f}^{j_k}(\tilde{K}), \tilde{\phi}^{\dar}_{\text{B},k}, \tilde{\phi}^{\rar}_{\text{T}})$ is a \textit{u}-anchor. Moreover, for every $k' \geq k$, a future iterate of $\tilde{K}$ will be \textit{u}-anchored to $\tilde{A}^{\dar}_{k'}$}
				\label{figure:InductiveAnchoring}
			\end{figure}	
		\end{proof}

		We now repeat the Inductive Anchoring with care to obtain the following:
		
		\begin{lemma}\label{lem:VeryAnchored}
			Let $\wh f$ be a lift to $\R^2$ of $f \in \mathrm{Homeo}_0(\T^2)$, with $\vec{0} \in \mathrm{int}(\rho(\wh f))$. Fix a CFS $\wh A \in \wh {\mathcal{A}}$, and a weakly unstable continuum $\widehat{K} \subset \R^2$.
			
			Then, there exists $j>0$ such that $\widehat{f}^{j}(\widehat{K})$ is \textit{u}-anchored to $\widehat{A}$.
		\end{lemma}

		\begin{proof}
			Let us assume without loss of generality that $\wh K$ is \textit{u}-anchored to $\wh A^{\rar}_0$. By Lemma \ref{lem:InductiveAnchoring}, we obtain for $k \geq k_0$ (where $k_0$ given by Lemma \ref{lem:InductiveAnchoring}) that the desired statement is true for every $\wh A^{\dar}_k$ and every $\wh{A}^{\uar}_k$. 
			
			As in Lemma \ref{lem:InductiveAnchoring} (see also Figure \ref{figure:InductiveAnchoring}), take $j_k$ such that
			\begin{equation}\label{equation:driftinganchor1}
				(\tilde{A}^{\dar}_k, \tilde{f}^{j_k}(\tilde{K}), \tilde{\phi}^{\dar}_{\text{B},k}, \tilde{\phi}^{\rar}_{\text{T}}) \text{ is a \textit{u}-anchor.}
			\end{equation}

			Now, note that $\tilde{A}^{\dar}_k$ is in turn partitioned in fundamental domains by taking lifts $\tl A^{\rar}_l$ of each of the respective $\widehat{A}^{\rar}_l$, which intersect $\tilde{A}^{\dar}_k$, and then taking leaves  $\tilde{\phi}^{\rar}_{\text{B},l} \in \partial_{\BB} \tl A^{\rar}_{l} \cap \tl A^{\dar}_k$ and defining: 
			$$\tilde{D}^{\dar}_l = \text{L}_{\tilde{A}^{\dar}_k }(\tilde{\phi}^{\rar}_{\text{B},l}) \cap \text{R}_{\tilde{A}^{\dar}_k }(\tilde{\phi}^{\rar}_{\text{B},l-1}).$$
			
			Observe that $\tilde{\phi}^{\downarrow}_{\text{B},k} \subset \tilde{D}^{\dar}_0$, and therefore for $l < 0$, we obtain that, 
			\[\tilde{\phi}^{\rar}_{\text{T}} \subset \text{R}(\tilde{\phi}^{\rar}_{\text{B},l}); \ \  \tilde{\phi}^{\downarrow}_{\text{B},k} \subset \text{R}(\tilde{\phi}^{\rar}_{\text{B},l}); \ \ 
			\tilde{\phi}^{\rar}_{\text{B},-1} \in \text{R}(\tl \phi^{\rar}_{\TT}) \cap \text{R}(\tl \phi^{\dar}_{\BB,k}).\]
			
			Again by the Anchoring Lemma, for every $l < 0$ there exists $j_l > j_k$ such that 
			$$ \tilde{f}^{j_l}(\tilde{K}) \cap \tilde{\phi}^{\rar}_{\text{T},l} \neq \varnothing, \ \text{  where } \ \  \tilde{\phi}^{\rar}_{\text{T},l} \subset \tl A^{\dar}_k \cap \partial_{\TT}\tl A^{\rar}_l,$$
			and we may then conclude (see Figure \ref{figure:VeryAnchored}) that
			$$ (\tilde{A}^{\rar}_{l}, \tilde{f}^{j_l}(\tilde{K}), \tilde{\phi}^{\rar}_{\text{T}}, \tilde{\phi}^{\rar}_{\text{T},l}) \text{ is a \textit{u}-anchor.} $$  
			
			\begin{figure}[h]
				\centering
				
				\def\svgwidth{.92\textwidth}
				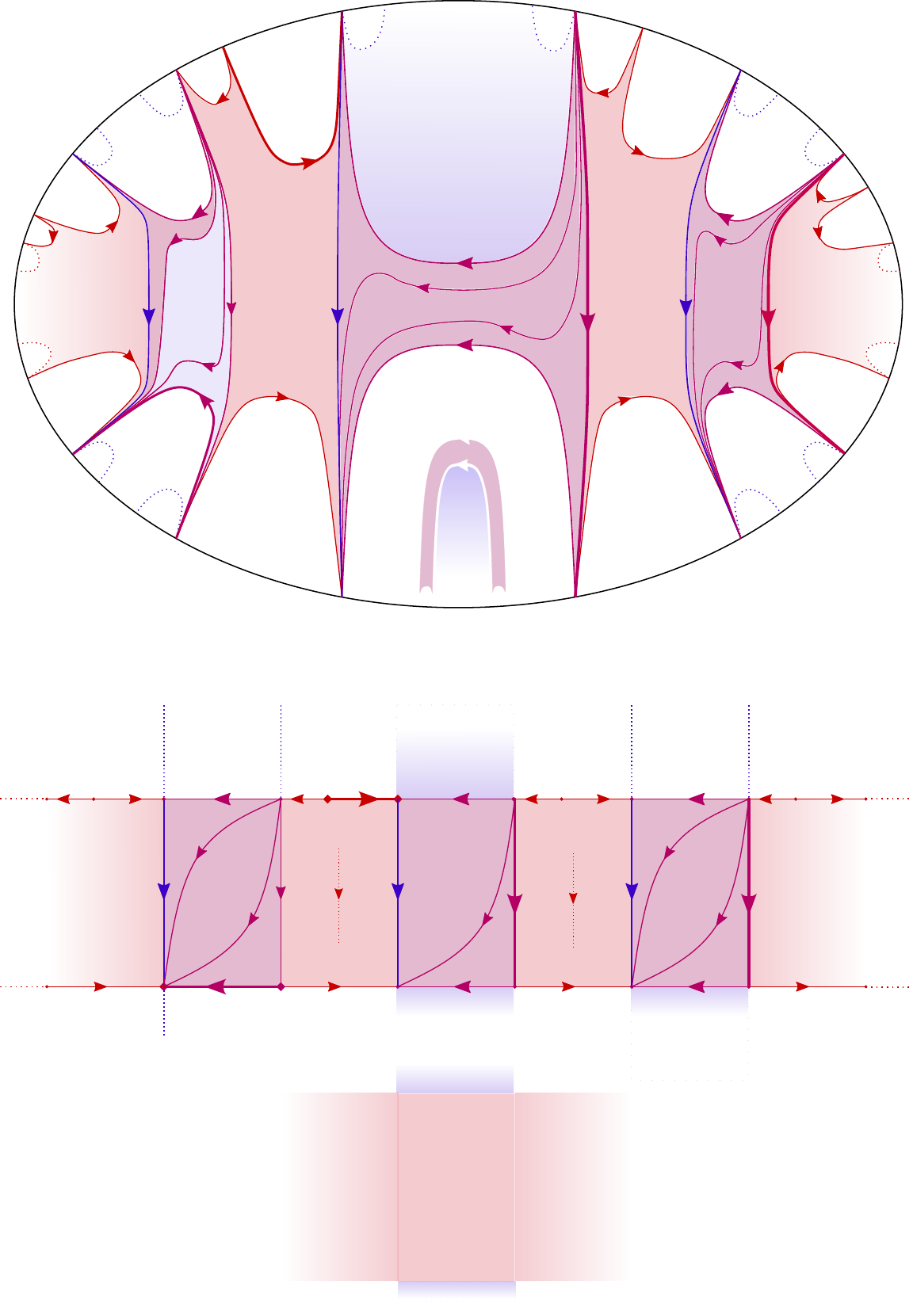
				
				\smallskip
				
				\caption {Note that $\tilde{\phi}^{\rar}_{\TT} \subset \BB_{\tilde{\phi}^{\rar}_{\BB,l}}(\tilde{A}^{\rar}_l)$, and thus $(\tilde{A}^{\rar}_l, \tilde{f}^{j_l}(\tilde{K}), \tilde{\phi}^{\rar}_{\text{T}}, \tilde{\phi}^{\rar}_{\text{T},l})$ is a \textit{u}-anchor. Similarly, $(\tilde{A}^{\rar}_l, \tilde{f}^{j_l}(\tilde{K}), \tilde{\phi}^{\rar}_{\text{B}}, \tilde{\phi}^{\rar}_{\text{T},l})$ is also a \textit{u}-anchor}
				\label{figure:VeryAnchored}
			\end{figure}

			We have just proved that $\widehat{f}^{j_l}(\widehat{K})$ is \textit{u}-anchored for $\widehat{A}^{\rar}_l$, for every $l < 0$. In the same fashion, we obtain that $\widehat{f}^{j_m}(\widehat{K})$ is \textit{u}-anchored for $\widehat{A}^{\lar}_m$, for every $m < 0$ (see Figure \ref{figure:AnchoringEverywhere} for details). Using that a future iterate of $\wh K$ is \textit{u}-anchored to some $\wh A^{\uar}_k \in \mathcal{A}^{\uar}$, and repeating the technique from Lemma \ref{lem:InductiveAnchoring}, we recover the same results, but with $l, m>0$. 
			
			\begin{figure}[h]
				\centering
				
				\def\svgwidth{.92\textwidth}
				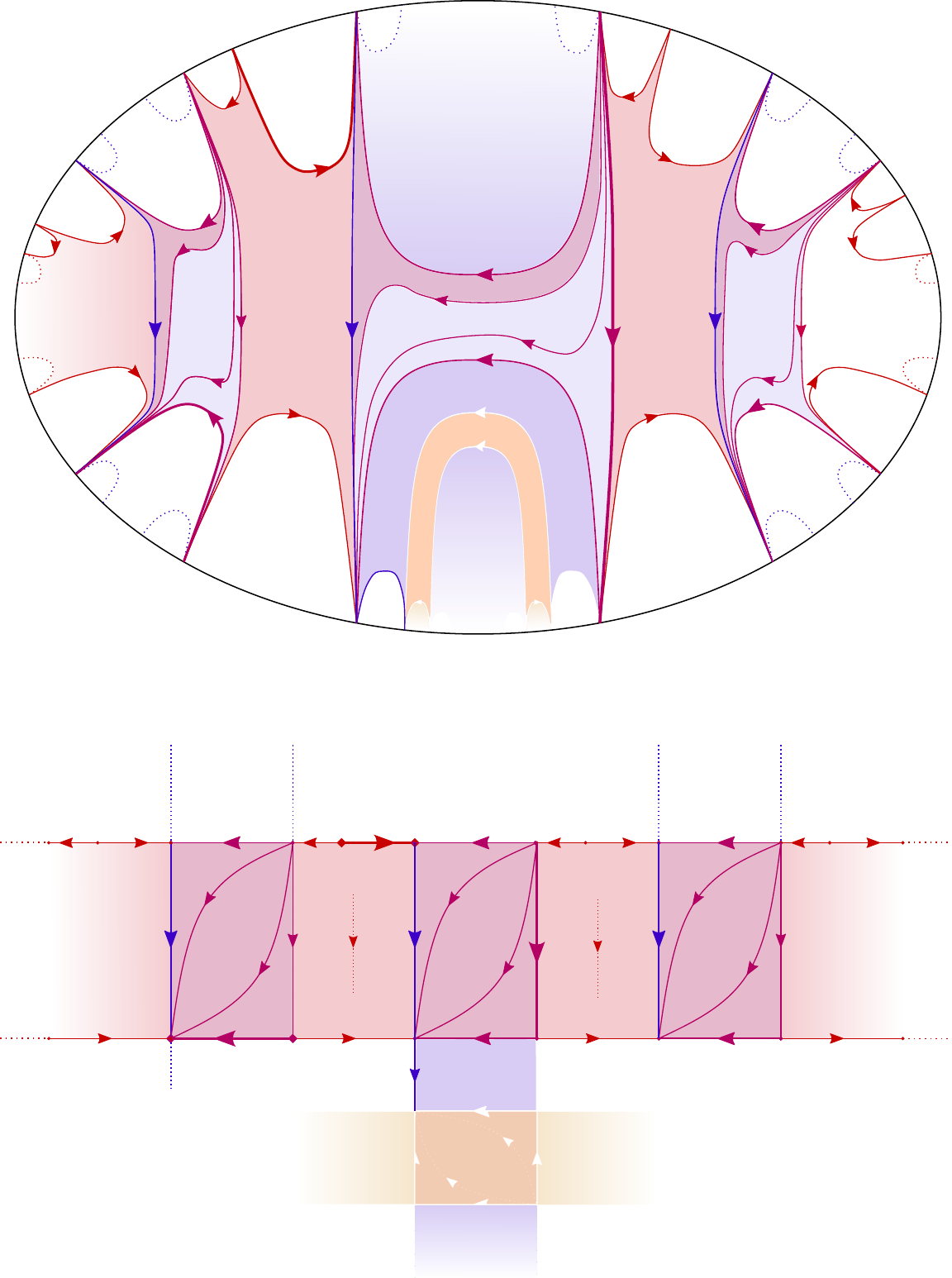
				
				\smallskip
				
				\caption {Example with $m =-1$. Note that $(\tilde{A}^{\lar}_m, \tilde{f}^{j_m}(\tilde{K}), \tilde{\phi}^{\lar}_{\text{B},m}, \tilde{\phi}^{\rar}_{\text{T}})$ and $(\tilde{A}^{\lar}_m, \tilde{f}^{j_m}(\tilde{K}), \tilde{\phi}^{\lar}_{\text{B},m}, \tilde{\phi}^{\rar}_{\text{B}})$ are \textit{u}-anchors.}
				\label{figure:AnchoringEverywhere}
			\end{figure}	
			
			We have proven that the \textit{u}-anchoring of $\wh K$ to a horizontal CFS, implies for \textit{every} horizontal CFS, the \textit{u}-anchoring of some future iterate of $\wh K$. Then, using Equation \ref{equation:driftinganchor1} we recall that a future iterate of $\wh K$ is \textit{u}-anchored to some vertical CFS, which then implies for every vertical CFS, the \textit{u}-anchoring of some future iterate of $\wh K$, which concludes the proof.

			
		\end{proof}

		We need one more result before the central one.
		
		\begin{lemma}\label{lem:TheRoute}
			Let $\wh K$ a \textit{u}-anchored continuum, let $\wh A ^{\rar}, \wh A^{\uar} \in \wh{\mathcal{A}}$ be two CFS, and let $\wh \phi^{\rar}_{\TT} \subset \partial_{\TT}\wh A^{\rar} \cap \wh A^{\uar}$. 
			
			Then, there exist lifts $\tl A^{\rar}, \tl A^{\uar}, \tl \phi^{\rar}_{\TT}$, and some $j >0$ such that 
			\begin{enumerate}
				\item $\tl \phi^{\rar}_{\TT} \subset \partial_{\TT} \tl A^{\rar} \cap \tl A^{\uar}$,
				\item $\tl K \subset \TT_{\tl \phi^{\rar}_{\TT}}(\tl A^{\rar})$,
				\item $\tl \phi^{\rar}_{\BB} \subset \partial_{\BB} \tl A^{\rar} \cap \tl A^{\uar}$
				\item $\tl f^{j}(\tl K) \cap \tl \phi^{\rar}_{\BB} \neq \varnothing$ 
			\end{enumerate}
		\end{lemma}
		
		Note that there are natural analogues to this statement for any pair of transversal CFS (exchanging $\TT$ for $\BB$ in half of the cases). 
		
		
		\begin{proof}
			Let us assume without loss of generality that $\wh K$ is \textit{u}-anchored to $\wh A^{\rar}_0$, by $\wh \phi_{\BB}$ and $\wh \phi_{\TT}$. Take a natural set of lifts $(\tl A^{\rar}_0, \tl K, \tl \phi_{\BB}, \tl \phi_{\TT})$ is \textit{u}-anchor. Let us also assume that $\wh A^{\rar} = \wh A^{\rar}_l$, and that $\wh A^{\rar} \cap \wh A^{\uar} \subset \wh D^{\rar}_n$ (the \textit{n}-th fundamental domain of $\wh A^{\rar}$).
			
			Now, take $k \in \Z$ with $k > n$, and such that $\wh K \subset \TT(\wh A^{\dar}_k)$. Let us also take $m < l$. We will now build \textit{a route} (sequence of CFS) in $\wh \R^{2}$, for an iterate of $\wh K$ to follow, and then to enter $\wh A^\rar_l$ \textit{through} $\wh \phi_{\TT}$. This will show us the desired lift $\tl A^{\rar}$ of $\wh A^{\rar}$ we need to take.  
			
			The route in $\wh \R^2$ is the following: we take $\wh A^\rar_0$, $\wh A^{\dar}_k$, $\wh A^{\lar}_{m}$, $\wh A^{\uar}_n, \wh A^{\rar}_l$, chosen such that $\wh A^\rar_0\subset T(\wh A^{\lar}_{m}), \wh A^\dar_k\subset T(\wh A^{\uar}_{n})$ and $\wh A^\lar_m\subset T(\wh A^{\rar}_{l})$ (see Figure \ref{figure:TheRoute} for details).  Then, take a natural set of lifts such that each of the successive intersections $\tl A^\rar_0 \cap \tl A^{\dar}_k$, $\tl A^{\dar}_k \cap \tl A^{\lar}_{m}$, $\tl A^{\lar}_{m} \cap \tl A^{\uar}_n$, $\tl A^{\uar}_n \cap \tl A^{\rar}_l$ are nonempty (you may think of this by taking {\color{black}a positively transverse path $\wh \gamma$ from $\wh K$ to $\wh \phi^{\rar}_{\TT}$ which is contained in the union of the CFS in the route and intersects them in the given order -see Figure \ref{figure:TheRoute}-}, and then lifting it to $\tl \D$). Use this to define a lift $\tl \phi^{\rar}_{\TT}$ of $\wh \phi^{\rar}_{\TT}$, such that $\tl \phi^{\rar}_{\TT} \subset \partial_{\TT} \tl A^{\rar}_l \cap \tl \phi^{\uar}_n$, and notice that by construction, 
			$$ \tl K \subset \TT(\tl A^{\dar}_k) \subset \TT(\tl A^{\lar}_{m}) \subset \TT(\tl A^{\uar}_n) \subset \TT_{\tl \phi^{\rar}_{\TT}}(\tl A^{\rar}_l).$$
			
			So far, the first two of the three desired properties are held. Now, four iterations of the Inductive Anchoring will give us four leaves (see Figure \ref{figure:TheRoute}),
			$$\tl \phi^{\dar}_{\BB} \subset \partial_{\BB} \tl A^{\dar}_k \cap \tl A^{\rar}_0; \ \  \tl \phi^{\lar}_{\BB} \subset \partial_{\BB} \tl A^{\lar}_m \cap \tl A^{\dar}_k; \ \ \tl \phi^{\uar}_{\BB} \subset \partial_{\BB} \tl A^{\uar}_n \cap \tl A^{\lar}_m; \ \  \tl \phi^{\rar}_{\BB} \subset \partial_{\BB} \tl A^{\rar}_l \cap \tl A^{\uar}_n,$$
			and some $j > 0$ such that the following quartets are \textit{u}-anchors:
			$$(\tl A^{\dar}_k, \tl f^{j}(\tl K), \tl \phi^{\dar}_{\BB}, \tl \phi_{\TT}); \ \ (\tl A^{\lar}_m, \tl f^{j}(\tl K), \tl \phi^{\lar}_{\BB}, \tl \phi_{\TT}); \ \  (\tl A^{\uar}_n, \tl f^{j}(\tl K), \tl \phi^{\uar}_{\BB}, \tl \phi_{\TT}); \ \  (\tl A^{\rar}_l, \tl f^{j}(\tl K), \tl \phi^{\rar}_{\BB}, \tl \phi_{\TT}),$$ 
			
		where the last anchor shows that $\tl f^{j}(\tl K) \cap \tl \phi^{\rar}_{\BB} \neq \varnothing$, and thus concludes the proof. 

			\begin{figure}[h]
				\centering
				
				\def\svgwidth{\textwidth}
				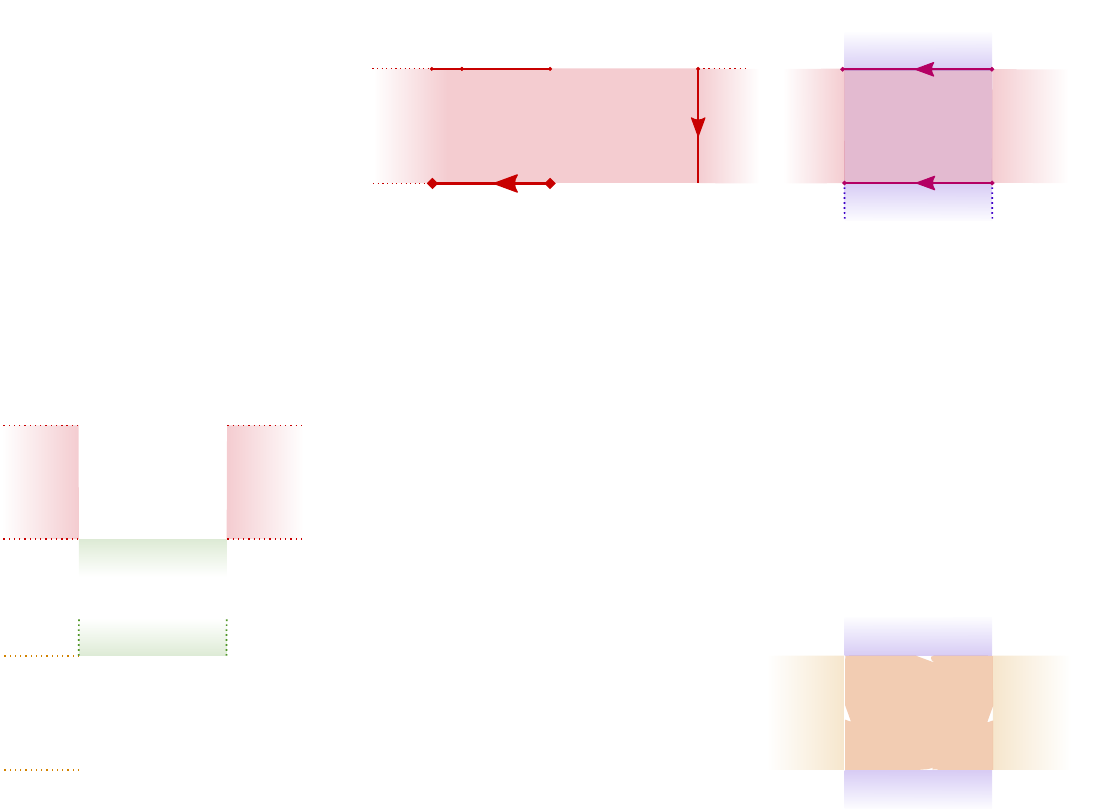
				
				\smallskip
				
				\caption {We lift $\wh \gamma$ to $\tl \D$, to find the lifts of $\wh A^{\dar}_k, \wh A^{\lar}_m, \wh A^{\uar}_n$ and $\wh A^{\rar}_l$. Note that by the Inductive Anchoring, $\tl f^{j}(\tl K)$ intersects $\tl \phi^{\dar}_{\BB}, \tl \phi^{\lar}_{\BB}, \tl \phi^{\uar}_{\BB}$ and $\tl \phi^{\rar}_{\BB}$.}
				\label{figure:TheRoute}
			\end{figure}
		\end{proof}

		We finish the subsection with a key result.
		
		\begin{proposition}[\textbf{Total Anchoring}]\label{prop:AnchoredEverywhere}
			Let $\wh f$ be a lift to $\R^2$ of $f \in \mathrm{Homeo}_0(\T^2)$, such that $\vec{0} \in \mathrm{int}(\rho(\wh f))$. Fix a weakly unstable continuum $\widehat{K} \subset \R^2$, a CFS $\widehat{A} \in \wh {\mathcal{A}}$, and a leaf $\widehat{\phi} \subset \widehat{A}$. 
			
			Then, there exist two lifts $\tl A$, $\tl \phi $ of $\wh A$ and $\wh \phi$, 
			and some $j > 0$, such that 
			$\tl{f}^{j}(\tl{K})$ is \textit{u}-anchored to $\tl A$, and $\tl{f}^{j}(\tl{K}) \subset \textnormal{R}_{\tl{A}}(\tl{\phi}).$  
		\end{proposition}
		
		\begin{proof}
			Let us assume without loss of generality that the given CFS $\wh A$ {\color{black}satisfies} $\wh A \in \mathcal{A}^{\rar}$, and let us write $\wh A = \wh A^{\rar}$. 
			Let us also assume that $\wh K$ is \textit{u}-anchored to $\wh A^{\rar}_0$, and take a natural set of lifts such that $(\tl A^{\rar}_0, \tl K, \tl \phi_{\BB}, \tl \phi_{\TT})$ is \textit{u}-anchor. 
			Let us recover once again the structure of fundamental domains for $\wh A^{\rar}$:
			$$ \wh D^{\rar}_k = \text{L}_{\wh A^{\rar}}(\wh \phi^{\dar}_{\TT,k}) \cap \text{R}_{\wh A^{\rar}}(\wh \phi^{\dar}_{\TT,k+1}).$$
			We use Lemma \ref{lem:PointsMoveSlowly}, to obtain $n \in \Z$ such that for every leaf $\wh \phi'$ in $\text{R}_{\wh A^{\rar}}(\wh \phi^{\dar}_{\TT,n+2})$, and any set of lifts $\tl \phi'$, $\tl \phi$ to {\color{black}$\tl \D$} belonging to the same lift of $\wh A^{\rar}$, we have
			\begin{equation}\label{eq:Controlled}
				\tl f^{j}(\tl \phi') \subset \mathrm{R}(\tl \phi), \text{ for every } 0 \leq j \leq 3. 
			\end{equation}

		Let us take $\wh A^{\uar}_n$, such that $\wh A^{\uar}_n \cap \wh A^{\rar} \subset \wh D^{\rar}_n$, and let us take $\wh \phi^{\rar}_{\TT} \subset \partial_{\TT} \wh A^{\rar} \cap \wh A^{\uar}_n$, $\wh \phi^{\rar}_{\BB} \subset \partial_{\BB} \wh A^{\rar} \cap \wh A^{\uar}_n$. We use Lemma \ref{lem:TheRoute} to find lifts $\tl A^{\rar}$, $\tl A^{\uar}_n$ such that $\tl K \subset \TT_{\tl \phi^{\rar}_{\TT}}(\tl A^{\rar})$, and some $j' >0$ such that $\tl f^{j'}(\tl K) \cap \tl \phi^{\rar}_{\BB} \neq \varnothing$. 
		
		Given $\tl \phi^{\rar}_{\TT} \subset \partial \tl D^{\rar}_n \subset \partial \tl A^{\rar}$, once again by the Anchoring Lemma we obtain that future iterates of $\tl K$ will intersect every leaf contained in $\tl D_{n'}$, whenever $n'>n$. Recall that 
		$$ \tl D^{\rar}_{n'} = \text{L}_{\tl A^{\rar}}(\tl \phi^{\dar}_{\TT,n'}) \cap \text{R}_{\tl A^{\rar}}(\tl \phi^{\dar}_{\TT,n'+1}).$$ 
		and let us take $j'' >0$ the smallest integer such that $\tl f^{j''}(\tl K) \cap \tl \phi^{\dar}_{\TT, n+2} \neq \varnothing$ (see Figure \ref{figure:Cling} for details. For ease of understanding the figure shows $\wh \R^2$, the arguments are made in $\tl \D$).

		\begin{figure}[h]
			\centering
			
			\def\svgwidth{\textwidth}
			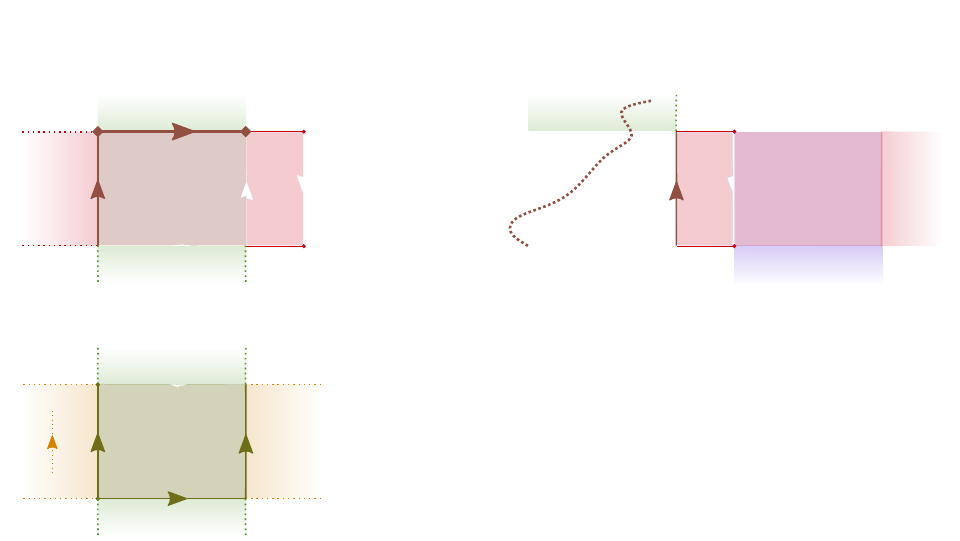
			
			\smallskip
			
			\caption {By the Anchoring Lemma, the natural lifts of $\wh \phi_0$ and $\wh \phi_1$ \textit{push} $\tl K$, and thus we have that $(\tl A^{\rar}, f^{j''+2}(\tl K), \tl \phi^{\rar}_{\BB,n+1},\tl \phi^{\rar}_{\TT,n+1})$ is a \textit{u}-anchor.}
			\label{figure:Cling}
		\end{figure}

		To finish, notice that 
		\begin{enumerate}
			\item $\tl f^{j''-1}(\tl K) \subset \RR(\tl \phi^{\dar}_{n+2})$, and thus $\tl f^{j''+2}(\tl K) \subset \RR(\tl \phi)$ (see Equation \ref{eq:Controlled}),
			\item $\tl f^{j''}(\tl K)$ is \textit{u}-anchored to $\tl A^{\uar}_{n+1}$ and to $\tl A^{\dar}_{n+1}$. By the Anchoring Lemma, $f^{j''+2}(\tl K)$ will intersect both $\tl \phi^{\rar}_{\BB,n+1} \subset \partial_{\BB}\tl A^{\rar} \cap \tl A^{\uar}_{n+1}$ and $\tl \phi^{\rar}_{\TT,n+1} \subset \partial_{\TT}\tl A^{\rar} \cap \tl A^{\dar}_{n+1}$, 
			\item $(\tl A^{\rar}, f^{j''+2}(\tl K), \tl \phi^{\rar}_{\BB,n+1},\tl \phi^{\rar}_{\TT,n+1})$ is a \textit{u}-anchor,
		\end{enumerate}
		so taking $j = j''+2$ concludes the proof. 
		\end{proof}

	Note that all techniques we have used, immediately yield analogous results for weakly stable continua.
		
	}

	\subsection{Mixing pairs of continua}
	
	Let us prove one last result before moving to the proof of Proposition \ref{prop:essfactor.isrot.mixing}. We emphasize that this result already has the mixing-like structure, and that it strongly uses Proposition \ref{prop:AnchoredEverywhere}.
	
	\begin{lemma}\label{lemma:WeaklyStableUnstableMixing}
		Let $\wh f$ be a lift to $\R^2$ of a torus homeomorphism $f \in \mathrm{Homeo}_0(\T^2)$, such that $\vec{0} \in \mathrm{int}(\rho(\wh f))$, and let $\wh K^{ws},\wh K^{wu}$ be a pair of respectively weakly stable and weakly unstable continua for $\wh f$. Then, there exists $l_0 \geq 0$ such that $$\widehat{f}^{l}(\wh K^{wu}) \cap \wh K^{ws} \neq \varnothing, \ \forall l\geq l_0.$$
	\end{lemma}
	
	\begin{proof}
		Fix a canonically foliated strip $\widehat{A} \in \wh{\mathcal{A}}$. By Proposition \ref{prop:AnchoredEverywhere}, there exists $l_{-} < 0$, and two leaves $\wh \phi ^{ws}_{\BB}, \wh \phi ^{ws}_{\TT}$  such that
		\[ (\wh A, \wh f^{l_-}(\wh K^{ws}), \wh \phi ^{ws}_{\BB}, \wh \phi ^{ws}_{\TT}) \text{ is a \textit{u}-anchor.}\] 
		
		Take {\bck{$\wh \phi \subset \wh A$}} such that 
		$$\wh \phi ^{ws}_{\BB}, \wh \phi ^{ws}_{\TT} \subset \partial \text{L}_{\wh A} (\wh \phi),$$
		and again by Proposition \ref{prop:AnchoredEverywhere}, there exists $l_{+} > 0$ and two leaves $\wh \phi ^{wu}_{\BB}, \wh \phi ^{wu}_{\TT}$ such that $\wh \phi ^{wu}_{\BB} \subset \partial \textnormal{R}_{\widehat{A}}(\widehat{\phi}), \wh \phi ^{wu}_{\TT} \subset \partial \textnormal{R}_{\widehat{A}}(\widehat{\phi})$ and such that
		\[ (\wh A, \wh f^{l_+}(\wh K^{wu}), \wh \phi ^{wu}_{\BB}, \wh \phi ^{wu}_{\TT}) \text{ is a \textit{u}-anchor.}\] 
		
		This implies that we may take natural lifts such that 
		\[ (\tl A, \tl f^{l_+}(\tl K^{wu}), \tl \phi ^{wu}_{\BB}, \tl \phi ^{wu}_{\TT}) \text{ is a \textit{u}-anchor,}\]
		\[ (\tl A, \tl f^{l_-}(\tl K^{ws}), \tl \phi ^{ws}_{\BB}, \tl \phi ^{ws}_{\TT}) \text{ is an \textit{s}-anchor,}\]
		and that there exists $\tl \phi \subset \tl A$ lift of $\wh \phi$ such that 
		\[ \tl \phi ^{wu}_{\BB} \subset \text{R}(\tl \phi); \  \tl \phi ^{wu}_{\TT} \subset \text{R}(\tl \phi); \ \tl \phi ^{ws}_{\BB} \subset \text{L}(\tl \phi) ; \  \tl \phi ^{ws}_{\TT} \subset \text{L}(\tl \phi).  
		\]
		\[\tl f^{l_+}(\tl K^{wu}) \subset \text{R}(\tl \phi); \  \tl f^{l_-}(\tl K^{ws}) \subset \text{L}(\tl \phi).\]
		See Figure \ref{figure:WeaklyStableUnstableMixing} for details.

		\begin{figure}[h]
			\centering
			
			\def\svgwidth{.92\textwidth}
			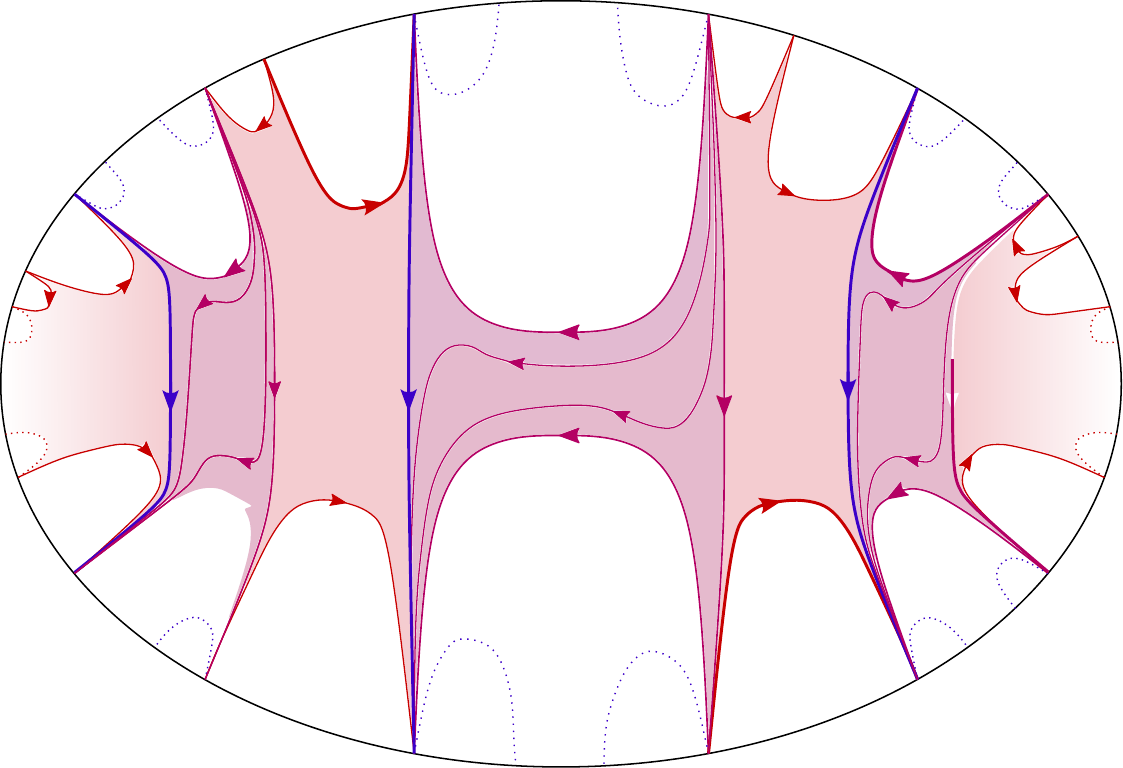
			
			\smallskip
			
			\caption {Configuration after having anchored both $\widehat{K}^{ws}, \widehat{K}^{wu}$ to the same canonically foliated strip. For this particular example, we have that $ \tilde{f}^{l}(\tilde{K}^{wu}) \cap \tilde{K}^{ws} \neq \varnothing, \text{ for every } l \geq l_+ - l_- + 4$. }
			\label{figure:WeaklyStableUnstableMixing}
		\end{figure}

		Take a leaf $\tl \phi ' \subset \tl A$ such that $\tl \phi ^{ws}_{\BB} \subset \text{R}(\tl \phi ') ; \  \tl \phi ^{ws}_{\TT} \subset \text{R}(\tl \phi ')$ {\bck{(once again, see Figure \ref{figure:WeaklyStableUnstableMixing})}}. By the Anchoring Lemma, we know there exists $l'_+ > l_+$ such that
		\[ \tl f^{l}(\tl K^{wu}) \cap \tl \phi ' \neq \varnothing \text { for every } l \geq l'_+.\]
		
		This implies that
		{\bck{\[ \tl f^{l}(\tl K^{wu}) \cap (\tl \phi ^{ws}_{\BB} \cup \tl f^{l_-}(\tl K^{ws}) \cup \tl \phi ^{ws}_{\TT}) \neq \varnothing \text { for every } l \geq l'_+,\]}
		}
		as the latter union separates the disk $\tl \D$. Now, the key is to observe that
		
		\[\tl f^{l_+}(\tl K^{wu}) \subset \text{R}(\tl \phi) \subset \text{L}(\tl \phi ^{ws}_{\BB}) \cap \text{L}(\tl \phi ^{ws}_{\TT}),\]
		which implies that for large enough values of $l$ (say $l \geq l''_+ \geq l'_+)$, we have that 
		\[\tl f^{l}(\tl K^{wu}) \cap (\tl \phi ^{ws}_{\BB} \cup \tl \phi ^{ws}_{\TT}) = \varnothing,\]
		which in turn implies that
		\[\tl f^{l}(\tl K^{wu}) \cap \tl f^{l_-}(\tl K^{ws}) \neq \varnothing, \text{ for every } l \geq l''_+,\]
		from where we conclude that 
		\[ \wh f^{l}(\wh K^{wu}) \cap \wh K^{ws} \neq \varnothing \text{ for every } l \geq l_0 = l''_+ - l_- > 0,\]
		which concludes the proof.   
	\end{proof}
	
	Note that this result, together with Lemma \ref{lem:stableunstablecontinuadense}, show that 
	
	\begin{remark}\label{rem:denseStableSets}
		Let $\wh z$ be the lift of a periodic point of a frice torus homeomorphism $g$. Then ${\color{black}W^s(\wh z)}$ and ${\color{black}W^{u}(\wh z)}$ are dense in $\R^2$.  
	\end{remark}

	We now finish with the proof of Proposition \ref{prop:essfactor.isrot.mixing}.
	
	\begin{proof}[Proof of Proposition \ref{prop:essfactor.isrot.mixing}]
		Take two open sets $\wh V_1, \wh V_2 \subset \R^2$. By Lemma \ref{lem:stableunstablecontinuadense}, we may take $\widehat{K}^{wu} \subset \wh V_1, \widehat{K}^{ws} \subset \wh V_2$, respectively weakly unstable and weakly stable continua. By Lemma \ref{lemma:WeaklyStableUnstableMixing}, there exists $l_0 > 0$ such that 
		\[ \wh g^{l}(\wh K ^{wu}) \cap \wh K^{ws} \neq \varnothing \text{ for every } l \geq l_0,\]
		which implies that 
		\[ \wh g^{l}(\wh V_1) \cap \wh V_2 \neq \varnothing \text{ for every } l \geq l_0,\]
		which completes the proof.
	\end{proof}

	\subsection{Proof of Theorem D}.

	The following result is central for the proof of Theorem \ref{thmD:TransitiveThenMixing}, and it is nothing more than an adaptation of Lemma \ref{lem:stableunstablecontinuadense}, which worked for frice torus homeomorphisms, to the context of transitive ones.
	
	\begin{lemma}\label{lem:TransitiveThenWsWu}
		Let $\wh f$ be a lift of a transitive homeomorphism $f \in \mathrm{Homeo}_0(\T^2)$, with $\vec{0} \in \mathrm{int}(\rho(\wh f))$. Then, for every open set $\wh V \subset \R^2$, there exist two nontrivial continua $\wh K^{ws}, \wh K^{wu} \subset \wh V$, which are respectively weakly stable and weakly unstable continua for $f$. 
	\end{lemma}
	
	\begin{proof}
		Let $V = \wh \pi (\wh V)$, and assume up to taking a connected component, that $\wh V$ is connected. By the realization result in \cite{franks89}, let us take the lifts $\wh z_1, \wh z_2$ of two $f$-periodic points $z_1,z_2$ with respective rotation vectors $(a_1,b_1), (a_2,b_2)$ with $a_1, b_1, b_2 < 0, a_2 > 0$. 
		
		Take the respective projections $\wh z_{1\sim} = \pi_{\sim}(\wh z_1),\wh z_{2\sim} = \pi_{\sim}(\wh z_2)$ to $\R^2_{\sim}$ for the essential factor. By Remark \ref{rem:denseStableSets}, we know that ${\color{black}W^s(z_{1\sim})}, {\color{black}W^s(z_{2\sim})}$ are both dense in $\R^2_{\sim}$. In particular, by taking their preimage by the semiconjugacy $\pi_{\sim}$, we will obtain that 
		\[{\color{black}W^{'s}(\wh z_1)} = \pi_{\sim}^{-1}({\color{black}W^s(\wh z_{1\sim})}) \text{ goes through every fundamental domain,}\]
		and similarly for ${\color{black}W^{'s}(\wh z_2)} = \pi_{\sim}^{-1}({\color{black}W^s(\wh z_{2\sim})})$  
		
		Given that $f$ is transitive, by \cite[Corollary E]{korotal} we know that $\mathrm{Ess}(f) = \mathbb{T}^2$, so there exists $k > 0$ such that 
		\[\underline{V} := \bigcup \limits_{j=0}^{j=k} f^j(V) \text{ is fully essential.}\]
		
		which means there exist two respective points in the finite orbits of $z_1$ and $z_2$ (which we will give the same name for the sake of simplicity and without loss of generality), such that 
		\[{\color{black}W^{'s}(z_1)} \cap V \neq \varnothing, \ {\color{black}W^{'s}(z_2)} \cap V \neq \varnothing.\]
		
		Up to retaking the lifts $\wh z_1, \wh z_2$, we may assume without loss of generality that ${\color{black}W^{'s}(\wh z_1)} \cap \wh V \neq \varnothing, {\color{black}W^{'s}(\wh z_2)} \cap \wh V \neq \varnothing$. In particular, there exists $m > 0$ such that 
		\[K^{'s}_m(\wh z_1) \cap \wh V \neq \varnothing, \  K^{'s}_m(\wh z_2) \cap \wh V \neq \varnothing,\]
		where we define 
		{\bck{\[ K^{'s}_m(\wh z_1) = \pi_{\sim}^{-1}(K^s_m(\wh z_{1\sim})), \ K^{'s}_m(\wh z_2) = \pi_{\sim}^{-1}(K^s_m(\wh z_{2\sim})).\]
		}
		} 
		Let us take a curve $\wh \gamma \subset \wh V$ which goes from a point in $K^{'s}_m(\wh z_1)$ to a point in $K^{'s}_m(\wh z_2)$. 
		{\bck{Note that there exists some $d > 0$ such that}}
		\[\sup_{j \in \Z^+} \mathrm{diam}(\wh f^j(K^{'s}_m(\wh z_1))) \leq d, \sup_{j \in \Z^+} \mathrm{diam}(\wh f^j(K^{'s}_m(\wh z_2))) \leq d, \]
		because each iterate intersects a uniformly bounded number of fundamental domains. This shows that 
		\[K^{'s}_m(\wh z_1) \subset K^s_d(\wh z_1), \ K^{'s}_m(\wh z_2) \subset K^s_d(\wh z_2),\]
		which allows us to use Lemma \ref{lemma:AnchorTypeA} and conclude that $\wh \gamma$ is a weakly unstable continuum. In identical fashion we can build $\wh \gamma' \subset \wh V$ a weakly stable continuum, which concludes the proof.

	\end{proof}

	We now finish the section with the proof of Theorem \ref{thmD:TransitiveThenMixing}. Let us recall the desired statement:
	
	\paragraph{\textbf{Theorem D}} Let $f \in \textnormal{Homeo}_0(\mathbb{T}^2)$ be a transitive homeomorphism. 
	
	\begin{enumerate}
		\item If $\rho(f)$ has nonempty interior, then $f$ is topologically mixing. 
		\item Furthermore, if $f$ has a lift $\wh f$ with $\vec{0} \in \mathrm{int}(\rho(\wh f))$, then $\wh f$ is also topologically mixing.
	\end{enumerate}
	
	\begin{proof}
		Note that the second item implies the first. {\bck{Indeed,}} suppose we have proven the second: we know that a power $f^j$ of $f$ has a lift $\wh{f^j}$ with $\vec{0}$ in the interior of its rotation set, which by the second item will be topologically mixing, which implies that $f^j$ is topologically mixing, which by Lemma \ref{lem:MixingPower} then implies that $f$ is also topologically mixing. 
		
		Let us then take $\wh f$ with $\vec{0} \in \mathrm{int}(\rho(\wh f))$, and two open sets $\wh V_1, \wh V_2 \subset \R^2$. By Lemma \ref{lem:TransitiveThenWsWu}, we can take $\wh K^{wu} \subset \wh V_1$, $\wh K^{ws} \in \wh V_2$, and by Lemma \ref{lemma:WeaklyStableUnstableMixing}, there exists $l_0 > 0$ such that 
		\[ \wh f^{l}(\wh K ^{wu}) \cap \wh K^{ws} \neq \varnothing \text{ for every } l \geq l_0,\]
		which implies that 
		\[ \wh f^{l}(\wh V_1) \cap \wh V_2 \neq \varnothing \text{ for every } l \geq l_0,\]
		which completes the proof. 
	\end{proof}

	\section{Heteroclinic pseudo-rectangles}\label{sec:HPR}
	
	We have already proven four out of the six properties in Theorem \ref{thmA:semiconjugation} for the essential factor $g$. In this section we construct the main tool for the two properties which are still left to prove:
	
	\begin{itemize}
		\item For each open set $U \subset \T^2$, $g$ has a Markovian horseshoe $X \subset U$, 
		\item $g$ is area preserving. 
	\end{itemize} 
	
	Assuming we have proven the density of Markovian horseshoes, the proof of the second item appears as a consequence almost automatically; the extensive proof is given in Proposition \ref{prop:AreaPreservingConjugate}.
	
	The proof of the first item is much more technical, and will require fine use of the anchoring techniques together with the construction of a new tool, to which this section is devoted. 
	
	The idea is to build a family of filled continua, which we will call fitted heteroclinic pseudo-rectangles (FHPR). Proposition \ref{prop:ChaoticRectanglesAreDense}, which shows the existence of an FHPR inside any open set of the plane for any frice torus homeomorphism, is the key result of the section. This will be crucial to prove the existence of a dense family of Markovian horseshoes for the essential factor $g$: Proposition \ref{prop:EssentialFactorHasDenseHorseshoes} {\bck{yields the existence of a Markovian horseshoe contained in each FHPR.}} 
	{\bck{We will adapt}} to this context, the classical notion of heteroclinic intersection for stable-unstable manifolds of hyperbolic periodic points.

	\subsection{Topological structure}
	
	\begin{definition}
		Let $S$ be a surface. We will say $R \subset S$ is a \textit{rectangle} if it is homeomorphic to $[0,1]^2$, as the image of a homeomorphism $h:[0,1]^2 \to h([0,1]^2) = R \subset S$. We will call \textit{sides} of $R$ to the image of the sides of $[0,1]^2$ by $h$. We will call \textit{horizontal} and \textit{vertical sides} to the sets $R_{\text{H}}$ and $R_{\text{V}}$, respectively the images of the horizontal and vertical sides of $[0,1]^2$ by $h$. 
	\end{definition}

	\begin{definition}[\textbf{Heteroclinic pseudo-rectangle}]\label{defi:ChaoticRectangle}
		An inessential continuum $P \subset \T^2$ will be a \emph{heteroclinic pseudo-rectangle} (HPR) for a torus homeomorphism $f$ in the General Hypothesis with associated MDTD $(f,I,\mathcal{F})$ if:
		\begin{enumerate}
			\item $P$ contains no singularities of $\F$,
			\item $P$ is the closure of an open disk $D$, 
			\item For any $z \in \partial P$, any neighbourhood $U$ of $z$ intersects the fully essential component $P_{\infty}$ of the complement of $P$, 
			\item (Sides). {\color{black}$\partial P = K^s_1 \cup K^s_2 \cup K^u_1 \cup K^u_2$}, respectively two stable and two unstable continua, contained in the (un)stable sets of four periodic points $z^s_1, z^s_2, z^{u}_1, z^{u}_2$, with different nontrivial associated rotation vectors.
			\item $K^{s}_i \cap K^u_j \neq \varnothing, \text{ for every } i,j = 1,2.$
		\end{enumerate}
	\end{definition}

	\begin{remark}
		From the fourth item of last definition, and the definition of stable set, we automatically obtain that $K^{\sigma}_1 \cap K^{\sigma}_2 = \varnothing, \text{ where } \sigma = s,u$. 
	\end{remark}

	Note that any lift $\wh P$ of a heteroclinic pseudo-rectangle is homeomorphic to $P$, we will then say that $\wh P \subset \R^2$ is a heteroclinic pseudo-rectangle if it is a connected component of $\wh{\pi}^{-1} (P)$, where $P$ is a heteroclinic pseudo-rectangle. For $\wh P \subset \R^2$, the third condition is equivalent to saying that for any neighbourhood $\wh U$ of a point $\wh z \in \partial \wh P$, the set $\wh U$ intersects the only unbounded component $\wh P_{\infty}$ of the complement of $\wh P$. In the same fashion, $\tilde P$ is a heteroclinic pseudo-rectangle when it is a lift of a heteroclinic pseudo-rectangle $\wh P$.

	\medskip
	
	\paragraph{\textbf{Prime end model.}} Given a heteroclinic pseudo-rectangle $P = \mathrm{cl}(D)$, we will denote by $\breve{P} \simeq \mathrm{cl}(\D)$ to the Caratheodory compactification (see Definition \ref{def:Caratheodory}) of the disk $D$, 
	and we will call $h$ the Riemann map from $\D$ to $D$.  
	
	{\color{black}
	\begin{definition}[Prime end Side]\label{def:PrimeEndSide}
		Let us define $\breve X^s_1 \subset \partial \D$ as the set of accessible prime ends $\breve{z}$ in $\partial \breve{P}$ for which there exists a ray $r$ landing at $\breve{z}$, with $h(r)$ landing at $z \in K^s_1$ {\color{black}(recall Definitions \ref{def:ray} through \ref{def:AccessiblePrimeEnd})}. Then, we define the \textit{prime end side} $\breve{K}^s_1 := \mathrm{cl}(\breve X^s_1)$. We define $\breve{K}^s_2$, $\breve{K}^u_1$, and $\breve{K}^u_2$ in the same fashion. 
	\end{definition}

	Let us state a result we will use in this section. 
	
	\begin{lemma}{\label{lem:Accessible}}
		Let $\breve{P}$ be the prime end model of an HPR. Let $\breve{z}, \breve{w}, \breve{z}', \breve{w}' \in \partial \D$ be four cyclically-ordered accessible prime ends, let $z, w, z', w' \in \partial D$ be their respective impressions, and let $X \subset \partial D$ be a closed set. Suppose that $z,z' \in X$, and that $w, w' \notin X$. 
		
		Then, $X$ is not connected.  
	\end{lemma}
	
	\begin{proof}
		Let us take four rays $r_{\breve{z}}, r_{\breve{w}}, r_{\breve{z}'}, r_{\breve{w}'}$ landing respectively at {\color{black}$\breve{z}, \breve{w}, \breve{z}', \breve{w}'$}, such that 
		
		\begin{itemize}
			\item The rays $h(r_{\breve{z}})$, $h(r_{\breve{w}})$, $h(r_{\breve{z}'})$, $h(r_{\breve{w}'})$ respectively land at {\color{black} accessible points $z, \ w, \ z', \ w'$ (see Remark \ref{rem:AccessibleImpression})}.  
			\item $ r_{\breve{z}}^{-1} \cdot r_{\breve{z}'} \ \wedge  \ r_{\breve{w}}^{-1} \cdot r_{\breve{w}'} = 1 \implies h(r_{\breve{z}}^{-1}) \cdot h(r_{\breve{z}'}) \ \wedge  \ h(r_{\breve{w}}^{-1}) \cdot h(r_{\breve{w}'}) = 1$
		\end{itemize}
		
		Since $X$ is closed, we know that $\textnormal{d}(w,X) > 0$ and $\textnormal{d}(w',X) > 0$, we then take two sufficiently small neighbourhoods {\color{black} $U_w, U_{w'} \subset \T^2$} respectively of $w, w'$ such that they intersect neither the crosscut $ h(r_{\breve{z}}^{-1}) \cdot h(r_{\breve{z}'})$ nor {\color{black}$X$}. Note that from the heteroclinic pseudo-rectangle definition, both $U_w$ and $U_{w'}$ intersect $P_{\infty}$. This means we can build a curve $\gamma$ from $\partial U_{w'}$ to $\partial U_w$ which is contained in $P_{\infty}$ except for its endpoints, and is therefore disjoint from $h(r_{\breve{z}}^{-1}) \cdot h(r_{\breve{z}'})$. We may then complete it with two simple arcs $\eta_w, \ \eta_{w'}$ respectively contained in $U_w, \ U_{w'}$ except for their endpoints, to obtain a simple closed curve $\gamma'$ which {\color{black}satisfies}
		$$h(r_{\breve{z}}^{-1}) \cdot h(r_{\breve{z}'}) \wedge \gamma' = h(r_{\breve{z}}^{-1}) \cdot h(r_{\breve{z}'}) \wedge h(r_{\breve{w}}^{-1}) \cdot h(r_{\breve{w}'})= 1, \textnormal{ and } \gamma' \cap X = \varnothing,$$
		which implies that $\gamma'$ separates $z$ from $z'$, and shows that {\color{black}$X$} is not connected. 
	\end{proof}
	
		\begin{figure}[h]
		\centering
		
		\def\svgwidth{\textwidth}
		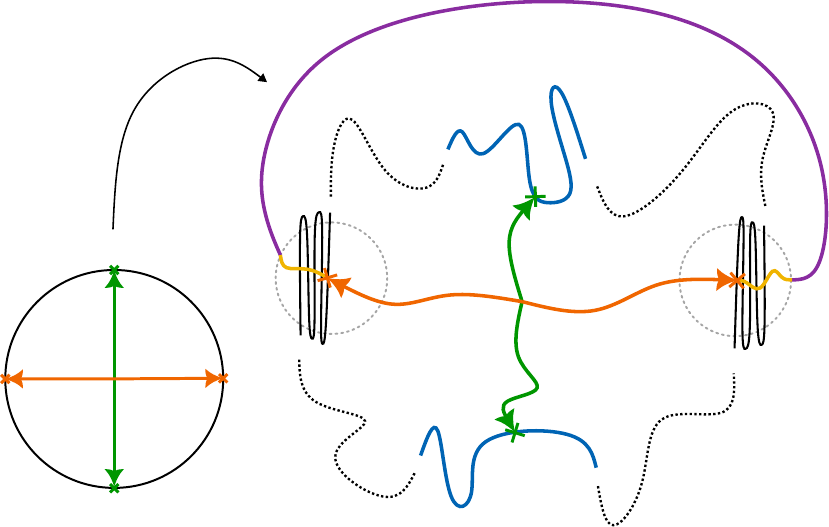
		
		\smallskip
		
		\caption {The curve $\gamma'$ separates $z$ from $z'$. }
		\label{figure:Accessible}
	\end{figure}	
	
	}

	The following {\color{black}lemma states that $\breve{P}$ has a rectangle-like structure}. 
	
	\begin{lemma}\label{lemma:PrimeEndChaoticRectangle}
		$\breve{K}^\sigma_i$ is a closed interval for every $\sigma \in \{s,u\}, \ i \in \{1,2\}$, and their one-to-one intersections are either empty or single points. 
	\end{lemma}
	
	\begin{proof}
		Recall for the proof that $K^{\sigma}_i, K^{\sigma}_j$ are at a positive distance. Note that by construction, each $\breve{K}^\sigma_i$ is a closed subset of the circle. 
		\begin{enumerate}
			\item \textbf{$\breve{K}^\sigma_i \neq \mathbb{S}^1$.} 
			
			Take $z \in K^{\sigma}_j$ on the \textit{opposite side} of $P$, and take a neighbourhood $U$ of $z$ which does not intersect $K^{\sigma}_i$ (note that $\mathrm{d}(K^{\sigma}_i, K^{\sigma}_j) > 0$). Given the set of accessible points in $\partial P$ is dense (see Remark \ref{rem:AcessiblePoints}), take an accessible point $z' \in U$, and note that $\breve{z}' \notin \breve{K}^{\sigma}_i$.
			
			\item \textbf{$\breve{K}^{\sigma}_i \neq \varnothing$.} 
			
			By symmetry, it is enough to prove this for $\breve{K}^s_1$. Note that there exists $z \in K^s_1$ with a neighbourhood $U$ of $z$ such that it intersects neither $K^{u}_1$ nor $K^{u}_2$, as we could otherwise write $K^s_1$ as the union of the two {\color{black}nonempty} disjoint closed sets $K^s_1 \cap K^{u}_1, \ K^s_1 \cap K^{u}_2$ {\color{black}(See Definition \ref{defi:ChaoticRectangle}, item (5))}, which would be a contradiction, because $K^s_1$ is connected. 
			Then, by Remark \ref{rem:AcessiblePoints}, $U$ must contain an accessible point $z' \in K^s_1$, and we then have $\breve{z}' \in \breve{K}^s_1$. 
			
			\item $\breve{K}^s_1 \cup \breve{K}^{s}_2 \cup \breve{K}^{u}_1 \cup \breve{K}^{u}_2 = \mathbb{S}^1.$
			
			This is because the set of accessible prime ends is dense in $\partial \breve{P} = \mathbb{S}^1$ (See Theorem \ref{thm:milnor}). 
			
			\item $\breve{K}^{\sigma}_i$ is connected. 
			
			Suppose by contradiction that it has two different connected components. {\color{black} By Definition \ref{def:PrimeEndSide}, these are two disjoint closed sets in $\partial \D$, each containing at least one accessible prime end by Theorem \ref{thm:milnor}. The complement of their union is an open set, and thus again by Theorem \ref{thm:milnor}, each of its connected components has an accessible prime end}. Let us then take four cyclically-ordered accessible prime ends $\breve{z},\breve{w},\breve{z}',\breve{w}'$, such that $\breve{z}, \breve{z}' \in \breve{K}^{\sigma}_i$, and $\breve{w}, \breve{w}' \notin \breve{K}^{\sigma}_i$. {\color{black}Again by Definition \ref{def:PrimeEndSide}, we may then take four rays $r_{\breve{z}}, r_{\breve{w}}, r_{\breve{z}'}, r_{\breve{w}'}$ landing respectively at {\color{black}$\breve{z}, \breve{w}, \breve{z}', \breve{w}'$}, such that 
			their images $h(r_{\breve{z}})$, $h(r_{\breve{w}})$, $h(r_{\breve{z}'})$, $h(r_{\breve{w}'})$ respectively land at {\color{black} accessible points} $z, \ w, \ z', \ w'$ with $z, \ z' \in K^{\sigma}_i$, $w, w', \notin K^{\sigma}_i$ {\color{black}(see Remark \ref{rem:AccessibleImpression})}. 
			
			Then, we can apply Lemma \ref{lem:Accessible} and conclude that $K^{\sigma}_i$ is not connected, which is a contradiction and concludes the proof.  
			}

			
			
			\item No intersection $\breve{K}^{\sigma}_i, \breve{K}^{\tau}_j$ contains an interval.
			
			Again by symmetry, let us suppose by contradiction that there exists an interval $I \subset \breve{K}^s_1 \cap \breve{K}^{u}_1$. Take two accessible prime ends $\breve{z},\breve{z}' \in \breve{K}^s_1 \cap I$, and respective rays $r_{\breve{z}}$, $r_{\breve{z}'}: [0,1] \to \breve{P}$ as in last item. Let $\theta_n$ be an arc of circumference centered at the origin, connecting  $r_{\breve{z}}$ to $r_{\breve{z}'}$ and at a distance at most  $1/n$ of  $I$. Take their images $h(\theta_n)$ by $h$. Up to taking a subsequence, the Hausdorff limit $\Theta$ of $h(\theta_n)$ is a nontrivial continuum. Provided this continuum is contained in $K^s_1$, we obtain our contradiction: we could repeat the construction with $\breve{K}^{u}_1$ and we could take subarcs $\theta'_n \subset \theta_n$ converging to a subcontinuum $\Theta' \subset K^{u}_1 \cap K^{s}_1$, which would mean we would have a dynamically bounded nontrivial continuum for $g$ {\color{black}(see Corollary \ref{cor:IntersectionStableUnstableTotDisconnected}).} 
			
			Let us then prove that the Hausdorff limit of $\theta_n$ is contained in $K^s_1$. This is again a {\color{black}connectedness} argument. Suppose by contradiction that there exists $w \notin K^s_1$ and $w_n \in \theta_n$ such that $w_n \to w$. We would have a neighbourhood $U_w$ of $w$ which is disjoint from $K^s_1$. We then take $n$ sufficiently large, such that $w_n \in U_w$, and build a curve from {\color{black}$h(r_{\breve{z}}(0))$} to $w$, which is disjoint from $K^s_1$. {\color{black}As in Lemma \ref{lem:Accessible}}, we may complete this curve to a simple closed curve $\gamma'$ which would separate $K^s_1$ {\color{black}and yield the desired contradiction}. This finishes the proof.        
		\end{enumerate}
	\end{proof}

	We will strongly use the following two lemmas for the two main {\color{black}results} of Section \ref{sec:DenseTopHorseshoes}: Proposition \ref{prop:EssentialFactorHasDenseHorseshoes} and Proposition \ref{prop:DenseRotationalHorseshoe}. 
	
	\begin{lemma}\label{lemma:RectangleNearPseudoRectangle}
		Let $P$ be a heteroclinic pseudo-rectangle. For every $\varepsilon > 0$ there exists a rectangle $R \subset P$ with sides $R^s_1, \ R^s_2, \ R^{u}_1, \ R^{u}_2$ such that 
		{\color{black}$$ \text{d}(z^{\sigma}_i, K^{\sigma}_i) < \varepsilon, \text{ for every } z^{\sigma}_i \in R^{\sigma}_i; \ \sigma = s,u; \ i = 1,2. $$ 
		}
	\end{lemma}
	
	\begin{proof}
		Go to the {\color{black}prime end} compactification $\breve{P} = \text{cl}(\mathbb{D})$, and let us assume without loss of generality that $\breve{K}^s_1, \ \breve{K}^{u}_1, \ \breve{K}^s_2, \  \breve{K}^{u}_2$ are positively and cyclically ordered. For every pair of values $(i,j) \in \{1,2\}^2$, let $\breve{z}_{ij}$ be the only prime end in $\breve{K}^s_i \cap \breve{K}^{u}_j$. Fix a pair $(i,j)$. Given that $\text{Imp}(\breve{z}_{ij})$ is defined by a chain of crosscuts with diameter going to 0, we may take $\gamma_{ij}$ a crosscut with $\text{diam}(\gamma_{ij}) < \frac{\varepsilon}{2}$, and note that the endpoints belong respectively to $K^s_i$ and $K^{u}_j$ provided $\varepsilon$ is sufficiently small, because the endpoints $\breve{z}^{s}_{ij}$, $\breve{z}^{u}_{ij}$ of $\breve{\gamma}_{ij}$, belong respectively to $\breve{K}^s_1$ and $\breve{K}^{u}_j$. We shall use $\breve{I}_{ij} \subset \partial \D$ to denote the interval determined by the endpoints of $\breve{\gamma}_{ij}$, which contains $\breve{z}_{ij}$. 
		
		Given that the set of accessible prime ends is dense, we may take accessible prime ends $\breve{z}^{'s}_{ij} \in \breve{K}^s_i \cap \breve{I}_{ij}, \ \breve{z}^{'u}_{ij} \in \breve{K}^u_i \cap \breve{I}_{ij}$ such that 
		$$ \breve{z}^{s}_{ij}, \ \breve{z}^{'s}_{ij}, \ \breve{z}^{'u}_{ij}, \ \breve{z}^{u}_{ij} \text{ are cyclically ordered,}$$
		and take disjoint respective rays $\breve{r}^s_{ij}, \ \breve{r}^u_{ij}$ from the origin to $\breve{z}^{'s}_{ij}, \ \breve{z}^{'u}_{ij}$, such that their images by $h$ land at the respective points $z^{'s}_{ij} \in K^s_i, z^{'u}_{ij} \in K^{u}_j$, and such that they each intersect $\breve{\gamma}_{ij}$ exactly once.
		
		Take $\breve{U}^s_1$ the angular region determined by $\breve{r}^s_{12}$ and $\breve{r}^s_{11}$ whose boundary is contained in the interior of $\breve{K}^s_1$. We now proceed as in Item (5) from Lemma \ref{lemma:PrimeEndChaoticRectangle}: take a sequence of curves $\{\breve{\theta}^s_{1,n}\}_{n \in \mathbb{Z}^+}$ from $\breve{r}^s_{12}$ to $\breve{r}^s_{11}$ and uniformly approaching the boundary, and note that up to taking a subsequence, the Hausdorff limit of $h(\breve{\theta}^s_{1,n})$ is included in $K^s_1$. This means we may take $n^s_1$ such that
		\begin{itemize}
			\item $h(\breve{\theta}^s_{1,n^s_1}) \subset B(K^s_1, \frac{\varepsilon}{2})$,
			\item $\breve{\theta}^s_{1,n} \cap \breve{\gamma}_{12} \neq \varnothing, \ \breve{\theta}^s_{1,n} \cap \breve{\gamma}_{11} \neq \varnothing$.
		\end{itemize} 
		
		Now, define $\breve{\theta}^s_1$ as a minimal-for-inclusion subarc of $\breve{\theta}^s_{1,n^s_1}$ going from $\breve{\gamma}_{11}$ to $\breve{\gamma}_{12}$, and define $\breve{\theta}^u_1, \ \breve{\theta}^s_2$, and $\breve{\theta}^u_2$ in the same fashion. Define $\breve{\gamma}'_{ij}$ as the subarc of $\breve \gamma_{ij}$ which goes from $\breve{\theta}^s_i$ to $\breve{\theta}^{u}_j$. See Figure \ref{figure:Waiter} for details. 
		
		\begin{figure}[h]
			\centering
			
			\def\svgwidth{.54\textwidth}
			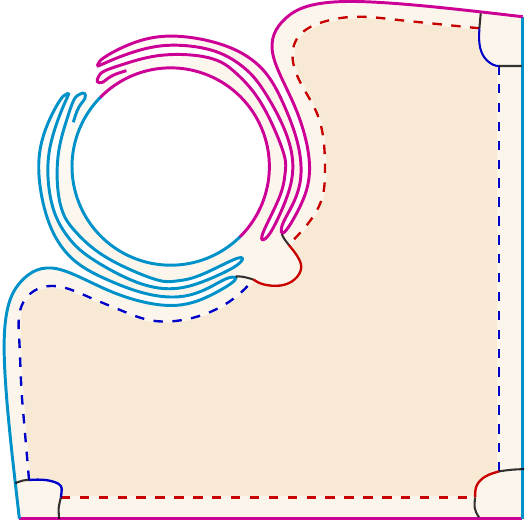
			
			\bigskip

			\caption{The stable and unstable sides of the HPR appear respectively in light blue and pink. The sides of the new rectangle appear in blue and red. Dotted blue and red lines come from taking curves near the sides in the {\color{black}prime end model}. Solid blue and red lines come from subarcs of crosscuts whose respective cross-sections contain each of the vertices of the rectangle in the {\color{black}prime end} model.} 
			\label{figure:Waiter}
		\end{figure}
		
		Finally, note that taking the correct orientation for each of these eight arcs, we have that 
		$$ \breve{\theta}^s_1 \cdot \breve{\gamma}'_{11} \cdot \breve{\theta}^u_1 \cdot \breve{\gamma}'_{21} \cdot \breve{\theta}^s_2 \cdot \breve{\gamma}'_{22} \cdot \breve{\theta}^u_2 \cdot \breve{\gamma}'_{12}$$
		is a simple closed curve, and moreover we may define 
		$$ R^s_1 = h(\breve{\theta}^s_1 \cdot \breve{\gamma}'_{11}), \ R^u_1 = h(\breve{\theta}^u_1 \cdot \breve{\gamma}'_{21}), \ R^s_2 = h(\breve{\theta}^s_2 \cdot \breve{\gamma}'_{22}), \ R^u_2 = h(\breve{\theta}^u_2 \cdot \breve{\gamma}'_{12}),$$ 
		and note that it is a rectangle satisfying the required properties.    
	\end{proof}
	
	{\color{black}We will slightly relax the definition of HPR in order to state a general lemma that can be applied in subtly different contexts.
	
	\begin{definition}\label{def:GeneralizedRectangle}
		Let $S$ be a surface. An inessential continuum $G \subset \T^2$ will be a \emph{generalized rectangle} if 
		\begin{enumerate}
			\item $G$ is the closure of an open disk $D$, 
			\item (Sides). $\partial G = K^{\textnormal{h}}_1 \cup K^{\textnormal{h}}_2 \cup K^{\textnormal{v}}_1 \cup K^{\textnormal{v}}_2$, each of them a continuum, with $K^{\textnormal{h}}_1 \cap K^{\textnormal{h}}_2 = \varnothing$, $K^{\textnormal{v}}_1 \cap K^{\textnormal{v}}_2 = \varnothing$. 
			\item The prime end sides associated to each of the continua $K^{\textnormal{h}}_1, K^{\textnormal{h}}_2, K^{\textnormal{v}}_1,  K^{\textnormal{v}}_2$ as in Definition \ref{def:PrimeEndSide}, define a rectangle in the prime end model $\breve{G}$ of $G$, as in Lemma \ref{lemma:PrimeEndChaoticRectangle}.

		\end{enumerate}	
		In this context, we will call $K^{\textnormal{h}}_1, K^{\textnormal{h}}_2$ \emph{horizontal sides}, and $K^{\textnormal{v}}_1, K^{\textnormal{v}}_2$ \emph{vertical sides} of $G$. 
	\end{definition}
	}
	
	\begin{lemma}\label{lemma:SeparateThenConnect}
		{\color{black}
		Let $S$ be a surface, and let $G \subset S$ be a generalized rectangle. Let $K$ be a continuum, and let $Q = K \cap G$ be a compact set such that 
		}
		\begin{itemize}
			\item {\color{black}$Q \cap K^{\textnormal{h}}_1 = \varnothing, \ Q \cap K^{\textnormal{h}}_2 = \varnothing$,
			
			\item $Q$ \textit{separates the horizontal sides:} for any continuum $Q' \subset P$ from $K^{\textnormal{h}}_1$ to $K^{\textnormal{h}}_2$, we have that $Q' \cap Q \neq \varnothing$.
			}	
		\end{itemize}
		Then, there exists a subcontinuum $Q^{\textnormal{h}} \subset Q$ connecting the {\color{black}vertical} sides, that is,
		{\color{black}
		$$ Q^{\textnormal{h}} \cap K^{\textnormal{v}}_1 \neq \varnothing, \ Q^{\textnormal{h}} \cap K^{\textnormal{v}}_2 \neq \varnothing$$}    
	\end{lemma}
		
	\begin{proof}
		Let $\delta_1 > 0$ be the distance between $Q$ and the union of the {\color{black}horizontal} sides. Define $Q_1$ as the closure of the union of every connected component of $Q$ which intersects {\color{black}$K^{\textnormal{v}}_1$}, and define $Q_2$ in the same fashion. Suppose by contradiction that these two compact sets are disjoint, and obtain that $\text{d}(Q_1,Q_2) = \delta_2 >0$. Take $\delta = \frac{\min \{\delta_1, \delta_2\}}{4}$, and take respective finite coverings $\mathscr{U}_1, \ \mathscr{U}_2$ of $Q_1$ and $Q_2$, by balls of radius $\delta$. Let $U_1$ be the union of the balls from $\mathscr{U}_1$, define $U_2$ in the same fashion, and note that $\text{cl}(U_1) \cap \text{cl}(U_2) = \varnothing$ by construction. 
		
		Going to the {\color{black}prime end} compactification {\color{black}$\breve{G}$} of {\color{black}$G$}, and because of Lemma \ref{lemma:PrimeEndChaoticRectangle}, we can take two crosscuts $\breve{\gamma}_1$ with endpoints in {\color{black}$\breve{K}^{\textnormal{v}}_1$}, $\breve{\gamma}_2$ with endpoints in {\color{black}$\breve{K}^{\textnormal{v}}_2$}, such that they are disjoint and each of them separates $U_1$ from $U_2$ {\color{black}as subsets of $G$}. This means that the remaining connected component of the complement of these two crosscuts contains both {\color{black}$\breve{K}^{\textnormal{h}}_1$} and {\color{black}$\breve{K}^{\textnormal{h}}_2$}, which means there {\color{black}exist} two accessible prime ends and then a crosscut $\breve{\gamma}_2$ from {\color{black}$\breve{K}^{\textnormal{h}}_1$} to {\color{black}$\breve{K}^{\textnormal{h}}_2$}, such that its image by $h$ is an arc disjoint from $Q$ and connecting {\color{black}$K^{\textnormal{h}}_1$} to {\color{black}$K^{\textnormal{h}}_1$}, which contradicts the hypothesis.  
	\end{proof}
	
	{\color{black}Note that an HPR and any of their lifts to $\R^2$ are generalized rectangles.}

	\subsection{Dense family of fitted heteroclinic pseudo-rectangles} 
	We will dedicate this subsection to prove that inside any open set in $\R^2$ there is a \textit{fitted} heteroclinic pseudo-rectangle, we only need the fitted condition to make the arguments in Section \ref{sec:DenseTopHorseshoes} simpler. Although our stable and unstable sets are dynamically defined, the proof of this density is mostly topological. 
	
	\begin{definition}\label{def:FittedHPR}
		We will say a heteroclinic pseudo-rectangle is \emph{fitted} if 
		$$ \rho(z^s_1) = (a^s_1,b^s_1), \ \rho(z^s_2) = (a^s_2,b^s_2), \ \rho(z^u_1) = (0,b^u_1),  \ \rho(z^u_2) = (0,b^u_2),$$ 
		where
		$$ a^s_1 < 0, \ a^s_2 < 0, \ b^s_1 < 0, \  b^s_2 < 0, \ b^{u}_1 < 0, \ b^{u}_2 > 0, \text{ and } \frac{b^s_1}{a^s_1} \neq \frac{b^s_2}{a^s_2}.$$
		
		We will sometimes call them FHPR for the sake of simplicity.
	\end{definition}
	
	Note from Definition \ref{defi:ChaoticRectangle} that an HPR for a power $g^j$ of $g$, is also an HPR for $g$. Therefore we may assume up to taking some power, that $g$ has a lift $\wh g$ such that $\vec{0} \in \rho(\wh g)$. We will then prove that: 
	
	\begin{proposition}\label{prop:ChaoticRectanglesAreDense}
		Let $\wh g$ be a lift to $\R^2$ of a frice torus homeomorphism $g$. If $\vec{0} \in \mathrm{int}(\rho(\wh g))$, then for every open set $\wh U \subset \R^2$, there exists a fitted heteroclinic pseudo-rectangle $\wh P \subset \wh U$ for $\wh g$.   
	\end{proposition}

	\begin{proof}
		This is quite an intricate construction, so we will divide it into steps. 
		
		\medskip
		
		\paragraph{\textbf{Step 1. Accumulating small stable continua.}}
		$\hbox{}$ \newline
		
		Let $(\wh g, \wh I, \wh \F)$ be an MDTD for $\wh g$. Recall that the set $\mathrm{Sing}(\wh \F)$ of singularities of the foliation, has no interior and does not disconnect the plane, because every nontrivial continuum is dynamically unbounded for $g$ (see Proposition \ref{prop:elementprop.essfactor}). Let us then assume up to taking a smaller set that $\wh U$ is an Euclidean ball $\wh U = B(\wh z, \varepsilon)$ with no singularities inside, and take $\wh U' = B(\wh z,\frac{\varepsilon}{2})$. Given that $\vec{0} \in \rho(\wh g)$, and using the realization of rational vectors in $\mathrm{int}(\rho(\wh g))$ in \cite{franks89}, we shall take $\{\wh z_n\}_{n \in \mathbb{Z}^+}$ a sequence of lifts of periodic points for $g$, such that 
		$$\rho_g(\wh z_n) = (a_n,b_n),$$
		such that $a_n < 0$ and $b_n < 0$ for every $n$, and such that the sequence of slopes $\{\frac{b_n}{a_n}\}_{n \in \mathbb{Z}^+}$ is injective. 
		Notice that 
		\begin{itemize}
			\item Each ${\color{black}\wh W^s(\wh z_n)}$ is dense in $\R^2$ {\color{black} by Remark \ref{rem:denseStableSets}}. 
			\item ${\color{black}\wh W^s(\wh z_n)} \cap {\color{black}\wh W^s(\wh z_m)} = \varnothing$ whenever $m \neq n$, because of Lemma \ref{lemma:StableSetGoesToZero} and the fact that our $\wh z_n$ have different rotation vectors, and therefore their orbits separate linearly in $\R^2$. 
		\end{itemize} 
		
		In particular, ${\color{black}\wh W^s(\wh z_n)} \cap \wh U' \neq \varnothing$ for every $n > 0$, {\color{black}then by Lemma \ref{lem:nadler}} we may take a sequence of stable continua $\{\wh K^s_n\}_{n \in \mathbb{Z}^+}$, with $\wh K^s_n \subset {\color{black}\wh W^s(\wh z_n)}$ and such that $\wh K^s_n$ is a stable continuum with $$\wh K^s_n \cap  \partial \wh U \neq \varnothing, \  \wh K^s_n \cap \partial \wh U' \neq \varnothing,$$ 
		minimal for the inclusion. Note that 
		$$\frac{\varepsilon}{2} \leq \text{diam}(\wh K^s_n) \leq 2\varepsilon \text{ for every } n,$$
		which implies that, up to taking a subsequence, we have that there exists $\wh K^s$ such that  
		$$ \wh K^s_n \xrightarrow[n \to +\infty]{\textnormal{Hff}}\wh K^s.$$
		\medskip
		We remark that $\wh K^s$ cannot intersect {\color{black}both} $\wh K^s_{n_1}$ and $\wh K^s_{n_2}$ for $n_1\not= n_2$: if it did intersect both, it would be a weakly unstable continuum by Lemma \ref{lemma:AnchorTypeA}, but we know $\wh K^s$ is a stable continuum by Proposition \ref{prop:HffLimitSmallstableSets}, so it would be a contradiction {\color{black}by Corollary \ref{cor:StableNotWeaklyUnstable}}. We may therefore assume that $\wh K^s$ is disjoint from all $\wh K^s_n$.  
		
		\medskip
		
		\paragraph{\textbf{Step 2. Understanding the local picture.}} 
		$\hbox{}$ \newline
		
		Notice that $\wh K^s$ intersects both $\partial \wh U$ and $\partial \wh U'$. Moreover,  $\wh K^s$ is a stable continuum by Proposition \ref{prop:HffLimitSmallstableSets}, and is therefore filled and has empty interior by Remark \ref{rem:FilledStableContinuum}, which implies that it is inessential when seen as a subset of the annulus $\wh S = \mathrm{cl}(\wh U) \backslash \wh U'$. This means that up to a change of coordinates, $\wh K^s$ is disjoint from an open \textit{radial region} $\wh V_0 \subset \wh S$ (i.e. an open region contained between two disjoint simple curves from $\partial \wh U$ to $\partial \wh U'$). 
		
		\textcolor{black}{$\wh S \backslash \wh K^s$ may have more than a single connected component, but there is only one that intersects both $\partial \wh U$ and $\partial \wh U'$. Indeed, if this was not the case, we could find two arcs $\gamma_1$ and $\gamma_2$, joining $\partial \wh U$ and $\partial \wh U'$, and lying in different connected components of $\wh S \backslash \wh K^s$. The two connected components of the complement of these arcs in the annulus must both contain points of  $\wh K^s$ which contradicts its connectedness. }
		Then, up to taking a subsequence, every continuum $\wh K^s_n$ belongs to the same connected component $\wh V$ of $\wh S \backslash  (\wh V_0 \cup \wh K^s)$. 
		
		{\color{black}Moreover, we obtain that $\wh K^s \subset \partial \wh V$, because $\displaystyle \bigcup \limits_{n \in \Z^{+}} \wh K^s_n \subset \wh V$, and $\wh K^{s}_n \xrightarrow[n \to +\infty]{\textnormal{Hff}} \wh K^{s}$ {\color{black}(see Figure \ref{figure:LocalPicture} for details).}
		{\color{black} By this construction, we obtain} that $\wh V$ is a disk, and that its boundary can be written as $\partial \wh V = \partial_{\textnormal{R}} \wh V \cup \partial_{\textnormal{T}} \wh V \cup \partial_{\textnormal{L}} \wh V \cup \partial_{\textnormal{B}} \wh V$ ,} where
		\begin{itemize}
			\item $\partial_{\textnormal{R}} \wh V$ is a closed arc of $\partial \wh U$,
			\item $\partial_{\textnormal{T}} \wh V$ is a closed arc from $\wh U$ to $\wh U'$, whose interior is in $\wh S$,   
			\item $\partial_{\textnormal{L}} \wh V$ is a closed arc of $\partial \wh U'$,
			\item $\partial_{\textnormal{B}} \wh V = \wh K^s$.  
		\end{itemize}

		We can take a closed arc $\wh \eta$ contained in $\wh S \backslash \mathrm{cl}(\wh V)$, except for its endpoints which belong respectively to $\partial_{\LL}\wh V \cap \wh K^s$, $\partial_{\RR}\wh V \cap \wh K^s$: {\color{black} To this end, we select two arcs, contained in $\partial \wh U$ and $\partial \wh U'$, respectively, each having exactly one endpoint on $\wh K^s$ (which is its unique point of intersection with $\wh K^s$) and the other endpoint on {\color{black}$\partial \wh V_0$}. We then join these arcs by a third simple curve contained in {\color{black}$\wh V_0$} except for its endpoints, thus obtaining the desired curve $\wh \eta$ (see Figure \ref{figure:LocalPicture}). We then} denote by $\wh V'$ the rectangle whose sides are $\partial_{\RR}\wh V$, $\partial_{\TT}\wh V$, $\partial_{\LL}\wh V$ and $\wh \eta$. By the Jordan-Schoenflies Theorem, we can take a new change of coordinates sending $\wh V'$ to the square $[0,1] \times [0,1]$; we can therefore assume that $\wh V$ is \textit{almost} the square $[0,1] \times [0,1]$, more precisely 
		$$ \partial_{\textnormal{R}} \wh V = \{1\} \times [0,1], \ \partial_{\textnormal{T}} \wh V = [0,1] \times \{1\} , \ \partial_{\textnormal{L}} \wh V = \{0\} \times [0,1], \ \partial_{\textnormal{B}} \wh V \subset [0,1] \times [0,1].$$
		
		See Figure \ref{figure:LocalPicture} for details. 
		
		Let us use the following notation:  
		$$\partial_{\textnormal{L}} \wh K^s_n = \wh K^s_n \cap \partial_{\textnormal{L}} \wh V, \ \partial_{\textnormal{R}} \wh K^s_n = \wh K^s_n \cap \partial_{\textnormal{R}} \wh V.$$
		
		Each of these two sets is naturally contained in a respective closed segment $I^{\textnormal{L}}_n, I^{\textnormal{R}}_n$, which are minimal for the inclusion. Again, up to taking a subsequence, we may assume that 
		$$\{I^{\textnormal{L}}_n\}_{n \in \mathbb{Z}^+}, \ \{I^{\textnormal{R}}_n\}_{n \in \mathbb{Z}^+} \text{ are decreasing-to-zero sequences of disjoint closed segments}.$$
		
		Note that
		$$\text{d}(\wh K^s_n, \wh K^s_{n+1}) =  \delta_n > 0$$ 
		because they are disjoint compact sets. Then, there exists an open connected component $\wh V_n$ \textit{between} $\wh K_n$ and $\wh K_{n+1}$, that is, {\color{black}there exists a curve from $\partial_{\textnormal{L}} \wh V$ to $\partial_{\textnormal{R}} \wh V$} which separates $\wh K_n$ from $\wh K_{n+1}$. Note that there exists only one of these components for each value of $n$, as they would otherwise separate $\wh K^s_n$ or $\wh K^s_{n+1}$. \textcolor{black}{Note that $\wh V_n$ is a generalized rectangle, with $\wh K_n$ and $\wh K_{n+1}$ as vertical sides and subarcs of $\partial \wh U$ and $\partial \wh U'$ as horizontal sides.}
		
		\begin{figure}[h]
			\centering
			
			\def\svgwidth{.92\textwidth}
			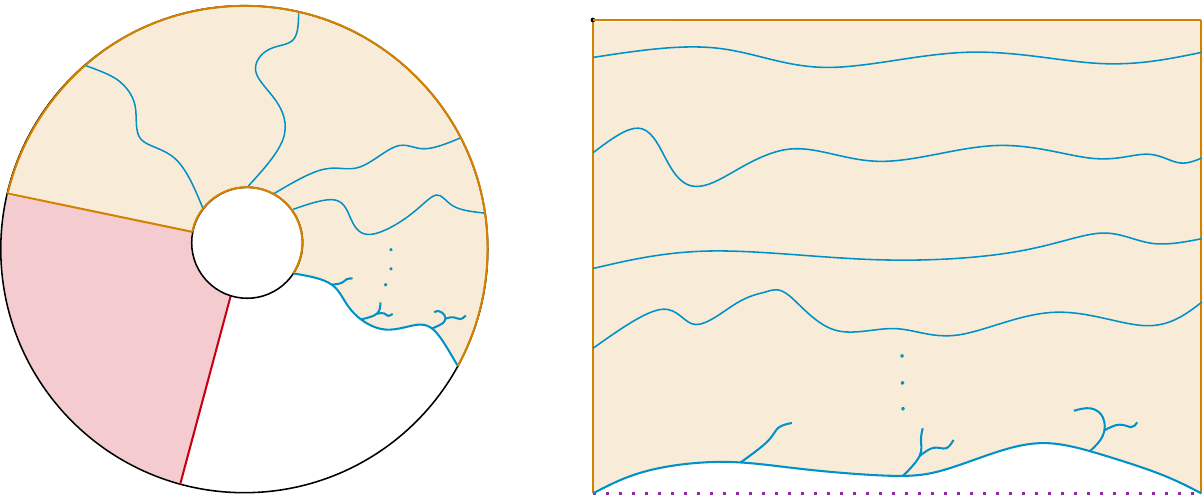
			
			\bigskip

			\caption{The complement of every $\wh K^s_n$ has well-defined upper and lower components, each of them accumulating at $\partial_{\LL}\wh V$ and $\partial_{\RR}\wh V$.} 
			\label{figure:LocalPicture}
		\end{figure}
		
		We will call $\wh V^+_n$ to the connected component of $\wh V \backslash \wh K^s_n$ which contains $\wh V_n$, and we will call $\wh V^-_n$ to the connected component of $\wh V \backslash \wh K^s_n$ which contains $\wh V_{n-1}$.
		
		Fix $n > 0$, and note that if $\delta < \delta_n$, then the filling of the $\delta$-neighbourhood of $\wh K^s_n$ does not intersect $\wh K^s_{n+1}$. This means there exists a curve $\gamma_n$ included in this filled neighbourhood except for its endpoints (in particular $\gamma_n \subset \wh V_n$ except for its endpoints), which respectively belong to $\partial_{\LL} \wh V_n$, and $\partial_{\RR} \wh V_n$. 
		
		This proves that we may take a sequence of curves $\{\gamma^{i}_n\}_{i \in \mathbb{Z}^{+}}$, each of them contained in $\wh V_n$ except for its endpoints, and intersecting both $\partial_{\textnormal{L}} \wh V$ and $\partial_{\textnormal{R}} \wh V$, such that the Hausdorff limit is a continuum $\wh K^+_n \subset \wh K^s_n$ which also intersects $\partial_{\textnormal{L}} \wh V$ and $\partial_{\textnormal{R}} \wh V$. In a identical fashion, by approximating $\wh K_n$ with curves contained in $\wh V_{n+1}$ except for its endpoints, we obtain $\wh K^-_n \subset \wh K^s_n$ which also intersects $\partial_{\textnormal{L}} \wh V$ and $\partial_{\textnormal{R}} \wh V$. Given that $\wh K^s_n$ is minimal for the inclusion, we obtain that $\wh K^+_n = \wh K^s_n$, and similarly $\wh K^-_n = \wh K^s_n$. This proves that 
		\begin{equation}\label{equation:InsideBoundaryEqualsOutsideBoundary}
			\wh K^s_n \subset \partial \wh V_{n-1}, \ \wh K^s_n \subset \partial \wh V_n
		\end{equation}  
		
		\medskip
		
		\paragraph{\textbf{Step 3. Intercalating unstable continua.}} 
		$\hbox{}$ \newline
		
		For each of these $\wh V_n$ {\color{black}with $n$ large enough}, in our current coordinates, take a small euclidean ball {\color{black}$\wh U_n$} contained in this component, centered at $w_n$ with $x-$coordinate equal to $\frac{1}{2}$, and radius much smaller than $\frac{1}{4}$.
		
		We now take a sequence $\{\wh z'_n\}_{n \in \mathbb{Z}^+}$ of periodic points such that 
		$$ \rho_g(\wh z'_n) = \Bigl ( 0,\frac{(-1)^n}{n} \Bigr ),$$
		As in Step 1, we obtain that 
		\begin{itemize}
			\item Each ${\color{black}\wh W^u(\wh z'_n)}$ is dense in $\R^2$, 
			\item ${\color{black}\wh W^u(\wh z'_n)} \cap {\color{black}\wh W^u(\wh z'_m)} = \varnothing$ whenever $m \neq n$
		\end{itemize}  
		In particular ${\color{black}\wh W^u(\wh z'_n)}$ must intersect {\color{black}$\wh U_n$}. Let us then take a sequence of unstable continua $\{\wh K^u_n\}_{n \in \mathbb{Z}^+}$ with $\wh K^{u}_n \subset \wh K^{u}(\wh z'_n)$, such that 
		$$ \wh K^{u}_n \cap \partial \wh V \neq \varnothing, \ \wh K^{u}_n \cap \partial \wh W_n \neq \varnothing,$$
		minimal for the inclusion, with that property. 
		
		\medskip
		
		\paragraph{\textbf{Step 4. Finding many transversal stable-unstable intersections.}} 
		$\hbox{}$ \newline
		
		Fix an integer $j > 0$. {\color{black}For any $n$ large enough}, let us define $\wh V^j_n$ as the connected component of $\wh V \backslash (\wh K^s_{n-j} \cup \wh K^s_{n+j})$ which is \textit{between} $\wh K^s_{n-j}$ and $\wh K^s_{n+j}$, that is, the only component of $\wh V \backslash (\wh K^s_{n-j} \cup \wh K^s_{n+j})$ where there is a curve joining $\partial_{\LL} \wh V$ to $\partial_{\RR} \wh V$, which separates $\wh K^s_{n-j}$ from $\wh K^s_{n+j}$.
		
		More over, let use the following notation: 
		$$ \partial_{\LL} \wh V^j_n = \partial \wh V^j_n \cap \partial_{\LL} \wh V, \ \partial_{\RR} \wh V^j_n = \partial \wh V^j_n \cap \partial_{\RR} \wh V,$$
		which are both segments respectively included in $\partial_{\LL} \wh V$ and $\partial_{\RR} \wh V$.
		
		We will prove the following: 
		
		\begin{lemma}\label{lem:TransversalSUIntersections}
			For up to finitely many values of $n$,
			\begin{equation}\label{equation:StableUnstableTransversal}
				\text{either } \wh K^{u}_n \cap \wh K^s_{n-j} \neq \varnothing, \text{ or } \wh K^{u}_n \cap \wh K^s_{n+j} \neq \varnothing,
			\end{equation}	
			\begin{equation}\label{equation:DontTouchTheSides}
				\wh K^{u}_n \cap \partial_{\LL} \wh V^j_n = \varnothing, \ \wh K^{u}_n \cap \partial_{\RR} \wh V^j_n = \varnothing
			\end{equation}
		\end{lemma}
		
		\begin{proof}

			Suppose by contradiction that the result from Equation \ref{equation:StableUnstableTransversal} does not hold. Then, up to taking a subsequence, we would have that $\wh K^{u}_n$ is included in $\wh V^j_n$. Note that this implies that 
			\begin{equation}\label{equation:IntercalatedUnstableSetsLimitDistance}
				d_n = \sup \{\textnormal{d}(z, \wh K^s) : \ z \in \wh K^{u}_n\} \xrightarrow{n \to \infty} 0.
			\end{equation} 
			
			We may assume up to taking a new subsequence, that every {\color{black}$\wh K^{u}_n$} would intersect {\color{black}$\partial_{\LL} \wh V^j_n$}. \textcolor{black}{Thus, if Equation~\ref{equation:StableUnstableTransversal} does not hold, then neither does Equation~\ref{equation:DontTouchTheSides}, so it suffices to prove the latter holds.}

Given that in the current coordinates, we have that for every $n > 0$,  
			{\color{black}
			\begin{equation}\label{eq:ControlledSizeKun}
				\frac{1}{4} \leq \textnormal{diam}(\wh K^{u}_n) \leq \sqrt{2},
			\end{equation}
			then}  
			we would have by Proposition \ref{prop:HffLimitSmallstableSets}, that {\color{black}up to taking a subsequence,} 
			$$ \wh K^{u}_n \xrightarrow[n \to +\infty]{\textnormal{Hff}} \wh K^{u}.$$
			
			{\color{black}
			We then obtain that $\wh K^{u}$ would be an unstable continuum contained in $\wh K^s$ by Equation \ref{equation:IntercalatedUnstableSetsLimitDistance}, which would also be nontrivial by Equation \ref{eq:ControlledSizeKun}. 
			}
			
			Then, our set {\color{black}$\wh K^{u}$} would be simultaneously stable and unstable, therefore being dynamically bounded for $g$, which would be a contradiction by Corollary \ref{cor:IntersectionStableUnstableTotDisconnected}. This proves that the result from  \ref{equation:StableUnstableTransversal} holds. {\color{black}If we suppose by contradiction that Equation \ref{equation:DontTouchTheSides} does not hold, the argument is identical. This concludes the proof.} 
			
			\smallskip

		\end{proof}
		\medskip
		
		\paragraph{\textbf{Step 5. Locating the sides.}} $\hbox{}$ \newline
		
		Start by applying Lemma \ref{lem:TransversalSUIntersections} to the value $j = 5$. Take a sufficiently large value of $n$ such that we avoid the finitely many pathological values from Step 4, that is, such that Equations \ref{equation:StableUnstableTransversal} and \ref{equation:DontTouchTheSides} are held for every $n' \geq n$. We suppose without loss of generality that $n$ is even, and we will work with the strip between $\wh K^s_n$ and $\wh K^s_{n+9}$.
		
		By Equation \ref{equation:StableUnstableTransversal}, we have that one of the following is true:
		\[\wh K^{u}_{n+4} \cap \wh K^{s}_n \neq \varnothing, \text{ or } \ \wh K^{u}_{n+4} \cap \wh K^{s}_{n+9} \neq \varnothing.\]
		
		We assume that the first one holds, as the proof for the other case is analogous. This implies that
		\[\wh K^{u}_{n+4} \cap \wh K^{s}_{n+1} \neq \varnothing, \ \wh K^{u}_{n+4} \cap \wh K^{s}_{n+3} \neq \varnothing, \ \wh K^{u}_{n+4} \cap \wh K^{s}_{n+4} \neq \varnothing.\]
		
		Again by Equation \ref{equation:StableUnstableTransversal}, we have that 
		\[\wh K^{u}_{n+1} \cap \wh K^{s}_n \neq \varnothing, \text{ or } \ \wh K^{u}_{n+1} \cap \wh K^{s}_{n+4} \neq \varnothing.\]
		
		Once again we assume that the first one holds, as the analysis for the other case is analogous. This implies that 
		\[\wh K^{u}_{n+1} \cap \wh K^{s}_n \neq \varnothing.\]
		
		Putting everything together, so far we have obtained that
		\begin{equation}
			\wh K^{u}_{n+4} \cap \wh K^{s}_n \neq \varnothing, \ \wh K^{u}_{n+4} \cap \wh K^{s}_{n+1} \neq \varnothing, \ \wh K^{u}_{n+1} \cap \wh K^{s}_n \neq \varnothing, \ \wh K^{u}_{n+1} \cap \wh K^{s}_{n+1} \neq \varnothing.
		\end{equation}
		
		Let $m = n+4$, and let us take an unstable subcontinuum $\wh K^{'u}_m \subset \wh K^{u}_m$ intersecting both $\wh K^s_n$ and $\wh K^s_{n+1}$, minimal for the inclusion. Each of these continua are contained in the closure of $\wh V_n$. For $m' = n+1$, define $\wh K^{'u}_{m'} \subset \wh K^{u}_{m'}$ in identical fashion. 
		
		Proceeding as in Step 4, we will define
		$$ \partial_{\LL} \wh V_n = \partial \wh V_n \cap \partial_{\LL} V, \ \partial_{\RR} \wh V_n = \partial \wh V_n \cap \partial_{\RR} V $$ 
		
		By Equation \ref{equation:DontTouchTheSides} we have that
		$$ \wh K^{'u}_m \cap \partial_{\LL} \wh V_n = \varnothing, \ \wh K^{'u}_m \cap \partial_{\RR}\wh V_n = \varnothing, \ \wh K^{'u}_{m'} \cap \partial_{\LL}\wh V_n = \varnothing, \ \wh K^{'u}_{m'} \cap \partial_{\RR}\wh V_n = \varnothing.$$ 
		
		This implies that, when seen as a subset of the closure of $\wh V_n$, the unstable continuum $\wh K^{'u}_m$ (and similarly $\wh K^{'u}_{m'}$) separates $\partial_{\LL} \wh V_n$ from $\partial_{\RR} \wh V_n$: otherwise {\color{black}by Lemma \ref{lemma:SeparateThenConnect}} there would exist a curve from $\partial_{\LL} \wh V_n$ to $\partial_{\RR} \wh V_n$ included in $\wh V_n$ which would separate $\wh K^{'u}_m$, which would be a contradiction.

		Thus, one of the following symmetrical statements holds:
		\begin{equation}
			\wh K^{'u}_{m} \text{ separates } \partial_{\LL} \wh V_n \text{ from } \wh K^{'u}_{m'}, \text{ and } \wh K^{'u}_{m'} \text{ separates } \partial_{\RR} \wh V_n \text{ from } \wh K^{'u}_{m},
		\end{equation}
		\begin{equation}
			\wh K^{'u}_{m'} \text{ separates } \partial_{\LL} \wh V_n \text{ from } \wh K^{'u}_{m}, \text{ and } \wh K^{'u}_{m} \text{ separates } \partial_{\RR} \wh V_n \text{ from } \wh K^{'u}_{m'}.
		\end{equation}
		Once again we can assume that the first one holds without loss of generality. We shall then say $\wh K^{'u}_{m}$ is \textit{on the left} of $\wh K^{'u}_{m'}$.

		\medskip
		
		\paragraph{\textbf{Step 6. Building the heteroclinic pseudo-rectangle.}} 
		$\hbox{}$ \newline
		
		We will now concentrate exclusively on the connected component $\wh V_n$ and the four continua 
		$$\wh K^s_n, \  \wh K^s_{n+1}, \ \wh K^{'u}_m \ \text{ and } \ \wh K^{'u}_{m'}.$$ Recall that {\color{black}$\wh K^{'u}_m$} is on the left of {\color{black}$\wh K^{'u}_{m'}$}, and that $\wh V_n$ is a disk. 
		
		Note that $\wh K^{'u}_m$ separates the disk $\wh V_n$ into (at least) two connected components, which we denote $\text{L}(\wh K^{'u}_m), \ \text{R}(\wh K^{'u}_m)$, containing respectively $\partial_{\text{L}} \wh V_n$ and $\partial_{\text{R}} \wh V_n$. Each of these connected components is a disk, as they are simply connected because their boundary is connected.
		
		Note that, again, $\wh K^{'u}_{m'}$ separates $\text{R}(\wh K^{'u}_m)$ into at least two connected components $\text{L}(\wh K^{'u}_{m'}) \cap \text{R}(\wh K^{'u}_m)$ and $\text{R}(\wh K^{'u}_{m'})$
		
		Given that both $\wh K^{'u}_m$ and $\wh K^{'u}_{m'}$ are minimal for the inclusion, we may replicate the argument from Equation \ref{equation:InsideBoundaryEqualsOutsideBoundary}:	as in the very end of Step 2, take a sequence of curves $\{\gamma^{'i}_m\}_{i \in \mathbb{Z}^+}$ such that each of them belongs to $\LL(\wh K^{'u}_m)$, except for its endpoints, which are respectively in $\wh K^s_n$ and $\wh K^s_{n+1}$, and such that there is a subsequence of $\{\gamma^{'i}_m\}$ which converges to a subcontinuum of $\wh K^{'u}_m$ in the Hausdorff topology, and recall that $\wh K^{'u}_m$ is minimal for inclusion by construction. Repeat this process for the other components to obtain 
		{\color{black}
		\begin{equation}\label{equation:InsideBoundaryEqualsOutsideBoundary2}
			\wh K^{'u}_m \subset \partial \text{L}(\wh K^{'u}_m), \ \wh K^{'u}_m \subset \partial \text{R}(\wh K^{'u}_m), \ \wh K^{'u}_{m'} \subset \partial \text{L}(\wh K^{'u}_{m'}), \ \wh K^{'u}_{m'} \subset \partial{R}(\wh K^{'u}_{m'})
		\end{equation}
		}
		Let us define {\color{black}$\wh D = \text{R}(\wh K^{'u}_m) \cap \text{L}(\wh K^{'u}_{m'})$}, $\wh P = \text{cl}(\wh D)$, and let us prove that it is a fitted heteroclinic pseudo-rectangle.

		Take $\wh K^{'s}_n = \partial \wh D \cap \wh K^{s}_n $ and $\wh K^{'s}_{n+1} = \partial \wh D \cap \wh K^s_{n+1}$. Each of these sets is connected: if $\wh K^{'s}_n$ is not connected, then it would be contained in the closure of two disjoint closed angular regions $R_1, \ R_2$ of the disk $\wh D$, which on their turn define two disjoint open regions $R'_1, \ R'_2$, which are the connected components in $\wh D$ of the complement of the regions $R_1, \ R_2$. Then, because of Equations \ref{equation:InsideBoundaryEqualsOutsideBoundary} and \ref{equation:InsideBoundaryEqualsOutsideBoundary2}, we can repeat the argument in Item (4) of Lemma \ref{lemma:PrimeEndChaoticRectangle} to obtain two simple curves $\gamma, \gamma': [0,1] \to \mathrm{cl}(\wh V)$ which are included in $\wh V$ except for their endpoints, such that 
		\begin{itemize}
			\item $\gamma \cap \gamma'$ is a single transversal intersection contained in $D$ (and therefore $\gamma \wedge \gamma' = \pm 1$),
			\item $\gamma$ has both endpoints in $\partial_{\text{T}} \wh V$, and $\gamma'$ has both endpoints in $\partial_{\text{L}} \wh V \cup \partial_{\text{R}} \wh V$,
		\end{itemize}
		which would be a contradiction. Proceeding in identical fashion we get that $\wh K^{'s}_{n+1}$ is also connected.  
		
		\medskip
		
		We finally have that
		\begin{itemize}
			\item $\wh P$ contains no singularities as $\wh U$ contains no singularities, 
			\item $\wh P$ is the closure of an open disk $\wh D$,
			\item $\partial \wh P = \wh K^{'u}_m \cup \wh K^{'u}_{m'} \cup \wh K^{'s}_n \cup \wh K^{'s}_{n+1}$, 
			\item Each neighbourhood of a boundary point intersects the unbounded component of the complement of $P$, because of Equations \ref{equation:InsideBoundaryEqualsOutsideBoundary} and \ref{equation:InsideBoundaryEqualsOutsideBoundary2}.
			\item $\wh K^{'s}_n \cap \wh K^{'u}_m \neq \varnothing$ by construction, and the same happens for the other three desired intersections. 
		\end{itemize}
		
		The five items from Definition \ref{defi:ChaoticRectangle} are then held, as well as the condition from Definition \ref{def:FittedHPR} because $m$ is even and $m'$ is odd, and therefore $\wh P$ is a fitted heteroclinic pseudo-rectangle, which concludes the proof.

	\end{proof}

	\section{The essential factor has a dense family of horseshoes}\label{sec:DenseTopHorseshoes}
	
	The goal of this section is to prove the following:
	
	\begin{proposition}\label{prop:EssentialFactorHasDenseHorseshoes}
		Let $f \in \text{Homeo}_0(\mathbb{\T}^2)$ with $\textnormal{int}(\rho(f)) \neq \varnothing$, and let $\wh g$ be a lift of the essential factor, with $\vec{0} \in \mathrm{int}(\rho(\wh g))$. Then, for every open set $\wh U \subset \mathbb{R}^2$, $\wh g$ has a Markovian horseshoe $\wh X \subset \wh U$. 
	\end{proposition}
	
	The idea resides on using the density of the family of fitted heteroclinic pseudo-rectangles, and proving that inside each one of these, there exists a Markovian horseshoe. Unlike the proof  of Proposition \ref{prop:ChaoticRectanglesAreDense}, which is almost purely topological, this one heavily relies on dynamical notions, and we will use the techniques and results developed in Sections \ref{section:stretching} and \ref{section:RotMixing}. 
	
	\medskip
	
	The path for this result is somewhat technical, so we divide it into two big steps, dedicating one of the following subsections to each of them:
	\begin{enumerate}
		\item Proving that inside every FHPR there exists a topological horseshoe, in the sense of \cite{kennedy} (this is Proposition \ref{prop:HeteroclinicPseudoRectangleHasTopHorseshoe}). These horseshoes capture some of the rotational behaviour of $g$. 
		\item To ensure the density of periodic points, we prove in a similar way that inside each FHPR there is a Markovian horseshoe, as in Definition \ref{def:RotationalHorseshoe}. This is done in Proposition \ref{prop:DenseRotationalHorseshoe}.  
	\end{enumerate}

	\subsection{Topological horseshoes}
	
	The central result to be proved in this subsection is the following:
	
	\begin{proposition}\label{prop:HeteroclinicPseudoRectangleHasTopHorseshoe}
		Let $P$ be a fitted heteroclinic pseudo-rectangle. Then, for every $m > 0$, there exists $j > 0$ such that $P$ contains a topological horseshoe in $m$ symbols for $g$, with period $j$.
	\end{proposition}
	
	\medskip
	
	As we have seen in Section \ref{section:Preliminaries}, we will use the notion of topological horseshoe defined by Kennedy and Yorke in \cite{kennedy}, for the restriction to a rectangle $R \subset P$ (taken from Lemma \ref{lemma:RectangleNearPseudoRectangle}), of a positive iterate $g^j$ of $g$. We need to check the 5 \textit{Horseshoe Hypothesis}, that is
	
	\begin{enumerate}
		\item $\T^2$ is a separable metric space.
		\item $R \subset P \subset \T^2$ is locally connected and compact.
		\item The map $g: R \to \T^2$ is continuous.
		\item The sets $R^s_1 \subset P$ and $R^s_2 \subset P$ are disjoint, nonempty and compact.
		\item (\textit{Multiple crossing.}) $g^j |_R$ has crossing number greater than or equal to $m>1$.
	\end{enumerate}
	
	that is, every connection contains at least $m$ mutually disjoint preconnections. {\color{black}Note that the first four properties are guaranteed by construction, we then only have to prove the multiple crossing property.}
	
	We recall that, from Lemma \ref{lemma:RectangleNearPseudoRectangle}, each of the sides of the rectangle $R$ are $\varepsilon$-near the respective stable and unstable continua $K^s_1, K^s_2, K^{u}_1, K^{u}_2$. Given $\varepsilon$ can be supposed to be as small as we want, we will proceed to prove the \textit{multiple crossing} for the actual HPR $P$: the actual needed value for $\varepsilon$ will be automatic from the construction, and specified at the end of it.  
	
	Thus, in our context, a connection will be a continuum $Q^{u} \subset P$ intersecting both $K^s_1$ and $K^s_2$. Note that $Q^{u}$ is not necessarily an unstable continuum, it is simply parallel to the unstable sides in a weak sense. 
	
	\begin{proof}[Proof of Proposition \ref{prop:HeteroclinicPseudoRectangleHasTopHorseshoe}]
		Fix $m > 0$. Let us divide this proof into steps.
		
		\medskip
		
		\paragraph{\textbf{Step 1. Reduction to transverse crossing. }}
		$\hbox{}$ \newline
		
		{\color{black}By Lemma \ref{lemma:SeparateThenConnect}}, it is enough to prove the 
		
		\smallskip
		
		\textit{Transverse crossing property.} For every connection $\wh Q^{u}$, there exists an integer $j > 0$ and $m$ different lifts $\wh P_1, ..., \wh P_m$ of $P$ to $\R^2$, such that for every integer $l$ with $1 \leq l \leq m$, and every continuum $\wh Q^s_l \subset \wh P_l$ connecting the unstable sides $\wh K^u_{1,l}, \ \wh K^u_{2,l} \subset \wh P_l$, we have that
		$$\wh g^j(\wh Q^{u}) \cap \wh Q^s_l \neq \varnothing.$$

		\medskip
		
		\paragraph{\textbf{Step 2. Past anchoring of the stable sides. }}
		$\hbox{}$ \newline
		
		{\color{black}Using Zorn's lemma}, let us define an \textit{extended stable side} $\wh K^{\ast s}_1 \subset {\color{black}W^s(\wh z^s_1)}$ as a stable continuum which is minimal for the inclusion, such that 
		$$\wh K^{s}_1 \subset \wh K^{\ast s}_1, \ \wh z^s_1 \in \wh K^{\ast s}_1,$$
		and let {\color{black}us define}
		$$d^s_1 = \text{diam}(\wh K^{\ast s}_1).$$ 
		
		Define in a similar fashion $\wh K^{\ast s}_2$, $\wh K^{\ast u}_1$, $\wh K^{\ast u}_2$, $d^s_2$, $d^u_1$, $d^u_2$, respectively the other extended stable side and the two extended unstable sides {\color{black}and their respective diameters. We then define} $\overline{d}=\max\{d^s_1, d^s_2, d^u_1, d^u_2\}$. Let us write 
		$$\wh P^{\ast} = \wh P \cup \wh K^{\ast s}_1 \cup \wh K^{\ast s}_2 \cup \wh K^{\ast u}_1 \cup \wh K^{\ast u}_2,$$
		and call it the \textit{extended heteroclinic pseudo-rectangle.} 
		
		Now, using the result and notation from Proposition \ref{prop:ExistenceofCFS} we can take two CFS $\wh A^{\rar}_k, \wh A^{\rar}_{k'} \in \wh {\mathcal{A}}^{\rar}$, such that 
		\begin{equation}\label{equation:HPRInTheMiddleOfStrips}
			\wh P^{\ast}  \subset \BB(\wh A^{\rar}_k), \ \wh P^{\ast} \subset \TT(\wh A^{\rar}_{k'}).	
		\end{equation}

		Now, take a sufficiently large negative integer $j_- < 0$, such that	
		\begin{equation}
			\wh g^{j_-}(\wh z^u_1) \subset \BB(\wh A^{\rar}_{k'}), \ \wh g^{j_-}(\wh z^u_2) \subset \TT(\wh A^{\rar}_{k}),
		\end{equation}
		\begin{equation}
			\text{d}(\wh g^{j_-}(\wh z^u_1), \wh A^{\rar}_{k'}) > L_{\overline{d}}, \ \text{d}(\wh g^{j_-}(\wh z^u_2), \wh A^{\rar}_{k}) > L_{\overline{d}}, 
		\end{equation}
		where $L_{\overline{d}}$ is given by Corollary \ref{corollary:SmallStableContinuaCantGrowMuch}. This implies that 
		\begin{equation}\label{eq:ContinuaFarAway}
			\wh g^{j_-}(\wh K^u_1) \subset \wh g^{j_-}(\wh K^{\ast u}_1) \subset \BB(\wh A^{\rar}_{k'}), \ \  \wh g^{j_-}(\wh K^u_2) \subset \wh g^{j_-}(\wh K^{\ast u}_2) \subset \TT(\wh A^{\rar}_{k}).
		\end{equation}
		
		\medskip
		
		Take a natural set of lifts $\tl P^{\ast}, \ {\color{black}\tl K^{{\ast} s}_1}, \ \tl K^{{\ast} s}_2, \ \tl K^{{\ast} u}_1, \ \tl K^{{\ast}u}_2, \ \tl A^{\rar}_k, \ \tl A^{\rar}_{k'}$ to $\tl \D$, respectively of $\wh P^{\ast}, \ {\color{black}\wh K^{{\ast} s}_1, \ \wh K^{{\ast} s}_2,} \ \wh K^{{\ast} u}_1, \ \wh K^{{\ast}u}_2, \ \wh A^{\rar}_k$ and $\wh A^{\rar}_{k'}$.  
		
		Given that the \textit{past} transverse path by the isotopy of any point, only crosses leaves from left to right, and by {\color{black}Equations \ref{equation:HPRInTheMiddleOfStrips} and \ref{eq:ContinuaFarAway}}, we obtain that there exist two leaves 
		$$\tl \phi^s_{\TT} \subset \partial_{\TT}\tl A^{\rar}_k, \ \tl \phi^s_{\BB} \subset \partial_{\BB}\tl A^{\rar}_{k'},$$ 
		such that
		\begin{equation}\label{equation:ChaoticRectangleAnchoredPast}
			\tl g^{j_-}(\tl K^{u}_1) \subset \BB_{\tl \phi^s_{\BB}} (\tl A^{\rar}_{k'}) \subset \BB(\tl A^{\rar}_k), \ \tl g^{j_-}(\tl K^{u}_2) \subset \TT_{\tl \phi^s_{\TT}} (\tl A^{\rar}_{k}) \subset \mathrm{T}(\tl A^{\rar}_{k'}),
		\end{equation}
		\begin{equation}
			\tl A^{\rar}_k \cup \tl A^{\rar}_{k'} \subset \LL(\tl \phi^s_{\BB}), \ \tl A^{\rar}_k \cup \tl A^{\rar}_{k'} \subset \LL(\tl \phi^s_{\TT}).
		\end{equation}

		This conveniently implies that any continuum intersecting both $\tl \phi^s_{\BB}$ and $\tl \phi^s_{\TT}$, will be automatically \textit{s}-anchored to $\tl A^{\rar}_k$ and $\tl A^{\rar}_{k'}$ by those same two leaves. See Figure \ref{figure:AnchorB} for details. 
		
		\begin{figure}[h]
			\centering
			
			\def\svgwidth{.85\textwidth}
			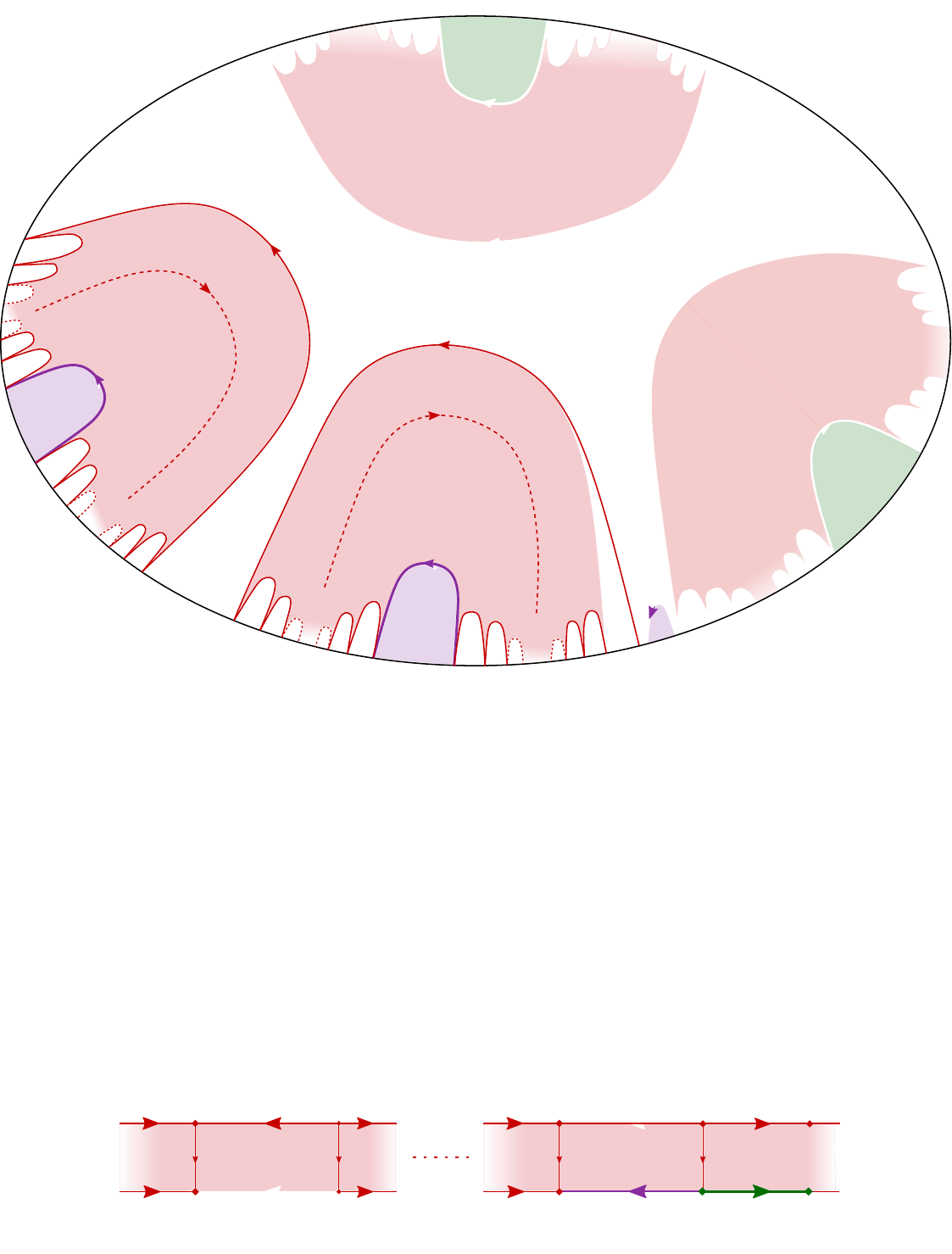
			
			\bigskip
			
			\caption{Stable sides are solid light blue, unstable sides are solid pink. Dotted lines complete the extended stable and unstable sides. Note that any continuum $\tilde{Q}^{s} \subset \tl g^{j_-}(\tl P)$ from $\tl g^{j_-}(\tl K^{u}_1) \text{ to } g^{j_-}(\tl K^{u}_2)$ {\color{black}satisfies} $(\tilde{A}^{\rar}_k, \tilde{Q}^s, \tilde{\phi}^s_{\BB}, \tilde{\phi}^s_{\TT})$ is an \textit{s}-anchor.}
			\label{figure:AnchorB}
		\end{figure}
		
		Now, note that from Equation \ref{equation:ChaoticRectangleAnchoredPast}, and the heteroclinic pseudo-rectangle definition, we obtain the following key partial result:
		\begin{equation}\label{equation:ChaoticRectangleKeyPartialQs}
			\text{Any } \tl Q^s \subset \tl g^{j_-}(\tl P) \text{ from } \tl g^{j_-}(\tl K^{u}_1) \text{ to } g^{j_-}(\tl K^{u}_2), \text{ must intersect } \tl \phi^s_{\BB} \text{ and } \tl \phi^s_{\TT}. 
		\end{equation}
		
		Note that the value $j_- = j_-(k,k') < 0$ exists for every pair $(k,k')$ satisfying Equation \ref{equation:HPRInTheMiddleOfStrips}. Be aware that we will use this fact in the next step, where we potentially retake the value of $k$. 
		
		\medskip
		
		\paragraph{\textbf{Step 3. Anchoring of the unstable sides to the same canonically foliated strip.}} 
		$\hbox{}$ \newline
		
		The rest of the proof is similar to the ending of Lemma \ref{lemma:WeaklyStableUnstableMixing}. 
		
		Nevertheless, we need one more key observation: there exists $M > 0$ such that if $\textnormal{d}(\wh P^{\ast}, \wh A^{\rar}_k) > M,$ then 
		\begin{equation}\label{equation:DisjointsSetsLeaves}
			\text{the sets of leaves} \ \wh {\Phi}^s_1, \ \wh {\Phi}^s_2 \ \text{ and } \  \wh {\Phi}^u_2 \ \text{are disjoint,}
		\end{equation} 
		where 
		$$\wh {\Phi}^s_1 = \{\wh \phi \subset \partial_{\BB} \wh A^{\rar}_k : \ \wh \phi \cap \wh I^{+}_{\wh {\mathcal{F}}}(\wh z^s_1) \neq \varnothing\},$$ 
		$$\wh {\Phi}^s_2 = \{\wh \phi \subset \partial_{\BB} \wh A^{\rar}_k : \ \wh \phi \cap \wh I^{+}_{\wh {\mathcal{F}}}(\wh z^s_2) \neq \varnothing\},$$
		$$\wh {\Phi}^u_2 = \{\wh \phi \subset \partial_{\BB} \wh A^{\rar}_k : \ \wh \phi \cap \wh I^{-}_{\wh {\mathcal{F}}}(\wh z^u_2) \neq \varnothing\}.$$ 
		
		This is due to the following facts put together:
		
		\begin{itemize}
			\item The diameter of the leaves is uniformly bounded,
			\item Fundamental domains from canonically foliated strips are also bounded
			\item The slopes of the rotation vectors of the periodic points in the extended stable sides are different, so the $x-coordinate$ of the intersection of the two transverse paths with horizontal lines separate linearly.
		\end{itemize}
		
		Take a value of $k$ satisfying Equation \ref{equation:DisjointsSetsLeaves}, and take $j_{+} = j_{+}(k) > 0$ such that

		\begin{equation}
			\wh g^{j_+}(\wh z^s_1) \subset \TT(\wh A^{\rar}_{k}), \ \wh g^{j_+}(\wh z^s_2) \subset \TT(\wh A^{\rar}_{k})
		\end{equation}
		\begin{equation}
			\text{d}(\wh g^{j_+}(\wh z^s_1), \wh A^{\rar}_{k}) > L_{\overline{d}}, \ \text{d}(\wh g^{j_+}(\wh z^s_2), \wh A^{\rar}_{k}) > L_{\overline{d}}, 
		\end{equation}
		
		where $L_{\overline{d}}$ is once again taken from Corollary \ref{corollary:SmallStableContinuaCantGrowMuch}. As in Step 2, this implies that 
		
		\begin{equation}\label{eq:StableContinuaFarAway}
			\wh g^{j_+}(\wh K^s_1) \subset \wh g^{j_+}(\wh K^{\ast s}_1) \subset \TT(\wh A^{\rar}_{k}), \ \  \wh g^{j_+}(\wh K^s_2) \subset \wh g^{j_+}(\wh K^{\ast s}_2) \subset \TT(\wh A^{\rar}_{k}).
		\end{equation}

		Recall that, from {\color{black}Lemma \ref{lem:StableDoesNotGrow}}, the future orbit of the extended stable sides will be at a uniformly bounded distance from their respective periodic points. This, together with {\color{black}Equations \ref{equation:DisjointsSetsLeaves} and \ref{eq:StableContinuaFarAway}},  implies that the future orbits of $\tl K^s_1$ and $\tl K^s_2$ will only go through different lifts of $\wh A^{\rar}_k$ in $\tl \D$, let us call them respectively $W^{'-1} \tl A^{\rar}_k$, and $W^{-1} \tl A^{\rar}_k$, {\color{black}in particular, this shows that there exist two leaves}
		
		
		\begin{equation}
			\tl \phi^{u}_{\BB} \subset \partial_{\TT} W^{'-1}\tl A^{\rar}_k \subset \BB(W^{-1} \tl A^{\rar}_{k}), \ \ \ \tl \phi^{u}_{\TT} \subset \partial_{\TT} W^{-1}\tl A^{\rar}_k,
		\end{equation}
		such that
		\begin{equation}\label{equation:ChaoticRectangleAnchoredFuture}
			\tl g^{j_+}(\tl K^{s}_2) \subset \LL(\tl \phi^{u}_{\BB}) \subset \BB(W^{-1} \tl A^{\rar}_k), \ \tl g^{j_+}(\tl K^{s}_1) \subset \LL(\tl \phi^{u}_{\TT}) = \TT_{\tl \phi^{u}_{\TT}} (W^{-1}\tl A^{\rar}_{k}).
		\end{equation}
		
		Once again, see Figure \ref{figure:AnchorB} for details. It can be helpful to compare it with Figure \ref{figure:AnchorTypeA}, from Lemma \ref{lemma:AnchorTypeA}. 
		
		Now, by retaking the lifts $\tl P' = W \tl P$, and therefore taking the corresponding lifts $\tl g^{j_+} (\tl P') = \tl g^{j_+} (W\tl P), \  \tl \phi^{'u}_{\BB} = W \tl \phi^{u}_{\BB}, \ \tl \phi^{'u}_{\TT} = W \tl \phi^{u}_{\TT},$ {\color{black}$\tl K^{'s}_1 = W \tl K^{s}_1$, $\tl K^{'s}_2 = W \tl K^{s}_2$} we obtain analogous results to the ones in Equation \ref{equation:ChaoticRectangleAnchoredFuture}, more precisely
		\begin{equation}
			\tl g^{j_+}(\tl K^{'s}_1) \subset \LL(\tl \phi^{'u}_{\BB}) \subset \BB(\tl A^{\rar}_k), \ \tl g^{j_+}(\tl K^{'s}_2) \subset \LL(\tl \phi^{'u}_{\TT}) = \TT_{\tl \phi^{u}_{\TT}} (\tl A^{\rar}_{k}).
		\end{equation}
		
		Moreover, {\color{black}since $\LL(\tl \phi^{s}_{\TT})$ and $\LL(\tl \phi^{s}_{\BB})$ are positively $f$-invariant,} we also have that 
		\begin{equation}
			\bigl( \tl \phi^{'u}_{\BB} \cup \tl g^{j_+}(\tl P') \cup \tl \phi^{'u}_{\TT} \bigr) \subset \LL(\tl \phi^{s}_{\BB}), \ \bigl( \tl \phi^{'u}_{\BB} \cup \tl g^{j_+}(\tl P') \cup \tl \phi^{'u}_{\TT} \bigr) \subset \LL(\tl \phi^{s}_{\TT})
		\end{equation}, 
		which in turn implies that 
		\begin{equation}\label{eq:OnTheLeft}
			\tl g^{j_+}(\tl P') \subset \LL(\tl \phi^{s}_{\BB}),  \  \tl g^{j_+}(\tl P') \subset \LL(\tl \phi^{s}_{\TT}), \text{ for every } j \geq j_+
		\end{equation}
		provided $k$ is a sufficiently large negative integer. It may help the heuristics of the proof to see that the configuration we have just obtained is virtually the same as the one obtained in Figure \ref{figure:WeaklyStableUnstableMixing}, from Lemma \ref{lemma:WeaklyStableUnstableMixing}.
		
		Note that we immediately obtain an analogous of Equation \ref{equation:ChaoticRectangleKeyPartialQs}, that is
		\begin{equation}\label{equation:ChaoticRectangleKeyPartialQu}
			\text{Any connection } \tl Q^u \subset \tl g^{j_+}(\tl P'), \text{ must intersect } \tl \phi^{'u}_{\BB} \text{ and } \tl \phi^{'u}_{\TT}.
		\end{equation}
		
		See Figure \ref{figure:TransverseCrossing} for details. 
		
		\medskip
		
		\paragraph{\textbf{Step 4. Finding the transverse crossing.} }
		$\hbox{}$ \newline
		
		The remainder of the proof is very similar to the one in Lemma \ref{lemma:WeaklyStableUnstableMixing}.
		
		Take any connection $\tl Q^{u} \subset \tl g^{j_+}(\tl P')$ (i.e. a continuum connecting the {\color{black}images of the} stable sides), and any continuum $\tl Q^s \subset \tl g^{j_-}(\tl P)$ connecting the respective unstable sides. Take a minimal-for-inclusion subcontinuum $\tl Q^{'s}$ intersecting both $\tl \phi^s_{\BB}$ and $\tl \phi^s_{\TT}$. Now, as in the Anchoring Lemma, take $$\tl z^s_{\BB} \in \tl Q^{'s} \cap \tl \phi ^s_{\BB}, \  \tl z^s_{\TT} \in \tl Q^{'s} \cap \tl \phi ^s_{\TT},$$ and define 
		$$ \Gamma _s = \tl \phi^+_{\tl z^s_{\BB}} \cup \tl Q^{'s} \cup \tl \phi^-_{\tl z^s_{\TT}}.$$
		
		We may then take a new leaf $\tl \phi_1 \subset \tl A^{\rar}_k$ such that 
		$$ \Gamma_s \subset \RR(\tl \phi_1),$$ 
		which in particular means that $\Gamma_s$ separates $\tl \phi^{'u}_{\BB}$ and $\tl \phi^{'u}_{\TT}$, from $\tl \phi_1$. Note that 
		\begin{equation}
			(\tl A^{\rar}_k, \tl Q^{u}, \tl \phi^{'u}_{\BB}, \tl \phi^{'u}_{\TT}) \text{ is a \textit{u}-anchor,}	
		\end{equation}
		from where we obtain that there exists $j_1 > 0$ such that
		{\color{black}
		\begin{equation}
			\tl g^{j} (\tl Q^{u}) \cap \tl \phi_1 \neq \varnothing, \text{ for every } j > j_1.	
		\end{equation}
		}
		This, together with Equation \ref{eq:OnTheLeft} and the fact that $\tl Q^{s} \subset \RR(\tl \phi_1)$, lets us conclude that
		{\color{black} 
		\begin{equation}
			\tl g^{j} (\tl Q^{u}) \cap \tl Q^s \neq \varnothing, \text{ for every } j \geq j_1.
		\end{equation}
		}
		
		\medskip
		
		Let us write $\wh P_1 = \wh P$. To finish proving the Transverse Crossing Property, simply take $m-1$ horizontal translates $\wh P_2, \hdots , \wh P_{m}$ of $\wh P$, where $\wh P_l = \wh P + (l-1,0)$, and note that taking the natural lifts $\tl P_2, \hdots, \tl P_{m}$ such that $ \tl g^{j_-}(\tl P_2), \hdots , \tl g^{j_-}(\tl P_{m})$ are \textit{s}-anchored to $\tl A^{\rar}_k$, we may take a new leaf $\tl \phi_m \subset \tl A^{\rar}_k$, which is \textit{much further to the visual right}, such that 
		$$ \tl g^{j_-} (\tl P_l) \subset \RR (\tl \phi_m), \text{ for every } l \text{ with } 1 \leq l \leq m,$$
		and conclude that there exists $j_{m} > 0$ such that 
		{\color{black}
		$$ \tl g^{j}(\tl Q^{u}) \cap \tl Q^s_l \neq \varnothing, \text{ for every } j \geq j_m, \ 0 \leq l \leq m$$
		}
		where $\tl Q^s_l \subset \tl g^{j_-}(\tl P_l)$ {\color{black}is any continuum connecting the unstable sides of $\tl g^{j_-}(\tl P_l)$}, (See Figure \ref{figure:TransverseCrossing} for details). This proves the Transverse Crossing Property for every $j \geq j_+ - j_- + j_{m}$. 
		
		Note that we can take $\varepsilon$ in Lemma \ref{lemma:RectangleNearPseudoRectangle}, such that the new sides $\tl R^s_1,\tl  R^s_2,\tl  R^{u}_1,\tl  R^{u}_2$ of the rectangle $\tl R$ still get anchored in the same way as the sides $\tl K^s_1,\tl  K^s_2,\tl  K^{u}_1,\tl  K^{u}_2$ of the HPR $\tl P$ (we use the uniform continuity of $\wh g$ and the fact that we have to control distances for a finite number of iterates of $\wh g$). This concludes the proof. 
		
		\begin{figure}[h]
			\centering
			
			\def\svgwidth{.92\textwidth}
			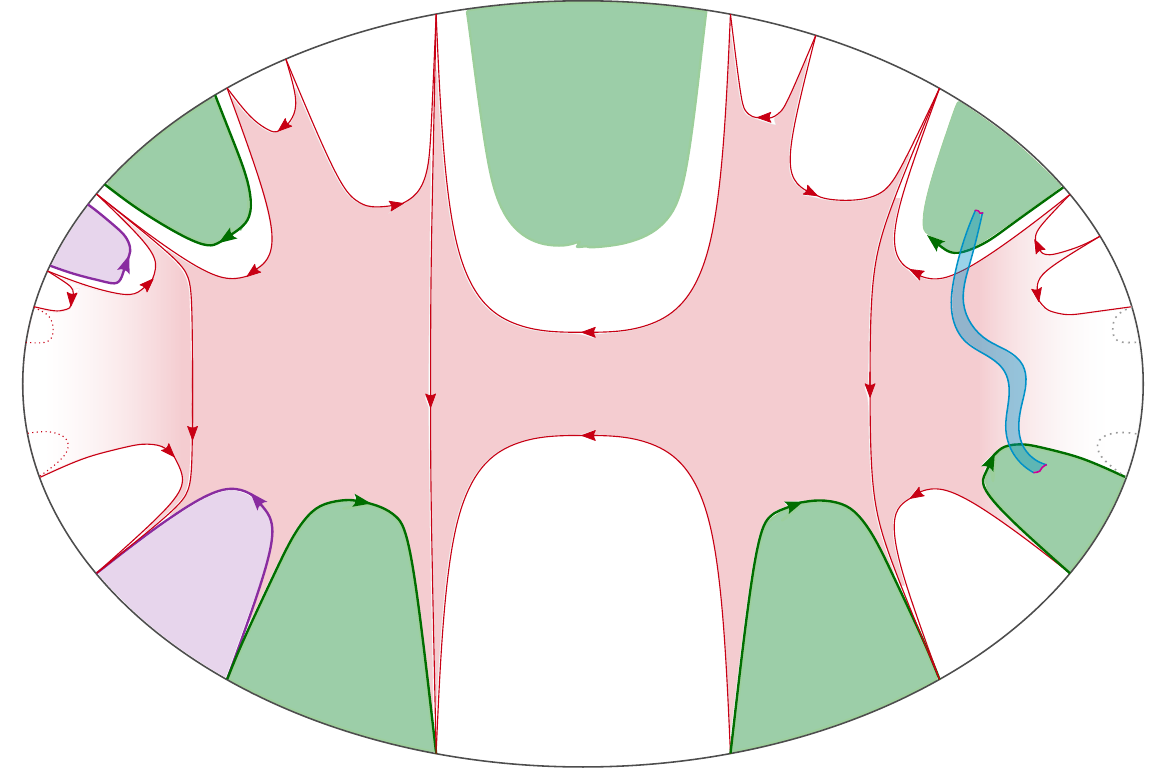
			
			\bigskip

			\caption{Once again, stable sides are light blue, unstable sides are pink. The unstable sides of $\tl g^{j_-}(\tl P)$ and its translates, are included in the green regions, the stable sides of $\tl g^{j_+}(\tl P')$ are included in the purple regions. The future orbit of $\tl g^{j_+}(\tl P')$ can not intersect any of the green leaves.}
			\label{figure:TransverseCrossing}
		\end{figure}

	\end{proof}
	
	Recall that, from the Kennedy-Yorke definition of topological horseshoe, we obtain positive topological entropy for $g|_P$, but we do not automatically recover the existence of periodic points inside $P$.

	\subsection{Markovian horseshoes}

	The key result from this subsection comes next. It is heavily based on the techniques displayed at Proposition \ref{prop:HeteroclinicPseudoRectangleHasTopHorseshoe}
	
	\begin{proposition}\label{prop:DenseRotationalHorseshoe}
		Let $\wh P \subset \R^2$ be a fitted heteroclinic pseudo-rectangle and let $m > 0$ be an integer. Then, there exists a rectangle $\wh R \subset \wh P$ which contains a Markovian horseshoe in $m$ symbols for $\wh g$.   
	\end{proposition}
	

	Once again we will make good use of the anchoring techniques. For the sake of clarity, we will divide the proof into steps. 
	
	\begin{proof}[Proof of Proposition \ref{prop:DenseRotationalHorseshoe}]
		$\hbox{}$ \newline
		
		\paragraph{\textbf{Step 1. Constructing an Initial Configuration. }} $\hbox{}$ \newline 
		
		Fix a lift $\tl P$ of $\wh P$, and take the lifts $\tl K^s_1, \tl K^s_2, \tl K^{u}_1, \tl K^{u}_2$ of its sides. Start by recovering a partial configuration we obtained in Proposition \ref{prop:HeteroclinicPseudoRectangleHasTopHorseshoe}, that is, take $$\tl A^{\rar}_k, \ \tl A^{\rar}_{k'}, \ \tl \phi^s_{\BB}, \ \tl \phi^s_{\TT}, \ \tl \phi^u_{\BB}, \ \tl \phi^u_{\TT}, \ j_- < 0, \ j_+ > 0, W,$$ such that they {\color{black}satisfy} every result from Equation \ref{equation:ChaoticRectangleAnchoredPast} to Equation \ref{equation:ChaoticRectangleAnchoredFuture}. See Figure \ref{figure:AnchorC} for details. 
		
		The key idea is to use Proposition \ref{prop:AnchoredEverywhere}. From its statement, let us take $\wh g^{j_+}(\wh P) = \tl \pi (\tl g^{j_+}(\tl P))$ as an unstable continuum, $\wh A^{\rar}_k = \tl \pi (\tl A^{\rar}_k)$ as our CFS, and a leaf $\wh \phi \subset \wh A^{\rar}_k$ such that $\wh g^{j_-}(\wh P) \cap \wh A^{\rar}_k \subset \LL_{\wh A^{\rar}_k}(\wh \phi)$. Then, Proposition \ref{prop:AnchoredEverywhere} allows us, by looking at the future iterates of $\tl g^{j_+}(\tl P)$ by $\tl g$, to find a new lift $W_1 \tl P$ by a deck transformation $W_1$ -think of $W_1$ as the deck transformation associated to a large loop going through four CFS, one in each of the four families in $\wh{\mathcal{A}}$, {\color{black}as in Lemma \ref{lem:TheRoute} (see also Figure \ref{figure:TheRoute})}, some $j_1 > 0$ and three leaves $$\tl \phi \subset W_1 \tl A^{\rar}_k, \ \tl \phi^{'u}_{\BB} \subset \partial_{\BB} W_1 \tl A^{\rar}_k, \  \tl \phi^{'u}_{\TT} \subset  \partial_{\TT} W_1 \tl A^{\rar}_k,$$
		(where $\tl \phi$ is the lift of $\wh \phi$ which belongs to $W_1 \tl A^{\rar}_k$, and $\tl \phi^{'u}_{\BB}, \tl \phi^{'u}_{\TT}$ are respectively the \textit{anchoring} leaves which appear in the statement of Proposition \ref{prop:AnchoredEverywhere} as $\tl \phi_{\BB}$ and $\tl \phi_{\TT}$), such that the following five equations hold:
		
		\begin{equation}\label{equation:InitialConfiguration1}
			W_1 \tl A^{\rar}_k \subset \RR(\tl \phi^{'u}_{\BB}) \cap \RR(\tl \phi^{'u}_{\TT}), \  W_1 \tl A^{\rar}_k \subset \LL(W_1 \tl \phi^{s}_{\BB}) \cap \LL(W_1 \tl \phi^{s}_{\TT}), 
		\end{equation}
		\begin{equation}\label{equation:InitialConfiguration2}
			\tl g^{j_-}(W_1 \tl P) \subset \LL(\tl \phi), \  \tl g^{j_+ + j_1}(\tl P) \subset \RR(\tl \phi), 
		\end{equation}  
		\begin{equation}\label{equation:InitialConfiguration3}
			\tl g^{j_-}(W_1 \tl K^{u}_1) \subset \RR(W_1 \tl \phi^{s}_{\BB}), \ \tl g^{j_-}(W_1 \tl K^{u}_2) \subset \RR(W_1 \tl \phi^{s}_{\TT}),
		\end{equation}
		\begin{equation}\label{equation:InitialConfiguration4}
			\tl g^{j_+ + j_1}(\tl K^{s}_1) \subset \LL (\tl \phi^{u}_{\BB}) \subset \TT(W_1 \tl A^{\rar}_k), \ \tl g^{j_+ + j_1}(\tl K^{s}_2) \subset \LL (\tl \phi^{u}_{\TT}) \subset \TT(W_1 \tl A^{\rar}_k),  
		\end{equation}
		\begin{equation}\label{equation:InitialConfiguration5}
			\text{For any continuum } \tl Q^u \subset \tl P \text{ from } \tl K^s_1 \text{ to } \tl K^s_2, \text{ we have that}  
		\end{equation}
		$$ (W_1 \tl A^{\rar}_k, \tl g^{j_+ + j_1}(\tl Q^{u}), \tl \phi^{'u}_{\BB}, \tl \phi^{'u}_{\TT}) \text{ is a \textit{u}-anchor.} $$
		
		To check Equations \ref{equation:InitialConfiguration1} and \ref{equation:InitialConfiguration2}, think of $W_1$ as the covering transformation associated to a loop in $\R^2$ as in Proposition \ref{prop:AnchoredEverywhere}: once we project to $\R^2$ and get that $\wh g^{j}(\wh P)$ is \textit{u}-anchored to $\wh A^{\rar}_k$, we use four CFS, one in each of our four directions and make the anchored continuum \textit{turn to the right} each time it enters a new CFS, until it enters again $\wh A^{\rar}_k$ coming from $\TT(\wh A^{\rar}_k)$. This is thoroughly described in Lemma \ref{lem:TheRoute} and Proposition \ref{prop:AnchoredEverywhere}.
		
		{\color{black}Equations \ref{equation:InitialConfiguration3} and \ref{equation:InitialConfiguration4}} come respectively as a consequence from Equations \ref{equation:ChaoticRectangleKeyPartialQs} and \ref{equation:ChaoticRectangleAnchoredFuture}, both appearing in Proposition \ref{prop:HeteroclinicPseudoRectangleHasTopHorseshoe}.  
		
		For {\color{black}Equation \ref{equation:InitialConfiguration5}}, we are applying Proposition \ref{prop:AnchoredEverywhere} to the continuum $\tl Q^{u}$, and using the fact that $\tl g^{j_{+}}(\tl Q^{u})$ is \textit{u}-anchored to $W^{-1} \tl A^{\rar}_k$. See Figures \ref{figure:AnchorC} and \ref{figure:InitialConfiguration} for details. 
		
		\begin{figure}[h]
			\centering
			
			\def\svgwidth{.92\textwidth}
			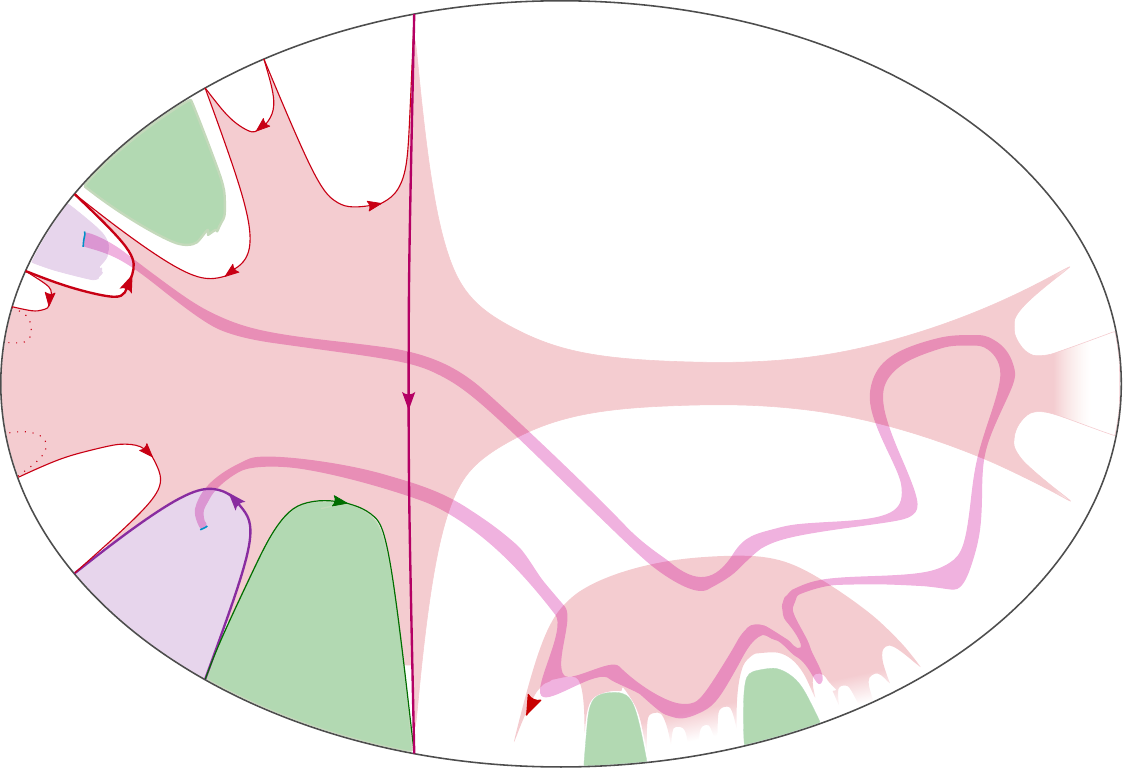
			
			\bigskip

			\caption{We recover the configuration and notation from Proposition \ref{prop:HeteroclinicPseudoRectangleHasTopHorseshoe}, use the Total Anchoring from Proposition \ref{prop:AnchoredEverywhere} to anchor a future iterate of $\tilde{g}^{j_+}(\tilde{P})$ to the CFS $W_1 \tilde{A}^{\rar}_k$. the deck transformation $W$ is taken as in Proposition \ref{prop:HeteroclinicPseudoRectangleHasTopHorseshoe}. Compare this with Figure \ref{figure:AnchorB}}.
			\label{figure:AnchorC}
		\end{figure}
		
		\bigskip

		For the sake of simplicity, let us rename $\tl O \coloneqq \tl g^{j_-}(W_1 \tl P),$ and  $\tl O' \coloneqq \tl g^{j_+ + j_1}(\tl P)$ for the remainder of the proof. Then, provided 
		\begin{enumerate}
			\item the statements from Equation \ref{equation:InitialConfiguration1} to Equation  \ref{equation:InitialConfiguration5} are held, 
			\item $\tl O' = \tl g^{l}(W_1^{-1} \tl O)$, where $l > 0$, and $W_1$ is a deck transformation,   
		\end{enumerate}
		we will say that $(\tl O, \tl O')$ is a pair in the \textit{Initial Configuration} relative to $W_1$, and with $\tl O$ \textit{on the left} of $\tl O'$. Note that any set $\tl Q \subset \tl O$ has its natural copy $\tl Q' \coloneqq \tl g^{l}(W^{-1} \tl Q) \subset \tl O'$, and vice versa.   
		
		Let us also rename $\tl \phi_{\BB} \coloneqq W_1 \tl \phi^{s}_{\BB}, \  \tl \phi_{\TT} \coloneqq  W_1 \tl \phi^{s}_{\TT}$, and $\tl A \coloneqq W _1 \tl A^{\rar}_k$. See Figure \ref{figure:InitialConfiguration}.  
		\medskip
		
		\begin{figure}[h]
			\centering
			
			\def\svgwidth{.92\textwidth}
			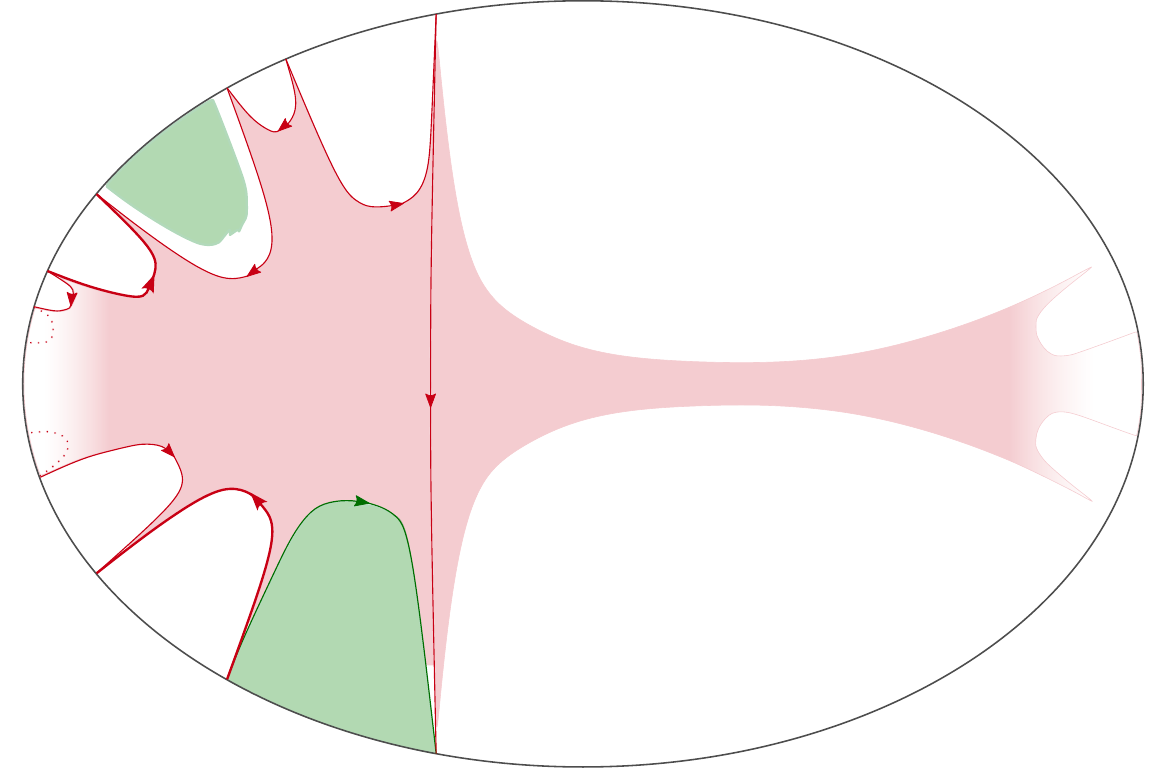
			
			\bigskip

			\caption{An Initial Configuration, with $\tilde{O}' = \tl g^l (W_1^{-1}\tl O)$. The leaf $\tl \phi$ separates $\tilde{O}$ from $\tilde{O}'$, and the leaf $\tl \phi '$ helps us find the desired Markovian intersection.}
			\label{figure:InitialConfiguration}
		\end{figure}
		
		\paragraph{\textbf{Step 2. Finding Markovian intersections.} } $\hbox{}$ \newline
		
		We will prove that given a pair $(\tl O, \tl O')$ in the Initial Configuration, there exists $j > 0$ and a rectangle $\tl{\underline{R}} \subset \tl O$ (with its natural copy $\underline{\tl R}' \subset \tl O'$) such that $\tl g^{j}(\underline {\tl R}') \cap \underline{\tl R}$ is Markovian. 
		
		First, note that by Lemma \ref{lemma:RectangleNearPseudoRectangle}, we may take a rectangle $\tl R \subset \tl O$, with horizontal sides $\tl I^s_1$, $\tl I^s_2$ sufficiently close to the stable sides $\tl K^s_1$, $\tl K^s_2$ of the pseudo-rectangle $\tl O$, and similarly vertical sides $\tl I^u_1$, $\tl I^u_2$ sufficiently close to the unstable sides $\tl K^u_1$, $\tl K^u_2$, and their corresponding copies $\tl R' \subset \tl O'$ with sides $\tl I^{'s}_1, \ \tl I^{'s}_2, \ \tl I^{'u}_1, \ \tl I^{'u}_2$, such that, (recalling Equations \ref{equation:InitialConfiguration3} and \ref{equation:InitialConfiguration4}), we obtain
		\begin{equation}\label{InitialConfigurationRectangle1}
			\tl I^{'s}_1 \subset \LL(\tl \phi^{u}_{\BB}) \subset \TT(\tl A), \  \tl I^{'s}_2 \subset \LL(\tl \phi^{u}_{\TT}) \subset \TT(\tl A), 
		\end{equation}    
		\begin{equation}\label{InitialConfigurationRectangle2}
			\tl I^{u}_1 \subset \RR(\tl \phi_{\BB}), \ \tl I^{u}_2 \subset \RR(\tl \phi_{\TT}),
		\end{equation}
		and therefore we have that
		\begin{equation}\label{equation:InitialConfigurationRectangle3}
			\text{For any continuum } \tl Q^{'u} \subset \tl R' \text{ from } \tl I^{'s}_1 \text{ to } \tl I^{'s}_2,
		\end{equation}
		$$ (\tl A, \tl Q^{'u}, \tl \phi^{'u}_{\BB}, \tl \phi^{'u}_{\TT}) \text{ is a \textit{u}-anchor.}$$
		
		Take a leaf $\tl \phi ' \subset \tl A$, such that $\tl O \subset \RR(\tl \phi')$. By the Anchoring Lemma (\ref{lemma:AnchoringLemma}), there exists $j > 0$ such that $\tl g^{j}(\tl Q^{'u}) \cap \tl \phi' \neq \varnothing$, for any such a continuum $\tl Q^{'u}$ as in Equation \ref{equation:InitialConfigurationRectangle3}. Given that $\tl R' \subset \RR (\tl \phi) \subset \LL(\tl \phi_{\BB}) \cap \LL(\tl \phi_{\TT})$, we obtain that $\tl g^{j}(\tl R') \subset \LL(\tl \phi_{\BB}) \cap \LL(\tl \phi_{\TT})$, which implies that $\tl g^{j}(\tl R')$ intersects neither $\tl \phi_{\BB}$ nor $\tl \phi_{\TT}$. Let us show that using this fact and Lemma \ref{lemma:SeparateThenConnect}, there must be a curve $\gamma^{'s} \subset \tl g^{j}(\tl R')$ parallel to the stable sides (i.e. going from one unstable side to the other, without intersecting the stable sides), such that $\gamma^{'s} \subset \LL_{\tl A}(\tl R)$. 
		
		The application of Lemma \ref{lemma:SeparateThenConnect} is as follows: take $\tl Q := \tl \phi' \cap \tl g ^{j} (\tl R')$. Now, note that $\tl Q$ is in the hypothesis of Lemma \ref{lemma:SeparateThenConnect}, as it does not intersect the stable sides and it also \textit{separates the stable sides} (i.e. it intersects every continuum going from one stable side to the other). Then by this same lemma we obtain that $\tl Q$ must have a subcontinuum $\tl Q' \subset \tl g ^{j} (\tl R')$ connecting the stable sides. As $\tl \phi'$ is a line (i.e. homeomorphic to $\R$ by a proper application) in $\tl \D$, we recover that $\tl Q'$ can be taken as a curve, which will be our desired $\gamma^{'s}$.     
		
		
		Observe that $\gamma^{'s}$ divides $\tl g^{j} (\tl R')$ into two complementary rectangles. Define $\tl g^{j}(\underline{\tl R}')$ as any of the two, say for example the one with horizontal sides $\tl g^{j}(\tl I^{'s}_1), \ \gamma^{'s}$ and for vertical sides the corresponding two subarcs $\gamma^{'u}_1, \ \gamma^{'u}_2$ of $\tl g^{j}(\tl I^{'u}_1), \tl g^{j}(\tl I^{'u}_2)$, which go from $\tl g^{j}(\tl I^{'s}_1)$ to $\gamma^{'s}$. 
		
		Finally, define $\underline{\tl R}$ as the natural copy of $\tl g^{j}(\underline{\tl R}')$ which is included in $\tl R$, 
		{\color{black} with respective horizontal sides $\tl I^s_1, \gamma^s$, and vertical sides $\gamma^{u}_1 \subset \RR (\tl \phi_{\BB})$, $\gamma^{u}_2 \subset \RR (\tl \phi_{\TT})$, which are respectively subarcs of $\tl I^{u}_1, \tl I^{u}_2$. 
		Note that $\tl \phi_{\BB}$ and $\tl \phi_{\TT}$ respectively separate each of the vertical sides of $\underline{\tl R}$ from $\tl g^{j}(\underline{\tl R}')$ (see Figure \ref{figure:InitialConfiguration}). This means we can complete both horizontal sides of $\underline{\tl R}$ to lines $\tl \eta_1$, $\tl \eta_2$, using for each side two arcs respetively contained in $\RR (\tl \phi_{\BB})$ and $\RR (\tl \phi_{\TT})$, note that there is a homeomorphism $h'$ from $\partial \underline{\tl R} \cup \tl \eta_1 \cup \tl \eta_2$ to its image in $\R^2$, satisfying 
		\[h'(\underline{\partial \tl R}) = \partial [0,1]^2, \  h'(\tl \eta_1) = \R \times \{0\}, \ h'(\tl \eta_2) = \R \times \{1\}.\] 
		
		Then, by construction of the configuration and Lemma \ref{lem:homma}, we get that 
		$$  \tl g^{j}(\underline{\tl R}') \cap \underline{\tl R} \text{ is Markovian.}$$
		}

		\medskip
		
		\paragraph{\textbf{Step 3. Iterating the process. }} $\hbox{}$ \newline
		
		The rest of the proof is inductive. We will give all the details to go from $m=1$ to $m=2$, and we will show this is enough. Recall that $\text{Dom}(\wh I)$ is an infinitely-punctured plane, and therefore the group of deck transformations of the universal covering is free. Given that every subgroup of a free group is also free, it is enough to find $l > 0$, two deck transformations $W_1$, $W_2$ with no relations whatsoever (take for example the ones induced by \textit{looping} around increasing sets of singularities), and using again Proposition \ref{prop:AnchoredEverywhere} obtain two pairs $(\tl O, \tl O')$, $(W_2 \tl O, \tl g^{l}(\tl O'))$ which are in the Initial Configuration respectively relative to $W_1$ and $W_2W_1$ in the respective canonically foliated strips $\tl A$, $\tl W_2 \tl A$, and a leaf $\tl \phi \subset \tl A$ such that
		\begin{equation}
			\tl O \subset \RR(\tl \phi), \ W_2 \tl O \subset \LL(\tl \phi).
		\end{equation}

		To finish the proof, observe that given $(W_2 \tl O, \tl g^{l}(\tl O'))$ is in the Initial Configuration, there exists $j > 0$ and a rectangle $\underline{\tl R}' \subset \tl O'$ such that 
		\begin{equation}
			\tl g^{l+j}(\underline{\tl R}') \cap W_2 \tl R \text{ is Markovian.}
		\end{equation} 
		{\color{black}By construction, we have that $W_2 \tl A \subset \LL(\tl \phi)$ (see Figure \ref{figure:IteratingTheProcess}), and by the Anchoring Lemma (same argument as in Figure \ref{figure:AnchorC}) we may take} the corresponding new stable (horizontal) side $\gamma^{'s} \subset \tl g^{l+j}(\underline{\tl R}')$ {\color{black}satisfying}
		$$\gamma^{'s} \subset \LL_{W_2 \tl A}(W_2 \tl R) \subset \mathrm{L}(\tl \phi).$$
		{\color{black} {\color{black}By construction, we know that the horizontal (stable) side $\tl I^s_1 \subset \partial \underline{\tl R}$, intersects both $\RR(\tl \phi_{\BB})$ and $\RR(\tl \phi_{\TT})$}. This means that $\RR(\tl \phi_{\BB}) \cup \underline{\tl R} \cup \RR(\tl \phi_{\TT})$ separates the stable sides of $g^{l+j}(\underline{\tl R}')$, which again by Lemma \ref{lem:homma} shows that 
		}
		\begin{equation}
			\tl g^{l+j}(\underline{\tl R}') \cap \underline{\tl R} \text{ is also Markovian,}
		\end{equation} 
		and finishes the proof for $m = 2$. Figure \ref{figure:IteratingTheProcess} illustrates this phenomenon. 
		
		\begin{figure}[h]
			\centering
			
			\def\svgwidth{.92\textwidth}
			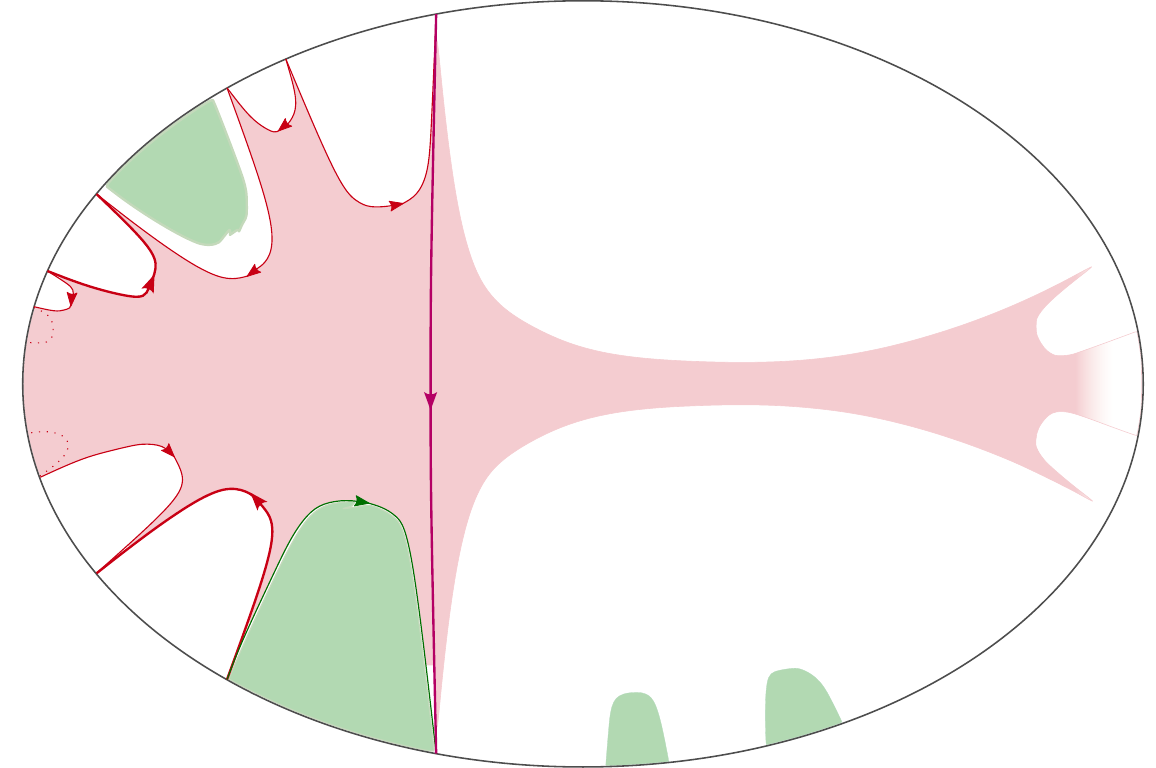
			
			\bigskip

			\caption{We use a similar strategy to the one in Proposition \ref{prop:AnchoredEverywhere}: \textit{push} the pseudo-rectangle until it enters a new copy of the same canonically foliated strip, and then anchor it there.}
			\label{figure:IteratingTheProcess}
		\end{figure}
		
		\medskip
		
		For $m > 2$, the argument is identical: take successive deck transformations $W_1, \hdots , W_m$ with no relations, and positive integers $l_2, \hdots , l_m$ such that 
		\begin{enumerate}
			\item $(\tl O, \tl O')$ is in the Initial Configuration relative to $W_1$, in the canonically foliated strip $\tl A$,  
			\item There exists a leaf $\tl \phi \subset \tl A$ such that $\tl O \subset \RR(\tl \phi), \ W_2 \tl O \subset \LL(\tl \phi)$,
			\item For every $2 \leq i \leq m$, $(\tl g^{l_i}(\tl O'), W_i \hdots W_2 \tl O)$ is in the Initial Configuration relative to $W_i \hdots W_1$, in the canonically foliated strip $W_i \hdots W_2 \tl A$, 
			\item For every $2 \leq i \leq m-1$ there exists a leaf $\tl \phi_i \subset W_i \hdots W_2 \tl A$ such that $W_i \hdots W_2 \tl O \subset \RR(\tl \phi), \ W_{i+1}W_i \hdots W_2 \tl O \subset \LL(\tl \phi)$. 
		\end{enumerate}
		
		To finish the proof, take $j > 0$ and a rectangle $\tl R'$ such that $$\tl g^{l_m+j}(\underline{\tl R}') \cap W_m \hdots W_2 (\underline{\tl R}) \text{ is Markovian,}$$
		which {\color{black}using item (4)} implies that 
		$$ \tl g^{l_m+j}(\underline{\tl R}') \cap W_i \hdots W_2 (\underline{\tl R}) \text{ is Markovian for every } 2 \leq i \leq m,$$
		and that  
		$$ \tl g^{l_m+j}(\underline{\tl R}') \cap \underline{\tl R} \text{ is also Markovian.}$$
		
		This finishes the proof.
		
	\end{proof}
	
	\medskip
	
	\begin{proof}[Proof of Proposition \ref{prop:EssentialFactorHasDenseHorseshoes}]
		By Proposition \ref{prop:ChaoticRectanglesAreDense}, we know that fitted heteroclinic pseudo-rectangles are dense, and by Proposition \ref{prop:DenseRotationalHorseshoe} we know that inside each of these there is a Markovian horseshoe, which concludes the proof. 
	\end{proof}

	As each Markovian horseshoe $\wh R \subset \R^2$ contains periodic points (at least one per periodic sequence by the two-sided shift $\sigma_m$ in $m$-symbols), we conclude that 
	
	\begin{corollary}
		If $\vec{0} \in \mathrm{int}(\rho(\wh g))$, then $\textnormal{Per}(\wh g)$ is dense in $\R^2$
	\end{corollary}
	
	which in particular implies that 
	
	\begin{corollary}
		If $\vec{0} \in \mathrm{int}(\rho(\wh g))$, then $\textnormal{Per}_0(g)$ is dense in $\mathbb{T}^2$
	\end{corollary}
	
	where $\textnormal{Per}_0(g)$ is the set of periodic points of $g$ with trivial associated rotation vector. Note that these periodic points \textit{do} turn around the singularities of the isotopy.
	
	\begin{remark}\label{rem:HypothesisForSection8}
		Apart from the General Hypothesis holding for $g$ (namely $g \in \textnormal{Homeo}_0(\T^2), \ \textnormal{int}(\rho(g)) \neq \varnothing$), and that $\vec{0} \in \mathrm{int}(\rho(\wh g))$, the properties of $g$ on which the construction of this section depend on are the ones in Section \ref{sec:EssFactor}, in particular
		\begin{itemize} 
			\item $\mathscr{D}_g(K) = \infty$ for every nontrivial continuum $K$ (which gives us the existence of stable and unstable continua for every point, which have uniformly large diameter), 
		\end{itemize}
		and the ones derived from it, as the existence of stable sets and the total anchoring of Proposition \ref{prop:AnchoredEverywhere}.
	\end{remark}
	
	This allows us to recover the following result, which was suggested by Pierre-Antoine Guihéneuf. 
	
	\begin{corollary}
		Let $v$ be a rational vector in $\mathrm{int}(\rho(\wh g))$. Then, $$\{\wh z \in \R^2 : \exists q \in \mathbb{Z}, w = qv \in \mathbb{Z}^2 \textnormal{ s.t. } \wh g^q(\wh z) = \wh z + w\}  \text{ is dense in }\R^2.$$
	\end{corollary}
	
	\begin{proof}
		Let $q$ be an integer such that $w = qv \in \mathbb{Z}^2$, and note that $w \in \textnormal{int}(\rho(\wh g^q))$. Now, take a new lift $\wh g^q_{w} = \wh g^q - w$, which is naturally isotopic to the identity as $\wh g^q$ also is. Last, note that $\wh g^q_{w}$ {\color{black}satisfies} every property from Remark \ref{rem:HypothesisForSection8}, which means that 
		\[X_v(\wh g) = \textnormal{Per}(\wh g^q_w) \text{ is dense in } \R^2,\]
		which concludes the proof.    
	\end{proof}
	
	Let us define $\textnormal{Per}_v(g) = \wh \pi (X_v(\wh g))$, where $\wh g$ is a lift of $g$, i.e. the set of periodic points of $g$ with rotation vector equal to $v$. We have then obtained that 
	
	\begin{corollary}
		For every rational vector $v \in \textnormal{int}(\rho(\wh g))$, we have that $\textnormal{Per}_v(g)$ is dense. 
	\end{corollary}

	\section{Proof of Theorem A}\label{sec:Conservative}

	\subsection{The essential factor is conservative}
	
	We will prove that, up to a continuous change of coordinates, the essential factor of a torus homeomorphism $f$ in the General Hypothesis preserves the Lebesgue measure. More precisely,
	
	\begin{proposition}\label{prop:AreaPreservingConjugate}
		Let $f : \T^2 \to \T^2$ be an isotopic-to-the-identity homeomorphism, such that $\rho(f)$ has nonempty interior, and let $g: \T^2 \to \T^2$ be its essential factor (see Definition \ref{def:EssentialFactor} and Remark \ref{rem:TorusModel}). 
		Then, there exists a homeomorphism $h: \T^2 \to \T^2$ such that $hgh^{-1}$ preserves the Lebesgue measure on $\T^2$. 
	\end{proposition} 
	
	\begin{proof}
		The whole idea of the proof is to find an atomless measure $\mu$ with total support, which is invariant under $g$. Once we have done that, we know by Oxtoby-Ulam's Theorem (see {\color{black} \cite{oxtobyulam41} for the original proof, \cite[Section 2]{leroux12} for a more recent one}), that the desired change of coordinates exists. We will strongly use the density of Markovian horseshoes for $g$. 
		
		Start by taking a countable base $\ \mathcal{U} = \{U_n\}_{n \in \mathbb{Z}^+}$ for the usual topology in $\T^2$ (i.e. balls of rational radius centered at rational points), and fix $n > 0$. By Proposition \ref{prop:ChaoticRectanglesAreDense}, we know that there exists a heteroclinic pseudo-rectangle $P$ included in $U_n$, and hence by Proposition \ref{prop:HeteroclinicPseudoRectangleHasTopHorseshoe} we know that there exists a compact subset $\Lambda \subset P$ which is invariant under some power $g^k$ of $g$, and such that it is an extension of the Bernoulli shift $\sigma$ in $2$ symbols by the surjection $\pi$, that is 
		
		\bigskip  
		
		\tikzset{every picture/.style={line width=0.75pt}} 
		
		\begin{tikzpicture}[x=0.75pt,y=0.75pt,yscale=-1,xscale=1]
			
			\draw    (273.07,38.23) -- (388.5,38.23) ;
			\draw [shift={(390.5,38.23)}, rotate = 180] [color={rgb, 255:red, 0; green, 0; blue, 0 }  ][line width=0.75]    (10.93,-3.29) .. controls (6.95,-1.4) and (3.31,-0.3) .. (0,0) .. controls (3.31,0.3) and (6.95,1.4) .. (10.93,3.29)   ;
			\draw    (273.67,121.19) -- (389.1,121.19) ;
			\draw [shift={(391.1,121.19)}, rotate = 180] [color={rgb, 255:red, 0; green, 0; blue, 0 }  ][line width=0.75]    (10.93,-3.29) .. controls (6.95,-1.4) and (3.31,-0.3) .. (0,0) .. controls (3.31,0.3) and (6.95,1.4) .. (10.93,3.29)   ;
			\draw    (253.83,50.25) -- (253.83,106.57) ;
			\draw [shift={(253.83,108.57)}, rotate = 270] [color={rgb, 255:red, 0; green, 0; blue, 0 }  ][line width=0.75]    (10.93,-3.29) .. controls (6.95,-1.4) and (3.31,-0.3) .. (0,0) .. controls (3.31,0.3) and (6.95,1.4) .. (10.93,3.29)   ;
			\draw    (412.54,50.86) -- (412.54,107.17) ;
			\draw [shift={(412.54,109.17)}, rotate = 270] [color={rgb, 255:red, 0; green, 0; blue, 0 }  ][line width=0.75]    (10.93,-3.29) .. controls (6.95,-1.4) and (3.31,-0.3) .. (0,0) .. controls (3.31,0.3) and (6.95,1.4) .. (10.93,3.29)   ;
			
			\draw (324.18,17) node [anchor=north west][inner sep=0.75pt]    {$g^k$};
			\draw (325.38,105) node [anchor=north west][inner sep=0.75pt]    {$\sigma$};
			\draw (248,29.23) node [anchor=north west][inner sep=0.75pt]    {$\Lambda$};
			\draw (403.34,29.23) node [anchor=north west][inner sep=0.75pt]    {$\Lambda$};
			\draw (248,118) node [anchor=north west][inner sep=0.75pt]    {$\Sigma_2$};
			\draw (403,118) node [anchor=north west][inner sep=0.75pt]    {$\Sigma_2$};
			\draw (82,63.67) node [anchor=north west][inner sep=0.75pt]   [align=left] { \ \ \ };
			\draw (239,69) node [anchor=north west][inner sep=0.75pt]    {$\pi$};
			\draw (397,69) node [anchor=north west][inner sep=0.75pt]    {$\pi$};
		\end{tikzpicture}
		
		where $\Sigma_2 = \{0,1\}^{\mathbb{Z}}$. Now, let us find a minimal subshift $\sigma_{\alpha}$ using a well-known technique. Choose an irrational number $0 < \alpha < 1$, and build a Sturmian sequence $\omega \in \Sigma_2$ with density of $1$'s equal to $\alpha$: the $n$-th coordinate of $\omega$ will be $1$ if and only if 
		\[\lfloor n\alpha \rfloor - \lfloor (n-1)\alpha \rfloor= 1,\]
		and it will be 0 otherwise.  
		
		We know that the closure $X$ of the orbit of $\omega$, is invariant under $\sigma$ and is also minimal for $\sigma$, from where we can obtain a minimal subshift $\sigma |_{X}$ (in particular it has no periodic orbits). 
		
		Take its preimage $\pi^{-1}(X) \subset U_n$ and look at the restriction $g^k|_{\pi^{-1}(X)}$. We know that it has no periodic orbits because its projection does not have either. By the Krylov–Bogolyubov theorem, we know there exists a $g^k$-invariant probability measure $\mu_n$ supported in $\pi^{-1}(X)$, which must have no atoms, as their existence would in turn imply the existence of periodic orbits for $g^k|_{\pi^{-1}(X)}$. Let us then define
		{\color{black} 
		\[ \mu'_n = \frac{1}{k} \sum \limits_{j=1}^{j=k} g^j_{\ast}(\mu_n).\]
		}
		Note that $\mathrm{supp}(\mu'_n) \cap U_n \neq \varnothing$ since $\mu'_n(U_n) \geq 1/n$, and that $\mu'_n$ is a $g$-invariant measure with no atoms. Let us now define the probability measure $\mu$ as 
		\[\mu = \sum \limits_{n \in \mathbb{Z}^+} \frac{\mu'_n}{2^n}.\]

		Now, note that the probability $\mu$ has the following properties: 
		\begin{itemize}
			\item It is invariant under $g$ because it is the countable sum of $g$-invariant measures, 
			\item It has total support: its support intersects every open ball because $\{U_n\}_{n \in \mathbb{Z}^+}$ is a base for the usual topology,
			\item It has no atoms because each of the countably many $\mu_n$ has no atoms.
		\end{itemize}  
		
		By Oxtoby-Ulam's Theorem, we know there exists a change of coordinates $h: \T^2 \to \T^2$ such that $h_{\ast}(\mu) = \textnormal{Leb}$, the usual Lebesgue measure for $\T^2$, which concludes the proof. 
	\end{proof}
	
	\subsection{The proof}
	
	We have finally gathered all the pieces to prove Theorem \ref{thmA:semiconjugation}. 
	
	\begin{proof}[Proof of Theorem \ref{thmA:semiconjugation}]
		Let $f: \T^2 \to \T^2$ be an isotopic-to-the-identity homeomorphism whose rotation set has nonempty interior, and let $g: \T^2 \to \T^2$ be its essential factor. Let us check that we have already obtained the required properties:
		\begin{itemize}
			\item $\rho(\wh g)=\rho(\wh f)$ for some well-chosen lifts. This comes from Proposition \ref{prop:elementprop.essfactor}.
			\item For every nontrivial continuum $K$, $\mathscr{D}_g(K) = \infty$. This is also proved in Proposition \ref{prop:elementprop.essfactor}. 
			\item $g$ is tight. This is Proposition \ref{prop:Tight}. 
			\item $g$ is rotationally mixing. This is proved in Proposition \ref{prop:essfactor.isrot.mixing}. 
			\item Density of Markovian horseshoes for $g$. This is proved in Proposition \ref{prop:EssentialFactorHasDenseHorseshoes}. We also obtain that $\overline{\textnormal{Per}(g)} = \T^2$.  
		\end{itemize}
		By Proposition \ref{prop:AreaPreservingConjugate}, we obtain that for some change of coordinates $h$, the conjugate $hgh^{-1}$ is area preserving (we can assume that $h$ is isotopic to the identity without loss of generality). Note that these first five properties are invariant under this topological conjugacy $h$. For the sake of simplicity, rename $hgh^{-1}$ as $g$, which {\color{black}satisfies} the six desired properties and thus finishes the proof.  
	\end{proof}

	\small
	\bibliographystyle{alpha}
	\bibliography{biblio}


\end{document}